\documentclass[a4paper,12pt]{report}
\usepackage{psfrag}
\usepackage{epsfig}

\usepackage{amsfonts,amssymb}
\usepackage{amsfonts,a4}
\usepackage{amsthm}
\usepackage{amsmath}

\newcommand{\tx}{\tilde{x}}

\newcommand{\pa}{\partial}

\newcommand{\ri}{{\rm i}}
\newcommand{\re}{{\rm e}}

\newcommand{\GH}{\Gamma_H}
\newcommand{\be}{\begin{eqnarray}}
\newcommand{\ben}{\begin{eqnarray*}}
\newcommand{\en}{\end{eqnarray}}
\newcommand{\enn}{\end{eqnarray*}}

\newcommand{\cF}{{\cal F}}
\newcommand{\hphi}{\hat{\phi}}
\newcommand{\hpsi}{\hat{\psi}}
\newcommand{\real}{\mathbb{R}}
\newcommand{\R}{\real}
\newcommand{\complex}{\mathbb{C}}

\newcommand{\C}{\mathbb{C}}
\newcommand{\dist}{{\mathrm{dist}}}

\newcommand{\supp}{{\mathrm{supp}}}
\newtheorem{remark}{Remark}[chapter]
\newtheorem{lemma}{Lemma}[chapter]
\newtheorem{theorem}{Theorem}[chapter]

\newcommand{\norm}[1]{\| #1\|}

\renewcommand{\Im}{{\rm Im}}
\renewcommand{\Re}{{\rm Re}}

\newcommand{\bnu}{\mbox{\boldmath $\nu$}}

\newtheorem{corollary}{Corollary}[chapter]
\newtheorem{definition}{Definition}[chapter]
\begin{document}
\renewcommand{\baselinestretch}{1.6}
%The simulation of scattering of acoustic or electromagnetic waves is of great importance for a 
\title{THE UNIVERSITY OF READING 
\newline {\sc Department of Mathematics}
\bigskip
\newline PhD Thesis}
\author{Thomas Baden-Riess
\\ 
\bigskip
A thesis submitted for the degree of Doctor of Philosophy}
\date{September 2006}
\maketitle

\begin{abstract}

In this thesis we study four problems in the area of scattering of time harmonic acoustic or electromagnetic waves by unbounded rough surfaces/unbounded inhomogeneous layers. Specifically 
the four problems we study are:
\newline i) A boundary value problem for the Helmholtz equation, in both 2 and 3 dimensions, modelling scattering of time harmonic waves due to a source that lies within a finite distance of the boundary and which decays along the boundary, by a layer of spatially varying refractive index above an unbounded rough surface on which the field vanishes. In particular, in the 2D case, the boundary value problem models the scattering of time harmonic electromagnetic waves by an inhomogeneous conducting or dielectric layer above a perfectly conducting unbounded rough surface, with the magnetic permeability a fixed positive constant in the media, in the transverse electric polarization case; 
\newline ii) a boundary value problem for the Helmholtz equation with an impedance boundary condition, in 2 and 3 dimensions, modelling the scattering of time harmonic acoustic waves due to a source that lies within a finite distance of the boundary and which decays along the boundary, by an unbounded rough impedance surface; 
\newline iii) a problem of scattering of time harmonic waves by a layer of spatially varying refractive index at the interface between semi-infinite half-spaces of fixed positive refractive index (the waves arising due to a source that lies within a finite distance of the layer and which decays along the layer). In the 2D case this models the scattering of time harmonic electromagnetic waves by an infinite inhomogeneous dielectric layer at the interface between semi-infinite homogeneous dielectric half-spaces, with the magnetic permeability a fixed positive constant in the media, in the transverse electric polarization case;
\newline iv) a boundary value problem for the Helmholtz equation with a Dirichlet boundary condition, in 3 dimensions, modelling the scattering of time harmonic acoustic waves due to a point source, by an unbounded, rough, sound soft surface.         

We study problems i), ii) and iiii) by variational methods; via analysis of equivalent variational formulations we prove these problems to be well-posed in the following cases: For i) we show that the problem is well-posed for arbitrary rough surfaces that are a finite perturbation of an infinite plane, in the case that the frequency is small or when the medium in the layer has some energy absorption; and when the rough surface is such that the resulting domain has the property that if $x$ is in the domain then so to is every point above $x$, we show the problem to be well-posed for arbitrary large frequency with certain restrictions on the rate of change of the refractive index; for ii) we show that the problem is well-posed for arbitrary rough Lipschitz surfaces that are a finite perturbation of an infinite plane, in the case that the frequency is small; and when the rough surface is the graph of a bounded Lipschitz function, we show the problem to be well-posed for arbitrary frequency; for iii) we establish that the problem is well-posed under certain restrictions on the variation of the index of refraction.

We study problem iv) via a Brakhage-Werner type integral equation formulation, based on an ansatz for the solution as a combined single- and double-layer potential, but replacing the usual fundamental solution of the Helmholtz equation with an appropriate half-space Green's function. We establish, in the case that the rough surface is the graph of a bounded Lipschitz function, that the problem is well-posed for arbitrary frequency.

An attractive feature of our results is that the bounds we derive, on the inf-sup constants of the sesquilinear forms in problems i), ii) and iii), and on the inverse operator associated with the single- and double-layer potentials in problem iv), are explicit in terms of the index of refraction, the geometry of the scatterer and the other parameters of the respective problems.        

\end{abstract}

%\textbf{Acknowledgements}

%I would like to begin by thanking the staff in the maths department at the university of East Anglia for their superb and inspiring introduction to mathematics; in particular I would like to thank my supervisor Graham Everest for his kindness to me. Also I would like to thank Geoffrey Burton and Edward Fraenkel for their help at Bath.

%At Reading I wish to thank my supervisor Simon Chandler-Wilde for being an outstanding supervisor; for treating me like an equal from the off; for having belief in me; for giving me every chance to do some top level mathematics; for always having time for me; for solving my mathematics problems when I was in a mess; for sending me off to various conferences; for making my life very easy over these last two years; for being a genial kind of chap; and for putting up with my unreasonable behaviour. 

%I wish also to thank Les Bunce for tutoring in the first year, and Markus Melenk, Bill Mclean, Peter Monk and Steve Langdon for their help with various aspects of this thesis. I acknowledge the Engineering and Physical Sciences Research Council for providing me with a research grant over the last 3 years.

%Finally I would like to thank my parents for their financial support, that allowed me to study  at both Norwich and Bath, and of course for everything else they have done for me.  

%large number of application areas ranging from medical imaging to seismic exploration.
\tableofcontents
\chapter{Introduction}\label{intro}
\section{Preamble}
Consider, if you will, the following problem: An aeroplane flying above the surface of the earth, generates some noise. This noise travels through the air striking the rough surface of the earth beneath it. The noise bounces off or scatters from the surface. Given that we know the exact nature of the noise produced by the aeroplane, and given that we know the exact shape of the earth's surface beneath it, can we predict the resultant propagation of noise as it strikes the earth's surface and scatters?

The above problem is a typical example of what are known as \textit{rough surface scattering problems.} Rough surface scattering problems arise frequently in the natural world and the study of these problems has been borne out of research in many diverse areas of science. The above example shows their importance to the science of sound propagation and noise control; on a much smaller scale, in the field of nano-technology, they are relevant in the study of the scattering of light from the surface of materials; and in the technology of solar heating, their understanding is important for the correct choice of solar paneling; in addition these problems crop up in medical imaging and seismic exploration.  

It is the overall aim then, of the mathematical and engineering community, to resolve these problems. This thesis is intended as a contribution to this subject. We are concerned primarily with the initial, mathematical and theoretical questions that should -- to a mathematician at least -- be answered, in this field, prior to the implementation of numerical and computational techniques that will simulate the process of rough surface scattering, and ultimately give answers to the problems stated above.

Thus, in what follows, we are concerned with the mathematical aspects of rough surface scattering problems. In particular we are interested in the correct mathematical formulation of these problems and we intend to analyze under which conditions they are well-posed. Specifically we study four rough surface scattering problems. Three of these we study by variational methods -- these are acoustic scattering by an impedance surface; electromagnetic scattering by inhomogeneous layers above a perfectly conducting rough surface (the Transverse electric polarization case); and the transmission problem -- and one by integral equation methods: acoustic scattering by a sound soft surface. We will shortly take a more precise look at what these problems are and look at the mathematical models that govern acoustic and electromagnetic propagation.

%In this thesis we are concerned with the mathematical theory associated with the problem of wave scattering by unbounded rough surfaces, so called \textit{rough surface scattering problems}, and also the related \textit{transmission problem.} In particular, we analyse, via variational methods and also integral equation methods, the well-posedness of the mathematical formulations of these problems.

 We wish to end this first section by introducing some nomenclature and notation used throughout and by setting the scene of our scattering problems. In accordance with the terminology of the engineering literature, we use the phrase \textit{rough surface} to denote a surface which is a (usually non-local) perturbation of an infinite plane surface, such that the whole surface lies within a finite distance of the original plane.
 
  Let $x=(x_1, \dots ,x_n)$ denote a point in $\mathbb{R}^n$, $(n=2,3)$, and let $\tilde x = (x_1, \dots ,x_{n-1})$ so that $x=(\tilde x,x_n)$. Further, for $H \in \mathbb{R} $ let $U_H= \{x:x_n>H\}$ and $\Gamma_H=\{x:x_n = H\}$. We will denote the region of space in which the acoustic or electromagnetic waves propagate, i.e. the air above the earth's surface in the aeroplane problem we described above, as $D$. Thus $D \subset \mathbb{R}^n$, %$n=2,3$,
and $D$ will be assumed to be a connected open set or \emph{domain}. Moreover we'll assume there exist constants $f_- < f_+$ such that 
\[
U_{f_+} \subset D  \subset U_{f_-}.
\]
 We let $\Gamma$ denote the boundary of $D$, i.e. the rough surface. The unit outward normal to $D$ will be denoted $\nu$. Finally, we define $S_H:= D\backslash \overline{U_H}$.
 % then it is constrained to lie between two horizontal lines or planes, one above it and one beneath it, these being defined by $x_2 =f_+$  and  $x_2 =f_-$ in 2D, and by $x_3 =f_+$  and  $x_3 =f_-$ in 3D, for some $f_+> f_-$. 
%\vspace*{-40ex}
\begin{figure}
 \psfrag{bound}{$\Gamma=\partial D$}
\psfrag{DD}{{\large $D$}}
\psfrag{GammaH}{$\Gamma_H$}
\psfrag{UH}{$U_H$}
\psfrag{SH}{$S_H$}
\psfrag{Gammaf}{$\Gamma_{f_-}$}
\psfrag{xlabel}{$\tilde x = (x_1,\dots,x_{n-1})$}
\psfrag{ylabel}{$x_n$}
    \begin{center}
    \epsfig{file=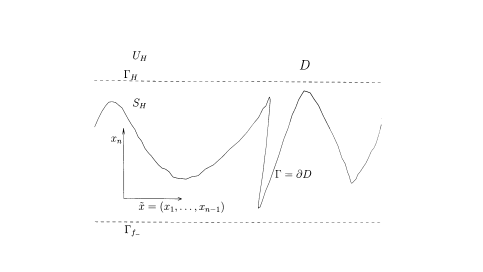}
    
    %sncw_thomas_monk_fig.eps,width=14cm}
    \end{center}

%\vspace*{-7ex}

\caption{Sketch of the geometry.}\label{fig1}
  \end{figure}
 
 This definition of $D$ is rather complicated -- see the picture below. It may help the reader, in coming to terms with this abstract description to keep in mind the following case, included in our definition of $D$. This is the case where the scattering surface $\Gamma$ is the graph of some bounded continuous function. For example in the 3D case, we have that for a bounded and continuous function $f:\mathbb{R}^{2} \to \mathbb{R}$,
\begin{equation}\label{Gamma_graph}
\Gamma:= \{x=(x_1,x_2,x_3) \in \mathbb{R}^3:x_3=f(x_1,x_2)\},
\end{equation}
after which the domain $D$ is given by    
\[
D= \{x=(x_1,x_2,x_3) \in \mathbb{R}^3:x_3>f(x_1,x_2)\}.
\]
Of course, in general one wishes to consider rough surfaces that are not simply the graphs of functions, but of more complex geometry that one encounters in reality. It is a major point of this thesis that we establish well-posedness results for wave scattering by rough surfaces that are \emph{not} the graphs of functions, under certain other restrictions (low frequency of the waves, for example). Nevertheless a great deal of the well-posedness results we establish in this thesis are restricted to this setting, where the rough surface is the graph of a bounded continuous function.

We aim in the next sections to turn our attention to the mathematical modelling of these problems. We'll begin by looking at acoustics; then we'll take a look at electromagnetics.

%However, since the majority of our results on the well-posedness of rough surface scattering problems are established in the case when the the rough surface is the graph of a bounded continuous function, and since also we wish at this stage to 

%We wish to begin by describing, in loose terms, a typical rough surface scatering problem, that of acoustic scattering by a rough sound soft surface. Thus the acoustic medium of propogation is thought to occuppy the perturbed half-space $D$. Given that we know the rough surface $\Gamma$ that bounds $D$, and supposing that a source of acoustic waves is present at some finite distance above the rough surface $Gamma$, and that we know too what this source is, then the problem is, in Laymen's terms, to predict the manner in which these acoustic waves bounce off or scatter from the surface.
  
%Mathematically, we say that, given a source $g \in L^2(D)$ of compact support, the problem we wish to analyze is to find an acoustic field $u$ such that
%\[
%\Del u +k^2 u =g \mbox{ in } D,    
%\]
%\[
%u=0 \mbox{ on } \Gamma,
%\]
%and such that $u$ satisfies an appropriate radiation condition.    

\section {The mathematical description of the scattering problems}
\subsection{Acoustics} 
The wave equation is the classical model that describes acoustic propagation. Let $U(x,t): D \times \mathbb{R} \to \mathbb{C}$ denote the perturbation of pressure at a point $x$ in $D$ and at a time $t>0$. It is this that we seek to find. In other words, for example, it represents the resultant sound distribution that we desired to predict in the aeroplane example earlier. Let $G(x,t):D\times \mathbb{R} \to \mathbb{C}$ denote the source of acoustic disturbance, i.e. the noise produced by the aeroplane, which we suppose we know or are given. Then the inhomogeneous wave equation relates the two via,
\begin{equation} \label{wave_eq}
\frac{1}{c^2}\frac{\partial ^2 U}{\partial t^2} - \Delta U = -G, \mbox{ in } D \times \mathbb{R}.
\end{equation} 
Here $c$ is the speed of sound in the medium. We may wish to assume that $c$ is a constant, which is appropriate if the region $D$ is occupied by one medium, such as air which is at rest; but also we will consider the case when the medium in $D$ varies throughout $D$, in which case $c$ depends on position. %Moreover in some instances we may wish to suppose that $c$ has an imaginary part. 
We note also at this point that the density perturbation $\rho$ of the air, also satisfies (\ref{wave_eq}), though with a different function $G$, and, provided the wave motion is initially irrotational, then, the velocity, $v$, of the air is the gradient of a scalar field $\Psi$ the \emph{velocity potential}, which also satisfies the wave equation, with a yet different function $G$. Further, the relationships between these three quantities are given by
\begin{equation} \label{relations}
v = \nabla \Psi, \mbox{   }  U = -\rho_0\Psi_{t}, \mbox{  } U=c^2\rho,
\end{equation}   
where $\rho_0$ denotes the density of the unperturbed state. We will suppose that our waves are time harmonic. This means we will assume that $G(x,t)=\Re(g(x)e^{-i\omega t})$ for a function $g: D\to \mathbb{C}$ and then look for solutions to the wave equation in the form $U(x,t) = \Re(u(x)e^{-i\omega t})$ for some function $u: D \to \mathbb{C}$. Here $\omega>0$ denotes the angular frequency of the waves. On making this assumption one sees that equation (\ref{wave_eq}) reduces to 
\begin{equation} \label{Helm_1}
\Delta u + k^2 u = g, \mbox{ in } D,
\end{equation}
where $k:=\omega/c$, is the \emph{wavenumber}. Equation (\ref{Helm_1}) is known as the inhomogeneous Helmholtz equation. We should mention at this stage that a slightly different model for acoustical scattering that takes into account the effect of dampening in the region $D$, and which takes as it's starting point the \emph{dissipative wave equation}, leads once again, with similar manipulations to those above, to the inhomogeneous Helmholtz equation (\ref{Helm_1}), but this time with $k$ being a complex valued function (see \cite{Col83a} pages 66-67). Thus the case when $k$ is complex valued in (\ref{Helm_1}) is also of interest in acoustics.
    
Thus given $D \subset \mathbb{R}^n$ and given $g:D \to \mathbb{C}$ our aim will be to find 
$u:D  \to \mathbb{C}$ satisfying (\ref{Helm_1}). That (\ref{Helm_1}) be satisfied in a 
classical sense, requires that we should find $u$ belonging to the space $C^2(D).$ Alternatively we may, to get a handle on solving the problem, seek to find a solution $u$ satisfying (\ref{Helm_1}) in a weaker sense, for example a distributional sense, in which case it would be more appropriate to look for a solution $u$ in $C^1(D)$ perhaps. When we precisely pose our problems, later on, it will be important to know in what exact function space to look for $u$. This will depend on the function space setting of $g$. For the meantime though we wish to brush over such issues and set up the basic problem. However we should mention here something about the nature of $g$. In our motivating problem of the aeroplane it is appropriate to view $g$ as a function of compact support. In fact, in our variational formulations of the problem, $g$ will be given more generally as a function in $L^2(D)$ although its support will lie at a finite distance from the boundary $\Gamma$. This means that the support of $g$ will lie in $\overline{S_H}$ for some $H\geq f_+$.  

Another source of acoustic excitation that we will consider in this thesis, is that due to a \emph{point source}. Here $g = \delta_y$ -- a delta function situated at $y \in D$. The interest in this sort of excitation again stems from the wish to study wave sources of compact support: any such source can be represented as superpositions of point sources located in the compact support.

Finally we should mention one important type of acoustic incidence, not covered in the work of this thesis. This is plane wave incidence. Generally the analysis that we apply in this thesis requires that sources should decay along the boundary; as such none of our results apply to scattering of plane waves.
 
 %incidence of  Also we should specify in which space the function $g$ is supposed to live. Is it in $L^2(D)$? Does it have compact support? etcetera. All of these questions are of importance and later on, when we firm up our mathematical descriptions of our problems, we will have to pay due attention to the respective function space settings to which $g$ and $u$ belong. However, at this stage we rather wish to brush over such points and just get a general feel of things.

\subsubsection{Boundary conditions.}
Finding a unique solution to equation (\ref{Helm_1}) will not be possible without requiring the solution $u$ to satisfy an appropriate boundary condition. We will look at two such boundary conditions: the Dirichlet boundary condition and the impedance boundary condition.

\subsubsection{Dirichlet boundary condition.} Here we require that $u=0$ on the boundary $\Gamma$. In this case $\Gamma$ is said to be a \emph{sound soft} surface. Physically it corresponds to there being no pressure on the surface. It is appropriate to assume this when there is a huge jump in pressure across a surface. An example of this arises in underwater acoustics: If we imagine a submarine beneath the sea emitting sound and if we wish to know just how this sound propagates, then, the problem domain $D$ would be the region occupied by the sea and the rough surface $\Gamma$ would be the surface of the sea. Given the large drop in pressure as one moves from the sea to the air above it, then, it would be appropriate in this case to assume that the pressure is zero on this rough surface $\Gamma$. 
  
\subsubsection{Impedance boundary condition.} In order to motivate this boundary condition let us just note that the relation between pressure $U$ and velocity potential $\Psi$ given in (\ref{relations}), can be simplified, on making our time harmonic assumptions that 
\[
U= \Re(u(x)e^{-i\omega t}) \mbox{ and } \Psi= \Re(\psi(x)e^{-i\omega t})
\]
for $u,\psi : D \to \mathbb{C}$.
The new relation between  $u$ and $\psi$ is then 
\begin{equation} \label{u_psi_relation}
u=i\omega\rho_0\psi. 
\end{equation}
 
Now, the classical Neumann boundary condition, that 
$$
\frac{\partial \psi}{\partial \nu}=0
$$
on the boundary $\Gamma$, assumes that the component of fluid velocity normal to the surface vanishes. This makes sense for a rigid surface. But for more general surfaces, the normal velocity is non-zero and the quantity $Z_s$, defined by
\begin{equation}\label{specific}
Z_s = \frac{u}{  \partial \psi/\partial \nu}
\end{equation}
is finite on the boundary. $Z_s$ is called the \emph{surface impedance} (e.g.\ \cite{Morse}). In general $Z_s$ depends on the variation of the acoustic field throughout the medium of propagation. Often however $Z_s$ depends only on the properties of the boundary surface (and on the angular frequency $\omega$, but we will assume that this is constant in our study of the impedance problem): specifically, for a given stretch of surface, for example a concrete road surface, the ratio $u/(\partial \psi/ \partial \nu)$ is constant; and then for a different stretch of surface, for example that covered by a field, it assumes yet a different constant value.  In this thesis we will always assume that $Z_s$ is independent of the distribution of the acoustic field, in which case we say that the boundary is \emph{locally reacting}. 

Thus defining $\beta$ as 
\begin{equation}\label{beta_def}
\beta = \frac{\rho_0 c}{Z_s}
\end{equation}
we see that, from the above discussion, $\beta$ is a function of the boundary $\Gamma$, and we will suppose that $\beta \in L^{\infty}(\Gamma)$. 
%Going back to equation (\ref{relations}), we see that bearing in mind that we are seeking time harmonic solutions, that $p$ and $\psi$ are related by
%\[
%p=i\omega\rho_0\psi.
%\] 
Using (\ref{beta_def}) and (\ref{u_psi_relation}), equation (\ref{specific}) can be rewritten as,
\[
\frac{\partial \psi}{\partial \nu}= ik\beta \psi \mbox{    or   }  \frac{\partial u}{\partial \nu}= ik\beta u, \mbox{ on } \Gamma, 
\]     
and we have arrived at the impedance boundary condition, also known as the Robin boundary condition or the third boundary condition.

It can be shown -- see for example \cite{Simon's_thesis} page 24 -- that if the ground/boundary is not to be a source of energy then a necessary condition on $\beta$ is that 
\[
\Re \beta \geq 0.
 \] 
In chapter 3 we'll make -- and to some extent justify -- some extra assumptions on $\beta$.

\subsubsection{The Radiation condition}
The solution to (\ref{Helm_1}) will still not be unique even when we impose one of the boundary conditions. A radiation condition is also required, many authors referring to this as an extra boundary condition at infinity. The role of the radiation condition is to pick out the \textit{physically realistic} solution. 

%A well known radiation condition used in acoustics is the \textit{Somerfeld radiation condition}. However this is not a valid characterization of outgouing waves for scattering problems at rough surfaces. Instead 
In this thesis we make use of a radiation condition called the \emph{upward propagating radiation condition}, (UPRC). To state this we introduce the fundamental solutions to the Helmholtz equation (\ref{Helm_1}) in the case when the wavenumber $k$ is a positive constant i.e. $k=k_+>0$. This is $\Phi$, given by   
\[\label{fund_sol}
\Phi(x,y)=\left\{\begin{array}{ll}\displaystyle\frac{\ri}{4}H_0^{(1)}(k_+|x-y|),&n=2,\\\nonumber
\displaystyle\frac{\exp(\ri k_+|x-y|)}{4\pi
|x-y|},&n=3,\end{array}\right. %\label{Weq2}
\]
for $x,y\in\real^n$, $x\neq y$, where $H_0^{(1)}$ is the Hankel
function of the first kind of order zero. $\Phi(x,y)$ is a solution to the Helmholtz equation
(\ref{Helm_1}) with $k=k_+$ in the special case when $D= \mathbb{R}^n$ and $g= \delta_y$, a point source located at $y \in \mathbb{R}^n$.
The (UPRC) then states that 
%and such that $u$ satisfies the upward propagating radiation
%condition (UPRC) proposed in \cite{chandprs99}, which states that
\begin{equation}
u(x)=2 \int_{\GH}\frac{\pa\Phi(x,y)}{\pa x_n} u(y)\,
ds(y),\label{uprc} \quad x\in U_H,
\end{equation}
for all $H$ such that the support of $g$ is contained in $S_H$. 

The (UPRC) was proposed in \cite{chandprs99}. In the case that the wavenumber $k$ has imaginary part, one can derive this representation for the solution of the Helmholtz equation in $U_H$, under mild assumptions on the growth of the solution at infinity: see \cite{chandmmas97}. 

In the case that $u|_{\Gamma_H}\in L^2(\Gamma_H)$ we can rewrite
(\ref{uprc}) in terms of the Fourier transform of $u|_{\Gamma_H}$.
For $\phi\in L^2(\Gamma_H)$, which we identify with
$L^2(\R^{n-1})$, we denote by $\hat \phi = {\cal F}\phi$ the
Fourier transform of $\phi$ which we define by
\begin{equation}
\cF\phi(\xi)=(2\pi)^{-(n-1)/2}\int_{\real^{n-1}}\exp(-\ri\tx\cdot\xi)\phi(\tx)\,d\tx,\quad
\xi \in \real^{n-1}.\label{FT}
\end{equation}
Our choice of normalization of the Fourier transform ensures that
$\cal F$ is a unitary operator on $L^2(\R^{n-1})$, so that, for
$\phi, \psi\in L^2(\R^{n-1})$,
\begin{equation} \label{plancherel}
\int_{\R^{n-1}} \phi\bar\psi d\tilde x = \int_{\R^{n-1}} \hat
\phi\bar{\hat \psi} d\xi.
\end{equation}
If $F_H:=u|_{\Gamma_H}\in L^2(\Gamma_H)$ then (see
\cite{chandsjam98,are04a} in the case $n=2$),  (\ref{uprc}) can be
rewritten as
\begin{eqnarray}\nonumber
u(x)=\frac{1}{(2\pi)^{(n-1)/2}}\int_{\real^{n-1}}\exp(\ri[(x_n-H)\sqrt{k_+^2-\xi^2}+
\tx\cdot\xi])\hat F_H(\xi)\,d\xi,\;\; x\in U_H. \\\label{uprcstar}
\end{eqnarray}
In this equation $\sqrt{k_+^2-\xi^2}=\ri\sqrt{\xi^2-k_+^2}$, when
$|\xi|>k_+$.

Equation (\ref{uprcstar}) is a representation for $u$, in the
upper half-plane $U_H$, as a superposition of upward propagating
homogeneous and inhomogeneous plane waves. A requirement that
(\ref{uprcstar}) holds is commonly used (e.g. \cite{DeSanto02}) as
a formal radiation condition in the physics and engineering
literature on rough surface scattering. The meaning of
(\ref{uprcstar}) is clear when $F_H\in L^2(\real^{n-1})$ so that
$\hat F_H\in L^2(\real^{n-1})$; indeed the integral
(\ref{uprcstar}) exists in the Lebesgue sense for all $x\in U_H$.
Recently Arens and Hohage \cite{are04a} have explained, in the
case $n=2$, in what precise sense (\ref{uprcstar}) can be
understood when $F_H\in BC(\Gamma_H)$ so that $\hat F_H$ must be
interpreted as a tempered distribution. Arens and Hohage also show the equivalence of this radiation condition with another known as the Pole Condition.

%As an alternative to the radiation condition some authors have used what is known as a limiting absorption principle in thier formulations. Although we will not do so in this thesis, we will nevertheless have to employ it in our arguments, since we will make use of previous results on scattering problems that were formulated with it. We should note that in the case that the wavenumber $k$ has imaginary part \emph{no} radiation condition  or limiting absorption principle is required. The following limiting absorption principle was used in :for $k>0$, denoting $u$, the solution of (\ref{Helm_1}) with an appropriate boundary condition, by $u^{(k)}$ to indicate it's dependence on $k$, we suppose that for all sufficiently small $\epsilon>0$ a solution $u^{k+i\epsilon)}$ exists and that, for all $x \in D$,
%\[
%u^{(k+i\epsilon)}(x) \to u^{(k)}(x), \epsilon \to 0.
%\]

In summary, our acoustic problems will be to look for a solution to the Helmholtz equation (\ref{Helm_1}), satisfying the radiation condition (\ref{uprcstar}), and satisfying one of the boundary conditions: if it is the Dirichlet boundary condition, then we will refer to this problem as the Dirichlet problem for the Helmholtz equation, or simply the Dirichlet problem; in the case where we use an impedance boundary condition, we will refer to the problem as the impedance problem for the Helmholtz equation or simply the impedance problem.

\subsection{Electromagnetics} 
In classical electromagnetics, Maxwell's equations relate the electric field intensity $\mathbf{E}(x,t) :D \times \mathbb{R} \to \mathbb{C}^n$, the magnetic field intensity $\mathbf{H}(x,t):D \times \mathbb{R} \to \mathbb{C}^n$, the electric displacement $\mathbf{D}(x,t) :D \times \mathbb{R} \to \mathbb{C}^n$ and the magnetic induction $\mathbf{B} :D \times \mathbb{R} \to \mathbb{C}^n$, to the cause of electromagnetic excitation, namely the charge density function $\rho :D \times \mathbb{R} \to \mathbb{C}$ and the current density function $\mathbf{J}:D \times \mathbb{R} \to \mathbb{C}^n$. Maxwell's equations are (see \cite{Monk} pages 1-9):
\begin{equation} \label{Max_1}
\frac{\partial \mathbf{B}}{\partial t} + \nabla \times \mathbf{E} = 0  \mbox{ in } D,
\end{equation}
\begin{equation}\label{Max_2}
\nabla.\mathbf{D} = \rho \mbox{ in } D,
\end{equation}
\begin{equation}\label{Max_3}
\frac{\partial \mathbf{D}}{\partial t} - \nabla \times \mathbf{H} = -\mathbf{J} \mbox{ in } D,
\end{equation}
\begin{equation}\label{Max_4}
\nabla .\mathbf{B} =0 \mbox{ in } D.
\end{equation}
Note that equations (\ref{Max_2}) and (\ref{Max_3}) can be combined to give
\begin{equation}\label{source_rel}
\nabla .\mathbf{J} + \frac{\partial \rho}{\partial t} =0.
\end{equation}

Making the assumption that the current density and charge density are time harmonic, i.e. that 
\[
\mathbf{J}(x,t)  = \Re (\exp(-i\omega t)\hat{\mathbf{J}}(x))
\]
and that 
\[
\rho(x,t) = \Re (\exp(-i\omega t)\hat{\rho}(x))
\]
for known $\hat{\mathbf{J}}:D \to \mathbb{C}^n$ and $\hat{\rho}:D \to \mathbb{C}$, and that also 
\[
\mathbf{E}(x,t)  = \Re (\exp(-i\omega t)\hat{\mathbf{E}}(x))
\] 
\[
\mathbf{D}(x,t)  = \Re (\exp(-i\omega t)\hat{\mathbf{D}}(x))
\]
\[
\mathbf{H}(x,t)  = \Re (\exp(-i\omega t)\hat{\mathbf{H}}(x))
\]
\[
\mathbf{B}(x,t)  = \Re (\exp(-i\omega t)\hat{\mathbf{B}}(x))
\]
for unknown functions $\hat{\mathbf{E}}:D \to \mathbb{C}^n$, $\hat{\mathbf{D}}:D \to \mathbb{C}^n$, $\hat{\mathbf{H}}:D \to \mathbb{C}^n$, and $\hat{\mathbf{B}}:D \to \mathbb{C}^n$, we obtain the time harmonic Maxwell's equations:
\begin{equation}\label{Max_5}
-i\omega \hat{\mathbf{B}} + \nabla \times \hat{\mathbf{E}} = 0 \mbox{ in }D,
\end{equation}
\begin{equation}\label{Max_6}
\nabla .\hat{\mathbf{D}} = \hat{\rho}  \mbox{ in } D,
\end{equation}
\begin{equation}\label{Max_7}
-i\omega \hat{\mathbf{D}}-\nabla \times \hat{\mathbf{H}} = -\hat{\mathbf{J}}\mbox{ in }D,
\end{equation}
\begin{equation}\label{Max_8}
\nabla .\hat{\mathbf{B}} = 0 \mbox{ in } D.
\end{equation}
We now reduce these 4 equations down to 2, eliminating the quantities $\hat{\mathbf{D}}$ and $\hat{\mathbf{B}}$, by supposing there hold two \emph{constitutive} laws that relate $\hat{\mathbf{E}}$ and $\hat{\mathbf{H}}$ to $\hat{\mathbf{D}}$ and $\hat{\mathbf{B}}$, respectively. These laws depend on the matter in the domain $D$ occupied by the electromagnetic field. In this thesis we suppose that the material occupying $D$ is \emph{inhomogeneous}, that is, it is a composition of different materials (e.g copper, air etc.); that the material is \emph{isotropic}, in other words the material properties do not depend on the direction of the field; and also we assume the material is linear. It then follows that the constitutive equations are (\cite{Monk} page 5)
\begin{equation}\label{const_1}
\hat{\mathbf{D}}= \epsilon \hat{\mathbf{E}}
\end{equation}
and 
\begin{equation}\label{const_2}
\hat{\mathbf{B}}= \mu \hat{\mathbf{H}},
\end{equation} 
where $\epsilon: D \to \mathbb{R}$ is positive and bounded and is known as the \emph{electric permittivity}; whilst $\mu>0$ is the \emph{magnetic permeability}, and is assumed to be a constant. One further constitutive equation is that 
\begin{equation}\label{const_3}
\hat{\mathbf{J}}= \sigma \hat{\mathbf{E}} + \hat{\mathbf{J}}_a, \mbox{ in D}
\end{equation}
where $\sigma:D \to \mathbb{R}$ is non-negative and is called the \emph{conductivity} and the vector function $\hat{\mathbf{J}}_a$ is the \emph{applied current density}. Regions of $D$ where $\sigma$ is strictly positive are termed \emph{conducting}.  Where $\sigma=0$ the material in $D$ is termed \emph{dielectric}.   

Now using the constitutive relations (\ref{const_1}), (\ref{const_2}), (\ref{const_3}) and also using equation (\ref{source_rel}) in its time harmonic form,
\[
\nabla.\hat{\mathbf{J}}- i\omega \hat{\rho}=0  \mbox{ in } D,
 \]
in the time harmonic Maxwell's equations (\ref{Max_5}), (\ref{Max_6}), (\ref{Max_7}), (\ref{Max_8}),
we derive that
\begin{equation} \label{Max_9}
-i\omega \mu\hat{\mathbf{H}} + \nabla \times \hat{\mathbf{E}} = 0 \mbox{ in } D,
\end{equation}
\begin{equation} \label{Max_10}
\nabla .(\epsilon \hat{\mathbf{E}})= \frac{1}{i\omega}\nabla .(\sigma \hat{\mathbf{E}} + \hat{\mathbf{J}}_a) \mbox{ in } D,
\end{equation}
\begin{equation} \label{Max_11}
-i\omega \epsilon\hat{\mathbf{E}} + \sigma \hat{\mathbf{E}}- \nabla \times \hat{\mathbf{H}} = -\hat{\mathbf{J}}_a \mbox{ in } D,
\end{equation}
\begin{equation} \label{Max_12}
\nabla .(\mu \hat{\mathbf{H}}) =0  \mbox{ in } D.
\end{equation}

We remark that by supposing the constitutive relations to hold, equations (\ref{Max_10}) and 
(\ref{Max_12}) are now redundant; they can be derived by taking the divergence of (\ref{Max_11}) and (\ref{Max_9}) respectively. Moreover we can eliminate the variable  $\hat{\mathbf{H}}$ by substituting (\ref{Max_9}) into (\ref{Max_11}), to arrive at one equation for $\hat{\mathbf{E}}$:
\begin{equation}\label{Max_13}
\nabla \times (\nabla \times \hat{\mathbf{E}}) -\omega^2\mu\epsilon\left[1 + \frac{i\sigma}{\omega \epsilon}\right]\hat{\mathbf{E}} = i\omega \mu\hat{\mathbf{J}}_a \mbox{ in } D,
\end{equation}
recalling that $\mu$ is assumed to be constant.
Letting $\mathbf{G}=i\omega \mu\hat{\mathbf{J}}_a$,
and letting 
\begin{equation}\label{k_def_em}
k^2 = \omega^2\mu\epsilon\left[1 + \frac{i\sigma}{\omega \epsilon}\right]
\end{equation}
we see that (\ref{Max_13}) becomes 
\begin{equation}\label{Max_14}
\nabla \times (\nabla \times \hat{\mathbf{E}}) - k^2\hat{\mathbf{E}} = \mathbf{G}.
\end{equation}
In this thesis we assume that the problem is two dimensional: precisely we suppose that the electric permittivity $\epsilon$, the magnetic permeability $\mu$ and the conductivity $\sigma$  are invariant in the $x_3$ direction. Moreover we only study the Transverse Electric (T.E.) case. In the T.E.\ case we seek the electric field intensity in the form $\hat{\mathbf{E}} = (0,0,E)$ where $E$ is supposed to be independent of the $x_3$ variable and also we assume that $\mathbf{G}= (0,0,-g)$. On making these assumptions we see that (\ref{Max_14}) becomes 
\begin{equation} \label{Helm_elect}
\Delta E +k^2 E =g  \mbox{ in } D.
\end{equation}
It is the solution $E$ to this equation that we will seek to find when we are given $g$. We see that we have once more arrived at the inhomogeneous Helmholtz equation. As in the last section it must be supplemented by boundary and radiation conditions.

\subsubsection{Boundary and radiation conditions.}
We will look at two, two-dimensional problems involving the scattering of electromagnetic waves. The first involves scattering by a perfectly conducting rough surface. In this case the appropriate boundary condition is that 
\[
\nu. \hat{\mathbf{E}}= 0 \mbox{ on } \Gamma.
\] 
Since we are assuming that $\hat{\mathbf{E}} = (0,0,E)$, this means we should require that $E=0$ on $\Gamma$. In addition we then impose the radiation condition (\ref{uprcstar}); for this it's necessary to assume that outside a neighbourhood of the boundary $\Gamma$ the quantity $k$ in (\ref{Helm_elect}) takes on  a constant positive value $k_+$. We will call this problem, that of scattering by an unbounded rough inhomogeneous layer.

The second electromagnetic problem we wish to study will be known as the transmission problem. Here the domain $D$ of electromagnetic propagation is assumed to be the whole of $\mathbb{R}^n$. As such no boundary conditions are required, but rather we impose the radiation condition both in the upward and downward directions, assuming that the function $k$ in (\ref{Helm_elect}) assumes positive constant values $k_+$ and $k_-$ above and below a strip of finite height within which $k$ may vary. We will return to to this idea and elaborate on it when we come to study the transmission problem in chapter 4.
%Here the
%fundamental solution of the Helmholtz equation $\Phi$ is given by
%$$
%\Phi(x,y)=\left\{\begin{array}{ll}\displaystyle\frac{\ri}{4}H_0^{(1)}(k|x-y|),&n=2,\\
%\displaystyle\frac{\exp(\ri k|x-y|)}{4\pi
%|x-y|},&n=3,\end{array}\right. %\label{Weq2}
%$$
%for $x,y\in\real^n$, $x\neq y$, where $H_0^{(1)}$ is the Hankel
%function of the first kind of order zero.

\section{Hadamard's criterion and numerical implementation}
It is our general aim to solve the problems that we posed in the last section. It is instructive to consider a little why we cannot find explicit solutions to these problems, via mathematical techniques. There are essentially two answers to this question. One is that, put simply, these problems are too difficult. Another involves the complex nature of the problem: for if, in the earlier aeroplane example, we consider the scattering surface to possess a complicated, that is realistic, geometry, and if the source of acoustic waves is similarly of a complex and realistic nature, then one can hardly expect that the scattered acoustic field will have such a simple form as to be able to be described by an explicit mathematical function. Indeed, such complicated scattered fields are best described by pictures generated on a computer.

Thus to solve such problems numerical and computational techniques are essential. On the theoretical side we should ensure that the mathematical problems we set are well-posed, in that they satisfy Hadamard's criterion. This states that for a given mathematical model:
\newline{1)there should exist a solution;
\newline2)the solution should be unique;  
\newline3)the solution should depend continuously on the data. For example any solution $u$ to equation (\ref{Helm_1}) should satisfy an inequality
\[
\Vert u \Vert_X \leq C \Vert g\Vert_Y, 
\]
where $C>0$ is a constant and $X$ and $Y$ are normed spaces to which $u$ and $g$ respectively belong.}

In this thesis we are primarily concerned with showing that the problems we state are well-posed and not with the approximate solution to these problems using numerical techniques. %Of course to say that the two sides of the coin, that is the Hadamard's criterion on the one hand and numerical implementation on the other are not related would be silly. 
However, our approach to our problems is geared toward the ultimate goal of numerical computation of the solution: specifically, the variational formulations that we derive from our problems in chapters 2, 3 and 4 should be suitable for finite element implementation; similarly the boundary integral equation that we derive from our problem in chapter 5 should be suitable for solution via boundary element methods. In particular the explicit bounds we establish, on the sesquilinear form, in chapters 2, 3 and 4, and on the inverse operator associated with the double- and single-layer potentials in chapter 5, should prove helpful in the analysis of the numerical implementation.

\section{Literature review: Overview}
We conclude this introduction with a broad review of the literature.

 The problem of rough surface scattering has long been studied and there have been many contributors to the subject. Principally research has focused on the use of numerical methods to solve these problems. The review of Warnick and Chew \cite{warnick01}, summarises numerical strategies, implemented over the past 30 years or so, that seek to simulate the scattering of electromagnetic (and also acoustic) waves by rough surfaces. Warnick and Chew roughly group these numerical strategies into three categories: differential equation methods, boundary integral equation methods and numerical methods based on analytical scattering approximations. 
 %Although our work is primarily concerned with the well-posedness of scattering problems, we should point out that the theoretical work on the problems in the first three chapters, is geared toward solving those problems by Differential equation methods, specifically finite element methods; also the problem that we study in chapter 4, will involve theory that will help implement a boundary integral equation method. 
A critical survey of scattering approximations is carried out in the review of Elfouhaily and Guerin \cite{elfguerin}.

In the review \cite{saillardsent01}, Saillard and Sentenac are interested in formulating rough surface scattering problems from a statistical point of view. Here, the rough surface is not a known quantity in the problem, but rather, one only has information on certain statistical properties of the surface, so that the shape of the rough surface is described by a random function of space coordinates and time. The problem is then to determine the statistical properties of the scattered field, such as it's mean value and mean intensity, as functions of the statistical properties of the surface. Obviously such a problem is of interest, since in reality, one often will not know the precise shape of the rough surface. In this thesis however, we always assume that the rough surface is known. Saillard and Sentenac then proceed to describe approximate methods for solving these problems numerically. They point out, however, that few authors have undertaken a rigourous mathematical study of the problem.  

In \cite{ogil91} Ogilvy reviews research in this area, again with an emphasis on random rough surfaces and on numerical techniques.  See also the books by Voronovich \cite{voron94}, Petit \cite{petit80} and Wilcox \cite{Wilcox84} and the review of DeSanto \cite{DeSanto02}.

Finally we should make mention of the closely related field of scattering by bounded obstacles. A very complete theory of this class of problems has been developed, especially by use of boundary integral equation methods, see for example \cite{Col83a}. 

We will return to the literature review, with a much closer scrutiny of the papers that are related to the work in this thesis, as we tackle the various problems in chapters 2, 3, 4 and 5.

\part{Variational Methods}
In the next three chapters we apply variational methods to three of our scattering problems. 
%sh to begin this section on variational methods by introducing some of the concepts involved, for example Sobolov spaces and weak solutions. 
An excellent introduction to the theory of variational methods can be found in Lawrence C. Evans's book `Partial Differential Equations' \cite{Evans} chapters 5 and 6. 

There are two main theorems that we will require:

\begin{theorem}\textbf{Lax-Milgram.} [e.g.\ \cite{Monk} Lemma 2.21.]
Let $H$ be a Hilbert space, with norm and inner product given by $\Vert \cdot\Vert$, $(\quad   , \quad )$ respectively. Suppose that $b:H\times H\to \mathbb{C}$ is a bounded sesquilinear form such that
for some $\alpha>0$ it holds that
\[
|b(u,u)| \geq \alpha \Vert u\Vert^2, \quad u \in H.
\]
Then for each $G \in H^*$ there exists a unique $u \in H$ such that
\[
b(u,v) = G(v) \quad v \in H,
\]  %\Theorem{Babu\v{s}ka} 
and 
\[
\Vert u \Vert \leq \alpha^{-1}\Vert G \Vert_{H^*},
\]
where $\Vert \cdot \Vert_{H^*}$ denotes the norm of $H^*$.
\end{theorem}
\begin{theorem}\textbf{Generalized Lax-Milgram Theorem.} [e.g.\ \cite{ihlenburg} Theorem 2.15.]
Let $H$ be a Hilbert space, with norm and inner product given by $\Vert \cdot\Vert$, $(\quad  , \quad )$ respectively. Suppose that $b:H\times H\to \mathbb{C}$ is a bounded sesquilinear form such that
for some $\alpha>0$ the \emph{inf-sup} condition holds:
\begin{equation}
 \alpha := \inf_{0\not=u\in H}\sup_{0\not=v\in
H}\frac{|b(u,v)|}{\Vert u\Vert\Vert v\Vert} >0;
\label{infsup}
\end{equation}
and the\emph{ transposed inf-sup} condition holds:
\begin{equation}\label{trans_is}
\sup_{0\not=u\in H}\frac{|b(u,v)|}{\Vert u\Vert}>0.
\end{equation}
Then for each $G \in H^*$ there exists a unique $u \in H$ such that
\[
b(u,v) = G(v) \quad v \in H,
\]  %\Theorem{Babu\v{s}ka} 
and 
\[
\Vert u \Vert \leq \alpha^{-1}\Vert G \Vert_{H^*}.
\]

\end{theorem}

\chapter{Scattering by unbounded, rough, inhomogeneous layers}\label{layer}
\section{Literature review} In this chapter we study, via variational methods, a boundary value problem for the Helmholtz equation modelling scattering of time harmonic waves by a layer of spatially-varying refractive index above a rough surface on which the field vanishes (we called this the problem of scattering by unbounded, rough, inhomogeneous layers in chapter 1 -- see the electromagnetics section). We recall from chapter 1 that in the 2D case this problem models the scattering of time harmonic electromagnetic waves by an inhomogeneous conducting or dielectric layer above a perfectly conducting rough surface in the transverse electric polarization case.
Moreover it is a model, in 2 and 3 dimensions, of time harmonic acoustic scattering by a rough surface in a medium in which the wavespeed varies with position or in which there is dissipation. 

We commence with a thorough survey of the literature on this problem. In fact this problem, in which we study the Helmholtz equation (\ref{Helm_1}) with $k$ a function of position, seems to have received little attention with the exception of \cite{chandprs98}. Mainly this problem has been studied in the special case when $k$ is constant throughout the region $D$, (in which case the problem reduces to what we have called the Dirichlet problem for the Helmholtz equation in chapter 1.) Thus, let us begin this survey by looking at contributions to this problem when $k$ is assumed constant. 
 
  The pioneering paper on this subject seem to be the uniqueness proof of Rellich \cite{Rel_2}. Here Rellich assumed that the rough surface roughly resembled a paraboloid. In \cite{odeh} Odeh proves uniqueness of solution in the case that the rough surface is smooth and is either a cone or approaches a flat boundary at infinity.  Willers proves the existence of a unique solution to this problem, in \cite{willers87}, making the assumption that the boundary is $C^2$ and is flat outside a compact set. 
  
  In another, somewhat related body of work existence of solution to the Dirichlet problem is established by the limiting absorption method, via a priori estimates in weighted Sobolev spaces (see Eidus and Vinnik \cite {eidvin}, Vogelsang \cite{vogel}, Minskii \cite{minskii} and the references therein.) The results obtained are still however limited in that one must assume that the rough surface approaches a flat boundary sufficiently rapidly at infinity and/or that the sign of $x.\nu(x)$ is constant on $\partial D$ outside a large sphere, where $\nu(x)$ denotes the unit normal at $x \in \partial D$.

%  With the exception of  the paper of Zhang, which we will speak about shortly, little attention appears to have been payed to this problem except in the special case when $k$ is constant throughout the domain $D$, in which case it reduces to the Dirichlet problem.     

%Thus it is appropriate, at this point to talk about previous work on the Dirichlet problem.

The most recent, and indeed most complete results in this field, have been developed by Chandler-Wilde and his collaborators. Principally Chandler-Wilde et al have concentrated on employing boundary integral equation (BIE) techniques to settle the question of unique existence of solution to these problems. However, the loss of compactness of the associated boundary integral operators in the case when the boundary is infinite, meant that the theory of boundary integral equations for scattering by bounded obstacles did not translate easily to the problem of rough surface scattering (c.f.\ the literature review in chapter 5). As such generalizations of part of the Riesz theory of compact operators have been developed - see the work of Arens, Chandler-Wilde and Haseloh, \cite{arenshcwI}, \cite{arenshcwII} - requiring only that the associated boundary integral operators be locally compact, and ensuring that existence of solution to the boundary integral equation follows from uniqueness.

In \cite{chandip95} Chandler-Wilde and Ross prove a uniqueness theorem for the Dirichlet problem, in an arbitrary domain $D$ with the assumption that $\Im k >0$. The same authors in \cite{chandmmas96}, then derive some existence results in 2D for mildly rough surfaces, using BIE techniques. 

In \cite{chandsjam98}, Chandler-Wilde and Zhang show uniqueness to the Dirichlet problem in a non-locally perturbed half-plane with piecewise Lyapunov boundary. This time the wavenumber $k$ is assumed real and the problem is formulated with a radiation condition. Moreover an integral equation formulation is proposed and existence of solution is established, for mildly rough surfaces, in 2 dimensions, by using the results of \cite{chandmmas96}. Finally, in \cite{chandprs99}, Chandler-Wilde, Ross and Zhang show existence of solution for domains with Lyapunov boundary, in 2D, but this time with no limit on the surface slope or amplitude. They do this by employing novel solvability results on integral equations, contained in the paper. See also \cite{zhangworking}, where similar results are obtained but with an alternative integral equation formulation.   

Recently Chandler-Wilde, Heinemeyer and Potthast looked at the Dirichlet problem in 3 dimensions, again by an integral equation approach, and were able in \cite{chandpottheim}, \cite{chandpottheim2}, to establish that the problem was well-posed in the case that the boundary is the graph of a Lyapunov function. We will in fact extend these results, to the case when the boundary is the graph of a Lipschitz function in chapter 5. Indeed, see the literature review in chapter 5 for further details on the use of BIE techniques.   

% Zhang, Ross, Arens and Haseloh were able to establish the well-posedness of the Dirichlet problem, via integral equation methods, in the 2D case when the boundary $\Gamma$ is the graph of a bounded Lyapunov function; and then, in Chandler-Wilde, Heinemeyer and Potthast were able to show the same results held in 3D.

In other recent work \cite{chandmonk} Chandler-Wilde and Monk adopted a different approach to this problem. Using variational methods they were able to establish the well-posedness of the Dirichlet problem, in both 2 and 3 dimensions, for much more general boundaries: specifically, for small wavenumber $k$ they showed the problem to be well-posed for totally general boundaries, that were not the graphs of functions and requiring no regularity; and for arbitrary wavenumber $k$ they established the same for those domains $D$ having the property that if $x \in D$ then so too is every point above $x$. To prove their results they first reformulated their boundary value problem as an equivalent variational problem on a strip; they then analyzed the variational problem and made use of the Lax-Milgram and generalised Lax-Milgram theorem of Babu\v{s}ka, to prove well-posedness. A key ingredient in their proofs was the derivation of an a priori bound on the solution.   

The purpose then of this first chapter is to extend these results to the problem of scattering by rough inhomogeneous layers; indeed we will make use of a lot of the results contained in \cite{chandmonk} and mimic their methods throughout.  

We should mention some papers, that are related, in terms of the methods they employ, to the paper of Chandler-Wilde and Monk. These are \cite{kirschipmp93} by Kirsch, and \cite{els02} by Elschner and Yamamoto, who study the Dirichlet problem; and the papers \cite{Bonnetmmas94} of Bonnet-Bendhia and Starling and \cite{Szembergmmas98} of Szemberg-Strycharz who study the diffraction grating or transmission problem as we've called it here. In all of these papers a variational approach is used. All of the authors begin by reducing their problems to a variational problem on a strip; however the assumption made in all of these papers, that the scattering surface/diffraction grating is periodic and that the source $g$ is quasi-periodic, leads to a variational problem over a bounded region, so that compact embedding arguments can be applied and the sesquilinear form that arises satisfies a G\aa rding inequality which simplifies the mathematical arguments. However we should say that the approach adopted in  \cite{chandmonk} -- and the one that we adopt here -- is very similar to the one adopted in \cite{kirschipmp93}, \cite{els02},  \cite{Bonnetmmas94} and \cite{Szembergmmas98}; in particular the use of the Dirichlet to Neumann map (see (\ref{Tdef}) and (\ref{DtN}) below) and the exploitation of its properties was done in \cite{kirschipmp93}, \cite{els02},  \cite{Bonnetmmas94} and \cite{Szembergmmas98}. 

 An attractive feature of our results and indeed of those in \cite{chandmonk} is the explicit bounds we obtain on the solution in terms of the data $g$, which exhibit explicitly dependence of constants on the wave number and on the geometry of the domain. Our methods of argument to obtain these bounds are inspired in part by the work of Melenk \cite{Melenk}, Cummings and Feng \cite{cumfeng}, Feng and Sheen \cite{fengsheen} and Chandler-Wilde and Monk \cite{chandmonk}. 
 
Finally we should also discuss the paper \cite{chandprs98} of Chandler-Wilde and Zhang; the authors here deal with the same problem as we deal with here, (i.e.\ the wavenumber $k$ varies throughout the domain) and employ an integral equation approach to establish well-posedness in 2D when the surface is flat so that the domain is a half-space. This paper would appear to present the best results on this problem to date. We should point out in what ways our results are an improvement on these: 
\newline 1) Our results work in both 2 and 3 dimensions; 
\newline 2) our rough surfaces are much more general: Specifically, when the maximal value of $k$ is small or $k$ has strictly positive imaginary part we show the problem to be well-posed for totally general boundaries, that are constrained to lie in a strip and which require no regularity; and for arbitrary $k \in L^{\infty}(D)$ subject to assumption 1 (see below) we establish well-posedness for those domains $D$ having the property that if $x \in D$ then so too is every point above $x$;
\newline 3) the assumption (see assumption 1 below) that we make in order to prove theorem \ref{th_main1} is slightly more general than the assumption 2.4 that is used in \cite{chandprs98}.     
 
Finally we should mention some other spin-off papers of \cite{chandmonk}: these are \cite{chandmonk2} in which the same authors apply similar methods to scattering by bounded obstacles; and \cite{HH} in which Claeys and Haddar use the same approach to tackle the problem of scattering from infinite rough tubular surfaces. 

Note that the results contained in the section `$V_H$-ellipticity of the sesqulinear form' were presented at the Waves 2005 Conference.  

\section{The boundary value problem and variational formulation}\label{varf}

%\begin{figure}
%\begin{center}
%\includegraphics[width=22cm,height=19cm]{paperpic2.jpg}
%\end{center}
%\caption{The problem and geometry considered: Note that $\Delta u + k^2 u =g$
% in $D$, with $k(x)=k_+$, $g=0$, for $x_n>H$ i.e. in the half space $U_H$ above $\Gamma_H$.} 
%\end{figure}

In this section we introduce the boundary value problem and its equivalent
variational formulation that will be analyzed in later sections. For
$x=(x_1,\ldots,x_n)\in\real^n$ ($n=2,3$) let
$\tx=(x_1,\ldots,x_{n-1})$ so that $x=(\tx,x_n)$.  For $H\in \real$
let $U_H=\left\{x\;:\; x_n>H\right\}$ and
$\GH=\left\{x\;:\;x_n=H\right\}$.  Let $D\subset \real^n$ be a
connected open set such that for some constants $f_-<f_+$ it holds
that
\begin{equation}
U_{f_+}\subset D \subset U_{f_-},
\end{equation}
and let $\Gamma=\partial D$ denote the boundary of $\partial D$. The variational
problem will be posed on the open set
$S_H:=D\setminus\overline{U}_H$, for some $H\ge f_+$, and we denote
the unit outward normal to $S_H$ by $\nu$.

Let $H_0^1(D)$
denote the standard Sobolev space, the completion of $C_0^\infty(D)$
in the norm $\|\cdot\|_{H^1(D)}$ defined by 
$$\Vert
u\Vert_{H^1(D)}=\left\{\int_{D}(|\nabla u|^2+|u|^2)dx\right\}^{1/2}.
$$ 
The main
function space in which we set our problem will be the Hilbert space
$V_H$, defined, for $H\ge f_+$, by $V_H:=\{\phi|_{S_H}\;:\;\phi\in
H_0^1(D)\}$, on which we will impose a wave number dependent scalar
product $(u,v)_{V_H} := \int_{S_H} (\nabla u \cdot \overline{\nabla
v} +k_+^2 u\bar v)\,dx$ and norm,
 $\Vert u\Vert_{V_H}=\{\int_{S_H}(|\nabla
u|^2+k_+^2|u|^2)dx\}^{1/2}$.

Recalling the basic model from chapter 1, we will make the assumption that the variation in $k$ is confined to a neighbourhood of the boundary. The following then is our exact formulation:

{\sc The Boundary Value Problem.} {\em Given $g\in L^2(D)$, and $k
\in L^{\infty}(D)$ such that for some $H\ge f_+$, it holds that the
support of $g$ lies in $\overline{S_H}$, and that $k(x)=k_+$,
$x\in\overline{U_H}$, for some $k_+>0$,
 find $u:D\to \complex$ such that $u|_{S_a}\in V_a$ for every $a>f_+$,
$$
\Delta u+k^2u=g \mbox{ in } D
$$ 
in a distributional sense, and the
radiation condition (\ref{uprcstar}) holds, with
$F_H=u|_{\Gamma_H}$.}
\begin{remark}
We note that, as one would hope, the solutions of the above
problem do not depend on the choice of $H$. Precisely, if $u$ is a
solution to the above problem for one value of $H\ge f_+$ for
which $\supp g\subset \overline{S_H}$ and $k(x)=k_+$, $x \in \overline{U_H}$ then $u$ is a solution for
all $H\ge f_+$ with this property. To see that this is true is a
matter of showing that, if (\ref{uprcstar}) holds for one $H$ with
$\supp g\subset \overline{S_H}$ and $k(x)=k_+$, $x \in \overline{U_H}$ 
then (\ref{uprcstar}) holds for
all $H$ with this property. It is shown in Lemma \ref{lemma3p2}
below that if (\ref{uprcstar}) holds, with $F_H=u|_{\Gamma_H}$,
for some $H\ge f_+$, then it holds for all larger values of $H$.
One way to show that (\ref{uprcstar}) holds also for every smaller
value of $H$, $\tilde H$ say, for which $\tilde H\ge f_+$ and
$\supp g\subset \overline{S_{\tilde H}}$ and $k(x)=k_+$, $x \in \overline{U_{\tilde{H}}}$, is to consider the
function
\begin{eqnarray*}
v(x) & := & u(x) -\\
&&\frac{1}{(2\pi)^{(n-1)/2}}\int_{\real^{n-1}}\exp(\ri[(x_n-\tilde
H)\sqrt{k_+^2-\xi^2}+ \tx\cdot\xi])\hat F_{\tilde H}(\xi)\,d\xi,\;\;
x\in U_{\tilde H},
\end{eqnarray*}
with $F_{\tilde H}:=u|_{\Gamma_{\tilde H}}$, and show that $v$ is
identically zero. To see this we note that, by Lemma
\ref{lemma3p2}, $v$ satisfies the above boundary value problem
with $D=U_{\tilde H}$ and $g=0$. That $v\equiv 0$ then follows
from Theorem \ref{th_main1} below.
\end{remark}

%\begin{remark}
%We note that, as one would hope (see \cite[Remark 2.1]{chandmonk}), the solutions of the above
%problem do not depend on the choice of $H$. Precisely, if $u$ is a
%solution to the above problem for one value of $H\ge f_+$ for
%which $\supp g\subset \overline{S_H}$ and $k=k_+$ in $\overline{U_H}$, then $u$ is a solution %for
%all $H\ge f_+$ with this property.
%\end{remark}

We should give some motivation for looking for a solution $u$ such that $u|_{S_a}\in V_a$ for every $a>f_+$. In the special case when the boundary is flat, so that $D$ is a half-space, $D=U_0$ say,  $k$ is constant and $g$ is smooth and has compact support we can explicitly construct the solution. Suppose $n=3$. Then if $\delta_y$ denotes a point source located at  $y= (\tilde y,y_3) \in D$, with $y_3>0$, then a solution to the problem, find $u :D \to \mathbb{C}$, such that 
  $$
  \Delta u + k^2u=\delta_y \mbox{ in } D,
  $$ 
and $u = 0 \mbox{ on } \Gamma$, is given by
%\vspace*{5ex}

%\setlength{\unitlength}{0.04in}
%\begin{picture}(90,30)(-50,6)

%\color{black} \multiput(-42,32)(2,-2){5}{\line(1,1){10}}
%\put(-39,39){\vector(1,-1){14}} \put(-24,24){$u^i$, incident wave}

%\put(0,35){\large $\Delta u + k^2u=-\delta_y$} \put(13,22){\large
%$u=0$}

%\put(28,15){$\textcolor{blue}{\partial D = \Gamma_h}$}
%\put(28,30){$\textcolor{blue}{D = U_h\subset \R^n}$}

%\thicklines
% \color{blue} \put(-50,18){\line(1,0){100}}

%\color{red} \put(-15,28){\circle*{1}} \put(-13,28){$y$}
%\put(-15,8){\circle*{1}} \put(-13,8){$y^\prime$}

%\color{green} \thinlines \put(-35,-2){\vector(1,0){22}}
%\put(-35,-2){\vector(0,1){12}} \thicklines
%\put(5,-2){\vector(0,1){5.8}}
% \put(-12.5,-3){$\tilde x = (x_1,...,x_{n-1})$}
%\put(-38,10.5){$x_n$}
%\color{black}
%\end{picture}
%\end{center}
$$
u(x) = G(x,y) := \Phi(x,y)-\Phi(x,y^\prime),
$$
where $y'= (y_1,y_2,-y_3)$ is the reflection of $y$ in the boundary $\Gamma: =\{ (x\in \mathbb{R}^3: x_3=0\}$,
and $G$ is the \emph {Green's function} for $U_0$. Note that   
\begin{equation}\label{motivation_bd}
|u(x)| \le \,C\,\frac{(1+x_3)(1+y_3)}{|x-y|^{2}}
\end{equation}
for some $C>0$, (see for example chapter 5, (\ref{Gbound})).

%\begin{slide}
%\begin{center}
%\framebox[1.05\width]{\bf The Scattering Problem: Special Case}

%\vspace*{5ex}

%\setlength{\unitlength}{0.04in}
%\begin{picture}(90,30)(-50,6)

%\color{black} \multiput(-42,32)(2,-2){5}{\line(1,1){10}}
%\put(-39,39){\vector(1,-1){14}} \put(-24,24){$u^i$, incident wave}

%\put(0,35){\large 
Moreover, a solution to the problem, find $u:D \to \mathbb{C}$ such that  
$$
\Delta u + k^2u=g \mbox{  in  } D,
$$ %\put(13,22){\large $u=0$}
and 
$$
u=0 \mbox{  on  } \Gamma,
$$
is, for compactly supported and smooth $g\in L^2(D)$, given by
$$
u(x) = \int_D G(x,y)g(y)dy, \quad x\in D.
$$
%with $G$ the {\bf Green's function} for $U_h$. 
It follows from the bound (\ref{motivation_bd})
%$$
%|u(x)| \le \,C\,\frac{x_n}{(1+|x|)^{(n+1)/2}},
%$$
that $u\in L^2(S_H)$ for every $H>0$, where $S_H := D\setminus
\overline{U_H}$ (one can deduce this by using the techniques of section 5.4 of chapter 5, for example). Further, by an application of Green's theorem it follows also that
$u\in H^1(S_H)$, for every $H>f_+$. This motivates, that in the general case, we seek a solution such that $u|_{S_a} \in V_a$, for all $a>f_+$. 

\bigskip

We now derive a variational formulation of the boundary value
problem above. To derive this alternative formulation we require a
preliminary lemma. In this lemma and subsequently throughout the thesis, we use standard
fractional Sobolev space notation, except that we adopt a wave
number dependent norm, equivalent to the usual norm, and reducing
to the usual norm if the unit of length measurement is chosen so
that $k_+=1$. Thus, identifying $\Gamma_H:=\{x:x_n = H\}$ with $\real^{n-1}$,
$H^s(\Gamma_H)$, for $s\in \real$, denotes the completion of
$C_0^\infty(\Gamma_H)$ in the norm $\|\cdot\|_{H^s(\Gamma_H)}$
defined by
$$
 \|\phi\|_{H^s(\Gamma_H)} = \left(
\int_{\real^{n-1}}(k_+^2+\xi^2)^s|{\cal F} \phi(\xi)|^2\,
d\xi\right)^{1/2}.
$$
  We recall \cite{adamsSS} that, for all $a>H\ge f_+$, there
exist continuous embeddings $\gamma_+:H^1(U_H\setminus U_a)\to
H^{1/2}(\Gamma_H)$ and $\gamma_-:V_H\to H^{1/2}(\Gamma_H)$ (the
trace operators) such that $\gamma_\pm \phi$ coincides with the
restriction of $\phi$ to $\Gamma_H$ when $\phi$ is $C^\infty$. In
the case when $H=f_+$, when $\Gamma_H$ may not be the boundary of
$S_H$ (if part of $\partial D$ coincides with $\Gamma_H$) we
understand this trace by first extending $\phi\in V_H$ by zero to
$U_{f_-}\setminus \overline U_{f_+}$. We recall also that, if $u_+\in
H^1(U_H\setminus U_{a})$, $u_-\in V_H$, and $\gamma_+u_+ =
\gamma_- u_-$, then $v\in V_a$, where $v(x) := u_+(x)$, $x \in
U_H\setminus U_a$, $:= u_-(x)$, $x \in S_H$. Conversely, if $v\in
V_a$ and $u_+:= v|_{U_H\setminus U_a}$, $u_-:= v|_{S_H}$, then
$\gamma_+u_+=\gamma_-u_-$. We introduce the operator $T$, which
will prove to be a Dirichlet to Neumann map on $\Gamma_H$, (see (\ref{DtN}) below), defined
by
\begin{equation}
T:=\cF^{-1}M_z\cF, \label{Tdef}
\end{equation}
where $M_z$ is the operation of multiplying by
\[
z(\xi):=\left\{\begin{array}{ll}-\ri\sqrt{k_+^2-\xi^2}&\mbox{if
}|\xi|\le k_+,
\\[3pt]
\sqrt{\xi^2-k_+^2}&\mbox{for }|\xi|>k_+.
\end{array}\right.
\]
We shall prove shortly in Lemma \ref{LTB} that
$T:H^{1/2}(\Gamma_H) \to H^{-1/2}(\Gamma_H)$ and is bounded.
We now state lemma 2.2 from \cite{chandmonk}. For completeness we include the proof.
\begin{lemma} \label{lemma3p2}
If (\ref{uprcstar}) holds, with $F_H\in H^{1/2}(\Gamma_H)$, then
$u\in H^1(U_H\setminus U_a)\cap C^2(U_H)$, for every $a>H$,
$$
\Delta u + k_+^2 u = 0 \mbox{ in } U_H,
$$
$\gamma_+u = F_H$, and
\begin{equation} \label{exact_rc}
\int_{\Gamma_H}  \bar vT\gamma_+u\,ds + k_+^2\int_{U_H} u\bar v \,dx
- \int_{U_H} \nabla u\cdot \nabla \bar v\, dx = 0, \quad v \in
C_0^{\infty}(D).
\end{equation}
Further, the restrictions of $u$ and $\nabla u$ to $\Gamma_a$ are
in $L^2(\Gamma_a)$, for all $a> H$, and
\begin{equation} \label{lem32}
\int_{\Gamma_a} \left[ \left|\frac{\partial u}{\partial
x_n}\right|^2 - \left|\nabla_{\tilde x} u\right|^2 +
k_+^2|u|^2\right]\,ds \le 2k_+\Im \int_{\Gamma_a} \bar u
\frac{\partial u}{\partial x_n}\, ds.
\end{equation}
Moreover, for all $a>H$, where $F_a\in H^{1/2}(\Gamma_a)$ denotes
the restriction of $u$ to $\Gamma_a$, (\ref{uprcstar}) holds with
$H$ replaced by $a$.
\end{lemma}
\begin{proof}
If $F_H\in L^2(\Gamma_H)$ then, as a function of $\xi$,
$\exp(i[(x_n-H)\sqrt{k_+^2-\xi^2}$ $+ \tx\cdot\xi])\hat
F_H(\xi)(1+\xi^2)^s\in L^1(\real^{n-1})$ for every $x\in U_H$ and
$s\ge 0$. It follows that (\ref{uprcstar}) is well-defined for
every $x\in U_H$, and that $u\in C^2(U_H)$, with all partial
derivatives computed by differentiating under the integral sign,
so that $\Delta u+k_+^2 u = 0$ in $U_H$. Thus, for $a>H$ and almost
all $\xi \in \real^{n-1}$,
\begin{eqnarray} \label{ft_ugama}
{\cal F}(u|_{\Gamma_a})(\xi) & = &
\exp(\ri(a-H)\sqrt{k_+^2-\xi^2}\,)\hat F_H(\xi),\\
\label{ft_unorm} {\cal F}\left(\left.\frac{\partial u}{\partial
x_n}\right|_{\Gamma_a}\right)(\xi) & = &
\ri\sqrt{k_+^2-\xi^2}\exp(\ri(a-H)\sqrt{k_+^2-\xi^2}\,)\hat F_H(\xi),\\
\nonumber {\cal F}(\nabla_{\tx}u|_{\Gamma_a})(\xi) & = & \ri
\xi\exp(\ri(a-H)\sqrt{k_+^2-\xi^2}\,)\hat F_H(\xi).
\end{eqnarray}
Therefore, by the Plancherel identity (\ref{plancherel}),
$u|_{\Gamma_a}$, $\nabla u|_{\Gamma_a}\in L^2(\Gamma_a)$ with
$$
\int_{\Gamma_a}|u|^2 ds = \int_{\real^{n-1}}
|\exp(2\ri(a-H)\sqrt{k_+^2-\xi^2}\,)|\; |\hat F_H(\xi)|^2\, d\xi \le
\int_{\Gamma_H}|F_H|^2\, ds
$$
and
\begin{equation} \label{nabu_bound1}
\hspace*{2ex}\int_{\Gamma_a}|\nabla u|^2 ds \le
\int_{\real^{n-1}}[|k_+^2-\xi^2|+\xi^2]
|\exp(2\ri(a-H)\sqrt{k_+^2-\xi^2}\,)|\; |\hat F_H(\xi)|^2\, d\xi,
\end{equation}
while
$$
\int_{\Gamma_a} \left[\left|\frac{\partial u}{\partial
x_n}\right|^2 -|\nabla_{\tx} u|^2 + k_+^2|u|^2\right]\,ds = 2
\int_{|\xi|<k_+} (k_+^2-\xi^2)|\hat F_H(\xi)|^2\,d\xi
$$
and
$$
\Im \int_{\Gamma_a} \bar u \frac{\partial u}{\partial x_n}\,ds =
\int_{|\xi|<k_+} \sqrt{k_+^2-\xi^2}\,|\hat F_H(\xi)|^2\,d\xi.
$$
Thus (\ref{lem32}) holds and
\begin{equation} \label{u_bound}
\int_{U_H\setminus U_a}|u|^2\, dx \le
(a-H)\int_{\Gamma_H}|F_H|^2\,ds.
\end{equation}
Further, from (\ref{nabu_bound1}) it follows that
\begin{eqnarray} \nonumber
\int_{U_H\setminus U_a}|\nabla u|^2\, dx & \le &
(a-H)k_+^2\int_{|\xi|<k_+}|\hat F_H(\xi)|^2\,d\xi + \\ \nonumber & &
\int_{|\xi|>k_+}\xi^2
\frac{1-\exp(-2[a-H]\sqrt{\xi^2-k_+^2}\,)}{\sqrt{\xi^2-k_+^2}}|\hat
F_H(\xi)|^2\,d\xi \\ \nonumber
%& \le & 2(a-H)k^2\int_{\xi^2<2k^2}|\hat F_H(\xi)|^2\,d\xi + \\
%\nonumber & & \sqrt{2}\int_{\xi^2>2k^2} |\xi| |\hat
%F_H(\xi)|^2\,d\xi \\
\label{nabu_bound2} & \le & \int_{\real^{n-1}}
(2(a-H)k_+^2+\sqrt{2}|\xi|)|\hat F_H(\xi)|^2 \, d\xi,
\end{eqnarray}
since $1-\re^{-z} \le z$ for $z\ge 0$ and $\sqrt{\xi^2-k_+^2}\ge
|\xi|/\sqrt{2}$ for $\xi^2\ge 2k_+^2$. Thus $u\in H^1(U_H\setminus
U_a)$ if $F_H\in H^{1/2}(\Gamma_H)$. That $u|_{\Gamma_H} = F_H$ is
clear when $F_H\in C_0^\infty(\Gamma_H)$, and $\gamma_+u = F_H$ for
all $F_H\in H^{1/2}(\Gamma_H)$ follows from the continuity of
$\gamma_+$, (\ref{u_bound}) and (\ref{nabu_bound2}), and the density
of $C_0^\infty(\Gamma_H)$ in $H^{1/2}(\Gamma_H)$. Similarly, in the
case that $F_H\in C_0^\infty(\Gamma_H)$ so that $u\in
C^\infty(\overline{U_H})$, it is easily seen that 
\begin{equation}\label{DtN}
T\gamma_+u =
-\partial u/\partial x_n|_{\Gamma_H}
\end{equation}
 and (\ref{exact_rc}) follows
by Green's theorem. The same equation for the general case follows
from the density of $C_0^\infty(\Gamma_H)$ in $H^{1/2}(\Gamma_H)$,
(\ref{u_bound}) and (\ref{nabu_bound2}) and the continuity of the
operator $T$.

That (\ref{uprcstar}) holds with $H$ replaced by $a$, for all $a>H$,
is clear from (\ref{ft_ugama}).
\end{proof}

Now suppose that $u$ satisfies the boundary value problem. Then
$u|_{S_a}\in V_a$ for every $a >f_+$ and, by definition, since
$\Delta u+k^2u=g$ in a distributional sense,
\begin{equation} \label{he_ds}
\int_D[g\bar v + \nabla u\cdot\nabla \bar v - k^2 u\bar v] dx = 0,
\quad v\in C_0^\infty(D).
\end{equation}
Applying Lemma \ref{lemma3p2}, and defining $w := u|_{S_H}$, it
follows that
$$
\int_{S_H} [g\bar v + \nabla w\cdot \nabla \bar v - k^2 w\bar v]\,
dx + \int_{\Gamma_H} \bar v T\gamma_- w \, ds = 0, \quad v\in
C_0^\infty(D).
$$
From the denseness of $\{\phi|_{S_H}:\phi\in C_0^\infty(D)\}$ in
$V_H$ and the continuity of $\gamma_-$, it follows that this
equation holds for all $v\in V_H$.

Let $\|\cdot\|_2$ and $(\cdot,\cdot)$ denote the norm and scalar
product on $L^2(S_H)$, so that $\Vert
v\Vert_2=\sqrt{\int_{S_H}|v|^2\,dx}$ and
$(u,v)=\int_{S_H}u\overline{v}\,dx$, and define the sesquilinear
form $b:V_H\times V_H\to \complex$ by
\begin{equation} \label{sesqui}
b(u,v) = (\nabla u,\nabla v) - (k^2u,v) + \int_{\GH} \gamma_-\bar v
T\gamma_-u \, ds.
\end{equation}
Then we have shown that if $u$ satisfies the boundary value
problem then $w:= u|_{S_H}$ is a solution of the following
variational problem: find $u\in V_H$ such that
\begin{equation} \label{weak_form}
b(u,v) = -(g,v), \quad v\in V_H.
\end{equation}

Conversely, suppose that $w$ is a solution to the variational
problem and define $u(x)$ to be $w(x)$ in $S_H$ and to be the
right hand side of (\ref{uprcstar}), with $F_H:= \gamma_- w$, in
$U_H$. Then, by Lemma \ref{lemma3p2}, $u\in H^1(U_H\setminus U_a)$
for every $a>H$, with $\gamma_+u = F_H = \gamma_-w$. Thus
$u|_{S_a}\in V_a$, $a\ge f_+$. Further, from (\ref{exact_rc}) and
(\ref{weak_form}) it follows that (\ref{he_ds}) holds, so that
$\Delta u + k^2u = g$ in $D$ in a distributional sense. Thus $u$
satisfies the boundary value problem.

We have thus proved the following theorem.
\begin{theorem} \label{th_equiv} If $u$ is a solution of the boundary value problem then
$u|_{S_H}$ satisfies the variational problem.  Conversely, if $u$
satisfies the variational problem, $F_H := \gamma_- u$, and the
definition of $u$ is extended to $D$ by setting $u(x)$ equal to the
right hand side of (\ref{uprcstar}), for $x\in U_H$, then
 the extended function satisfies the boundary value problem, with $g$ extended by zero from $S_H$ to $D$ and $k$ extended from $S_H$ to $D$ by taking the value $k_+$ in $\overline{U_H}$.
\end{theorem}
It remains to prove the mapping properties of $T$.
\begin{lemma}\label{LTB}
The Dirichlet-to-Neumann map $T$ defined by (\ref{Tdef}) is a
bounded linear map from $H^{1/2}(\Gamma_H)$ to
$H^{-1/2}(\Gamma_H)$, with $\|T\|=1$.
\end{lemma}

\begin{proof} From the definitions of $T$ and the Sobolev norms we see that, as a map from
$H^{1/2}(\Gamma_H)$ to $H^{-1/2}(\Gamma_H)$,
\begin{equation} \label{Tnorm}
\Vert T\Vert =\max_{\xi\in
\real^{n-1}}\frac{|\sqrt{k_+^2-\xi^2}|}{|\sqrt{k_+^2+\xi^2}|} =1.
\end{equation}
\end{proof}
\section{$V_H$-Ellipticity of the sesquilinear form}\label{lowk}
In this section we shall investigate under what conditions the sesquilinear form $b$ is
$V_H$-elliptic (we shall give explicit restrictions on  $k \in
L^{\infty}(S_H)$ to guarantee this). From the point of view of
numerical solution by e.g.\ finite element methods, the ellipticity
we establish is of course highly desirable, guaranteeing, by C\'ea's
lemma, unique existence and stability of the numerical solution
method.

Let $V_H^*$ denote the dual space of $V_H$, i.e. the space of
continuous anti-linear functionals on $V_H$. Then our analysis will
also apply to the following slightly more general problem: given
$\mathcal{G}\in V_H^*$ find $u\in V_H$ such that
\begin{equation} \label{var_prob2}
b(u,v) = \mathcal{G}(v), \quad v\in V_H.
\end{equation}
%We now introduce some extra assumptions on $k \in L^{\infty}(D)$, that will be needed in order to show that the boundary value problem is well-posed.
It will be assumed in the remainder of this chapter that $k\in
L^\infty(D)$ satisfies that $\Re(k^2)\geq 0$, $\Im(k^2)\geq 0$, which
is certainly the case in the electromagnetic case where $k^2$ is
given by (\ref{k_def_em}), i.e.
\[
k^2 = \omega^2\mu\epsilon[1+\ri\sigma/(\omega\epsilon)].
\]
Under these assumptions there exist constants $k_\infty\ge k_0 \ge
0$ and $\theta \in [0,\pi/2]$ such that
$$
k_0 \le |k(x)|\le k_\infty, \quad \arg (k^2(x)) \ge \theta,
$$
for almost all $x\in S_H$. It is convenient to introduce the dimensionless parameters
\[
\kappa_{\infty}:=k_{\infty}(H-f_-), \quad \kappa_0 := k_0(H-f_-), \mbox{ and } \kappa_+:=k_+(H-f_-).
\]

We shall prove the following theorem.
\begin{theorem}\label{thm1}
Suppose that either $ \kappa_{\infty} <\sqrt{2}$ or $\theta>0$. Then, for some constant $\alpha>0$,
$$
|b(u,u)| \ge \alpha ||u||^2_{V_H}, \quad u\in V_H,
$$
so that the variational problem (\ref{var_prob2}) is
uniquely solvable. Moreover, the solution satisfies the estimate
\begin{equation} \label{ap_2}
\Vert u\Vert_{V_H}\leq C \Vert \mathcal{G}\Vert_{V_H^*}
\end{equation}
where $C:= \alpha^{-1}$ satisfies $C\leq (2+\kappa_+^2)/(2-
\kappa_{\infty}^2)$ if $\kappa_\infty<\sqrt{2}$, and satisfies $C
\le \csc \theta (1+\kappa_+^2/\max(2,\kappa_0^2))$ if $\theta>0$. In
particular, the scattering problem (\ref{weak_form}) is uniquely
solvable and the solution satisfies the bound
\begin{equation} \label{apriori2_2}
k_+\Vert u\Vert_{V_H}\leq
\frac{\kappa_+}{\sqrt{2}}\,C\Vert
g\Vert_{2}.
\end{equation}
\end{theorem}

We begin by recalling some results from \cite{chandmonk}; namely a trace theorem and a Friedrich's inequality, that are needed to prove Theorem \ref{thm1}.

\begin{lemma} \label{WL3} For all $u\in V_H$,
$$
\|\gamma_- u\|_{H^{1/2}(\Gamma_H)} \le
\|u\|_{V_H} \mbox{ and }
\|u\|_2 \leq\frac{H-f_-}{\sqrt{2}} \left\| \frac{\partial
u}{\partial x_n}\right\|_2.
$$
\end{lemma}
We next state another lemma from \cite{chandmonk} whose proof we include for completeness.
\begin{lemma}\label{WL1}
For all $\phi, \psi \in H^{1/2}(\GH)$,
$$
\int_{\GH} \phi T \psi ds = \int_{\GH} \psi T \phi ds.
$$
 For all $\phi\in H^{1/2}(\GH)$,
\[
\Re\int_{\GH}\bar\phi\,T\phi\,ds\ge 0, \quad   \Im\int_{\GH}\bar\phi\,T\phi\,ds\leq 0.
\]
\end{lemma}
\begin{proof}  Let $\hphi=\cF\phi$, $\hat \psi = \cF\psi$.  Then
$\cF(T\phi)=z\hphi$. Thus, using the Plancherel identity
(\ref{plancherel}) and since ${\hat{\bar \psi}}(\xi) =
\overline{\hat\psi(-\xi)}$ and $z$ is even,
\begin{eqnarray*}
\int_{\GH}\psi\,T\phi\,ds=
\int_{\real^{n-1}}\hpsi(-\xi)z(\xi)\hphi(\xi)\,d\xi=
\int_{\real^{n-1}}\hpsi(\xi)z(\xi)\hphi(-\xi)\,d\xi=
\int_{\GH}\phi\,T\psi\,ds.
\end{eqnarray*}
In particular, putting $\psi = \bar\phi$,
\begin{eqnarray*}
\int_{\GH}\bar\phi\,T\phi\,ds&=&
\int_{\real^{n-1}}z(\xi)|\hphi(\xi)|^2\,d\xi\\
&=&
\int_{|\xi|>k}\sqrt{\xi^2-k^2}|\hphi(\xi)|^2\,d\xi-\ri\int_{|\xi|<k}
\sqrt{k^2-\xi^2}|\hphi(\xi)|^2\,d\xi,
\end{eqnarray*}
from which the second result follows.
%To prove the second result,
%note that if
% \[
%\Im\int_{\GH}\bar\phi T\phi\,ds=0
%\]
% then the above equality shows that
%$\hat{\phi}(\xi)=0$ for almost every $\xi$ with $|\xi|<k$ and further
%\begin{eqnarray*}
%\int_{\GH}\bar\psi\,T\phi\,ds&=&
%\int_{|\xi|>k}\sqrt{\xi^2-k^2}\,\overline{\hpsi}(\xi)\,\hphi(\xi)\,d\xi\\
%&=&\overline{\int_{|\xi|>k}\sqrt{\xi^2-k^2}
%\,\hpsi(\xi)\,\overline{\hphi}(\xi)\,d\xi}\\
%&=&\overline{\int_{\GH}\bar\phi\,T\psi\,ds},
%\end{eqnarray*}
%so proving the second assertion of the lemma.
\end{proof}

The above lemma implies that %$b(\cdot,\cdot)$ is bounded, giving an
%explicit value for the bound, and that 
$b(\cdot,\cdot)$ has the
following important symmetry property.

\begin{corollary} \label{symmetry}
For all $u,v\in V_H$,
$
b(v,u) = b(\bar u, \bar v).
$
\end{corollary}
We are now in a position to show that the sesquilinear form is bounded, establishing an explicit value for the bound. 
\begin{lemma}\label{WL4}
For all $u,v\in V_H$,
\[
\hspace*{-4ex} |b(u,v)|\leq \left[
\frac{k^2_{\infty}}{k_+^2}%\right\}
+ 1\right] \Vert u\Vert_{V_H}\Vert v \Vert_{V_H}
\]
so that the sesquilinear form $b(.,.)$ is bounded.
\end{lemma}

\begin{proof}
From the definition of the sesquilinear form $b(.,.)$ and the
Cauchy-Schwarz inequality we have
\[
|b(u,v)|\leq \|\nabla u\|_2\|\nabla v\|_2+\frac{k^2_{\infty}k_+^2}{k_+^2}\|u\|_2\|v\|_2 +
\|\gamma_-u\|_{H^{1/2}(\Gamma_H)}
\|T\|\,\|\gamma_-v\|_{H^{1/2}(\Gamma_H)}.
\]
Using %(\ref{LTB}), 
(\ref{Tnorm}), and Lemma
\ref{WL3} we obtain the desired estimate.
\end{proof}

Our last lemma of this section shows that the sesquilinear form $b(.,.)$ is
$V_H$-elliptic provided that $\kappa_{\infty}$ is not too large or $\arg(k^2)$ is strictly positive.

\begin{lemma}\label{WL5}
i)
For all $u\in V_H$,
\[
|b(u,u)| \ge \frac{2- \kappa^2_{\infty}}{2 +\kappa_+^2}\Vert
u\Vert_{V_H}^2.
\]
ii)If $\theta> 0$ then, for all $u\in V_H$,
\[
|b(u,u)| \ge \frac{\sin\theta}{1+\kappa_+^2/\max(2,\kappa_0^2)} \Vert u\Vert_{V_H}^2.
\]
\end{lemma}

\begin{proof}
i) By Lemma \ref{WL1}, $\Re \,b(u,u)\ge \Vert
u\Vert_{V_H}^2-k_+^2\Vert u\Vert_{2}^2- k^2_{\infty}\Vert
u\Vert_{2}^2$. The result follows from Lemma \ref{WL3} which implies
that $\|u\|_{V_H}^2 \ge k_+^2(2/\kappa_+^2+1)\|u\|_2^2$.

ii) Choose $\alpha\ge 0$ and define $\beta\in (0,\theta]$ by
$$
\tan\beta = \frac{\sin\theta}{\alpha+\cos\theta},
$$
so that $\alpha\sin\beta = \sin(\theta-\beta)$ and
$$
\sin\beta = \frac{\sin\theta}{\sqrt{\alpha^2+2\alpha\cos\theta +1}} \ge \frac{\sin\theta}{1+\alpha}.
$$
Then, by Lemma \ref{WL1}, and since $\pi/2 -\beta  \in [0,\pi/2),$
\[
\Re \left(\re^{i(\pi/2-\beta )}\int_{\Gamma_H}\gamma_-\bar uT\gamma_-uds\right)\geq 0.
\]
Hence
\begin{eqnarray*}
R & := &\Re \left(\re^{i(\pi/2-\beta )} b(u,u)\right) \ge \sin\beta\Vert \nabla u\Vert_2^2 +
\int_{S_H}\sin(\arg(k^2)-\beta)|k^2|\,|u|^2dx \\
&\ge&
\sin\beta\Vert \nabla u\Vert_2^2 + \sin(\theta-\beta)\frac{k_0^2}{k_+^2}k_+^2||u||_2
= \sin\beta \left(\Vert\nabla u\Vert_2^2 + \alpha\frac{k_0^2}{k_+^2}k_+^2||u||^2_2\right).
\end{eqnarray*}
Thus, and by Lemma \ref{WL3}, for $0\le \gamma\le 1$,
$$
R \ge \sin\beta\,\left(\gamma\,\Vert \nabla u\Vert_2^2 + \frac{2(1-\gamma) +\alpha\kappa_0^2}{\kappa_+^2}\;k_+^2||u||_2\right).
$$

Choosing first $\gamma=1$ and $\alpha=\kappa_+^2/\kappa_0^2$, we see that
$$
R \ge \sin\beta \Vert u\Vert_{V_H}^2\ge \frac{\sin\theta}{1+\kappa_+^2/\kappa_0^2}\,\Vert u\Vert_{V_H}^2.
$$
Alternatively, choosing $\gamma = 2/(2+\kappa_+^2)$ and $\alpha=0$, so that $\beta=\theta$, we see that
$$
R \ge \frac{\sin\theta}{1+\kappa_+^2/2}\,\Vert u\Vert_{V_H}^2.
$$

\end{proof}

Theorem \ref{thm1} now follows from Lemmas \ref{WL4} and \ref{WL5}
and the Lax-Milgram lemma. The final bound (\ref{apriori2_2}) is a
consequence of Lemma \ref{WL3} which implies, in the
 particular case that $\mathcal{G}(v)
:= -(g,v)$, for some $g\in L^2(S_H)$, that
\[
\Vert \mathcal{G}\Vert_{V_H^*}=\sup_{v\in V_H}\frac{|(v,g)|}{\Vert
v\Vert_{V_H}} \leq\sup_{v\in V_H} \frac{\Vert v\Vert_{2}\Vert
g\Vert_{2}}{\Vert v\Vert_{V_H}} \leq \frac{H-f_-}{\sqrt{2}}\Vert
g\Vert_{2}.
\]

\section{Analysis of the variational problem at arbitrary frequency}
\label{hik}

%The sesquilinear form $b(.,.)$ is not $V_H$-elliptic if $k \in L^{\infty}(S_H)$ is large.  
In this section we will consider the case where there is no restriction on $ k_{\infty}$ and where $\theta$ may be identically zero. We do however impose some additional constraints on the vertical decay of $k \in L^{\infty}(D)$, and also on the domain. Under these assumptions we then prove that the boundary value problem and the equivalent
variational problem are uniquely solvable by using the generalized
Lax-Milgram theory of Babu\v{s}ka. 

The domains $D$ for which we will
establish this result are those which, in addition to our assumption
throughout that $U_{f_+}\subset D \subset U_{f_-}$, satisfy the
condition that
\begin{equation} \label{dom_cond}
x \in D \Rightarrow x+s e_n\in D, \mbox{ for all } s>0,
\end{equation}
where $e_n$ denotes the unit vector in the direction $x_n$.
Condition (\ref{dom_cond}) is satisfied if $\Gamma$ is the graph of
a continuous function, but certainly does not require that this be
the case. Nor does (\ref{dom_cond}) impose any regularity on
$\partial D$. 

In what follows we always assume that $k_0>0$, and moreover that $\Re(k^2)\geq k_0^2$. 
Recall that $H\geq f_+$ is such that the support of $g$ lies in $\overline{S_H}$ and such that $k=k_+$ in $\overline{U_H}$. We now state the assumption we make on the vertical decay of $k$ in addition to the assumptions that $k \in L^{\infty}(D)$, takes the value $k_+$ in $\overline{U_H}$, and satisfies $\Re(k^2)\geq k_0^2$:
\newline \textbf{Assumption 1.} There exist $0<\lambda_1<4/(H-f_-)^3$, $0\leq\lambda_2$ such that $k \in L^{\infty}(D)$ satisfies
\[
\Re(k^2(x)) = \pi(x) -\lambda_1 x_n - \int^{x_n}_{-\infty}\lambda_2\Im(k^2)(\tilde x,t) dt\quad \mbox{ for almost all }x \in D,
\]
where $\pi:D \to \mathbb{R}$ is monotonic non-decreasing, i.e.\ for all $h>0$,
\[
\mbox{ess}\inf_{x \in D}[\pi(x+e_nh)-\pi(x)]\geq 0.
\]
Here $\Im(k^2)$ is extended onto $\mathbb{R}^n$ by taking the value zero on $\mathbb{R}^n \backslash D$.
\begin{remark}
If assumption 1 holds and $k^2 \in C^1(D)$, then
\[
\frac{\partial \Re(k^2)}{\partial x_n} \geq -\lambda_1 -\lambda_2\Im(k^2) \quad \mbox{ in } D.
\]
\end{remark}
%In order to motivate the assumption that we will impose on $k \in L^{\infty}(D)$, we will consider a simple case when our boundary value problem is ill-posed. Consider the boundary value problem in 2D, with the boundary $\partial D$, being flat, so that $D$ is a half plane.
%Suppose that $\Im(k^2)=0$, that $k^2 \in C^1(D)$ such that $\partial k^2/\partial x_1 =0, \partial k^2/\partial x_2=-\lambda_1$.
%Then looking for solutions of 
\begin{remark} Assumption 1 can be justified to some extent: Suppose the domain $D$ to be two dimensional and the boundary $\Gamma$ to be flat. Let $k^2 \in C(\overline{D})$ be such that $\Im(k^2)=0$, such that $\partial k^2/\partial x_1 =0$ and $\partial k^2/\partial x_2= -\lambda_1$ where $\lambda_1> a^3/(H-f_-)^3$, with $-a\approx -1.987$ being the largest negative zero of the Airy function $Ai(z) + Bi(z)/{\sqrt{3}}$ (so that assumption 1 is violated, and note also $\kappa_{\infty}>\sqrt{2}$). In this case, using separation of variables, one can show that the boundary value problem is not well-posed.  
\end{remark}

%Finally let us add that from now on we always assume that $k_0>0$, so that $\Re(k^2)\geq k_0^2>0$. 
Our main result in this section is then the following:

\begin{theorem}\label{th_main1} If (\ref{dom_cond}) and Assumption 1 hold then the variational
problem (\ref{var_prob2}) has a unique solution $u\in V_H$ for every
${\cal G}\in V_H^*$ and
\begin{equation}
\Vert u\Vert_{V_H}\le C\Vert {\cal G}\Vert_{V_H^*}
\label{aprioriest}
\end{equation}
where
\[
C=  \left(1+k_0^{-1}\left[k_+ + \frac{k_{\infty}^2}{k_+}\right]\sqrt{[(\kappa_+^2 + \kappa_{\infty}^2)A^{-1}B+1]\kappa_0^2A^{-1}B}\right)
\]
where $A=2-\lambda_1(H-f_-)^3/2$ and $B= 2\kappa_++1+2\sqrt{2}+ \lambda_2(H-f_-) +2\kappa_{\infty}^2A^{-1}$.
In particular, the boundary value problem and the equivalent
variational problem (\ref{weak_form}) have exactly one solution, and
the solution satisfies the bound
\[
k_0\Vert u\Vert_{V_H}\leq  \sqrt{[(\kappa_+^2 + \kappa_{\infty}^2)A^{-1}B+1]\kappa_0^2A^{-1}B}\Vert g\Vert_{2}.
\]
\end{theorem}
%Moreover, we shall estimate all stability and inf-sup constants
%explicitly, as functions of the dimensionless parameter $\kappa$, so
%that the $k$-dependence of constants will be apparent.
%In order to motivate the assumption that we will impose on $k \in L^{\infty}(D)$, we will consider a simple case when our boundary value problem is ill-posed. Consider the boundary value problem in 2D, with the boundary $\partial D$, being flat, so that $D$ is a half plane.
%Suppose that $\Im(k^2)=0$, that $k^2 \in C^1(D)$ such that $\partial k^2/\partial x_1 =0, \partial k^2/\partial x_2=-\lambda_1$.
%Then looking for solutions of 
 
To apply the generalized Lax-Milgram theorem %(e.g. \cite[Theorem
%2.15]{ihlenburg}) 
we need to show that $b$ is bounded, which we
have done in Lemma \ref{WL4}; to establish the inf-sup condition %(\ref{infsup})
that
\begin{equation}
\alpha := \inf_{0\not=u\in V_H}\sup_{0\not=v\in
V_H}\frac{|b(u,v)|}{\Vert u\Vert_{V_H}\Vert v\Vert_{V_H}} >0;
\label{infsup_layer}
\end{equation}
and to establish the ``transposed'' inf-sup condition. %(\ref{trans_is}). 
It follows
easily from Corollary \ref{symmetry} that the transposed inf-sup
condition follows automatically if (\ref{infsup_layer}) holds.

%\begin{lemma}\label{WEXL6}
%If $v\in V_H$ and $b(v,v)=0$ then
%\begin{equation}
%b(u,v)=\overline{b(v,u)}, \quad u\in V_H.\label{WEXeqc}
%\end{equation}
%\end{lemma}

%\begin{proof} If $b(v,v)=0$ then $\Im b(v,v)=\Im
%\int_{\GH}\gamma_- \bar v\,T\gamma_- v\,ds=0$ and (\ref{WEXeqc})
%follows from Lemma \ref{WL1}.
%\end{proof}

\begin{lemma}\label{WEXL7} If (\ref{infsup_layer}) holds then, for all non-zero $v\in V_H$,
\[
\sup_{0\not=u\in V_H}\frac{|b(u,v)|}{\Vert u\Vert_{V_H}}>0.
\]
\end{lemma}

\begin{proof}
If (\ref{infsup_layer}) holds and $v\in V_H$ is non-zero then
\[
\sup_{0\not=u\in V_H}\frac{|b(u,v)|}{\Vert u\Vert_{V_H}}=
\sup_{0\not=u\in V_H}\frac{|b(\bar v,u)|}{\Vert u\Vert_{V_H}}\geq
\alpha \Vert v\Vert_{V_H}>0.
\]
This proves the lemma.
\end{proof}

The following result follows from \cite[Theorem 2.15]{ihlenburg}
and Lemmas \ref{WL4} and \ref{WEXL7}.

\begin{corollary} \label{cor_infsup} If (\ref{infsup_layer}) holds then
the variational problem (\ref{var_prob2}) has exactly one solution
$u\in V_H$ for all ${\cal G}\in V_H^*$. Moreover
$$
\|u\|_{V_H} \le \alpha^{-1}\|{\cal G}\|_{V_H^*}.
$$
\end{corollary}

%Our approach to establishing (\ref{infsup}) will be to first show
%that (\ref{infsup}) holds in the case that $\Gamma$ is the graph of
%a smooth function, given by (\ref{}) with $f\in BC(\real^{n-1})\cap
%C^\infty(\real^{n-1})$.
To show (\ref{infsup_layer}) we will establish an
a priori bound for solutions of (\ref{var_prob2}), from which the
inf-sup condition will follow by the following easily established
lemma (see \cite[Remark 2.20]{ihlenburg}).

\begin{lemma} \label{ihl_rem}
Suppose that there exists $C>0$ such that, for all $u\in V_H$ and
${\cal G}\in V_H^*$ satisfying (\ref{var_prob2}) it holds that
\begin{equation} \label{boundA}
\|u\|_{V_H} \le C \|{\cal G}\|_{V_H^*}.
\end{equation}
Then the inf-sup condition (\ref{infsup_layer}) holds with $\alpha\ge
C^{-1}$.
\end{lemma}

The following lemma reduces the problem of establishing
(\ref{boundA}) to that of establishing an a priori bound for
solutions of the special case (\ref{weak_form}).

\begin{lemma} \label{special_case_enough}
Suppose there exists $\tilde C>0$ such that, for all $u\in V_H$ and
$g\in L^2(S_H)$ satisfying (\ref{weak_form}) it holds that
\begin{equation} \label{boundB}
\|u\|_{V_H} \le k_0^{-1}\tilde C\, \|g\|_2.
\end{equation}
Then, for all $u\in V_H$ and ${\cal G}\in V_H^*$ satisfying
(\ref{var_prob2}), the bound (\ref{boundA}) holds with
$$
C  \le \left(1 +  k_0^{-1}\tilde C\left[k_+ + \frac{k_{\infty}^2}{ k_+}\right]\right).
%$\frac{k_+^2}{k_0^2}\left[1 + 2  \,\tilde C\frac{k_{\infty}^2}{k_0k_+}\right].
$$
\end{lemma}
\begin{proof}
 Suppose $u\in V_H$ is a solution of
\begin{equation}
b(u,v)={\cal G}(v),\quad v\in V_H, \label{uweak}
\end{equation}
where ${\cal G}\in V_H^*$.   Let $b_0:V_H\times V_H\to \C$ be
defined by
\[
b_0(u,v)=(\nabla u,\nabla
v)+k_+^2(u,v)+\int_{\Gamma_H}\gamma_-\overline{v}\,T\gamma_-u\,ds,
\quad u,v\in V_H.
\]
It follows from Lemma \ref{WL1} that $b_0$ is $V_H$-elliptic, in
fact that
\[
\Re \,b_0(v,v)\ge \Vert v\Vert_{V_H}^2, \quad v\in V_H.
\]
Thus the problem of finding $u_0\in V_H$ such that
\begin{equation}
b_0(u_0,v)={\cal G}(v),\quad v\in V_H, \label{u0weak}
\end{equation}
has a unique solution which satisfies
\begin{equation}
\Vert u_0\Vert_{V_H}\leq \Vert{\cal G}\Vert_{V_H^*}. \label{uoap}
\end{equation}
Furthermore, defining $w=u-u_0$ and using (\ref{uweak}) and
(\ref{u0weak}), we see that
\[
b(w,v)=b(u,v)-b(u_0,v)={\cal G}(v)-({\cal
G}(v)-k_+^2(u_0,v)-(k^2u_0,v))=((k_+^2 + k^2)u_0,v),
\]
for all $v\in V_H$. Thus $w$ satisfies (\ref{weak_form}) with
$g=-(k_+^2 +k^2)u_0$. It follows, using (\ref{uoap}) and (\ref{boundB}),
% and Lemma \ref{WL3}DO WE NEED THIS, 
that
\begin{eqnarray} \label{boundC}
\Vert w\Vert_{V_H}\leq k_0^{-1} \tilde C(k_+^2 +k_{\infty}^2)\Vert u_0\Vert_2\leq k_0^{-1}\tilde C\left[k_+ + \frac{k_{\infty}^2}{k_+ }\right]
\Vert {\cal G}\Vert_{V_H^*}.%= %2\tilde C\frac{k_{\infty}^2k_+}{k_0^3}\Vert {\cal G}\Vert_{V_H^*}.
\end{eqnarray}
The bound (\ref{boundA}), with 
$$ 
C \le \left(1 +  k_0^{-1}\tilde C\left[k_+ + \frac{k_{\infty}^2}{ k_+}\right]\right),
$$ 
follows
from  (\ref{uoap}) and (\ref{boundC}).
\end{proof}

Following these preliminary lemmas we turn now to establishing the a
priori bound (\ref{boundB}), at first just for the case when
$\Gamma$ is the graph of a smooth function so that 
\begin{equation}\label{Gamma_def_layer}
\Gamma= \{ (\tilde x,f(\tilde x)): \tilde x \in \mathbb{R}^{n-1}\},
\end{equation}
where $f \in C^{\infty}(\mathbb{R}^{n-1})$;
and when $k \in C^{\infty}(D)$. We recall that $\nu$ is
the outward unit normal to $S_H$ and $\nu_n=\nu\cdot e_n$ is the
$n$th (vertical) component of $\nu$.
\begin{lemma}\label{rellich}
Suppose $\Gamma$ is given by (\ref{Gamma_def_layer}) with $f\in
C^\infty(\real^{n-1})$.  Let $H\ge f_+$, $g\in L^2(S_H)$ and
suppose that $k \in C^{\infty}(D)$ is such that $k=k_+$ in $\overline{U_H}$ and such that it satisfies assumption 1. Then, if $w\in V_H$ satisfies
\begin{equation}
b(w,\phi)= -(g,\phi),\quad\phi\in V_H, %\label{bgprob}
\end{equation}
then
\[
\Vert w\Vert_{V_H}\leq k_0^{-1} \sqrt{[(\kappa_+^2 + \kappa_{\infty}^2)A^{-1}B +1]\kappa_0^2A^{-1}B}\Vert g\Vert_{2}
\]
where $A= 2- {\lambda_1(H-f_-)^3}/{2}$ and $B=2\kappa_+ +1 +2\sqrt{2} +\lambda_2(H-f_-) + 2\kappa_{\infty}^2A^{-1}$.
%If in addition
%$\nu_n\leq -\alpha^2<0$ for some constant $\alpha>0$ and if $f_-$ is
%chosen so that $(f(x_1)-f_-)>(H_{min}-f_-)$ for all $x_1$ then
%\[
%\left\Vert\frac{\partial w}{\partial \nu}\right\Vert_{L^2(\Gamma)}
%\leq \frac{(H-f_-)}{\alpha\sqrt{H_{min}-f_-}}\left(\kappa +\frac{1}{2}+\sqrt{2}\right)
%\Vert g\Vert_2.
%\]
\end{lemma}

\begin{proof}%The proof of this lemma is
%motivated  by \cite{Melenk,cum04}, where a Rellich identity is used
%to prove estimates for solutions of the Helmholtz equation posed on
%bounded domains, by the proofs of the basic inequalities for rough
%surface scattering problems in \cite{chandsjam98,zhangsjam98}, and
%by the estimates derived for the diffraction grating problem in
%\cite{els02}.

 Let $r=|\tx|$. For $A\ge 1$ let $\phi_A\in
C_0^{\infty}(\real)$ be such that $0\leq \phi_A\leq 1$,
$\phi_A(r)=1$ if $r\le A$ and $\phi_A(r)=0$ if $r\ge A+1$ and
finally such that $\Vert\phi_A'\Vert_{\infty}\leq M$ for some fixed
$M$ independent of $A$.

Extending the definition of $w$ to $D$ by defining $w$ in $U_H$ by
(\ref{uprcstar}) with $F_H:= \gamma_-w$, it follows from Theorem
\ref{th_equiv} that $w$ satisfies the boundary value problem, with
$g$ extended by zero from $S_H$ to $D$ and $k$ extended from $S_H$ to $D$ by taking the value $k_+$ in $U_H$. By standard local
regularity results (e.g. \cite{mclean00} Theorem 4.18) it holds, since $g\in
L^2(D)$, $w=0$ on $\Gamma$, $k \in C^{0,1}_{\rm loc}(D)$ and the boundary is smooth, that $w\in
H^2_{\rm loc}(D)$. Further, $w\in H^2(U_b\setminus U_c)$ for
$c>b>f_+$.
% (though $w\in H^2(S_c)$ is not clear without some
%further constraint on the  of $\Gamma$ at infinity).
Moreover, by Lemma \ref{lemma3p2}, $w$ is given by the right hand
side of (\ref{uprcstar}) in $U_b$ for all $b>H$ if $H$ is replaced
in (\ref{uprcstar}) by $b$ and $F_b$ denotes the restriction of
$w$ to $\Gamma_b$. Thus $w$ satisfies the boundary value problem
with $H$ replaced by $b$, for all $b>H$, and so, by Theorem
\ref{th_equiv},
\begin{equation} \label{eqstst}
\int_{S_b}(\nabla w\cdot\nabla\bar v-k^2 w\bar v)\,dx =
-\int_{\Gamma_b}\gamma_-\bar v\, T\gamma_- w \,ds - \int_{S_b} \bar
v g \,dx,
\end{equation}
for all $b\ge H$.

In view of this regularity and since $w$ satisfies the boundary
value problem, we have, for all $a>H$,
\begin{eqnarray*}
\lefteqn{2\Re\int_{S_a}\phi_A(r)(x_n-f_-)g\frac{\partial\bar{w}}{\partial x_n}\,dx}\\
&=& 2\Re\int_{S_a}\phi_A(r)(x_n-f_-)(\Delta w+k^2w)
\frac{\partial\bar{w}}{\partial x_n}\,dx\\
&=&
\int_{S_a}\left\{2\Re\left\{\nabla\cdot\left(\phi_A(r)(x_n-f_-)\frac{\partial\bar{w}}{\partial
x_n} \nabla w\right)\right\}-2\phi_A(r)\left|\frac{\partial
w}{\partial x_n}\right|^2
\right.\\&&\left.-2\Re\left[(x_n-f_-)\phi_A(r)\frac{\partial \nabla
\bar w}{\partial x_n}.\nabla w\right]-2
\phi_A'(r)(x_n-f_-)\frac{\tx}{|\tx|}\cdot
\Re\left(\nabla_{\tx}w\frac{\partial \bar{w}}{\partial
x_n}\right)\right\}\,dx\\&&+ 2\Re\int_{S_a}\Re (k^2)(x_n-f_-)\phi_A(r)\frac{\partial \bar w}{\partial
x_n}w + i\Im(k^2)(x_n-f_-)\phi_A(r)\frac{\partial \bar w}{\partial
x_n}wdx.
\end{eqnarray*}
 Using the divergence theorem and integration by parts
 \begin{eqnarray*}
\lefteqn{2\Re\int_{S_a}\phi_A(r)(x_n-f_-)g\frac{\partial\bar{w}}{\partial x_n}\,dx}\\
&=&(a-f_-)\int_{\Gamma_a}\phi_A(r)\left\{ \left|\frac{\partial
w}{\partial x_n}\right|^2- \left|\nabla_{\tx}w\right|^2+
k_+^2\left|w\right|^2\right\}\,ds\\&-&
\int_{\Gamma}(x_n-f_-)\phi_A(r)\left\{\nu_n|\nabla
w|^2-2\Re\left(\frac{\partial \bar{w}}{\partial
x_n}\frac{\partial w}{\partial \nu}\right)\right\}\,ds\\&
+&\int_{S_a}\left\{\phi_A(r)\left(|\nabla
w|^2-\Re (k^2)|w|^2-2\left|\frac{\partial w}{\partial
x_n}\right|^2\right)\right.\\&&
\left.-2\phi_A'(r)(x_n-f_-)\Re\left(\frac{\partial\bar{w}}{\partial
x_n} \frac{\partial w}{\partial r}\right)\right\}\,dx\\&
-&\int_{S_a}\phi_A(r)\frac{\partial \Re (k^2)}{\partial x_n}(x_n -f_-)|w|^2dx\\& +&2\Re\int_{S_a} i\Im (k^2)\phi_A(r)(x_n-f_-)\frac{\partial \bar w}{\partial
x_n}wdx.
\end{eqnarray*}
Using the fact that $w=0$ on $\Gamma$, so that $\nabla w=(\partial
w/\partial \nu)\nu$ and
\[
\frac{\partial w}{\partial x_n}=e_n\cdot\nabla w=e_n\cdot
\nu\frac{\partial w}{\partial \nu}=\nu_n\frac{\partial w}{\partial
\nu},
\]
and rearranging terms we find that
\begin{eqnarray*}
\lefteqn{-\int_{\Gamma}\phi_A(r)(x_n-f_-)\nu_n\left|\frac{\partial
w}{\partial \nu}\right|^2\,ds
+2\int_{S_a}\phi_A(r)\left|\frac{\partial w}{\partial x_n}\right|^2\,dx }\\&& +\int_{S_a}\phi_A(r)\frac{\partial \Re (k^2)}{\partial x_n}(x_n -f_-)|w|^2dx \\
&=& (a-f_-)\int_{\Gamma_a}\phi_A(r)\left\{ \left|\frac{\partial
w}{\partial x_n}\right|^2- \left|\nabla_{\tx}w\right|^2+
k_+^2\left|w\right|^2\right\}\,ds\\&&
+\int_{S_a}\left\{\phi_A(r)\left(|\nabla w|^2-\Re(k^2)|w|^2\right)
-2\phi_A'(r)(x_n-f_-)\Re\left(\frac{\partial\bar{w}}{\partial
x_n} \frac{\partial w}{\partial
r}\right)\right\}\,dx\\&&-2\Re\int_{S_a}\phi_A(r)(x_n-f_-)g\frac{\partial\bar{w}}{\partial
x_n}\,dx +2\Re\int_{S_a}i\Im (k^2)(x_n-f_-)\frac{\partial \bar w}{\partial
x_n}wdx.
\end{eqnarray*}
We now wish to let $A\to \infty$.  The only problem is the term
involving $\phi'_A$ which we estimate as follows. Let
$S_a^b=\left\{x\in S_a\;:\; |\tx|< b\right\}$ for $b\ge 1$. Then
\[
\left|\int_{S_a}\left\{2\phi_A'(r)(x_n-f_-)\Re\left(\frac{\partial\bar{w}}{\partial
x_n} \frac{\partial w}{\partial r}\right)\right\}\,dx\right|\leq
2M(a-f_-)\int_{S_a^{A+1}\setminus \overline{S}_a^A}|\nabla
w|^2\,dx\to 0
\]
as $A\to\infty$, where the convergence follows from the fact that
$w\in H^1(S_H)$.  In addition since $w\in H^2(U_b\setminus U_c)$,
for $c>a>b>f_+$, $\nabla w|_{\Gamma_a}\in H^{1/2}(\Gamma_a)$ and
so, by the Lebesgue dominated and monotone convergence theorems, (note that $\partial \Re(k^2)/\partial x_n$ is bounded below by assumption 2),
\begin{eqnarray}
&&-\int_{\Gamma}(x_n-f_-)\nu_n\left|\frac{\partial
w}{\partial \nu}\right|^2\,ds
+2\int_{S_a}\left|\frac{\partial w}{\partial x_n}\right|^2\,dx  +\int_{S_a}\frac{\partial \Re (k^2)}{\partial x_n}(x_n -f_-)|w|^2dx\nonumber\\
&&= (a-f_-)\int_{\Gamma_a}\left\{ \left|\frac{\partial w}{\partial
x_n}\right|^2- \left|\nabla_{\tx}w\right|^2+
k_+^2\left|w\right|^2\right\}\,ds\nonumber\\&&
+\int_{S_a}\left(|\nabla
w|^2-\Re(k^2)|w|^2-2\Re\left((x_n-f_-)g\frac{\partial\bar{w}}{\partial
x_n}\right)\right)\,dx \nonumber\\&&+2\Re \int_{S_a}i\Im (k^2)(x_n-f_-)\frac{\partial \bar w}{\partial
x_n}w dx.\nonumber\\\label{s13}
\end{eqnarray}
Now, since $w$ satisfies the boundary value problem, including the
radiation condition (\ref{uprcstar}), applying Lemma \ref{lemma3p2}
it follows that
\begin{eqnarray}
\int_{\Gamma_a}\left\{ \left|\frac{\partial w}{\partial
x_n}\right|^2- \left|\nabla_{\tx}w\right|^2+
k_+^2\left|w\right|^2\right\}\,ds&\leq& 
%2k_+\Im\int_{\Gamma_a}
%\overline{w}\frac{\partial w}{\partial x_n}\,ds\nonumber\\&=&
-2k_+\Im\int_{\Gamma_a}\gamma_-\bar{w}T\gamma_-w\,ds.
\label{s14}
\end{eqnarray}
%on applying the Plancherel identity (\ref{plancherel}), noting
%(\ref{ft_ugama}) and (\ref{ft_unorm}). 
Further, setting $v=w$ in
(\ref{eqstst}) we get
\begin{equation}
\int_{S_b}\left(|\nabla
w|^2-k^2|w|^2\right)\,dx=-\int_{\Gamma_b}\gamma_-\bar{w}T\gamma_-w\,ds-
\int_{S_b}g\bar{w}\,dx, \label{s15}
\end{equation}
for $b\ge H$, so that, by Lemma \ref{WL1},
\begin{equation}
\int_{S_b}[|\nabla w|^2-\Re (k^2)|w|^2]\,dx\leq
-\Re\int_{S_b}g\bar{w}\,dx \label{s16}
\end{equation}
and
\begin{equation}
-\int_{S_b}\Im (k^2)|w|^2dx + \Im\int_{\Gamma_b}\gamma_-\bar{w}T\gamma_-w\,ds=-\Im\int_{S_b}g\bar{w}\,dx,
\end{equation}
which means, in view of Lemma \ref{lemma3p2} and the fact that $\Im (k^2) \geq 0$, that 
\begin{eqnarray}\label{s17a}
 \int_{S_b}\Im (k^2)|w|^2dx\leq\Im\int_{S_b}g\bar{w}\,dx, 
\end{eqnarray}
and that 
\begin{equation}
 \quad -2k_+\Im\int_{\Gamma_b}\gamma_-\bar{w}T\gamma_-w\,ds\leq 2k_+\Im\int_{S_b}g\bar{w}\,dx. \label{s17}
\end{equation} 
Using (\ref{s17}) in (\ref{s14}) and then using the resulting
equation, (\ref{s17a}) and (\ref{s16}) in (\ref{s13}), noting that $\supp
g\subset \overline{S_H}$, and using Assumption 1 and the Cauchy-Schwarz inequality, we get that
\begin{eqnarray*}
& -& \int_\Gamma (x_n-f_-)\nu_n\left|\frac{\partial w}{\partial
\nu}\right|^2\,ds + 2\int_{S_a}\left|\frac{\partial w}{\partial
x_n}\right|^2\,dx -\lambda_1\int_{S_a}(x_n -f_-)|w|^2 dx\\&&\le  2(a-f_-)k_+\Im\,\int_{S_H}g\bar w\,dx
  - \Re \int_{S_H} \left[g\bar w + 2(x_n-f_-)g\frac{\partial \bar
w}{\partial x_n}\right]\,dx \\  && +2\left|\int_{S_a}\Im (k^2)(x_n-f_-)\frac{\partial \bar w }{\partial
x_n}w dx\right| + \lambda_2\int_{S_a}\Im(k^2)(x_n-f_-)|w|^2dx
\\ &&\leq   2(a-f_-)k_+\Im\,\int_{S_H}g\bar w\,dx
  - \Re \int_{S_H} \left[g\bar w + 2(x_n-f_-)g\frac{\partial \bar
w}{\partial x_n}\right]\,dx \\
&&+ 2(a-f_-)\left(\Im\int_{S_H}g\bar wdx\right)^{\frac{1}{2}}k_{\infty}\left(\int_{S_a}\left|\frac{\partial w}{\partial x_n}\right|^2dx\right)^{\frac{1}{2}}\\&& + \lambda_2(a-f_-)\Im\int_{S_H}g \bar wdx.
\end{eqnarray*}
Since this equation holds for all $a>H$ and $\nu_n < 0$, on
$\Gamma$, 
%\[
%\frac{\partial \Re k^2}{\partial x_n} \geq -\lambda_1 -\lambda_2\Im (k^2,
%\]
it follows by the Cauchy-Schwarz inequality that
\begin{eqnarray*}
2\left\Vert \frac{\partial w}{\partial x_n}\right\Vert_2^2 -\lambda_1(H-f_-)\Vert w \Vert^2_2  &\leq&\left(
2\kappa_+ \Vert w\Vert_2+\Vert
w\Vert_2+2(H-f_-)\left\Vert\frac{\partial w}{\partial
x_n}\right\Vert_2 \right)\Vert g\Vert_2 \\  &+& \lambda_2(H-f_-)\Vert w\Vert_2  \Vert g\Vert_2 \\&+& 2\kappa_{\infty}\Vert g\Vert_2^{\frac{1}{2}}\Vert w\Vert_2^{\frac{1}{2}}\left\Vert\frac{\partial w}{\partial x_n}\right\Vert_2 .
%\lambda_2(H-f_-)\int_{S_H}\Im (k^2)|w|^2dx \\ &&+  2(H-f_-)\left(\int_{S_H}\Im (k^2)|w|^2dx\right)^{\frac{1}{2}}k^2_{\infty}\left\Vert \frac{\partial w}{\partial x_n}\right\Vert_2.
\end{eqnarray*}
Now using lemma \ref{WL3} to estimate $\Vert w\Vert_2$ we obtain
\begin{eqnarray*}
\left[2-\frac{\lambda_1(H-f_-)^3}{2}\right]\left\Vert \frac{\partial w}{\partial x_n}\right\Vert_2&&\leq [2\kappa_+ +1+2\sqrt{2} + \lambda_2(H-f_-)]\frac{(H-f_-)}{\sqrt{2}}\Vert g\Vert_2 \\ &&+ 2\kappa_{\infty}\sqrt{\frac{(H-f_-)}{\sqrt{2}}}\Vert g\Vert_2^{\frac{1}{2}}\left\Vert \frac{\partial w}{\partial x_n}\right\Vert_2^{\frac{1}{2}}.
%(H-f_-)\left(
%\frac{1}{\sqrt{2}}\kappa +\frac{1}{2\sqrt{2}}+1\right) \Vert g\Vert_2\left\Vert \frac{\partial w}{\partial x_n}\right\Vert_2 + \lambda_2(H-f_-)\Vert g\Vert_2\Vert w\Vert_2 + %\frac{\partial w}{\partial x_n}\Vert_2 + 
%+2(H-f_-)k^2_{\infty}\Vert g \Vert_2^{\frac{1}{2}}\Vert w\Vert_2^{\frac{1}{2}}\left\Vert\frac{\partial w}{\partial x_n}\right\Vert_2.
%2\kappa_{\infty}(\Vert g\Vert_2\Vert w\Vert_2)^{\frac{1}{2}}\Vert \frac{\partial w}{\partial x_n}\Vert_2
\label{dwdxest}
\end{eqnarray*}
%Using Lemma \ref{WL3} again we see that for $\tau>0$
%\begin{equation}
%\left[2-\frac{\lambda_1(H-f_-)^3}{2}\right]\left\Vert \frac{\partial w}{\partial x_n}\right\Vert^2_2\leq (H-f_-)\left(
%\frac{1}{\sqrt{2}}\kappa +\frac{1}{2\sqrt{2}}+1\right) \Vert g\Vert_2\left\Vert \frac{\partial w}{\partial x_n}\right\Vert_2 + \lambda_2\frac{(H-f_-)^2}{\sqrt{2}}\Vert g\Vert_2\left\Vert \frac{\partial w}{\partial x_n}\Vert_2 + %\frac{\partial w}{\partial x_n}\Vert_2 + 
%+\tau^{-1}\kappa^2_{\infty}\frac{(H-f_-)}{\sqrt{2}}\Vert g \Vert_2\left\Vert \frac{\partial w}{\partial x_n}\Vert_2 + \tau \left\Vert\frac{\partial w}{\partial x_n}\right\Vert^2_2.
%2\kappa_{\infty}(\Vert g\Vert_2\Vert w\Vert_2)^{\frac{1}{2}}\Vert \frac{\partial w}{\partial x_n}\Vert_2
%\label{dwdxest}
%\end{equation}
Now, recalling the definition of the constant $A$, it holds for all $\tau \geq0$, that 
\begin{eqnarray*}
(1-A^{-1}\tau^{-1})\left\Vert \frac{\partial w}{\partial x_n}\right\Vert_2\leq A^{-1}(2\kappa_+ +1 + 2\sqrt{2}+ \lambda_2(H-f_-) +\kappa_{\infty}^2\tau)\frac{(H-f_-)}{\sqrt{2}}\Vert g \Vert_2.
\end{eqnarray*}
Now choosing $\tau = 2A^{-1}$, recalling the definition of $B$ and using Lemma \ref{WL3} again shows that
\[
\left\Vert w\right\Vert_2\leq (H-f_-)^2A^{-1}B\Vert g\Vert_2.
%\left( \frac{1}{2}\kappa
%+\frac{1}{4}+\frac{1}{\sqrt{2}}\right) \Vert g\Vert_2.
\]
Using the above inequality in (\ref{s16}) shows that
\begin{eqnarray*}
\Vert w\Vert_{V_H}^2&\leq &(k^2_+ +k_{\infty}^2)\Vert w\Vert_2^2+\Vert
g\Vert_2\Vert w\Vert_2
\\&\leq &
(\kappa_+^2 + \kappa_{\infty}^2)(H-f_-)^2A^{-2}B^2\Vert g\Vert_2^2 + (H-f_-)^2A^{-1}B\Vert g\Vert_2^2.
%\frac{(H-f_-)^2}{4} \left(\frac{\kappa ^2}{2}(2\kappa
%+1+2\sqrt{2})^2+2\kappa +1+2\sqrt{2}\right)\Vert g\Vert_2^2.
\end{eqnarray*}
The required bound now follows.%Thus, for $\kappa\geq 1$,
%$$
%\Vert w\Vert_{V_H}^2\leq \frac{(H-f_-)^2}{2}(\kappa +1)^4\Vert
%g\Vert_2^2.
%$$
%The same bound holds for $\kappa<1$ by Theorem \ref{thm1}.
%This finishes the proof of the first inequality.  The second inequality is proved
%by noting that from (\ref{s13}) and estimates following that equation we have
%\begin{eqnarray*}
%\alpha^2(H_{min}-f_-)\left\Vert \frac{\partial w}{\partial x_n}\right\Vert_{L^2(\Gamma)}^2
%&\leq& -\int_{\Gamma}(x_n-f_-)\nu_n\left|\frac{\partial w}{\partial \nu}\right|^2\,d\tx\\
%&\leq& (H-f_-)\left(\sqrt{2}\kappa +\frac{1}{\sqrt{2}}+2\right)\Vert g\Vert_2\left\Vert
%\frac{\partial w}{\partial x_n}\Vert_{2}\right\Vert_2.
%\end{eqnarray*}
%Use of (\ref{dwdxest}) completes the proof of the second estimate of the lemma.
\end{proof}
%\begin{remark} The above argument works under milder assumptions on
%the boundary $\Gamma$, in particular that $\Gamma$ is the graph of a
%function $f\in C^2(\real^{n-1})$, so that $\Gamma$ is of class
%$C^2$. This assumption is enough \cite{gilbtrud83} to deduce the
%necessary local regularity result that $w\in H^2_{\rm loc}(D)$.
%\end{remark}

Combining lemmas \ref{rellich}, \ref{special_case_enough} and
\ref{ihl_rem} with Corollary \ref{cor_infsup}, we have the following
result.
\begin{lemma}\label{apriori}
If $\Gamma$ and $k \in L^{\infty}(D)$ satisfy the conditions of Lemma \ref{rellich} then the
variational problem (\ref{var_prob2}) has a unique solution $u\in
V_H$ for every ${\cal G}\in V_H^*$ and the solution satisfies the
estimate (\ref{aprioriest}).

%In addition the inf-sup condition (\ref{infsup}) is satisfied with
%\[
%\beta\ge \frac{1}{C}.
%\]
\end{lemma}
%\begin{remark}
%The above result, combined with Lemma \ref{ihl_rem}, implies that
%$\beta$, the inf-sup constant for $b(\cdot,\cdot)$, satisfies
%$\beta^{-1} \leq C =O(k^3)$ as $k\to\infty$. This high power of
%the wave number is, we suspect, not optimal. For an interior
%problem in a smooth starlike and bounded domain in $\real^2$ or
%$\real^3$ with impedance boundary data it is known that the
%constant in the corresponding bound satisfies the estimate $C =
%O(k)$ (for example this can be proved by combining estimate (2) of
%Theorem 1 of \cite{cum04} with the argument of Lemma
%\ref{special_case_enough}, involving a function corresponding to
%$u_0$).
% For a somewhat analogous one-dimensional problem
%the inf-sup constant is also $O(k)$ as $k\to\infty$ (Theorem 4.2
%of \cite{ihlenburg}).
%\end{remark}
%\begin{remark}
%For large $k$ we see that $C =O(k^3)$.  This high power of the wave
%number is surprising.  For an interior problem in a smooth starlike
%and bounded domain in $\real^2$ or $\real^3$ with impedance boundary
%data it is known that the same estimate holds with a constant
%bounded by $C(1+k^2)$ (for example this can be proved by combining
%estimate (2) of Theorem 1 of \cite{cum04} with the argument of Lemma
%\ref{special_case_enough}, involving a function corresponding to
%$u_0$).
% In one dimension
%the constant is just $O(1+k)$ (Theorem 4.2 of \cite{ihlenburg}).
%\end{remark}
We now proceed to establish that lemmas \ref{rellich} and
\ref{apriori} hold for arbitrary $k \in L^{\infty}(D)$ satisfying assumption 1.
\begin{lemma}\label{arb_k}
Suppose $\Gamma$ is given by (\ref{Gamma_def_layer}) with $f\in
C^\infty(\real^{n-1})$.  Let $H\ge f_+$, $g\in L^2(S_H)$ and
suppose that $k \in L^{\infty}(D)$ satisfies assumption 1. Then, if $w\in V_H$ satisfies
\begin{equation}
b(w,\phi)= -(g,\phi), \quad \phi\in V_H, \label{bgprob}
\end{equation}
then
\[
\Vert w\Vert_{V_H}\leq k_0^{-1} \sqrt{[\kappa_+^2 + \kappa_{\infty}^2]A^{-1}B +1]\kappa_0^2A^{-1}B}\Vert g\Vert_{2}
\]
where $A= 2- {\lambda_1(H-f_-)^3}/{2}$ and $B=2\kappa_+ +1 +2\sqrt{2} +\lambda_2(H-f_-) + 2\kappa_{\infty}^2A^{-1}$.\end{lemma}
\begin{proof}
%Let $k \in L^{\infty}(S_H)$ be such that $\frac{\partial \Re k^2}{\partial x_n} \geq -\lambda_1 -\lambda_2\Im k^2$, in a distributional sense, i.e.
%\[
%-\int_{S_H}\Re k^2\frac{\partial \bar \phi}{\partial x_n} dx \geq  -\lambda_1 -\lambda_2\Im k^2.
%\]
Extending the definition of $w$ to $D$ by defining $w$ in $U_H$ by
(\ref{uprcstar}) with $F_H:= \gamma_-w$ and extending $g$ by zero from $S_H$ to $D$, it follows from Theorem
\ref{th_equiv} and lemma \ref{lemma3p2} (cf. the proof of lemma \ref{rellich}) that
%that $w$ satisfies the boundary value problem, with
%$g$ extended by zero from $S_H$ to $D$ and $k$ extended from $S_H$ to $D$ by taking the value $k_+$ in $U_H$. 
%By standard local
%regularity results (e.g. \cite{mclean00} Theorem 4.18) it holds, since $g\in
%L^2(D)$, $w=0$ on $\Gamma$, $k \in C^{0,1}_{\rm loc}(D)$ and the boundary is smooth, that $w\in
%H^2_{\rm loc}(D)$. Further, $w\in H^2(U_b\setminus U_c)$ for
%$c>b>f_+$.
% (though $w\in H^2(S_c)$ is not clear without some
%further constraint on the  of $\Gamma$ at infinity).
%Moreover, by Lemma \ref{lemma3p2}, $w$ is given by the right hand
%side of (\ref{uprcstar}) in $U_b$ for all $b>H$ if $H$ is replaced
%in (\ref{uprcstar}) by $b$ and $F_b$ denotes the restriction of
%$w$ to $\Gamma_b$. Thus $w$ satisfies the boundary value problem
%with $H$ replaced by $b$, for all $b>H$, and so, by Theorem
%\ref{th_equiv},
%Let $b(w,v)= -(g,v)$.  Extending the definition of $w$ onto $U_H$ using the radiation condition \ref{fuck}, and extending $g$ by zero and $k$ by $k_+$ onto $U_H$, it holds by lemma fuck 
$\forall b>H$,  
\begin{equation} \label{int_S_b}
\int_{S_b}\nabla w.\nabla \bar v - k^2w\bar v dx + \int_{\Gamma_b}\gamma_-\bar v T\gamma_-wds  
=-(g,v), \quad v\in V_H.
\end{equation}
For $x \in \mathbb{R}^n \backslash D$ let $k(x) = k_0$, so that $k$ is now a function on $\mathbb{R}^n$. Note that assumption 1 now holds for almost all $x \in \mathbb{R}^n$.
%\[
%\Re(k^2(x) = \pi(x) - \lambda_1 x_n - \Im(k^2)x_n,
%\]
%where $\pi: \mathbb{R}^n \to \mathbb{R}$ is a monotonic non-decreasing function.

For $\delta>0$, let $\psi_{\delta} \in C_0^{\infty}(\mathbb{R}^n)$ be such that $\psi_{\delta} >0$, $\psi_{\delta}(x)=0$ if $|x|>\delta$,  such that $\int_{\mathbb{R}^n}\psi_{\delta}(x)dx =1$ and such that $\psi_{\delta}(x) =\psi_{\delta}(-x)$ for $x \in \mathbb{R}^n$. 
Let $k^2_{\delta} \in C^{\infty}(\mathbb{R}^n)$ be given by
\[
k_{\delta}^2 := k^2 \ast \psi_{\delta} =\Re(k^2)\ast \psi_{\delta} +i%\Im (k^2_{\delta}):= 
\Im(k^2)\ast \psi_{\delta}. 
%\quad \Re(k^2_{\delta}):=\pi \ast \psi_{\delta} - \lambda_1x_n - \Im(k^2_{\delta})x_n.
\]
Since $b>H$, then for all $x \in \Gamma_b$ there exists $\mu>0$ such that if $z \in B_{\mu}(x)$ then $k(z)=k_+$. Thus it follows from the definitions of convolution and $\psi_{\delta}$ that $k_{\delta} =k_+$ on $\Gamma_b$ provided we choose $\delta \le \mu$. 
Also $\Vert k_{\delta} \Vert_{L^{\infty}(D)} \leq k_{\infty}$ and 
\[
\Re(k^2_{\delta})= \int_{|y|<\delta}\Re(k^2)(x-y)\psi_{\delta}(y)dy>k_0^2.
\]
Moreover, if $s=t+y_n$, then
%since $\Im(k^2) =0$ on $\mathbb{R}^n \backslash D$,
%\[
%\lim_{x_n\to \infty}\left[\int_{-\infty}^{x_n}\Im(k^2) dz_n\right] \ast \psi_{\delta}(x) =0
%\]
%and using that differentiation and convolution commute, we see that
\begin{eqnarray*}
\int_{\mathbb{R}^n}\psi_{\delta}(y)\int_{-\infty}^{x_n-y_n}\Im(k^2(\tilde x-\tilde y,t))dtdy &=&
\int_{\mathbb{R}^n}\psi_{\delta}(y)\int_{-\infty}^{x_n}\Im(k^2(\tilde x-\tilde y,s-y_n))dsdy\\&=&
\int_{-\infty}^{x_n}\int_{\mathbb{R}^n}\psi_{\delta}(y)\Im(k^2(\tilde x-\tilde y,s-y_n))dyds\\ &=& \int_{-\infty}^{x_n}\Im(k^2)\ast\psi_{\delta}((\tilde x,s)) ds.
%\left[\int_j{-\infty}^{x_n}\Im(k^2)dz_n\right]\ast \psi_{\delta}(x)= \int_{-\infty}^{x_n}\frac{\partial}{\partial x_n}\left\{\left[\int_{-\infty}^{x_n}\Im(k^2)dz_n\right]\ast \psi_{\delta}(x)\right\}dz_n = \int_{-\infty}^{x_n}
%\Im(k^2) \ast \psi_{\delta} dz_n.
\end{eqnarray*}
%where we have used the fact that $\Im (k^2)=0$ on $\mathbb{R}^n \backslash D$ and assuming that $\delta>0$ is sufficiently small.
The symmetry property of $\psi_{\delta}$ then ensures that
\[
\Re(k^2_{\delta})(x)= \pi \ast \psi_{\delta}(x) -\lambda_1x_n -\int_{-\infty}^{x_n}\lambda_2 \Im(k^2_{\delta})(\tilde x, t)dt,
\]
with $\pi \ast \psi_{\delta}$ monotonic non-decreasing:
for if $h>0$ and $x \in \mathbb{R}^n$ then 
\begin{eqnarray*}
\pi\ast\psi_{\delta}(x+e_nh)-  \pi\ast\psi_{\delta}(x)=
\int_{\mathbb{R}^n}[\pi(x-y+e_n h) -\pi(x-y)]\psi_{\delta}(y)dy \geq 0,
\end{eqnarray*} 
because $\pi$ is assumed monotonic non-decreasing. Thus $k_{\delta}\in C^{\infty}(D)$ satisfies assumption 1 and so all of the hypotheses of lemma \ref{rellich}, with $H$ replaced by $b$. 

Now, fix $\epsilon>0$, and choose $w_n \in C_0^{\infty}(D)$ such that $\Vert w- w_n\Vert_{V_b}<\epsilon$. 
Thus (\ref{int_S_b}) can be rewritten as
\begin{eqnarray*}
\int_{S_b}\nabla w_n.\nabla \bar v - k_{\delta}^2w_n\bar v dx + \int_{\Gamma_b}\gamma_-\bar v Tw_nds  
&=&-(g,v) + b_b(w_n-w,v) \\
&+&\int_{S_b}(k^2-k^2_{\delta})w_n\bar v dx, \quad v\in V_H,
\end{eqnarray*}
where $b_b: V_b \times V_b \to \mathbb{C}$ is defined by (\ref{sesqui}) with $H$ replaced by $b$.
Now by lemma \ref{apriori} there exist unique $w',w'' \in V_H$ such that 
\begin{equation} %\label{int_S_b}
\int_{S_b}\nabla w'.\nabla \bar v - k_{\delta}^2w'\bar v dx + \int_{\Gamma_b}\gamma_-\bar v T\gamma_-w'ds  
=-(g,v) + \int_{S_b}(k^2-k^2_{\delta})w_n\bar v dx, \quad v\in V_H,
\end{equation}
and
\begin{equation} %\label{int_S_b}
\int_{S_b}\nabla w''.\nabla \bar v - k_{\delta}^2w''\bar v dx + \int_{\Gamma_b}\gamma_-\bar v T\gamma_-w''ds  
=b(w_n-w,v), \quad v\in V_H.
\end{equation}
Evidently $w_n=w'+w''$.
Hence by lemmas \ref{rellich} and \ref{apriori} again and using lemma \ref{WL4} 
\begin{eqnarray*}
\Vert w_n\Vert_{V_b} & \leq & \Vert w'\Vert_{V_b} +\Vert w''\Vert_{V_b}
\\& \leq& k_0^{-1}\sqrt{[(\kappa_+^2 + \kappa_{\infty}^2)A^{-1}B+1]\kappa_0^2A^{-1}B}\left[\Vert g\Vert_2 + \Vert (k^2-k^2_{\delta})w_n\Vert_{L^2(S_b)}\right]\\
 &+& C\Vert w_n-w\Vert_{V_b},
\end{eqnarray*}
where $C$ is independent of $\delta>0$. 
Since $k^2 \in L^2(\supp w_n)$ and since $w_n \in C^{\infty}_0(D)$ and so is bounded, standard arguments show that if $\delta >0$ is sufficiently small then
\begin{eqnarray*}
\Vert (k^2-k^2_{\delta})w_n\Vert_{L^2(S_b)} & = & \int_{S_b}|k^2 -k^2_{\delta}|^2|w_n|^2dx < \epsilon. 
%\\& \leq& \int_{S_b}|\pi(x)- \pi\ast\psi_{\delta}(x) -(\Im(k^2)-\Im(k^2)\ast\psi_{\delta})x_n + i(\Im(k^2)-\Im(k^2)\ast\psi_{\delta})|^2|w_n|^2dx,
\end{eqnarray*}
%Since$(k^2) \in L^2(\supp w_n)$ and since $w_n \in C^{\infty}_0(D)$ and so is bounded, standard arguments show that if $\delta >0$ is sufficiently small then 
%\[
%\Vert (k^2-k^2_{\delta})w_n\Vert_2 <\epsilon.
%\]
Thus for all $\epsilon>0$ 
\begin{eqnarray*}
\Vert w\Vert_{V_b} & \leq & \Vert w_n-w\Vert_{V_b} + \Vert w_n\Vert_{V_b}\\
&& \leq \epsilon +k_0^{-1}\sqrt{[(\kappa_+^2 + \kappa_{\infty}^2)A^{-1}B+1]\kappa_0^2A^{-1}B}\left[\Vert g\Vert_2 + \epsilon\right] + C\epsilon, 
\end{eqnarray*}
which implies the result by arbitrariness of $\epsilon >0$ and $b>H$.
%Extend $k$ onto the whole of $\mathbb{R}^n$ be letting it take the values $k_+$ in $U_H$, and $k_0$ in $\mathbb{R}^n \backslash D.$ 
\end{proof}
Combining lemmas \ref{arb_k}, \ref{special_case_enough} and
\ref{ihl_rem} with Corollary \ref{cor_infsup}, we have the following
result.
\begin{lemma}\label{apriori_arb_k}
If $\Gamma$ and $k \in L^{\infty}(D)$ satisfy the conditions of Lemma \ref{arb_k} then the
variational problem (\ref{var_prob2}) has a unique solution $u\in
V_H$ for every ${\cal G}\in V_H^*$ and the solution satisfies the
estimate (\ref{aprioriest}).
\end{lemma}

We proceed now to establish that lemmas \ref{arb_k} and
\ref{apriori_arb_k} hold for much more general boundaries, namely those
satisfying (\ref{dom_cond}). To establish this we first recall the
following technical lemma from \cite{chandmonk}.

\begin{lemma} \label{lem_tech}
If (\ref{dom_cond}) holds then, for every $\phi\in C_0^\infty(D)$,
there exists $f\in C^\infty(\real^{n-1})$ such that
$$
\supp \phi \subset D^\prime := \{x\in \real^n:x_n>f(\tx),
\,\tx\in\real^{n-1}\}
$$
and $ U_{f_+}\subset D^\prime\subset D. $
\end{lemma}

With this preliminary lemma we can proceed to show that Lemma
\ref{arb_k} holds whenever (\ref{dom_cond}) holds.

\begin{lemma}\label{rellich2}
Suppose (\ref{dom_cond}) holds, $H\ge f_+$, $g\in L^2(S_H)$ $k\in L^{\infty}(D)$ satisfies assumption 1, and
$w\in V_H$ satisfies
\begin{equation} \label{var_last}
b(w,\phi)= -(g,\phi), \quad \phi\in V_H.
\end{equation}
Then
\[
\Vert w\Vert_{V_H} \leq k_0^{-1}\sqrt{[(\kappa_+^2 + \kappa_{\infty}^2)A^{-1}B+1]\kappa_0^2A^{-1}B}\Vert g\Vert_{2}
\]
%where $\tilde C= \frac{\kappa}{\sqrt{2}}(\kappa +1)^2$.
%If in addition
%$\nu_n\leq -\alpha^2<0$ for some constant $\alpha>0$ and if $f_-$ is
%chosen so that $(f(x_1)-f_-)>(H_{min}-f_-)$ for all $x_1$ then
%\[
%\left\Vert\frac{\partial w}{\partial \nu}\right\Vert_{L^2(\Gamma)}
%\leq \frac{(H-f_-)}{\alpha\sqrt{H_{min}-f_-}}\left(\kappa +\frac{1}{2}+\sqrt{2}\right)
%\Vert g\Vert_2.
%\]
\end{lemma}
\begin{proof}
Let $\tilde V:= \{\phi|_{S_H}:\phi\in C_0^\infty(D)\}$. Then $\tilde
V$ is dense in $V_H$. Suppose $w$ satisfies (\ref{var_last}) and
choose a sequence $(w_m)\subset\tilde V$ such that
$\|w_m-w\|_{V_H}\to 0$ as $m\to\infty$. Then $w_m = \phi_m|_{S_H}$,
with $\phi_m\in C_0^\infty(D)$, and, by Lemma \ref{lem_tech}, there
exists $f_m\in C^\infty(\real^{n-1})$ such that $\supp \phi_m\subset
D_m$ and $U_{f_+}\subset D_m\subset D$, where $D_m := \{x\in
\real^n:x_n>f_m(\tx), \,\tx\in\real^{n-1}\}$. Let $V_H^{(m)}$ and
$b_m$ denote the space and sesquilinear form corresponding to the
domain $D_m$. That is, where $S_H^{(m)}:=
D_m\setminus\overline{U_H}$, $V_H^{(m)}$ is defined by
$V_H^{(m)}:=\{\phi|_{S_H^{(m)}}:\phi\in H_0^1(D_m)\}$ and $b_m$ is
given by (\ref{sesqui}) with $S_H$ and $V_H$ replaced by $S_H^{(m)}$
and $V_H^{(m)}$, respectively. Then $S_H^{(m)}\subset S_H$ and, if
$v_m\in V_H^{(m)}$ and $v$ denotes $v_m$ extended by zero from
$S_H^{(m)}$ to $S_H$, it holds that $v\in V_H$. Via this extension
by zero, we can regard $V_H^{(m)}$ as a subspace of $V_H$ and regard
$w_m$ as an element of $V_H^{(m)}$.

For all $v\in V_H^{(m)}\subset V_H$, we have
$$
b_m(w_m,v) = b(w_m,v) = -(g,v)-b(w-w_m,v).
$$
By Lemma \ref{apriori} there exist unique $w_m^\prime$,
$w_m^{\prime\prime}\in V_H^{(m)}$ such that
$$
b_m(w_m^\prime,v) = -(g,v),\quad v\in V_H^{(m)},
$$
and
$$
b_m(w_m^{\prime\prime},v) = -b(w-w_m,v),\quad v\in V_H^{(m)}.
$$
Clearly $w_m=w_m^{\prime}+w_m^{\prime\prime}$ and, by Lemma
\ref{rellich},
\[
\Vert w_m^\prime\Vert_{V_H^{(m)}} \leq k_0^{-1}\sqrt{[(\kappa_+^2 + \kappa_{\infty}^2)A^{-1}B+1]\kappa_0^2A^{-1}B} \Vert
g\Vert_{2}
\]
while, by Lemmas \ref{apriori} and \ref{WL3},
$$
\|w_m^{\prime\prime}\|_{V_H^{(m)}}\leq C\Vert w-w_m\Vert_{V_H},
$$
where $C$ is independent of $m$.

Thus
$$
\|w\|_{V_H} = \lim_{m\to\infty}\|w_m\|_{V_H^{(m)}}  \leq k_0^{-1}\sqrt{[(\kappa_+^2 + \kappa_{\infty}^2)A^{-1}B+1]\kappa_0^2A^{-1}B}
\Vert
g\Vert_{2}
$$
\end{proof}

Theorem \ref{th_main1} now follows by combining Lemmas
\ref{rellich2}, \ref{special_case_enough} and \ref{ihl_rem} with
Corollary \ref{cor_infsup}.\\

\chapter{The Impedance problem}\label{imp}
\section{Literature review}
In this chapter we study a boundary value problem for the Helmholtz equation with an impedance boundary condition (we called this the Impedance problem in chapter 1). Our aim is once again to extend the methods and results of \cite{chandmonk} to this problem. Thus in terms of style and approach this work follows on from \cite{chandmonk}, \cite{kirschipmp93}, \cite{els02},  \cite{Bonnetmmas94} and \cite{Szembergmmas98} (c.f.\ the literature review of chapter 2).

So let us turn  to the issue of prior work that was specifically done on the impedance problem. In \cite{chandmmas97}, Chandler-Wilde showed the impedance problem to be well-posed in 2D when the boundary is flat, including, in the problem formulation, the case of plane wave incidence. The method applied to obtain these results was to reformulate the problem as an equivalent second kind boundary integral equation on the real line; then to prove uniqueness of solution; and then to infer existence of solution by utilizing the results of \cite{chandima93}, which established some novel solvabiltity results for integral equations on the real line. In \cite{zhangworking} Chandler-Wilde and Zhang were able to show, again using boundary integral equation techniques, that the problem was well-posed in 2D, this time with the boundary being the graph of a bounded $C^{1,1}$ function.

In \cite{Cha02b}, Chandler-Wilde and Peplow consider a 2D impedance problem when the boundary is flat outside a compact set on which the relative admittance $\beta $ is constant. They prove uniqueness of solution and reformulate the problem as an equivalent boundary integral equation.
 
In \cite{Ned_1} and \cite{Ned_2} Dur\'an, Muga and N\'ed\'elec look at the impedance problem (in 2D and 3D respectively) in the special case that the boundary is flat so that the problem domain is a half plane or half space, obtaining unique existence of solution to their problem. In relation to the paper \cite{chandmmas97} and also the work that we present here, their problem set-up differs in that they assume that $\Re(\beta)<0$, whereas in \cite{chandmmas97} and again here we make the assumption that $\Re(\beta)\geq 0$: Indeed the problem of \cite{chandmmas97} and the one we study here are ill-posed if this condition is violated. However the authors of \cite{Ned_1} and \cite{Ned_2} are able to get round this by employing a different radiation condition to the one used here and in \cite{chandmmas97}.

On the numerical side, Chandler-Wilde, Langdon and Ritter, in \cite{chandrs04} show stability and convergence for a boundary element method for the impedance problem in a half plane, with the admittance $\beta$ being piecewise constant. 

Thus in our work there are two main novel aspects; in the first place the results apply in both 2 and 3 dimensions; and in the second place the boundaries to which our results apply are more general: Specifically we prove the problem to be well-posed when the boundary is the graph of a Lipschitz function for all wavenumber $k$; and also for small wavenumber $k$, we establish well-posedness in the case when the boundary is simply Lipschitz (and confined to a strip as usual.)  
      
\section{The Boundary value problem and variational formulation} In this section we shall define some notation related to the rough surface scattering problem and write down the boundary value problem and equivalent variational formulation that will be analyzed in later sections. We recall the usual notation: For $x=(x_1, \dots ,x_n) \in \mathbb{R}^n$ $(n=2,3)$ let $\tilde{x} =(x_1, \dots ,x_{n-1})$ so that $x=(\tilde{x}, x_n).$ For $H \in \mathbb{R}$ let $U_H = \{x:x_n > H\}$ and $\Gamma_H=\{x:x_n=H \}$.
Let $D \subset \mathbb{R}^n$ be an open connected set, with boundary $\Gamma$, such that for some constants $f_-<f_+$ it holds that 
\begin{eqnarray*}
U_{f_+} \subset D \subset U_{f_-}.
\end{eqnarray*}
In order to make sense of boundary integrals, we will require that for some $\mu>0$ and $N\in \mathbb{N}$, $D$ be an $(L,\mu,N)$ Lipschitz domain, in the sense of the following definition.
%and such that $S_H:=\Omega \backslash\overline{U_H}$ is an "$(L,\epsilon,N)$" Lipschitz domain, in the sense of the following definition:
\begin{definition}\label{LipDef} Given $L \in \mathbb{R}$, $\mu>0$ and $N \in \mathbb{N}$, the set $\Omega$ is said to be an $(L,\mu,N)$ Lipschitz domain if there exists a locally finite open cover $\{{O_j}\}_{j\in J}$ of $\Gamma$, such that
\newline i)For each $y \in \Gamma$, the open ball of radius $\mu$ and centre $y$ is a subset of $O_i$, for some $i \in J$.
\newline ii) For each $j \in J$, $O_j \cap \Omega= O_j \cap \Omega_j$, where $\Omega_j$ is, after a rotation, the epigraph of a Lipschitz function  $f:\mathbb{R}^{n-1} \to \mathbb{R}$ such that 
\begin{equation} \label{Lipconst}
|f(\tilde x) - f(\tilde y)| \leq L|\tilde x- \tilde y|, \quad \tilde x, \tilde y \in \mathbb{R}^{n-1}.
\end{equation}
\newline iii)Every collection of $N+1$ of the sets $O_j$ has empty intersection.
\end{definition}

In fact we will mainly be concerned with the case when $\Gamma$ is the graph of a Lipschitz function: %$f$, satisfying (\ref{Lipconst}):
\begin{equation} \label{Gamma}
\Gamma:=\{(\tilde{x},x_n):x_n = f(\tilde{x}), \tilde{x} \in \mathbb{R}^{n-1} \},  
\end{equation}
where $f:\mathbb{R}^{n-1} \to \mathbb{R}$ satisfies (\ref{Lipconst}), in which case $D$ is an $(L,\mu,1)$ Lipschitz domain, for all $\mu>0.$
\begin{remark}
It is shown in Adams \cite{adamsSS} that definition \ref{LipDef} is equivalent to $\Omega$ having the `strong local Lipschitz property' as defined in \cite{adamsSS}. As Adams remarks, in the case that $\Omega$ is bounded, definition \ref{LipDef} reduces to the standard definition of a Lipschitz domain.%The above is the definition of a Lipschitz domain given in \cite{adamsSS}, and, if $S_H$ were bounded, would reduce to the usual and more simple definition of bounded Lipschitz domains.
\end{remark}
%where $f:\mathbb{R}^{n-1} \to \mathbb{R}$ is a function such that there exists $L \in \mathbb{R}$ such that
%\begin{equation} \label{Lipconst}
%|f(\tilde x) - f(\tilde y)| \leq L|\tilde x- \tilde y|, \mbox{ for all } \tilde x, \tilde y \in \mathbb{R}^{n-1}.
%\end{equation}
% $S_H$ being a is a Lipschitz domain, in the sense of the following definition:
%\begin{definition}\label{LipDef} Given $L \in \mathbb{R}$, $\epsilon>0$, $N \in \mathbb{N}$, the set $S_H$ is said to be a $(L,\epsilon,\mathbb{N})$ Lipschitz domain if there exists a countable family of open sets  $\{{O_j}\}_{j\in J}$, such that
%\newline i)For each $y \in \Gamma$, the open ball of radius $\epsilon$ and centre $y$, is a subset of $O_i$ for some $i \in J$.
%\newline ii) For each $j \in J$, $O_j \cap S_H= O_j \cap (S_H)_j$, where $(S_H)_j$ is, after a rotation, the epigraph of a function $f_j$ satisfying (\ref{Lipconst}).
%\newline iii)Every collection of $N+1$ of the sets $\{O_j\}_{j \in J}$ has empty intersection.
%\end{definition}

We denote by $\nu$ the outward unit normal to $D$, which exists almost everywhere by Rademacher's theorem. The variational problem will be posed on the open set $S_H:= D \backslash \overline{U_H}$, for some $H \geq f_+ + \mu$, so that $S_H$ will be an $(L, \mu, N+1)$ Lipschitz domain. 

We will refer to any function $f:\mathbb{R}^{n-1} \to \mathbb{R}$ that satisfies (\ref{Lipconst}) as a Lipschitz function with Lipschitz constant $L$. Moreover we introduce the notation 
\[
J_f(\tilde x) = \sqrt{1+|\nabla_{\tilde x}f(\tilde x)|^2} \quad \tilde x \in \mathbb{R}^{n-1},
\]
and define $L'=\sqrt{1+L^2}$, so that $J_f\leq L'$.

Again, in this chapter it will be convenient to work with wavenumber dependent norms; thus we will equip the standard Sobolev space $H^1(S_H)$    
%\[
%H^1(S_H):=\{v \in L^2(S_H): \nabla v \in L^2(S_H)\},
%\]
with the $k$-dependent norm, equivalent to the usual norm, given by
\begin{eqnarray*}
\Vert v\Vert_{H^1(S_H)}:= \left\{k^2\Vert v \Vert^2_{L^2(S_H)} + \Vert \nabla v\Vert^2_{L^2(S_H)}\right\}^{\frac{1}{2}},  \quad v\in H^1(S_H).
\end{eqnarray*}
Let $\mathcal{D}({S_H}):= \{v|_{S_H}: v \in C_{0}^{\infty}(\mathbb{R}^n)\}$, so that $\mathcal{D}({S_H})$ is dense in $H^1(S_H)$. 
Let $\gamma^* : \mathcal{D}({S_H}) \to L^2(\Gamma)$ be defined by $\gamma^* \phi = \phi|_{\Gamma}$ for $\phi \in \mathcal{D}({S_H})$.
Then with $S_H$ being an $(L, \mu, N+1)$ Lipschitz domain, it's possible to show (see lemma \ref{traces} below) that $\gamma^*$ extends to a bounded linear operator $\gamma^* :H^1(S_H) \to L^2(\Gamma)$.
\begin{remark}
Trace results on unbounded domains, appear to be not well written up in the literature, so we prove the above statement in section 3. 
In fact we expect that a stronger result holds, but stating this would require defining Sobolev spaces on the boundary. The above result will be sufficient for our needs. 
\end{remark}

We are now in a position to state our boundary value problem. 
\newline {\sc The Boundary Value Problem.} Let $D$ be an $(L, \mu, N)$ Lipschitz domain for some $L>0, \mu>0$ and $N \in \mathbb{N}$. Given $g \in L^2(D)$, whose support lies in $S_H$ for some $H \geq f_{+}+ \mu$, and given $\beta \in L^{\infty}(\Gamma)$, %such that $\Re ( \beta ) \geq \eta >0$, 
find $u: D \rightarrow \mathbb{C}$ such that $u|_{S_a} \in H^1(S_a)$ for every $a\geq f_+ + \mu$,
\begin{equation} \label{BVP_2}
\Delta u + k^2 u = g \mbox{ in }D, \quad
\frac{\partial u }{\partial \nu}= ik\beta u \mbox{ on }\Gamma ,%\int_{\Omega}   \nabla u.\nabla\bar v -k^2u\bar v dx- \int_{\Gamma}ik \beta \gamma^* u\bar v ds = -\int_{\Omega} g \bar v dx, \mbox{   }  v \in D(\overline \Omega),
\end{equation}
in a distributional sense (see (\ref{variat}) below), and the radiation condition (\ref{uprcstar}) holds with $F_H = u|_{\Gamma_H}$ (and with $k_+$ replaced by $k$).
\begin{remark} Recall from the acoustics subsection in chapter 1  that $\beta \in L^{\infty}(\Gamma)$, known as the surface admittance, must satisfy that $\Re (\beta) \geq 0$ if the surface is not to be a source of energy. Additional assumptions on $\beta$ will be added later on to prove well-posedness of the boundary value problem.
\end{remark}
%(\ref{variat}) is an appropriate weak formulation of (\ref{Helmimp}), because if $u \in H^2(\Omega)$ then by the first Green identity (\cite{mclean00} Theorem 4.4) (\ref{Helmimp}) will hold almost everywhere.

\begin{remark}
We note that, as one would hope, the solutions of the above
problem do not depend on the choice of $H$. Precisely, if $u$ is a
solution to the above problem for one value of $H\geq f_+ + \mu$ for
which $\supp g\subset \overline{S_H}$ then $u$ is a solution for
all $H \geq f_+ + \mu$ with this property. To see that this is true is a
matter of showing that, if (\ref{uprcstar}) holds for one $H$ with
$\supp g\subset \overline{S_H}$, then (\ref{uprcstar}) holds for
all $H$ with this property. It was shown in Lemma \ref{lemma3p2} in chapter 2, that if (\ref{uprcstar}) holds, with $F_H=u|_{\Gamma_H}$,
for some $H\geq f_+ + \mu$, then it holds for all larger values of $H$.
One way to show that (\ref{uprcstar}) holds also for every smaller
value of $H$, $\tilde H$ say, for which $\tilde H\geq f_+ + \mu$ and
$\supp g\subset \overline{S_{\tilde H}}$, is to consider the
function
\begin{eqnarray*}
v(x) & := & u(x) -\\
&&\frac{1}{(2\pi)^{(n-1)/2}}\int_{\real^{n-1}}\exp(\ri[(x_n-\tilde
H)\sqrt{k^2-\xi^2}+ \tx\cdot\xi])\hat F_{\tilde H}(\xi)\,d\xi,\;\;
x\in U_{\tilde H},
\end{eqnarray*}
with $F_{\tilde H}:=u|_{\Gamma_{\tilde H}}$, and show that $v$ is
identically zero. To see this we note that, by Lemma
\ref{lemma3p2}, $v$ satisfies the above boundary value problem
with $D=U_{\tilde H}$ and $g=0$. That $v\equiv 0$ then follows
from Theorem \ref{th_main1}, chapter 2.
%by Lemma \ref{lemma3p2} below, $v|_{\Gamma_{\tilde H}} = 0$, so that for any $a>\tilde H$, $v|_{U_{\tilde H}\backslash \overline{U_a}} \in \{\phi|_{U_{\tilde H} \backslash \overline{U_a}} : \phi \in H^1_0(U_{\tilde H})\}$, and, by lemma \ref{lemma3p2} again, $v$ satisfies the boundary value problem of \cite{chandmonk}
%with $D=U_{\tilde H}$ and $g=0$. That $v\equiv 0$ then follows
%from Theorem 4.1 of \cite{chandmonk}.
\end{remark}

We now derive a variational formulation of the boundary value
problem above. As in chapter 2 (but this time with $k$ replacing $k_+$) we use standard
fractional Sobolev space notation, except that we adopt a wave
number dependent norm, equivalent to the usual norm, and reducing
to the usual norm if the unit of length measurement is chosen so
that $k=1$. Thus, identifying $\Gamma_H:=\{x:x_n = H\}$ with $\real^{n-1}$,
$H^s(\Gamma_H)$, for $s\in \real$, denotes the completion of
$C_0^\infty(\Gamma_H)$ in the norm $\|\cdot\|_{H^s(\Gamma_H)}$
defined by
$$
 \|\phi\|_{H^s(\Gamma_H)} = \left(
\int_{\real^{n-1}}(k^2+\xi^2)^s|{\cal F} \phi(\xi)|^2\,
d\xi\right)^{1/2}.
$$
  We recall \cite{adamsSS} that, for all $a>H\ge f_+ + \mu$, there
exist continuous embeddings $\gamma_+:H^1(U_H\setminus U_a)\to
H^{1/2}(\Gamma_H)$ and $\gamma_-:H^1(S_H)\to H^{1/2}(\Gamma_H)$ (the
trace operators) such that $\gamma_\pm \phi$ coincides with the
restriction of $\phi$ to $\Gamma_H$ when $\phi$ is $C^\infty$. We recall also that, if $u_+\in
H^1(U_H\setminus U_{a})$, $u_-\in H^1(S_H)$, and $\gamma_+u_+ =
\gamma_- u_-$, then $v\in H^1(S_a)$, where $v(x) := u_+(x)$, $x \in
U_H\setminus U_a$, $:= u_-(x)$, $x \in S_H$. Conversely, if $v\in
H^1(S_a)$ and $u_+:= v|_{U_H\setminus U_a}$, $u_-:= v|_{S_H}$, then
$\gamma_+u_+=\gamma_-u_-$. We recall the operator $T$ (see (\ref{Tdef}) but this time with $k$ replacing $k_+$) a Dirichlet to Neumann map on $\Gamma_H$, (see (\ref{DtN}) above), defined
by
\begin{equation}
T:=\cF^{-1}M_z\cF, \label{Tdef_inp}
\end{equation}
where $M_z$ is the operation of multiplying by
\[
z(\xi):=\left\{\begin{array}{ll}-\ri\sqrt{k^2-\xi^2}&\mbox{if
}|\xi|\le k,
\\[3pt]
\sqrt{\xi^2-k^2}&\mbox{for }|\xi|>k.
\end{array}\right.
\]
We proved in Lemma \ref{LTB} that
$T:H^{1/2}(\Gamma_H) \to H^{-1/2}(\Gamma_H)$ and is bounded.

We now derive a variational formulation of the boundary value problem making use of lemma \ref{lemma3p2}. Suppose that $u$ satisfies the boundary value problem. Then $u|_{ S_a} \in H^1(S_a) $ for every $a \geq f_+ + \mu$, and, by definition, since (\ref{BVP_2}) holds in a distributional sense,
\begin{equation} \label{variat}
\int_{D} g\bar v  +\nabla u.\nabla\bar v -k^2u\bar v dx- \int_{\Gamma}ik \beta \gamma^*u\bar v ds = 0,\quad v \in C^{\infty}_0(\mathbb{R}^n).
\end{equation}
Defining $w:=u|_{ S_H}$, and applying lemma \ref{lemma3p2} it follows that 
\begin{equation} \label{nolabel}
\hspace*{5ex}\int_{S_H} g\bar v + \nabla w.\nabla\bar v - k^2w\bar v dx + \int_{\Gamma_H}\bar vT \gamma_- w ds - \int_{\Gamma}ik \beta  \gamma^* w \bar v ds = 0, \quad v \in \mathcal{D}({S_H}). 
\end{equation}
%In fact (\ref{nolabel}) holds with $w=u|_{S_a}$ and $H$ replaced by $a$ for all $a>H$. 
%Further, in the case that $u|_{S_H} \in D(\overline{S_H})$, so that $u|_{\Gamma_H}=F_H \in C^{\infty}_0(\Gamma_H)$, $T\gamma_-u=-({\partial u}/{\partial x_n})|_{\Gamma_H}$, by lemma \ref{lemma3p2}. Also $u \in H^2(S_a)$,for all $a>f_+$, so that using Green's first identity, ${\partial u}/\partial \nu = ik\beta\gamma^*u$. Now, letting $\Upsilon=\Gamma \cap \Gamma_H$, $\partial u/\partial \nu|_{\Upsilon} =-\partial u/\partial x_n$ so 
%\begin{equation} \label{cancel}
%-\int_{\Upsilon} ik\beta\gamma^*u\bar v ds + \int_{\Upsilon}\bar v T\gamma_-uds= 0,
%\end{equation}
%and the same result holds for the solution $u$ of the BVP, by the density of $C^{\infty}_0(S_H)$ in $H^1(S_H)$, and the continuity of $\gamma_-,\gamma^*,$ and $T$.
% by the density 
From the denseness of $\mathcal{D}({S_H})$ in $H^1(S_H)$, and the continuity of $\gamma_-$ and $\gamma^*$, it follows that this equation holds for all $v \in H^1(S_H)$.%\[
%\gamma_-:H^1(S_H) \rightarrow H^{\frac{1}{2}}(\Gamma_H), \mbox{      } \gamma^*:H^1(S_H) \rightarrow L^2(\Gamma),
%\]
%we conclude that this equation holds for all $v \in H^1(S_H)$. 

Again, we let ${\Vert \cdot \Vert}_2$ and $( \cdot, \cdot )$ denote the norm and scalar product on $L^2(S_H)$ so that ${\Vert v\Vert}_2 = \sqrt{\int_{S_H} |{v}|^2 dx}$ and 
\[
(u,v)= \int_{S_H} u \bar v dx,
\] 
and define the sesquilinear form $c: H^1(S_H) \times H^1(S_H) \rightarrow \mathbb{C}$ by
\begin{equation} \label{sesq_2}
\hspace*{3ex}c(u,v)= (\nabla u,\nabla v) - k^2 (u,v) + \int_{\Gamma_H} \gamma_- \bar v T \gamma_-u ds - \int_{\Gamma} ik \beta \gamma^*u \gamma^* \bar v ds. 
\end{equation}
%From (\ref{nolabel}) and (\ref{cancel}), the denseness of $D(\overline{S_H})$ in $H^1(S_H)$, and the continuity of the trace operators
%\[
%\gamma_-:H^1(S_H) \rightarrow H^{\frac{1}{2}}(\Gamma_H), \mbox{      } \gamma^*:H^1(S_H) \rightarrow L^2(\Gamma),
%\]
Then we have shown that if $u$ satisfies the boundary value problem then $w:=u|_{S_H}$ is a solution of the following variational problem: find $u \in H^1(S_H)$ such that
\begin{equation} \label{weak_form_imp}
c(u,v) = -(g,v), \quad v \in H^1(S_H).
\end{equation}
Conversely, suppose that $w$ is a solution to the variational problem and define $u(x)$ to be $w(x)$ in $S_H$, and to be the right hand side of (\ref{uprcstar}), in $U_H$, with $F_H:=\gamma_-w$ (and with $k$ replacing $k_+$). %Then $\gamma_+u = F_H = \gamma_-w$, 
%and arguing precisely as above we see that $u|_{S_H}$ satisfies (\ref{cancel}) so that (\ref{nolabel}) holds. 
Then by lemma \ref{lemma3p2}, $u \in H^1(U_H \backslash \overline{U_a})$ for every $a>H$, with $\gamma_+u = F_H = \gamma_-w$. Thus $u|_{S_a} \in H^1(S_a)$ for all $a \geq f_+ + \mu$. From (\ref{exact_rc}) and (\ref{nolabel}), it follows that (\ref{variat}) holds, so $u$ satisfies (\ref{BVP_2}).

 We have thus proved the following theorem.

\begin{theorem} \label{th_equiv_2} If $u$ is a solution of the boundary value problem then $u|_{ S_H}$ satisfies the variational problem. Conversely, if $u$ satisfies the variational problem, %$u$ satisfies (\ref{nolabel}), and if 
$F_H := \gamma_-u$, and the definition of $u$ is extended to $D$ by setting $u(x)$ equal to the right hand side of (\ref{uprcstar}), for $x \in U_H$, then the extended function satisfies the boundary value problem, with $g$ extended by zero from $S_H$ to $D$. 
%In addition the extended function restricted to $S_a$ satisfies the variational problem (\ref{weak_form}), with $H$ replaced by $a$, for all $a>H$.
\end{theorem}
%\medskip

%Let us show that the map 
%\[
%M_ik \beta :H^{\frac{1}{2}}( \Gamma ) \rightarrow H^{-\frac{1}{2}}( \Gamma), \mbox{ defined by } M_{ik \beta} u = ik \beta u, \mbox{ for all } u \in H^{ \frac{1}{2}}( \Gamma),
%\]
 %is continuous. Let $f:\mathbb{R}^{n-1} \rightarrow \mathbb{R}$ be the Lipschitz function that defines $\Gamma$, and let $u_f(\tilde{x})= u(\tilde{x}, f(\tilde{x})) $ for $\tilde{x} \in \mathbb{R}^{n-1}$ , and let $C_f=\sqrt{1 + {\mid \nabla f \mid}^2  }$.Thus 
%\begin{eqnarray*}
%{{\parallel M_{ik \beta} u \parallel}^2_{H^-{\frac{1}{2}}(\Gamma)}} & = & \int_{\mathbb{R}^{n-1}} {(k^2 + {\xi}^2)}^{-\frac{1}{2}} {\mid F(ik \beta u_f c_f) \mid} ^2 d \xi \\
 % & = & \int_{\mathbb{R}^{n-1}} \frac{k^2}{{(k^2 + {\xi}^2)}^{\frac{1}{2}}} { \mid F( \beta u_f c_f)\mid}^2 d \xi \\& \leq & \int_{\mathbb{R}^{n-1}} k {\mid F( \beta u_f c_f \mid} ^2 d \xi \mbox{since} k \leq (k^2 + {\xi}^2)^{\frac{1}{2}} \\\ & = & \int_{\mathbb{R}^{n-1}} k \mid \beta u_f c_f \mid ^2 d \tilde{x}\\ & \leq & k {{\parallel \beta \parallel}_{L^{\infty}(\Gamma)}}^2 \int_{\mathbb{R}^{n-1}} {\mid u_fc_f \mid}^2 d \tilde{x} \\ & \leq & {{\parallel \beta \parallel}_L^{\infty}(\Gamma)}^2 \int_{mathbb{R}^{n-1}} (k^2 + {\xi}^2)^{\frac{1}{2}} {\mid F(u_fc_f)\mid}^2 d \xi \\ & = & {{\parallel \beta \parallel}_L^{\infty}(\Gamma)}^2 {\mid C_f \mid}^2 {\parallel u \parallel}_{H^{\frac{1}{2}}(\Gamma)}^2 \\
%\end{eqnarray*}
%Thus $ {\parallel M_{ik \beta } u \parallel}_{H^\frac{-1}{2}}( \Gamma) \leq {\parallel \beta %\parallel}_{L^{\infty}}( \Gamma) |C_f | {\parallel u \parallel}^{H^{\frac{1}{2}}( \Gamma)}$. 

\section{Analysis of the variational problem for low frequency}
Let $H^1(S_H)^*$ denote the dual space of $H^1(S_H)$, i.e.\ the space of
continuous anti-linear functionals on $H^1(S_H)$. Then, again, our analysis of the variational problem will
also apply to the following slightly more general situation: given
${\cal G}\in H^1(S_H)^*$ find $u\in H^1(S_H)$ such that
\begin{equation} \label{var_prob2_2}
c(u,v) = {\cal G}(v), \quad v\in H^1(S_H).
\end{equation}
We define the dimensionless wave number 
\[
\kappa= k(H-f_-),
\]
and the angle $\Phi \in [-\pi/2,0]$, by
 \[
\Phi := \mbox{min}\{0,\mbox{ess}\inf_{y \in \Gamma}\mbox{arg}\beta(y)\}.
\]
In order to show the boundary value problem is well-posed we will make some assumptions on $\beta \in L^{\infty}(\Gamma)$. In the 2D case, when $\Gamma$ is a straight line and $\beta$ is a constant it is well-known that the boundary value problem is ill-posed if $\beta=-is$, for some $s>0$. This motivates the following condition, that
\[
\mbox{dist}[\beta(\Gamma), \{-is:s\geq 0\}]>0.
\]
Together with the conditions that $\Re (\beta)\geq 0$ and $\beta \in L^{\infty}(\Gamma)$ we have assumption 2:
\newline Assumption 2 (A2). For some $\alpha_1 \in [0,\pi/2), \eta>0$, it holds that  $$\Im[e^{i\alpha_1}\beta]\geq \eta.
$$

We then let $\eta_{\alpha} = \eta\sec\alpha_1$. 

A different assumption that we will make is: 
\newline Assumption 3 (A3). For some $ \eta>0$,
$$
\Re(\beta) \geq \eta.
$$ 
\begin{remark}
Our analysis of the variational problem under assumption (A2), will not be applicable to the limiting case when $\alpha_1 = \pi/2$, so a different study of the variational problem will be made under (A3). Note that if $\beta$ satisfies (A3), then it satisfies (A2), but results with less restriction on $\kappa$ are obtained under (A3).
\end{remark}
%$\Im [e^{i\alpha_1} \beta ] \geq \eta.$
\begin{remark}
The assumption that $\Re (\beta) >0 $ corresponds, in physical terms, to the boundary absorbing some energy, which in practice is true. Assumption (A3) is a slightly coarser condition.  
\end{remark}

%We introduce $\alpha_2\in (-\Phi, \pi/2]$ such that 
%\begin{equation}\label{alpha_2}
%\tan \alpha_2= \frac{1}{\eta}\left[\eta\tan(-\Phi) + \frac{8\kappa}{6+\kappa^2}\frac{(2+ \kappa^2)}{(2-\kappa^2)}\right].
%\end{equation}
%We then define the sesquilinear forms $c_1,c_2: H^1(S_H) \times H^1(S_H) \to \mathbb{C}$ via 
%\[
%c_1(u,v) = e^{i\alpha_1}b(u,v),\quad c_2(u,v) = e^{i\alpha_2}b(u,v)  \quad u, v \in H^1(S_H).
%\]
%(Note that the definition of $\alpha_2$ was motivated in trying to show coercivity of $c_2$.)
%We shall prove the following theorem:  
 %For this section we introduce the following notation: Define $f_{r}:= f_-  r$, where $r>0 $  a fixed but small quantity. Then define $S:= \mathbb{R}^{n-1} \times (H-f_r)$,$\kappa_r:= k(H-f_r)$.  
Our main theorem of this section is then the following:
\begin{theorem} \label{small_k}
i)Suppose (A2) holds and that 
\[
\kappa <\frac{2\eta_{\alpha}}{1+ \sqrt{1+ 2\eta_{\alpha}^2}}
\]
Then for some constant $C_1 >0$
\[
|c(u,u)| \geq C_1^{-1}\Vert u \Vert^2_{H^1(S_H)}, \quad u \in H^1(S_H),
\] 
%$c_1$ is coercive 
so that the variational problem (\ref{var_prob2_2}) is uniquely solvable, and the solution satisfies the estimate
\begin{equation} \label{est1_small}
\Vert u\Vert_{H^1(S_H)} \leq C_1%\frac{2\eta_{\alpha} + \eta_{\alpha}\kappa^2 +4\kappa + \sqrt{[\eta_{\alpha}(2+ \kappa^2) - 4\kappa]^2 + 16\kappa^3\eta_{\alpha}}}{6\eta_{\alpha} - \eta_{\alpha}\kappa^2 -4\kappa - \sqrt{[\eta_{\alpha}(2+ \kappa^2) - 4\kappa]^2 + 16\kappa^3\eta_{\alpha}}}
\Vert \mathcal{G} \Vert_{H^1(S_H)^*},
\end{equation}
where
\[
C_1\leq \sec\alpha_1\left[\frac{2\eta_{\alpha} + \eta_{\alpha}\kappa^2 +4\kappa + \sqrt{[\eta_{\alpha}(2+ \kappa^2) - 4\kappa]^2 + 16\kappa^3\eta_{\alpha}}}{6\eta_{\alpha} - \eta_{\alpha}\kappa^2 -4\kappa - \sqrt{[\eta_{\alpha}(2+ \kappa^2) - 4\kappa]^2 + 16\kappa^3\eta_{\alpha}}}\right].
\]
In particular, the scattering problem (\ref{weak_form_imp}) is uniquely solvable and the solution satisfies the bound
\begin{equation} \label{est1_smallg}
k\Vert u\Vert_{H^1(S_H)} \leq C_1 \Vert g\Vert_2.
\end{equation}
ii) Suppose (A3) holds and that %$k < \sqrt{2}/(H-f_-)$ (equivalently 
$\kappa <\sqrt{2}$.  
Then for some constant $C_2 >0$
\[
|c(u,u)| \geq C_2^{-1}\Vert u \Vert^2_{H^1(S_H)}, \quad u \in H^1(S_H),
\] 
%$c_2$ is coercive 
so that the variational problem (\ref{var_prob2_2}) is uniquely solvable, and the solution satisfies the estimate
\begin{equation} \label{est2_small}
\Vert u\Vert_{H^1(S_H)}\leq C_2
%\frac{6+ \kappa^2}{2-\kappa^2}\left[1+\frac{1}{\eta^2}\left(\eta \tan(-\Phi) + \frac{8\kappa}{6+\kappa^2}\frac{(2+\kappa^2)}{2-\kappa^2)}\right)^2\right]^{\frac{1}{2}} 
\Vert \mathcal{G} \Vert_{H^1(S_H)^*},
\end{equation}
where
\[
C_2 \leq
 \frac{6+ \kappa^2}{2-\kappa^2}\left[1+\frac{1}{\eta^2}\left(\eta \tan(-\Phi) + \frac{8\kappa}{6+\kappa^2}\frac{(2+\kappa^2)}{(2-\kappa^2)}\right)^2\right]^{\frac{1}{2}}.
 \]%\begin{eqnarray} \label{smallk_est}
%& & \left(1-\phi -2\left((1+ 2\varepsilon)\frac{\kappa^2}{2} +1\right)^{-1}
%\left[(1+2\varepsilon +\theta)\frac{\kappa^2}{2}\right]\right)\Vert w\Vert^2_{H^1(S_H)}\\\nonumber  \leq  
%&& \left(2\left((1+ 2\varepsilon)\frac{\kappa^2}{2} +1\right)^{-1}\left[\left(1+\frac{1}{\varepsilon}\right)^2\frac{1}{2\eta^2\theta}\right] +\left(1+ \frac{\Vert\beta\Vert}{\eta}\right)^2\frac{1}{4\phi}\right)\Vert \mathcal{G}\Vert^2_{H^1(S_H)^*},\\\nonumber
%\end{eqnarray}
%where $\phi$,$\theta$,$\varepsilon>0$, are to be chosen appropriately.
%& &    \hspace*{5ex}\left(1-2(1+\phi)\left\{(1-\theta)\left(\frac{(1+2\epsilon)}{(1-\theta)}\frac{\kappa^2}{2} +1\right)\right\}^{-1}(1+2\epsilon)\frac{\kappa^2}{2}\right)\Vert w\Vert^2_{H^1(S_H)}\\\nonumber & \leq & \nonumber
%\left\{\left\{(1-\theta)\left(\frac{(1+2\epsilon)}{(1-\theta)}\frac{\kappa^2}{2} +1\right)\right\}^{-1}
%\left(1+\frac{1}{\epsilon}\right)^2\frac{\kappa^2}{4\theta(\eta k)^2} +\left(1+ \frac{\Vert\beta\Vert}{\eta}\right)^2\frac{1}{8k^2 \phi}\right\}\Vert \mathcal{G}\Vert^2_{H^1(S_H)^*},\\\nonumber
In particular the scattering problem (\ref{weak_form_imp}) is uniquely solvable, and the solution satisfies the bound 
\begin{equation} \label{est2_smallg}
k\Vert u\Vert_{H^1(S_H)} \leq C_2 \Vert g\Vert_2.
\end{equation}
%\Vert \mathcal{G}\Vert_{H^1(S_H)^*} \leq \left(\sqrt{\frac{3}{2}}(H-f_-) + \sqrt{2(H-f_-)\frac{NL'}{k}\left(\frac{1}{2} + \frac{2}{\epsilon k}\right)}\right)\Vert g\Vert_2.
%\]
%\Vert w \Vert^2_{H^1(S_H)} \leq\left(1-\left(\frac{8 \kappa_r^2 +4\theta(H-f_r)^2}{4\kappa_r^2 +1 - \phi} +\frac{(2k^2 + \theta)\delta}{4\kappa^2_r + 1 -\phi}\right)\right)^{-1}\left(\frac{C\chi^4}{64\eta^2k^2\delta^2 \phi(4\kappa^2_r +1- \phi)} + \left(1+\frac{\Vert\beta\Vert}{\eta}\right)^2\frac{1}{4\theta}\right),

\end{theorem}

Theorem \ref{small_k} will be proved via a sequence of lemmas. Our first aim is to show that the sesquilinear form $c$ is bounded. For this we will need the following trace results which are proved by combining standard methods of proof used for trace theorems on bounded domains, together with the proof of (\cite{chandmonk} lemma 3.4). The proof can be found in the appendix. 
\begin{lemma} \label{traces}
%For $u,v \in W_H$
%\[
%b(u,v) = (\nabla u,\nabla v) -k^2(u,v) + \int_{\Gamma_H\backslash \Upsilon}\gamma_- \bar vT \gamma_- u ds - \int_{\Gamma \backslash \Upsilon} ik\beta \gamma^* w\gamma^* \bar v ds.
%\]

Let $D$ be an $(L,\mu,N)$ Lipschitz domain, and let $S_H= D \backslash \overline{U_H}$ for $H \geq f_+ + \mu$.
For $u \in H^1(S_H)$,
\begin{eqnarray*}
\Vert \gamma_- u\Vert_{H^{\frac{1}{2}}(\Gamma_H )} \leq \sqrt{\left(1+ \frac{1}{k\mu}\right)}\Vert u \Vert_{H^1(S_H)},
\end{eqnarray*} 
and, the map $\gamma^*: \mathcal{D}({S_H}) \to L^2(\Gamma)$ such that $\gamma^*u$ is $u$ restricted to $\Gamma$, for $u \in \mathcal{D}({S_H})$, extends to a bounded linear operator $\gamma^*: H^1(S_H) \to L^2(\Gamma)$ with  
\begin{eqnarray*}
k\Vert \gamma^* u \Vert^2_{L^2(\Gamma)} \leq NL'\left(1+\frac{1}{ k\mu}\right)\Vert u \Vert^2_{H^1(S_H)}.
\end{eqnarray*}
\end{lemma}

Let $B:=\Vert \beta \Vert_{L^{\infty}(\Gamma)}$. Let us now show that the sesquilinear form
$c(.,.)$ is bounded.
\begin{lemma}\label{bounded}
Let $D$ be an $(L, \mu, N)$ Lipschitz domain and let $S_H= D\backslash \overline{U_H}$ for  $H\geq f_+ +\mu$. For all $u,v\in H^1(S_H)$,
\[
|c(u,v)|\leq \left(1+\left(1+BNL'\right)\left[1+ \frac{1}{k\mu}\right]\right)%2\left(1+ \frac{1}{k^2(H-f_+)^2}\right) + \Vert\beta \Vert NL'\left(\frac{1}{2} + \frac{2}{\epsilon k}\right)\right)
\Vert u\Vert_{H^1(S_H)}\Vert v \Vert_{H^1(S_H)}.
\]

\end{lemma}

\begin{proof}
This follows from the definition of $c(.,.)$, the Cauchy-Schwarz inequality, Lemma \ref{traces} and the mapping properties of $T$.
\end{proof}

%Throughout this paper we will make frequent use of Cauchy's inequality with $\epsilon$, by which we mean
%\[
%ab \leq \frac{a^2}{2 \epsilon} + \frac{\epsilon b^2}{2} \mbox{ for all } \epsilon > 0 \mbox{ where } a,b >0.\]

We now prove an important Friedrich's type inequality.
\begin{lemma} \label{friedrichs}
%If $\Gamma$ is given by , then for all $w \in H^1(S_H)$
%\begin{eqnarray*}
%\Vert w \Vert_2 \leq 4(H-f_r)^2 \left\Vert\frac{\partial w}{\partial x_n}\right\Vert^2_2 + 2(H-f_r)C\Vert w\Vert_{L^2(\Gamma)}\Vert w\Vert \chi \Vert w\Vert_{H^1(S_H)}.
%\end{eqnarray*}
Let $S_H$ be an $(L,\mu,N+1)$ Lipschitz domain. 
%$f:\mathbb{R}^{n-1} \to \mathbb{R}$ where $f$ is Lipschitz continuous. 
Then for all $w \in H^1(S_H)$ and for all $\zeta>0$
\[
\Vert w\Vert^2_2 \leq (1+\zeta)\frac{(H-f_-)^2}{2} \left\Vert \frac{ \partial w}{ \partial x_n}\right\Vert^2_2
+ \left(1+\frac{1}{\zeta}\right)(H-f_-) \Vert w\Vert^2_{L^2 ( \Gamma)}.
\]
\end{lemma}
\begin{proof}
%Let $(\tilde x,x_n) \in S$. Then
%\begin{eqnarray*}
%|w(\tilde x,x_n)|^2 = 2\Re \int_{f_r}^{x_n} \bar{w(\tilde x,x_n)} \frac{\partial w(\tilde x,x_n)}{\partial x_n} dx_n.
%\end{eqnarray*}
%Hence 
%\begin{eqnarray*}
%\int_{S} |w(\tilde x,x_n)|^2 dx = (H-f_r)2\int_{\mathbb{R}^{n-1}}\int_{f_r}^{H} |w|\left|\frac{\partial w}{\partial x_n}\right| dx_n d \tilde x.
%\end{eqnarray*}
%So using the Cauchy-Schwarz inequality
%\begin{eqnarray*}
%\Vert w\Vert^2_{L^2(S_H)} \leq \Vert w\Vert_{L^2(S)} \leq 2(H-f_r)\left(\Vert w\Vert_{L^2(S_H)}\Vert \frac{\partial w}{\partial x_n}\Vert_{L^2(S_H)} + \Vert w\Vert_{L^2(S \backslash S_H)}\Vert\frac{\partial w}{\partial x_n}\Vert_{L^2(S \backslash S_H)}\right).
%\end{eqnarray*}
%So using Cauchy's inequality with appropriately chosen $\epsilon$, and using (\ref{tr_bd}) and (\ref{tr_bd2}), we get
%\begin{eqnarray*}
%\frac{1}{2} \vert w\Vert_{L^2(S_H)} \leq 2(H-f_r)^2\Vert\frac{\partial w}{\partial x_n}\Vert_L^2(S_H) +2(H-f_r)C\Vert w\Vert_{L^2(\Gamma)}\Vert w \Vert_{H^1(S_H)}.
%\end{eqnarray*}

%i)Assume at first that $\Gamma$ is given by (\ref{Gamma}). 
For $x=(\tilde{x},x_n) \in S_H$, define $x_B : S_H \to \mathbb{R}$ by  $x_B=\max\{t:t\leq x_n \mbox{  and  } (\tilde x,t) \in \Gamma\}$. Let us show that $x_B$ is Borel measurable: Any point on the graph of a Lipschitz function $f$ is the vertex of a cone situated in the hypergraph of $f$ and with semi-major axis directed vertically. To see that this is true, consider such a cone, with slope greater than $L$, the Lipschitz constant of $f$; If a point on the graph of $f$ were also in the cone, then, by the mean value theorem, $f$ would have to assume a gradient greater than $L$, at some point.

Now fix $\alpha \in \mathbb{R}$, and suppose $x=(\tilde x,x_n) \in x_B^{-1}(\alpha,\infty)$. Let $B_1$ be an open ball with centre $(\tilde x,x_B)$ such that if $(\tilde y,t) \in B_1$ then $t>\alpha$. Let $B_2$ be an open ball, with centre $(\tilde x,x_n)$ such that $B_1\cap B_2=\emptyset.$

Now denote by $C_1$ the cone with vertex $(\tilde x,x_B)$ as described above and note that $C_1 \cap B_1$ is also a cone with vertex $(\tilde x,x_B)$. Let $P:\mathbb{R}^{n} \to \mathbb{R}^{n-1}$ be the projection, $P(\tilde y,y_n) = \tilde y$. Then its not difficult to show that $P(C_1 \cap B_1)$ must contain either a ball with centre $\tilde x$ or a ball with centre $\tilde y\not=\tilde x$, but with $\tilde x$ contained in this ball, or a cone with vertex $\tilde x$. Denote by $\hat D$ one of these sets which $P(C_1 \cap B_1)$ contains. 

Now, if $y \in \hat D  \times \mathbb{R} \cap B_2$, then $x_B(y) > \alpha$. This is enough to show that $x \in x_B^{-1}(\alpha,\infty)$ is the limit of a sequence of points $q_k \in \mathbb{Q}^n$ such that $q_k \in x_B^{-1}(\alpha,\infty)$ for all $k$.
 
Now let $q\in x_B^{-1}(\alpha,\infty) \cap \mathbb{Q}^n$. Define the open set 
$$O_q =\bigcup \{\mbox{open balls with centre } q \mbox{ contained in }S_H\}.
$$
For $n \in \mathbb{N}$ define the closed set $F_n= \Gamma \cap \{(\tilde y,y_n)\in\mathbb{R}^{n} :y_n \geq \alpha+ 1/n\}$. Then define the Borel set $O_{q,n}= P(F_n) \times \mathbb{R} \cap O_q$ and see that if $y \in O_{q,n}$ then $x_B(y) \geq \alpha + 1/n$. We claim that 
$$
x_B^{-1}(\alpha,\infty)= \bigcup_{q \in \mathbb{Q}^n \cap x_B^{-1}(\alpha,\infty)}\bigcup_{n \in \mathbb{N}}O_{q,n} ,
$$
so that $x_B$ is Borel measurable.

To verify this let $y \in x_B^{-1}(\alpha ,\infty)$ and fix an open ball $B$ with centre $y$ in $S_H$. Then find $q \in \mathbb{Q}^n$ such that $q\in x_B^{-1} (\alpha,\infty)$ and such that there exists an open ball with centre $q$ containing $y$ and contained within $B$. It now follows that $y \in O_{q,n}$, for some $n \in \mathbb{N}$.  
   
% which is well-defined because $\Gamma$ is closed. 

Then for $w \in \mathcal{D}({S_H})$,
\begin{eqnarray*}
\left|w(x) \right| ^2 & = & {\left| \int_{x_B}^{x_n} \frac{\partial w(\tilde x,y_n)}{\partial y_n} dy_n + w( \tilde{x},x_B) \right|}^2 \\
& \leq & \left(\int_{x_B}^{x_n} \left| \frac{\partial w(\tilde x,y_n)}{\partial y_n} \right| dy_n + \left|w( \tilde{x},x_B) \right| \right)^2 \\
& \leq & (x_n - x_B) \int_{x_B}^{x_n} \left| \frac{ \partial w(\tilde x,y_n)}{ \partial y_n} \right|^2 dy_n \\ & +& 2|w(\tilde{x},x_B)|\int_{x_B}^{x_n} \left| \frac{\partial w(\tilde x,y_n)}{\partial y_n} \right| dy_n  + \left|w(\tilde{x}, x_B) \right|^2\\
& \leq & (1+\zeta)(x_n - x_B) \int_{x_B}^{x_n} \left| \frac{ \partial w(\tilde x, y_n)}{ \partial y_n} \right|^2 dy_n + \left(1+\frac{1}{\zeta}\right)\left|w(\tilde{x}, x_B) \right|^2 \\
& \leq & (1+\zeta)(x_n - f_-) \int_{\mathbb{R}}1_{S_H}(\tilde x,y_n) \left| \frac{ \partial w(\tilde x,y_n)}{ \partial y_n} \right|^2 dy_n \\& +& \left(1+\frac{1}{\zeta}\right)\left|w(\tilde{x}, x_B) \right|^2.
\end{eqnarray*}
So that, %since $f(x) \geq f_-$,
since $\int_{\mathbb{R}}1_{S_{H}} (x_n -f_-) dx_n \leq (H-f_-)^2/2$, we have, using Fubini's Theorem,
\begin{eqnarray*}
&&\int_{S_H}{\left|w(x) \right|}^2 dx \\ & \leq & (1 + \zeta)\int_{ \mathbb{R}^{n-1}}\int_{\mathbb{R}}1_{S_H}(\tilde x,x_n)(x_n-f_-)dx_n \left(\int_{\mathbb{R}}1_{S_H}(\tilde x,y_n){\left| \frac{ \partial w(\tilde x,y_n)}{\partial{y_n}} \right|} ^2 dy_n\right) d \tilde{x} \\ %& + & 2(H-f_-) \int_{ \mathbb{R}^{n-1}} \left(\int_f^H \left| \frac{ \partial w}{\partial x_n} \right| dx_n \right) \left|w(\tilde{x},f) \right| d \tilde{x} \\
 & + & \left(1+ \frac{1}{\zeta}\right) \int_{\mathbb{R}}\int_{\mathbb{R}^{n-1}} 1_{S_H}(\tilde x,x_n){\left| w(\tilde{x}, x_B) \right|}^2 d \tilde{x}dx_n \\
& \leq & (1+ \zeta)\frac{(H-f_-)^2}{2}\int_{\mathbb{R}^{n-1}}\int_{\mathbb{R}}1_{S_H}(\tilde x,y_n) \left| \frac{\partial w(\tilde x,y_n)}{\partial y_n}\right|^2dy_n d \tilde x \\& + &\left(1 + \frac{1}{\zeta}\right)(H-f_-)\int_{\Gamma} |w(s) |^2 ds, 
\end{eqnarray*}
and the result follows for $w \in \mathcal{D}(S_H)$. By the density of this space in $H^1(S_H)$, the result holds for all $w \in H^1(S_H)$.
\end{proof}

Recalling Lemma \ref{WL1} from chapter 2 it's easy to verify the following important symmetry property of $c$: for all $u,v \in H^1(S_H)$
\begin{equation} \label{sym_2}
c(v,u) = c(\bar u, \bar v).
\end{equation}

We now introduce $\alpha_2\in (-\Phi, \pi/2]$ such that 
\begin{equation}\label{alpha_2}
\tan \alpha_2= \frac{1}{\eta}\left[\eta\tan(-\Phi) + \frac{8\kappa}{6+\kappa^2}\frac{(2+ \kappa^2)}{(2-\kappa^2)}\right].
\end{equation}
We then define the sesquilinear forms $c_1,c_2: H^1(S_H) \times H^1(S_H) \to \mathbb{C}$ via 
\[
c_1(u,v) = e^{i\alpha_1}c(u,v),\quad c_2(u,v) = e^{i\alpha_2}c(u,v)  \quad u, v \in H^1(S_H).
\]
(Note that the definition of $\alpha_2$ was motivated in trying to show ellipticity of $c_2$
and recall that $\alpha_1$ was defined when we wrote down (A2).)
 We now show that the sesquilinear forms $c_1,c_2$ are elliptic for small $\kappa$.
\begin{lemma}\label{coerc}
i)If (A1) holds
%\[
%\kappa <\frac{2\eta_{\alpha}}{1+ \sqrt{1+ 2\eta_{\alpha}^2}}
%\]
then for all $w \in H^1(S_H)$  
\begin{eqnarray*}
\Re [c_1(w,w)] \geq C\Vert w\Vert_{H^1(S_H)},
\end{eqnarray*}
where
\begin{eqnarray*}
C= \cos\alpha_1\left(\frac{6\eta_{\alpha} - \eta_{\alpha}\kappa^2 -4\kappa - \sqrt{[\eta_{\alpha}(2+ \kappa^2) - 4\kappa]^2 + 16\kappa^3\eta_{\alpha}}}{2\eta_{\alpha} + \eta_{\alpha}\kappa^2 +4\kappa + \sqrt{[\eta_{\alpha}(2+ \kappa^2) - 4\kappa]^2 + 16\kappa^3\eta_{\alpha}}}\right).
\end{eqnarray*}

ii) If (A2) holds %and $\kappa < \sqrt{2}$
then for all $w \in H^1(S_H)$ 
\[
\Re [c_2(w,w)] \geq C\Vert w\Vert_{H^1(S_H)},
\]
where
\[
C=\frac{2-\kappa^2}{6+ \kappa^2}\left[1+\frac{1}{\eta^2}\left(\eta \tan(-\Phi) + \frac{8\kappa}{6+\kappa^2}\frac{(2+\kappa^2)}{(2-\kappa^2)}\right)^2\right]^{-\frac{1}{2}}.
\]
\end{lemma}
\begin{proof}
i)For $w \in H^1(S_H)$ 
\begin{eqnarray*}
\Re [c_1(w,w)] &=& \cos \alpha_1 [\Vert \nabla w\Vert^2_2 -k^2\Vert w\Vert^2_2] + \Re\left[e^{i\alpha_1}\int_{\Gamma_H} \gamma_- \bar wT\gamma_-w ds\right] \\ &+& k\int_{\Gamma}\Im[e^{i\alpha_1}\beta]|w|^2ds.
\end{eqnarray*}
By lemma \ref{WL1}, $\arg\{\int_{\Gamma_H}\gamma_-\bar wT\gamma_-wds\} \in [-\pi/2,0]$, so that 
$$\Re\left[e^{i\alpha_1}\int_{\Gamma_H} \gamma_- \bar wT\gamma_-w ds\right] \geq 0,
$$
because $\alpha_1 \in [0,\pi/2)$. Hence, using lemma \ref{friedrichs} with $\zeta>0$, noting $\Im[e^{i\alpha_1}\beta]\geq \eta $, and, where $\theta>0$, we have
\begin{eqnarray*}
\Re [c_1(w,w)] & \geq &\cos \alpha_1\left(1-(1+ \zeta)\frac{\kappa^2}{2}(1+\theta)\right)\Vert \nabla w\Vert_2^2 \\ & + & k(\eta - \cos\alpha_1(1+ \zeta^{-1})(1+ \theta)\kappa)\int_{\Gamma}|w|^2ds \\
& + &\cos \alpha_1 \theta k^2\Vert w\Vert^2_2.
\end{eqnarray*}
If we now choose 
$$
 \theta = \frac{2-\kappa^2[1+\zeta]}{2+\kappa^2[1+\zeta]}
$$
and
\begin{eqnarray*}
\zeta=\frac{1}{2\eta_{\alpha} \kappa^2}\left(-\left[\eta_{\alpha}\left(2+\kappa^2\right) - 
4\kappa\right] + \sqrt{\left[\eta_{\alpha}\left(2+ \kappa^2\right) - 4\kappa\right]^2 
+ 16\kappa^3\eta_{\alpha}}\right),
\end{eqnarray*}
%and 
%$$ \theta = \frac{2-\kappa^2[1+\xi]}{2+\kappa^2[1+\xi]}
%$$
then $\eta -\cos \alpha_1(1+\zeta^{-1})(1+\theta)\kappa=0$, and
one obtains the (optimal) ellipticity bound. 
\newline
ii) As in part i) we have, for $w \in H^1(S_H)$, and $\theta>0, \zeta>0$ 
\begin{eqnarray} 
\Re [c_2(w,w)]  \nonumber & \geq & \cos \alpha_2\left(1- (1+ \zeta)\frac{\kappa^2}{2}(1+\theta)\right)\Vert \nabla w\Vert_2^2 \\\nonumber &-& k\cos\alpha_2(1+ \zeta^{-1})(1+ \theta)\kappa\int_{\Gamma}|w|^2ds \\ \label{coerc1}
 & + & k\int_{\Gamma}\Im[e^{i\alpha_2}\beta]|w|^2ds  + \cos \alpha_2\theta k^2\Vert w\Vert^2_2.
\end{eqnarray}
Now, if $\eta(\alpha_2):= \mbox{ess}\inf_{y\in \Gamma} \Im [e^{i\alpha_2}\beta]$, it's evident, since $\alpha_2 >-\Phi$, that
\[
\eta(\alpha_2) \geq \Im\left[e^{i(\alpha_2+\Phi)}\frac{\eta}{\cos \Phi}\right] = \frac{\eta \sin(\alpha_2 +\Phi)}{\cos \Phi} = \eta[\sin \alpha_2 + \cos \alpha_2 \tan \Phi].
\]
So making the (optimal) choices
\[
\zeta= \frac{1}{\kappa^2}- \frac{1}{2} \quad \mbox{ and }
\quad \theta =\frac{2-\kappa^2}{6+\kappa^2}
\]
(\ref{coerc1}) becomes 
\begin{eqnarray*}
\Re [c_2(w,w)]& \geq &\cos \alpha_2\left[\frac{2-\kappa^2}{6+\kappa^2}\right]\Vert w\Vert_{H^1(S_H)} \\&+& k\left[\eta(\sin\alpha_2 + \cos\alpha_2\tan\Phi)\right.\\&-&\left.\cos\alpha_2\frac{(2+\kappa^2)}{(2-\kappa^2)}\left(\frac{8\kappa}{6+\kappa^2}\right)\right]\int_{\Gamma} |w|^2ds, 
\end{eqnarray*}
so that the desired bound holds because (\ref{alpha_2}) implies that 
\[
\left[\eta(\sin\alpha_2 +\cos\alpha_2\tan\Phi)-\cos\alpha_2\frac{(2+\kappa^2)}{(2-\kappa^2)}\left(\frac{8\kappa}{6+\kappa^2}\right)\right]=0,
\]
and because 
\[
\cos \alpha_2 = \frac{1}{\sqrt{1+\tan^2\alpha_2}}.
\]
\end{proof}

Using Lemmas \ref{bounded} and \ref{coerc}, we can now prove Theorem \ref{small_k}.
\begin{proof}
By Lemma \ref{coerc} and under the assumption that $\kappa< 2\eta_\alpha/(1+\sqrt{1+2\eta_\alpha^2})$ (respectively $\kappa<\sqrt{2}$), one can verify that $c_1$, (resp. $c_2$) is elliptic, which in turn implies the ellipticity of $c$. Lemma \ref{bounded} implies that $c$ is bounded and hence by the Lax-Milgram lemma the existence of a unique solution $u$ to (\ref{var_prob2_2}) is assured assuming (A2), (resp. (A3)). The estimates (\ref{est1_small}), (\ref{est2_small}) also follow from the Lax-Milgram lemma. In the particular case $\mathcal{G}(v)= -(g,v)$, for some $g \in L^2(S_H)$ we have
\[
\Vert \mathcal{G} \Vert_{H^1(S_H)^*}= \sup_{\phi \in H^1(S_H)}\frac{|(g,\phi)|}{\Vert \phi\Vert_{H^1(S_H)}}\leq \Vert g\Vert_2\sup_{\phi \in H^1(S_H)}\frac{\Vert \phi\Vert_2}{\Vert \phi \Vert_{H^1(S_H)}}\leq \frac{1}{k}\Vert g\Vert_2, 
\]
%which cannot be improved on (e.g. consider $g(x)= e^{-\epsilon k|\tilde x|}$ and let $\epsilon \to 0$), 
so that (\ref{est1_smallg}) and (\ref{est2_smallg}) hold.
\end{proof}
%follows by combining lemmas \ref{smallKestimate}, \ref{ihl_rem} with Corollary \ref{cor_infsup}. In the case that $\mathcal{G}(v) = -(g,v)$, for $g\in L^2(S_H)$, 
%\begin{eqnarray*}
%\Vert \mathcal{G} \Vert_{H^1(S_H)^*} & = & \mbox{sup}_{v \in H^1(S_H)}\frac{|(v,g)|}{\Vert v \Vert_{H^1(S_H)}}\\ &\leq & \mbox{sup}_{v \in H^1(S_H)}\frac{\Vert v\Vert_2\Vert g\Vert_2}{\Vert v \Vert_{H^1(S_H)}}\\ & \leq & \left(\sqrt{\frac{3}{2}}(H-f_-) + \sqrt{2(H-f_-)\frac{NL'}{k}\left(\frac{1}{2} +  \frac{2}{\epsilon k}\right)}\right)\Vert g\Vert_2,
%\end{eqnarray*}
%where we have used lemmas \ref{traces} and \ref{friedrichs}.
\section{Analysis of the variational problem at arbitrary frequency} The sesquilinear form $c$ is not elliptic if the wavenumber $k$ is large.
In this section, we will assume that $\Gamma$ is the graph of a Lipschitz function. Under this restriction, and assuming (A3), but for arbitrary wave number $k$, we will employ Babu\v{s}ka's generalised Lax-Milgram theorem to show that the boundary value problem is well-posed.%derive the apriori bound (\ref{boundA}), which, by the results of section 3, will complete the existence and uniqness proof to solutions of (\ref{var_prob2}). We introduce here the parameter $\Phi$
%\[
%\Phi := \mbox{min}\{0,\mbox{ess}\inf_{y \in \Gamma}\mbox{arg}\beta(y)\}.
%\]

%Moreover the derived bound on the solution $u$, in terms of the data $\cal{G}$ and $\beta$  and the wave number $k$.
Our main result is:

\begin{theorem} \label{mainresultof}
If $\Gamma$ is given by (\ref{Gamma}) with $f$ satisfying (\ref{LipDef}), and (A2) holds then the variational problem (\ref{var_prob2_2}) has a unique solution $u \in H^1(S_H)$ for every $\cal{G}$ $\in {H^1(S_H)}^*$ and 
\begin{equation} \label{aprioriest_2}
\Vert u\Vert_{H^1(S_H)} \leq \sec \Phi(1+2E)\Vert \mathcal{G}\Vert_{ {H^1(S_H)}^*}
\end{equation}
where 
\begin{equation}\label{E}
E=\left(2\sqrt{2}\kappa\left[\frac{2+ \kappa^2(1+B^2(1+L))}{\eta} + \kappa[\sqrt{2}+ \sec\Phi]\right] + \frac{\sec\Phi}{4\sqrt{2}}\right).
\end{equation}
% is a dimensionless parameter given by (\ref{A}), (\ref{B}) and (\ref{E}) below.
In particular the boundary value problem and the equivalent variational problem have exactly one solution, and the solution satisfies the bound 
\[
k\Vert u\Vert_{H^1(S_H)} \leq E\Vert g \Vert_2.
\]
\end{theorem}
%In order to show existence and uniqueness of solution to (\ref{weak_form}), we will use the generalized Lax-Milgram Theorem (e.g \cite{ihlenburg}, Theorem 2.15).
To apply the generalised Lax-Milgram theorem %(e.g \cite{ihlenburg}, Theorem 2.15) 
we need to show that $c$ is bounded, which we have done in lemma \ref{bounded};
to establish the inf-sup condition 
that 
\begin{equation}
\alpha := \inf_{0\not=u\in H^1(S_H)}\sup_{0\not=v\in
H^1(S_H)}\frac{|c(u,v)|}{\Vert u\Vert_{H^1(S_H)}\Vert v\Vert_{H^1(S_H)}} >0;
\label{infsup_imp}
\end{equation}
and to establish the ``transposed'' inf-sup condition. It follows
easily from (\ref{sym_2}) that this transposed inf-sup
condition follows automatically if (\ref{infsup_imp}) holds.

\begin{lemma}\label{WEXL7_2} If (\ref{infsup_imp}) holds then, for all non-zero $v\in H^1(S_H)$,
\[
\sup_{0\not=u\in H^1(S_H)}\frac{|c(u,v)|}{\Vert u\Vert_{H^1(S_H)}}>0.
\]
\end{lemma}

\begin{proof}
If (\ref{infsup_imp}) holds and $v\in H^1(S_H)$ is non-zero then
\[
\sup_{0\not=u\in H^1(S_H)}\frac{|c(u,v)|}{\Vert u\Vert_{H^1(S_H)}}=
\sup_{0\not=u\in H^1(S_H)}\frac{|c(\bar v,u)|}{\Vert u\Vert_{H^1(S_H)}}\geq
\alpha \Vert v\Vert_{H^1(S_H)}>0.
\]
This proves the lemma.
\end{proof}

The following result follows from \cite[Theorem 2.15]{ihlenburg}
and Lemmas \ref{bounded} and \ref{WEXL7_2}.

\begin{corollary} \label{cor_infsup_2} If (\ref{infsup_imp}) holds then
the variational problem (\ref{var_prob2_2}) has exactly one solution
$u\in H^1(S_H)$ for all ${\cal G}\in H^1(S_H)^*$. Moreover
$$
\|u\|_{H^1(S_H)} \le \alpha^{-1}\|{\cal G}\|_{H^1(S_H)^*}.
$$
\end{corollary}
To show (\ref{infsup_imp}) we will establish an
apriori bound for solutions of (\ref{var_prob2_2}), from which the
inf-sup condition will follow by the following easily established
lemma (see \cite[Remark 2.20]{ihlenburg}).

\begin{lemma} \label{ihl_rem_2}
Suppose that there exists $C>0$ such that, for all $u\in H^1(S_H)$ and
${\cal G}\in H^1(S_H)^*$ satisfying (\ref{var_prob2_2}) it holds that
\begin{equation} \label{boundA_2}
\|u\|_{H^1(S_H)} \le C \|{\cal G}\|_{H^1(S_H)^*}.
\end{equation}
Then the inf-sup condition (\ref{infsup_imp}) holds with $\alpha\ge
C^{-1}$.
\end{lemma}

The following lemma reduces the problem of establishing
(\ref{boundA_2}) to that of establishing an a priori bound for
solutions of the special case (\ref{weak_form_imp}).

\begin{lemma} \label{special_case_enough_2}
Suppose there exists $ C^*>0$ such that, for all $u\in H^1(S_H)$ and
$g\in L^2(S_H)$ satisfying (\ref{weak_form_imp}) it holds that
\begin{equation} \label{boundB_2}
\|u\|_{H^1(S_H)} \le k^{-1}C^*\, \|g\|_2.
\end{equation}
Then, for all $u\in H^1(S_H)$ and ${\cal G}\in H^1(S_H)^*$ satisfying
(\ref{var_prob2_2}), the bound (\ref{boundA_2}) holds with
$$
C \le \sec\Phi(1+2C^*).
$$
\end{lemma}
\begin{proof}
% Suppose $u\in H^1(S_H)$ is a solution of
%\begin{equation}
%b(u,v)={\cal G}(v),\quad v\in H^1(S_H), \label{uweak}
%\end{equation}
%where ${\cal G}\in {H^1(S_H)}^*$.   
Let $\hat{c}:H^1(S_H)\times H^1(S_H)\to \C$ be defined by 
\begin{eqnarray*}
\hat{c}(u,v) :& =  &e^{-i\Phi}[ c(u,v) +2k^2(u,v)] \\ & = &
e^{-i\Phi}\left[(\nabla u,\nabla
v)+k^2(u,v)+\int_{\Gamma_H}\gamma_-\bar v\,T\gamma_-u\,ds - \int_{ \Gamma} ik \beta \gamma^*u \gamma^*\bar v ds\right],
\end{eqnarray*}
for  $u,v\in H^1(S_H)$.
%Now define $c:H^1(S_H) \times H^1(S_H) \to \mathbb{C}$ by $c(u,v)= \exp(-i\Phi)b_0(u,v)$, for $u,v \in H^1(S_H)$. Given $\mathcal{G} \in H^1(S_H)^*$, $u_0 \in H^1(S_H)$ satisfies 
%\begin{eqnarray}
%b_0(u_0,v) = \mathcal{G}(v), \hspace*{5ex} v \in H^1(S_H) \label{ProbA}
%\end{eqnarray}
%if, and only if 
%\begin{equation}
% c(u_0,v) =\exp(-i\Phi)\mathcal{G}(v), \hspace*{5ex} v \in H^1(S_H)\label{ProbB}.
%\end{equation}
For $u \in H^1(S_H)$ we see that
\[
\Re [\hat{c}(u,u)]\geq \Re(e^{-i\Phi}\Vert u\Vert^2_{H^1(S_H)}) = \cos\Phi\Vert u\Vert^2_{H^1(S_H)}.
\]
This follows because by lemma \ref{WL1}, $\arg \left\{\int_{\Gamma_H} \gamma_-\bar uT \gamma_-uds\right\} \in [-\pi/2,0]$, whilst $\Phi \in (-\pi/2,0]$, so that, noting the definition of $\Phi$, it holds that 
\[
\Re\left(e^{-i\Phi}\int_{\Gamma_H}\gamma_-\bar uT\gamma_-u ds\right) \geq0, \quad \Re\left(-e^{-i\Phi}\int_{\Gamma}ik\beta |u|^2ds\right) \geq 0.
\]
Thus given $\mathcal{G} \in H^1(S_H)^*$, it follows, by the Lax-Milgram lemma, that there exists unique $u_0 \in H^1(S_H)$ satisfying
\begin{equation} \label{c}
\hat{c}(u_0,v) = \mathcal{G}(v), \quad v\in H^1(S_H),
\end{equation}  
and moreover $u_0$ satisfies the estimate
%It follows by the Lax-Milgram lemma that there exists unique $u_0 \in H^1(S_H)$ satisfying (\ref{ProbB}) and therefore, equivalently (\ref{ProbA}) and moreover $u_0$ satisfies the estimate
\begin{equation} \label{sec_bd}
\Vert u_0 \Vert_{H^1(S_H)} \leq \sec\Phi\Vert \mathcal{G}\Vert_{H^1(S_H)^*}.
\end{equation}
Now suppose $u\in H^1(S_H)$ and $\mathcal{G} \in H^1(S_H)^*$ satisfy
\begin{equation}
c(u,v)={\cal G}(v),\quad v\in H^1(S_H), \label{uweak_2}
\end{equation}
and denote by $u_0 \in H^1(S_H)$ the unique solution of (\ref{c}). 
Then, defining $w=u-e^{-i\Phi}u_0$, we see that
\begin{eqnarray*}
c(w,v)&=&c(u,v)-\hat{c}(u_0,v) +e^{-i\Phi}2k^2(u_0,v)\\&=&{\cal G}(v)-{\cal
G}(v)+e^{-i\Phi}2k^2(u_0,v)=e^{-i\Phi}2k^2(u_0,v),
\end{eqnarray*}
for all $v\in H^1(S_H)$. Thus $w$ satisfies (\ref{weak_form_imp}) with
$g=-e^{-i\Phi}2k^2u_0$. It follows, using (\ref{boundB_2}) and (\ref{sec_bd}), that
\begin{equation} \label{boundC_2}
\Vert w\Vert_{H^1(S_H)}\leq k^{-1}C^* \Vert 2k^2u_0\Vert_2\leq 2C^*\sec\Phi\Vert {\cal G}\Vert_{H^1(S_H)^*}.
\end{equation}
The bound (\ref{boundA_2}), with $ C \le \sec\Phi(1+2C^*) $, follows
from  (\ref{sec_bd}) and (\ref{boundC_2}).
\end{proof}

We now turn to establishing the a
priori bound (\ref{boundB_2}), at first just for the case when
$\Gamma$ is the graph of a smooth Lipschitz function and $\beta \in C^{\infty}(\Gamma)$.
%, according to the following definition:
%\begin{definition}
% We recall that on any manifold $\Gamma$ given by (\ref{Gamma}), there exists a well-defined coordinate system $(s_1, \dots,s_{n-1})$ given by 
%\[
%s_i=\int_{0}^{x_i}\sqrt{1+\left|\frac{\partial f}{\partial x_i}\right|^2}dx_i.
%\]
%We will say that $\beta: \Gamma \to \mathbb{C}$ is such that $\beta \in C^{\infty}(\Gamma)$ if $\beta$ is infinitely differentiable with respect to the coordinate system $(s_1, \dots ,s_{n-1})$ given by\[
%s_i=\int_{0}^{x_i}\sqrt{1+\left|\frac{\partial f}{\partial x_i}\right|^2}dx_i.
%\]
%\end{definition}
\begin{remark}
If $v \in H^1(S_H)$, then $\gamma^*v \in L^2(\Gamma)$ by lemma \ref{traces}. For lemma \ref{rellich_2} it will be necessary to know that $\gamma^* v \in H^{\frac{1}{2}}_{loc}(\Gamma)$, as defined in \cite{mclean00}. This follows from \cite{mclean00} theorem 3.37.
\end{remark}% In the case that the Lipschitz function $f \in C^{\infty}(\mathbb{R}^{n-1})$, then $\alpha$, viewed as a function of the coordinates $(\tilde x,f(\tilde x))$ will also be smooth. However in the case that $f$ is merely Lipshitz, $\alpha$ will not neccesarily be smooth iwith respect to this coordinate system.  
%and $\beta \in C^{\infty}(\Gamma).$
%Let us further remark that if $v \in H^1(S_a)$ for $a > H\geq f_+$, then $\gamma^*v \in H^{\frac{1}{2}}_{\rm{loc}}(\Gamma)$, using \cite{mclean00} Theorem 3.37.  

We recall that $\nu$ is
the outward unit normal to $S_H$ and $\nu_n=\nu\cdot e_n$ is the
$n$th (vertical) component of $\nu$. %Moreover let $L^{\prime}:= \sqrt{1+L^2}$ and note that $1/(-\nu_n) \leq \sqrt{1+L^2}$, because $\Vert \nabla f \Vert_{L^{\infty}( \mathbb{R}^{n-1})} \leq L$. 

\begin{lemma}\label{rellich_2}
Suppose $\Gamma$ is given by (\ref{Gamma}) with $f$ satisfying (\ref{Lipconst}) and with $f\in
C^\infty(\real^{n-1})$.
%and $f$ Lipschitz continuous, with $|f(x)-f(y)| \leq L|x-y|$ for all $x,y \in \mathbb{R}^{n-1}$.
  Let $H \geq f_+ + \mu$, $g\in L^2(S_H)$ and let $\beta \in C^{\infty}(\Gamma)$ be such that (A2) holds.
Suppose  $w\in H^1(S_H)$ satisfies
\begin{equation}
b(w,\phi)= -(g,\phi), \quad \phi\in H^1(S_H). \label{bgprob_2}
\end{equation}
Then
\[
k\Vert w\Vert_{H^1(S_H)}\leq \left(2\sqrt{2}\kappa\left[\frac{2+ \kappa^2(1+B^2(1+L))}{\eta} + \kappa[\sqrt{2}+ \sec\Phi]\right] + \frac{\sec\Phi}{4\sqrt{2}}\right) \Vert g\Vert_{2}
\]
%where $E$, a dimensionless constant, is given explicitly by (\ref{E}), (\ref{A}) and (\ref{B}) below.

\end{lemma}

\begin{proof}

%Extending the definition of $w$ to $\Omega$ by defining $w$ in $U_H$ by
%(\ref{UPRC2nd}) with $F_H:= \gamma_-w$, it follows from Theorem 
%\ref{th_equiv} that $w$ satisfies the boundary value problem, with
%$g$ extended by zero from $S_H$ to $\Omega$. 
%Letting $v=w$ in the boundary value problem, and taking imaginary parts gives
%\[
%\eta \int_{\Gamma}k \left| w \right|^2 ds \leq \int_{\Gamma} k \Re(\beta) \left|w \right|^2 ds \leq \Im \int_{S_H} g \bar w dx\]
%so that
%Putting $w=\phi$ in (\ref{bgprob}), taking imaginary parts, and using lemma \ref{lemma3p2} (cf proof of lemma \ref{smallKestimate}),
Setting $\phi=w$ in (\ref{bgprob_2}) and, multiplying through by $e^{-i\Phi}$, and taking real parts (c.f.\ the proof of lemma \ref{special_case_enough_2}) we derive the estimate
\begin{equation}\label{secbd}
\Vert \nabla w\Vert_2^2 \leq k^2 \Vert w\Vert_2^2 + \sec\Phi \Vert g\Vert_2\Vert w\Vert_2.
\end{equation}%\begin{equation} \label{gammabd2}
%\int_{\Gamma} \left|w \right|^2 ds \leq \frac{1}{ \eta k}\int_{S_H} \left|g \right|\left|w\right| dx \leq \frac{1}{\eta k}\Vert g\Vert_{L^2(S_H)}\Vert w\Vert_{L^2(S_H)}.
%\end{equation}
Setting $\phi=w$ in (\ref{bgprob_2}) and taking imaginary parts, and writing $\gamma^* w$ as $w$, gives
\[
\Im\int_{\Gamma_H}\gamma_-\bar wT \gamma_-wds - \int_{\Gamma}k\Re(\beta)|w|^2ds = -\Im(g,w),
\] 
so that from lemma \ref{WL1} and assuming (A2) we get
\begin{equation} \label{w_gamma_bd}
\eta \int_{\Gamma}k|w|^2 ds \leq \Vert g\Vert_2\Vert w\Vert_2.
\end{equation}
From lemma \ref{friedrichs} with $\zeta=1$, we have
\begin{equation} \label{Frieds}
\hspace*{5ex}k^2{\Vert w\Vert}^2_2  \leq \kappa^2 {\left\Vert \frac{ \partial w}{ \partial x_n} \right\Vert}^2_2 + 2k\kappa\int_{\Gamma}|w|^2ds.
\end{equation}
%so that use of Cauchy's inequality with $\tau=1$, gives
%\begin{equation} \label{important}
%{\Vert w\Vert}^2_{L^2(S_H)}  \leq 2(H-f_-)^2 {\left\Vert \frac{ \partial w}{ \partial x_n} \right\Vert}^2_{L^2(S_H)} + \frac{4(H-f_-)^2}{(\eta k)^2} {\Vert g\Vert}^2_{L^2(S_H)} 
%\end{equation}
%Use of Cauchy's inequality with $\epsilon$ gives
%\[
%\parallel w \parallel_2^2 \leq \frac{3}{2}(H-f_-)^2 \parallel \frac{ \partial w}{\partial x_n} \parallel^2_2 + \frac{2(H-f_-)}{ \epsilon k } \parallel g \parallel^2_2
%\]

Extending the definition of $w$ to $D$ by defining $w$ in $U_H$ by
(\ref{uprcstar}) with $F_H:= \gamma_-w$, it follows from Theorem 
\ref{th_equiv_2} that $w$ satisfies the boundary value problem, with
$g$ extended by zero from $S_H$ to $D$. 

With $\beta \in C^{\infty}(\Gamma)$ and $ w|_{\Gamma} \in H_{\rm{loc}}^{\frac{1}{2}}(\Gamma)$ it follows (e.g.\ \cite{mclean00}, Theorem 3.20) that $\beta w \in H^{\frac{1}{2}}_{\rm loc}(\Gamma)$. Together with the assumptions that $g\in
L^2(D)$, and that the boundary is smooth, regularity theory implies that $w\in
H^2_{\rm loc}(D)$ (e.g.\ \cite{mclean00} Theorem 4.18). 
%Further, $w\in H^2(U_b\setminus U_c)$ for
%$c>b>f_+$.
% (though $w\in H^2(S_c)$ is not clear without some
%further constraint on the  of $\Gamma$ at infinity).
%Moreover, by Lemma \ref{lemma3p2}, $w$ is given by the right hand
%side of (\ref{UPRC2nd}) in $U_b$ for all $b>H$ if $H$ is replaced
%in (\ref{UPRC2nd}) by $b$ and $F_b$ denotes the restriction of
%$w$ to $\Gamma_b$. Thus $w$ satisfies the boundary value problem
%with $H$ replaced by $b$, for all $b>H$, and so, by Theorem
%\ref{th_equiv},
%\begin{eqnarray}
%\int_{S_b}(\nabla w\cdot\nabla\bar v-k^2 w\bar v)dx  =  \label{eqstst} & - & \int_{\Gamma_b}\gamma_-\bar v\, T\gamma_- w \,ds + \int_{\Gamma} ik \beta \gamma^*w \gamma^* \bar v ds \\\nonumber  & - & \int_{S_b} g\bar
%v\,  dx,\\ \nonumber 
%\end{eqnarray}
%for all $b\ge H$.

Let $r=|\tx|$. For $A\ge 1$ let $\phi_A\in
C_0^{\infty}(\real)$ be such that $0\leq \phi_A\leq 1$,
$\phi_A(r)=1$ if $r\le A$ and $\phi_A(r)=0$ if $r\ge A+1$ and
finally such that $\Vert\phi_A'\Vert_{\infty}\leq M$ for some fixed
$M$ independent of $A$.

In view of this regularity and since $w$ satisfies the boundary
value problem, we have 
\begin{eqnarray*}
\lefteqn{2\Re\int_{S_H}\phi_A(r)(x_n-H)g\frac{\partial\bar w}{\partial x_n}\,dx}\\
&=& 2\Re\int_{S_H}\phi_A(r)(x_n-H)(\Delta w+k^2w)
\frac{\partial\bar w}{\partial x_n}\,dx\\
&=&
\int_{S_H}\left\{2\Re\left\{\nabla\cdot\left(\phi_A(r)(x_n-H)\frac{\partial\bar w}{\partial
x_n} \nabla w\right)\right\}-2\phi_A(r)\left|\frac{\partial
w}{\partial x_n}\right|^2
\right.\\&&\left.-2\Re\left[(x_n-H)\phi_A(r)\frac{\partial \nabla
\bar w}{\partial x_n}\cdot \nabla  w\right]\right.\\&&\left.-2
\phi_A'(r)(x_n-H)\frac{\tx}{|\tx|}\cdot
\Re\left(\nabla_{\tx}w\frac{\partial \bar w}{\partial
x_n}\right)\right.\\&&\left.+2\Re\left[k^2(x_n-H)\phi_A(r)\frac{\partial \bar w}{\partial
x_n} w\right]\right\}\,dx.
\end{eqnarray*}
 Using the divergence theorem and integration by parts
 \begin{eqnarray*}
\lefteqn{2\Re\int_{S_H}\phi_A(r)(x_n-H)g\frac{\partial\bar w}{\partial x_n}\,dx}\\
&=&%(a-f_-)\int_{\Gamma_a}\phi_A(r)\left\{ \left|\frac{\partial
%w}{\partial x_n}\right|^2- \left|\nabla_{\tx}w\right|^2+
%k^2\left|w\right|^2\right\}\,ds\\&&-
-\int_{\Gamma}(x_n-H)\phi_A(r)\left\{\nu_n(|\nabla
w|^2 -k^2 |w|^2) -2\Re\left(\frac{\partial \bar w}{\partial
x_n}\frac{\partial w}{\partial \nu}\right)\right\}\,ds\\&&
+\int_{S_H}\left\{\phi_A(r)\left(|\nabla
w|^2-k^2|w|^2-2\left|\frac{\partial w}{\partial
x_n}\right|^2\right)
\right.\\&& \left.-2\phi_A'(r)(x_n-H)\Re\left(\frac{\partial\bar w}{\partial
x_n} \frac{\partial w}{\partial r}\right)\right\}\,dx.
\end{eqnarray*}
Now, on $\Gamma\cap\mbox{supp}\phi_A(r)$ 
\[
\frac{\partial w}{\partial x_n} = e_n.\nabla w = e_n.\left(\nabla_{\Gamma}w +\frac{\partial w}{\partial \nu}\nu\right),
\]
where $\nabla_{\Gamma}w$, the tangential part of $\nabla w$, is given by
$$
\nabla_{\Gamma}w:= \nabla w -\frac{\partial w}{ \partial \nu}\nu.
$$
% is the tangential part of $\nabla w$.
 So
\[
\Re \left(\frac{\partial \bar w}{\partial x_n}\frac{\partial w}{\partial \nu}\right) = \left|\frac{\partial w}{\partial \nu}\right|^2\nu_n + \Re\left((e_n.\nabla_{\Gamma}\bar w) \frac{\partial w}{\partial \nu}\right).
\]
Also 
\[
|\nabla w|^2= |\nabla_{\Gamma} w|^2 + \left|\frac{\partial w}{\partial \nu}\right|^2,
\]
so that 
\[
\nu_n|\nabla w|^2 - 2\Re\left(\frac{\partial \bar w}{\partial x_n}\frac{\partial w}{\partial \nu}\right)= \nu_n|\nabla_{\Gamma} w|^2 - \nu_n\left|\frac{\partial w}{\partial \nu}\right|^2 -2\Re\left((e_n \cdot \nabla_{\Gamma} \bar w)\frac{\partial w}{\partial \nu}\right).
\]
Rearranging terms and noting that $\partial w / \partial \nu = ik\beta w$ on $\mbox{supp}\phi_A(r) \cap \Gamma$, we find that
\begin{eqnarray} \label{main}
\hspace*{5ex}&&{2\int_{S_H}\phi_A(r)\left|\frac{\partial w}{\partial x_n}\right|^2\,dx} - \int_{\Gamma}\phi_A(r)(H-x_n)\nu_n|\nabla_{\Gamma}w|^2ds\\ \nonumber
 %(a-f_-)\int_{\Gamma_a}\phi_A(r)\left\{ \left|\frac{\partial
%w}{\partial x_n}\right|^2- \left|\nabla_{\tx}w\right|^2+
%k^2\left|w\right|^2\right\}\,ds\\&&
&& =\int_{S_H}\left\{\phi_A(r)\left(|\nabla w|^2-k^2|w|^2\right)
-2\phi_A'(r)(x_n-H)\Re\left(\frac{\partial\bar w}{\partial
x_n} \frac{\partial w}{\partial
r}\right)\right\}\,dx\\ \nonumber &&-2\Re\int_{S_H}\phi_A(r)(x_n-H)g\frac{\partial\bar w}{\partial
x_n}\,dx 
\\\nonumber&&-\int_{ \Gamma}(H-x_n) \phi_A(r)\nu_nk^2(1+|\beta|^2)|w|^2ds \\\nonumber && -2k \int_{\Gamma}(H-x_n)\phi_A(r)\Re((e_n.\nabla_{\Gamma} \bar w)i \beta w) ds.
\end{eqnarray}
Now, where $L'=\sqrt{1+L^2}$, 
\begin{equation} \label{nu_Lprime}
-e_n.\nu = -\nu_n \geq \frac{1}{L'}
\end{equation}
and
\begin{equation} \label{grad_gamma}
|e_n.\nabla_{\Gamma} w| \leq \frac{L}{L'}|\nabla_{\Gamma}w|,
\end{equation}
so
%using (\ref{neat}) and (\ref{end11}) we obtain
\begin{eqnarray} \label{labelcat}
\nonumber &&\left|2k\int_{\Gamma}(H-x_n)\phi_A(r)\Re((e_n.\nabla_{\Gamma}\bar w)i\beta w)ds\right|
\\& \leq &  \frac{2kL}{L'}\int_{\Gamma}\phi_A(r)|\nabla_{\Gamma}w|B|w|(H-x_n)ds
\\\nonumber & \leq & \frac{1}{L'}\int_{\Gamma}\phi_A(r)|\nabla_{\Gamma} w|^2(H-x_n) ds\\\nonumber & + & \frac{k^2L^2}{L'}\int_{\Gamma}\phi_A(r)(H-x_n)B^2|w|^2ds,
\end{eqnarray}
while 
\begin{eqnarray} \nonumber
\left|2\Re\int_{S_H}\phi_A(r)(H-x_n)g\frac{\partial \bar w}{\partial x_n}dx\right|
&\leq& \int_{S_H}\phi_A(r)\left|\frac{\partial w}{\partial x_n}\right|^2 dx \\&+& \int_{S_H}\phi_A(r)(H-x_n)^2|g|^2 dx. \label{labelmouse}
\end{eqnarray}
Combining (\ref{main}), (\ref{labelcat}), (\ref{labelmouse}) and noting (\ref{nu_Lprime}) we have 
%and letting $A \to \infty$ and using Lebesque's dominated and monotone convergence theorems
\begin{eqnarray} 
&&{\int_{S_H}\phi_A(r)\left|\frac{\partial w}{\partial x_n}\right|^2\,dx}\\\nonumber 
%- \int_{\Gamma}(H-x_n)\nu_n|\nabla_{\Gamma}w|^2ds\\ \nonumber
& \leq & %(a-f_-)\int_{\Gamma_a}\phi_A(r)\left\{ \left|\frac{\partial
%w}{\partial x_n}\right|^2- \left|\nabla_{\tx}w\right|^2+
%k^2\left|w\right|^2\right\}\,ds\\&&
\int_{S_H}\left\{\phi_A(r)\left(|\nabla w|^2-k^2|w|^2\right)
-2\phi_A'(r)(x_n-H)\Re\left(\frac{\partial\bar w}{\partial
x_n} \frac{\partial w}{\partial
r}\right)\right\}\,dx\\\nonumber & + &\int_{S_H}\phi_A(r)(H-x_n)^2 |g|^2 dx \\\nonumber &
- & \int_{\Gamma}(H-x_n) \phi_A(r)\nu_n k^2(1+B^2)|w|^2ds +\frac{k^2L^2}{L'} \int_{\Gamma}(H-x_n)\phi_A(r)B^2|w|^2 ds.
\end{eqnarray}
%\begin{eqnarray*}
%\int_{S_H}\phi_A(r)\left|\frac{\partial w}{\partial x_n}\right|^2dx \leq \phi_A(r)\sec\Phi\Vert g\Vert_2\Vert w\Vert_2 + \frac{\kappa^2}{k^2}\phi_A(r)\Vert g \Vert_2^2 + k\kappa\int_{\Gamma}\phi_A(r)|\nu_n|(1+B^2)|w|^2ds + \frac{k\kappa L^2}{L'}\int_{\Gamma}\phi_A(r)B^2|w|^2ds.
%\end{eqnarray*}

%\begin{eqnarray*}
%\lefteqn{2\int_{S_a}\phi_A(r)\left|\frac{\partial w}{\partial x_n}\right|^2\,dx}\\\nonumber & \leq &  (a-f_-)\int_{\Gamma_a}\phi_A(r)\left\{ \left|\frac{\partial
%w}{\partial x_n}\right|^2- \left|\nabla_{\tx}w\right|^2+
%k^2\left|w\right|^2\right\}\,ds\\\nonumber &&
%+\int_{S_a}\left\{\phi_A(r)\left(|\nabla w|^2-k^2|w|^2\right)
%-2\phi_A'(r)(x_n-f_-)\Re\left(\frac{\partial\bar w}{\partial
%x_n} \frac{\partial w}{\partial
%r}\right)\right\}\,dx\\\nonumber &&-2\Re\int_{S_a}\phi_A(r)(x_n-f_-)g\frac{\partial\bar w}{\partial
%x_n}\,dx 
%+ 2(H-f_-)\int_{\Gamma}\phi_A(r)L^{\prime}|k\beta w|^2ds \\ \nonumber 
%& + & \frac{3}{2}(a-f_-)\left\{4 \int_{S_H}\phi_A(r) \left| g \right| \left| \frac{ \partial \bar w}{\partial x_n} \right| dx +4 \int_{\Gamma} \phi_A(r)L^{\prime}{\left|k \beta w \right|}^2 ds + 2\int_{ \Gamma}\phi_A(r) k^2 {\left| w \right|}^2 ds \right. \\\nonumber 
%&& \left. +  \int_{S_a}4\phi_A^{\prime}(r)\left|\frac{\partial \bar w}{\partial x_n} \frac{\partial w}{\partial r}\right| dx    -2\int_{\Gamma_a}\phi_A(r) \left\{\left|\frac{\partial w}{\partial x_n}\right|^2 - |\nabla_{\tilde x} w|^2 +k^2|w|^2 \right\}ds\right\}. \\ \nonumber
%\end{eqnarray*}

We now wish to let $A\to \infty$.  The only problem is the term
involving $\phi'_A$ which we estimate as follows. Let
$S_H^b=\left\{x\in S_H\;:\; |\tx|< b\right\}$ for $b\ge 1$. Then
\begin{eqnarray*}
&&\left|\int_{S_H}\left\{2\phi_A'(r)(x_n-H)\Re\left(\frac{\partial\bar{w}}{\partial
x_n} \frac{\partial w}{\partial r}\right)\right\}\,dx\right|\\&&\leq
2M(H-f_-)\int_{S_H^{A+1}\setminus \overline{S_H^A}}|\nabla
w|^2\,dx\to 0
\end{eqnarray*}
as $A\to\infty$, where the convergence follows from the fact that
$w\in H^1(S_H)$. Now letting $A\to\infty$ and using Lebesgue's dominated convergence theorem gives,
\begin{eqnarray} 
\nonumber {\int_{S_H}\left|\frac{\partial w}{\partial x_n}\right|^2\,dx} 
%- \int_{\Gamma}(H-x_n)\nu_n|\nabla_{\Gamma}w|^2ds\\ \nonumber
& \leq & %(a-f_-)\int_{\Gamma_a}\phi_A(r)\left\{ \left|\frac{\partial
%w}{\partial x_n}\right|^2- \left|\nabla_{\tx}w\right|^2+
%k^2\left|w\right|^2\right\}\,ds\\&&
\int_{S_H}\left(|\nabla w|^2-k^2|w|^2\right)
 +\int_{S_H}(H-x_n)^2 |g|^2 dx \\\nonumber
&-&\int_{\Gamma}(H-x_n) \nu_n k^2(1+B^2)|w|^2ds\\ &+&\frac{k^2L^2}{L'} \int_{\Gamma}(H-x_n)B^2|w|^2 ds.
\end{eqnarray}
Use of (\ref{secbd}) leads to
\begin{eqnarray} \label{fish}
\nonumber\int_{S_H}\left| \frac{\partial w}{\partial x_n}\right|^2 dx & \leq & \sec\Phi\Vert g\Vert_2\Vert w\Vert_2 + \frac{\kappa^2}{k^2}\Vert g \Vert_2^2 \\ & + & k\kappa\int_{\Gamma}|\nu_n|(1+B^2)|w|^2ds   +  \frac{k\kappa L^2}{L'}\int_{\Gamma}B^2|w|^2ds\\\nonumber
& \leq & \sec\Phi\Vert g\Vert_2\Vert w\Vert_2 + \frac{\kappa^2}{k^2}\Vert g\Vert^2_2 + k\kappa\int_{\Gamma}\left[1+B^2\left(1+ \frac{L^2}{L'}\right)\right]|w|^2ds\\\nonumber
& \leq &\sec\Phi\Vert g\Vert_2\Vert w\Vert_2 + \frac{\kappa^2}{k^2}\Vert g\Vert^2_2 + k\kappa[1+B^2(1+L)]\int_{\Gamma}|w|^2ds.
\end{eqnarray}
Combining (\ref{Frieds}) and (\ref{fish}) gives
\begin{eqnarray*}
k^2\Vert w\Vert_2^2 \leq \kappa^2\sec\Phi\Vert g\Vert_2\Vert w\Vert_2 + \frac{\kappa^4}{k^2}\Vert g\Vert^2_2 + k\kappa[2 + \kappa^2(1+B^2(1+L))]\int_{\Gamma}|w|^2ds.
\end{eqnarray*}
Using (\ref{w_gamma_bd}) we get 
%\begin{eqnarray*}
%k\int_{\Gamma}|w|^2ds \leq \frac{1}{\eta}\Vert g\Vert_2\Vert w\Vert_2,
%\end{eqnarray*}
%we get 
\begin{eqnarray*}
k^2\Vert w\Vert_2^2 & \leq & \left[ \kappa^2\sec\Phi + \frac{2\kappa+\kappa^3(1+B^2(1+L))}{\eta}\right]\Vert g\Vert_2\Vert w\Vert_2 + \frac{\kappa^4}{k^2}\Vert g\Vert_2^2 \\ & \leq & \frac{1}{2}k^2\Vert w\Vert_2^2 \\&+& \frac{\kappa^2}{k^2}\left[\kappa^2 + \frac{1}{2}\left[\kappa\sec\Phi + \frac{2+ \kappa^2(1+ B^2(1+L))}{\eta}\right]^2\right]\Vert g\Vert_2^2,
\end{eqnarray*}%$w\in H^2(U_b\setminus U_c)$,
so that, using $\sqrt{a^2 + b^2} \leq a+ b$, for $a,b>0,$
\begin{eqnarray*}
k\Vert w\Vert_2 & \leq & \frac{\kappa}{k}\left[2\kappa^2 + \left[\kappa\sec\Phi + \frac{2+\kappa^2(1+B^2(1+L))}{\eta}\right]^2\right]^\frac{1}{2}\Vert g\Vert_2\\ & \leq &
\frac{\kappa}{k}\left[\sqrt{2}\kappa + \left[\kappa\sec\Phi + \frac{2+ \kappa^2(1+B^2(1+L))}{\eta}\right]\right]\Vert g\Vert_2.%for $c>a>b>f_+$, 
\end{eqnarray*}%$\nabla w|_{\Gamma_a}\in L^2(\Gamma_a)$, then, by Lebesgue's dominated 
Defining
\[
 F:= {\kappa}\left[\kappa\left[\sqrt{2}+\sec\Phi\right] + \frac{2+\kappa^2(1+B^2(1+L))}{\eta}\right],
 \]
and using (\ref{secbd}) we get
\begin{eqnarray*}
k^2\Vert w\Vert_{H^1(S_H)}^2 & \leq &  2k^4\Vert w\Vert_2^2 + \sec\Phi k^2\Vert g\Vert_2\Vert w\Vert_2 \\
& \leq & [2F^2 + \sec\Phi F]\Vert g\Vert_2^2   
\end{eqnarray*}
so that
\[
k\Vert w\Vert_{H^1(S_H)} \leq \left(2\sqrt{2}F + \frac{\sec\Phi}{4\sqrt{2}}\right)\Vert g \Vert_2.
\]%convergence theorem,
\end{proof}

Combining lemmas \ref{rellich_2}, \ref{special_case_enough_2} and
\ref{ihl_rem_2} with Corollary \ref{cor_infsup_2}, we have the following
result.
\begin{lemma}\label{apriori_2}
If $\Gamma$ is given by (\ref{Gamma}) with $f$ satisfying (\ref{Lipconst}) and with $f \in C^{\infty}(\mathbb{R}^{n-1})$ and $\beta \in C^{\infty}(\Gamma)$ such that (A2) holds, then the
variational problem (\ref{var_prob2_2}) has a unique solution $u\in
H^1(S_H)$ for every ${\cal G}\in H^1(S_H)^*$ and the solution satisfies the
estimate (\ref{aprioriest_2}).

%In addition the inf-sup condition (\ref{infsup}) is satisfied with
%\[
%\beta\ge \frac{1}{C}.
%\]
\end{lemma}
%Throughout the paper we denote by $\psi_{ \delta} \in C^{\infty}( \mathbb{R}^{n})$ the function that satisfies $\psi_{ \delta} \geq 0$ on $\mathbb{R}^{n}$, $\psi_{\delta}(x) = 0 $ for $|x|> \delta$ and such that
%\[
%\int_{ \mathbb{R}^{n}} \psi_{\delta}(x) dx = 1.
%\]
Before we extend Lemma \ref{apriori_2} to non-smooth surfaces we will need two preliminary and standard lemmas. The first concerns approximation of a Lipschitz function by smooth Lipschitz functions.
%Firstly we will need to use the fact that a Lipschitz function can be approximated by smooth Lipschitz functions:
\begin{lemma} \label{smooth_Lip}
Let $f:\mathbb{R}^{n-1} \rightarrow \mathbb{R}$ be a Lipschitz function with Lipschitz constant $L$. 
% a Lipschitz continuous function, such that 
%\[
%|f(x) - f(y)| \leq L |x-y|, \mbox{ for all } x,y \in \mathbb{R}^{n-1}
%\]
%for some $L \in \mathbb{R}$. 
Then for all $\epsilon > 0$, there exists $f_{ \epsilon} : \mathbb{R}^{n-1} \rightarrow \mathbb{R} $ such that
\newline i) $f_{ \epsilon } \in C^{\infty} ( \mathbb{R}^{n-1})$,
\newline ii) $f_{ \epsilon }$ is Lipschitz and $|f_{ \epsilon}(\tilde x) - f_{ \epsilon}(\tilde y)| \leq L |\tilde x- \tilde y|$, $\quad$  $\tilde x,\tilde y \in \mathbb{R}^{n-1}$,
\newline iii) $f_{ \epsilon} \geq f + \epsilon/6$, 
%+ r$, for some $r>0$,
\newline iv) $\Vert f_{ \epsilon} - f \Vert_{L^{\infty}(\mathbb{R}^{n-1})} < \epsilon$.
\newline v) For $i \in \{1, \dots,n-1\}$, $\tilde \epsilon>0$, and compact $K \subset \mathbb{R}^{n-1}$,
\[
\left\Vert \frac{\partial f}{\partial x_i}- \frac{\partial f_{\epsilon}}{\partial x_i}\right\Vert_{L^p(K)} <\tilde \epsilon, 
\]
for $1< p<\infty$, provided $\epsilon$ is sufficiently small.
\newline vi) $\nabla_{\tilde x}f_{\epsilon}$ is uniformly H\"{o}lder continuous for any index $\alpha \in (0,1)$ i.e.\
\[
\sup_{\tilde x,\tilde z \in \mathbb{R}^{n-1}, \tilde x\neq\tilde z}\frac{|\nabla_{\tilde x}f(\tilde x) -\nabla_{\tilde x}f(\tilde z)|}{|\tilde x-\tilde z|^{\alpha}}< \infty,
\]
so that $f_\epsilon$ is a Lyapunov function.
\end{lemma}
\begin{proof}
Let $\delta = { \epsilon}/(3L)$. Let $\psi_{\delta} \in C_0^{\infty}(\mathbb{R}^{n-1})$ be such that $\psi_{\delta} >0$, $\psi_{\delta}(x)=0$ if $|x|>\delta$ and such that $\int_{\mathbb{R}^n}\psi_{\delta}(x)dx =1.$ Then $\psi_{ \delta} \ast f \in C^{ \infty}( \mathbb{R}^{n-1})$ (e.g.\cite{mclean00} Theorem 3.3). For $\tilde x \in \mathbb{R}^{n-1}$
\begin{eqnarray*}
| \psi_{ \delta} \ast f(\tilde x) - f(\tilde x) | & = & \left| \int_{|\tilde y|< \delta} (f(\tilde x-\tilde y) - f(\tilde x)) \psi_{\delta} (\tilde y) d \tilde y \right| \\
& \leq & \int_{ |\tilde y|< \delta} L|\tilde y| \psi_{ \delta}(\tilde y) d \tilde y   \leq  L \delta = \frac{\epsilon}{3}.
\end{eqnarray*}
Defining $ f_{\epsilon} : \mathbb{R}^{n-1} \rightarrow \mathbb{R}$ by $f_{ \epsilon} = \psi_{ \delta} \ast f + \frac{ \epsilon}{2}$, we see that $f_{ \epsilon} \geq  f +\epsilon/6$ and that 
\[
  \Vert f_{\epsilon} - f \Vert_{L^{\infty}(\mathbb{R}^{n-1})} < \epsilon.
 \]
 
For $\tilde x, \tilde z \in \mathbb{R}^{n-1}$,
\begin{eqnarray*}
|\psi_{\delta} \ast f (\tilde x) - \psi_{\delta} \ast f(\tilde z) |& = & \left|\int_{\mathbb{R}^{n-1}} (f(\tilde x- \tilde y)-f(\tilde z- \tilde y)) \psi_{\delta} (\tilde y) d\tilde y \right|\\
 & \leq & \int_{ \mathbb{R}^{n-1}} L|\tilde x- \tilde z| \psi_{ \delta} (\tilde y)d \tilde y  \leq  L|\tilde x- \tilde z|.
 \end{eqnarray*}

For v), we note that $\partial f_{\epsilon}/{\partial x_i}= \psi_{\delta}*\partial f/\partial x_i$ (e.g. \cite{Federer} page 347), and that $\partial f/\partial x_i \in L^p(K)$ for $1< p<\infty$. Hence for $\phi \in L^q(K)$, where $p^{-1} + q^{-1} =1$,
\begin{eqnarray*}
&&\left|\int_K \left(\frac{\partial f}{\partial x_i} - \frac{\partial f_{\epsilon}}{\partial x_i}\right)\phi d\tilde x\right|\\ & = & \left|\int_K\int_{|\tilde y| <\delta}\left(\frac{\partial f}{\partial x_i}(\tilde x)- \frac{\partial f}{\partial x_i}(\tilde x- \tilde y)\right) \psi_{\delta}(\tilde y) d\tilde y\phi(\tilde x) d\tilde x\right|\\ & = & 
\left|\int_{|\tilde y| <\delta}\int_K \psi_{\delta}(\tilde y)\left(\frac{\partial f}{\partial x_i}(\tilde x)- \frac{\partial f}{\partial x_i}(\tilde x- \tilde y)\right) \phi(\tilde x) d\tilde xd \tilde y \right| \\& \leq & 
\int_{|\tilde y| <\delta} \psi_{\delta}(\tilde y)\left(\int_K \left|\frac{\partial f}{\partial x_i}(\tilde x)- \frac{\partial f}{\partial x_i}(\tilde x-\tilde y)\right|^pd\tilde x\right)^{\frac{1}{p}}\left(\int_K|\phi(\tilde x)|^qd\tilde x\right)^{\frac{1}{q}}d\tilde y ,
\\%& \leq &
%\Vert \phi \Vert_{L^q(K)}  \left(\int_K \left|\frac{\partial f}{\partial x_i}(x)- \frac{\partial f}{\partial x_i}(x-y)\right|^pdx\right)^{\frac{1}{p}}.
\end{eqnarray*}  
using H\"{o}lders inequality. It follows that
\begin{eqnarray*} 
\left\Vert \frac{\partial f}{\partial x_i} - \frac{\partial f_{\epsilon}}{\partial x_i}\right\Vert_{L^p(K)} \leq \sup_{|\tilde y|<\delta}\left(\int_K \left|\frac{\partial f}{\partial x_i}(\tilde x)- \frac{\partial f}{\partial x_i}(\tilde x-\tilde y)\right|^pd\tilde x\right)^{\frac{1}{p}}<\tilde \epsilon,
\end{eqnarray*}
provided $\delta >0 $ and hence $\epsilon >0 $ is sufficiently small. This latter fact is easily shown in the case when $\partial f/\partial x_i \in C^{\infty}(K)$, and holds in general by the density of $C^{\infty}(K)$ in $L^p(K)$. 

Finally for vi) we see that for $\alpha \in (0,1)$ and $\tilde x,\tilde z \in \mathbb{R}^{n-1}$,
\begin{eqnarray*}
\frac{|\nabla_{\tilde x}f(\tilde x) -\nabla_{\tilde x}f(\tilde z)|}{|\tilde x-\tilde z|^{\alpha}}&=& \frac{\left|\int_{\mathbb{R}^{n-1}}\nabla_{\tilde x}\psi_{\delta}(\tilde y)[f(\tilde x-\tilde y)-f(\tilde z-\tilde y)]d \tilde y\right|}{|\tilde z-\tilde x|^{\alpha}}\\&\leq&C_{\epsilon}L|\tilde z-\tilde x|^{1-\alpha},
\end{eqnarray*}
where 
$$
C_{\epsilon}= \int_{\mathbb{R}^{n-1}}|\nabla_{\tilde x}\psi_{\delta}(\tilde y)|d\tilde y.
$$
Part vi) now follows by noting that if $|\tilde x-\tilde z|>1$ say, the result is trivial.
\end{proof}
%\begin{lemma}
%Let $H > f_+$, $g \in L^2(S_H)$ and let $\beta \in C^{\infty}(\Gamma)$ such that $\Re(\beta)\geq \eta>0.$ Suppose $w \in H^1(S_H)$
%satisfies
%\[
%b(w,v) =-(g,v) \mbox{ for all } v \in H^1(S_H).
%\]
%Then
%\[
%\Vert w \Vert_{H^1(S_H)} \leq E(k) \Vert g \Vert_2
%\]
%\end{lemma}

In the next lemma we extend, by reflection, a test function onto a larger domain. We will not make use of standard extension theorems because we need to know explicit bounds.
%but one, we will need to extend a continuous function on a given domain, to a continuous function on a larger domain, such that the norm of the extended function can be controlled by the norm of the original function. Since we will need to know explicit bounds between norms we will not make use of standard extension theorems but instead prove the following simple extension by reflection lemma.
\begin{lemma} \label{ext}
Let $H\geq f_++ \mu$ and suppose $\Gamma$ is given by (\ref{Gamma}) with $f$ Lipschitz with Lipschitz constant $L$. Let  %Fix arbitrary $a > H \geq f_+$, 
%and suppose that $\epsilon >0 $ is such that the function $f_{\epsilon}:\mathbb{R}^{n-1} \to \mathbb{R}$ from lemma \ref{smooth_Lip} satisfies
%Let $
$f^* \in C^{\infty}(\mathbb{R}^{n-1})$ be such that, for $\tilde x,\tilde y \in \mathbb{R}^{n-1}$, 
$$|f^*(\tilde x) - f^*(\tilde y)| \leq L|\tilde x-\tilde y|,$$
$$f^*(\tilde x) \geq f(\tilde x)
$$
 and such that 
%be a bounded Lipschitz function  such that there exists $r>0$ such that
\begin{equation} \label{constraint}
%f_{\epsilon}<a-\frac{(a-f_+)}{2}, 
f^*(\tilde x) + (f^*(\tilde x) -f(\tilde x)) <H. 
%f_m>f+\delta, 
\end{equation} 
Let $S_H^*=D^* \backslash \overline{U_H}$, where $D^*$ is the epigraph of $f^*$. Then, for all $v \in \mathcal{D}({S_H^*})$, we can extend $v$ to a function on $S_H$ such that $v \in H^1(S_H)$, $v|_{S_H\backslash S_H^*} \in \mathcal{D}({S_H \backslash \overline{S_H^*}})$ and 
\begin{equation} \label{nice}
\Vert v\Vert_{H^1(S_H \backslash \overline{S_H^*})} \leq 2\sqrt{(1+4(n-1)L^2)}\Vert v\Vert_{H^1(S_H^*)}.
\end{equation}
% and
%\begin{equation}
%\int_{\mathbb{R}^{n-1}}|v(\tilde{x},f(\tilde x))|^2 d\tilde x \leq \Vert v\Vert^2_{H^1(S_H)}.
%\end{equation}
\end{lemma}
\begin{proof}
For $v \in \mathcal{D}({S_H^*})\subseteq \mathcal{D}(\overline{D^*})$ and for $(\tilde x,x_n ) \in \mathbb{R}^n \setminus D^*$ define $v_E(\tilde x,x_n):=v(\tilde x,2f^*(\tilde x) -x_n)$, so that 
\begin{equation} \label{v_E_1}
\frac{\partial v_E}{\partial x_n}(\tilde x,x_n)= -\frac{\partial v}{\partial x_n}(\tilde x,2f^*(\tilde x) -x_n),
\end{equation}
and for $i\in\{1,\dots,n-1\}$, 
\begin{equation}\label{v_E_2}
\frac{\partial v_E}{\partial x_i}(\tilde x,x_n)= \frac{\partial v}{\partial{x_i}}(\tilde x,2f^*(\tilde x) -x_n) +\frac{\partial v}{\partial x_n}(\tilde x,2f^*(\tilde x)-x_n)2\frac{\partial f^*}{\partial x_i}(\tilde x).
\end{equation}
Hence $\partial v_E / \partial x_i \in \mathcal{D}({S_H \backslash \overline{S_H^*}}) \subseteq L^2(S_H \backslash \overline{S_H^*})$, for $i \in \{1,\dots,n\}$.
Now, if $\hat v(x):= v(x) $ on $D^*$ and $\hat v(x):= v_E(x)$ on $\mathbb{R}^n \setminus D^*$, then, fixing $\phi \in C^{\infty}_0(\mathbb{R}^n)$, $i \in \{1,\dots,n\}$, and where $\mathbf{i}$ denotes the unit vector in the direction $x_i$, we have  
\begin{eqnarray*}
\int_{\mathbb{R}^{n}} \hat v\frac{\partial \bar \phi}{\partial x_i}dx & = & \int_{D^*}v\mathbf{i}\cdot\nabla \bar \phi dx + \int_{\mathbb{R}^{n}\setminus D^*} v_E\mathbf{i}\cdot\nabla \bar\phi dx\\ & = & -\int_{D^*}\frac{\partial v}{\partial x_i} \bar \phi dx - \int_{\mathbb{R}^{n}\setminus D^*} \frac{\partial v_E}{\partial x_i}\bar\phi dx= -\int_{\mathbb{R}^{n}} \bar \phi\frac{\partial \hat v}{\partial x_i}dx,
\end{eqnarray*}
using the divergence theorem and the fact that $v_E= v$ on the graph of $f^*$.
This shows that $\hat v \in H^1(\mathbb{R}^{n})$, so that $v:= \hat v|_{S_H} \in H^1(S_H)$.
The estimates 
\begin{eqnarray*}
\Vert v_E \Vert_{L^2(S_H\setminus S_H^*)} \leq \Vert v \Vert_{L^2(S_H^*)},
\left\Vert \frac{\partial v_E}{\partial x_n} \right\Vert_{L^2(S_H\setminus S_H^*)} \leq \left\Vert \frac{\partial v}{\partial x_n} \right\Vert_{L^2(S_H^*)},\\
\left\Vert \frac{\partial v_E}{\partial x_i}  \right\Vert_{L^2(S_H\setminus S_H^*)} \leq \sqrt{2}\left\{\left\Vert \frac{\partial v}{\partial x_i} \right\Vert_{L^2(S_H^*)} + 2L\left\Vert \frac{\partial v}{\partial x_n}\right\Vert_{L^2(S_H^*)}\right\}, \quad i\in \{1,\dots,n-1\},
\end{eqnarray*}
follow from (\ref{v_E_1}), (\ref{v_E_2}), (\ref{constraint}) and the fact that $\Vert \partial f^*/\partial x_i\Vert_{L^{\infty}(\mathbb{R}^{n-1})} \leq L $, and combine to give (\ref{nice}).
\end{proof}

We next show that lemmas \ref{rellich_2} and \ref{apriori_2} hold for domains with boundaries given by arbitrary Lipschitz graphs. 
\begin{lemma} \label{rough}
%Let $\Omega$ be an $(L,\mu,N)$ Lipschitz domain.
Suppose $\Gamma$ is given by (\ref{Gamma}) with $f$ Lipschitz with Lipschitz constant $L$, $H\geq f_+ + \mu$, $g\in L^2(S_H)$, and $\beta \in C(\Gamma)$ is such that $\beta $ is the restriction to $\Gamma$ of 
$$
\beta^* \in C^{\infty}(\mathbb{R}^n)\cap L^{\infty}(\mathbb{R}^n)\mbox{ such that } \frac{\partial \beta^*}{\partial x_n} = 0, \mbox{ and such that }\Re(\beta^*) \geq \eta>0.
$$ Suppose that $w \in H^1(S_H)$ satisfies
\begin{equation}\label{horrible}
c(w,v)= -(g,v), \quad v \in H^1(S_H).
\end{equation}
Then
\begin{eqnarray*}
 k\Vert w \Vert_{H^1(S_H)} \leq E \Vert g\Vert_2,
\end{eqnarray*}
where $E$ is given by (\ref{E}).
\end{lemma}

\begin{proof} 
Fix a sequence $\epsilon_m \to 0$ such that $\epsilon_{m+1} <{\epsilon_m}/6$, for $m \in \mathbb{N}$. By Lemma \ref{smooth_Lip}, there exists a sequence of Lipschitz functions $f_{\epsilon_m}\in C^{\infty}(\mathbb{R}^{n-1})$, with Lipschitz constant $L$, such that $\Vert f-f_{\epsilon_m}\Vert_{L^{\infty}(\mathbb{R}^{n-1})}< \epsilon_m$, such that $f_{\epsilon_m}\geq f + \epsilon_m/6$, and we may assume that $ 2 f_{\epsilon_m}  -f <H $  for all $m \in \mathbb{N}$. Note that the $f_{\epsilon_m}$ are decreasing. For each $m \in \mathbb{N}$, let $D_m \subseteq \mathbb{R}^n$, denote the epigraph of $f_{\epsilon_m}$, let $S_H^m=D_m \backslash \overline{U_H}$ and let $\Gamma_m = \partial D_m$.
Let $c_m:H^1(S_H^m) \times H^1(S_H^m) \to \mathbb{C}$, be defined by (\ref{sesq_2}) with $S_H, \Gamma $ replaced by $S_H^m, \Gamma_m$ and $\beta $ replaced by $\beta^*$.

Fix $m \in \mathbb{N}$. 
%Note that if $v \in H^1(S_a)$, then $v$ restricted to $S_a^m$ belongs to $H^1(S_a^m)$: let us write this simply as $v \in H^1(S_a^m)$. 
Every $v \in \mathcal{D}({S_H^m})$ can be extended to an element of $H^1(S_H)$ by lemma \ref{ext} such that
\begin{equation} \label{vbd}
\hspace{5ex} \Vert v\Vert_{H^1(S_H\backslash \overline{S_H^m})} \leq 2\sqrt{(1+4(n-1)L^2)}\Vert v\Vert_{H^1(S_H^m)}. 
%\mbox{ and } \left(\int_{\mathbb{R}^{n-1}} |v(\tilde x, f(\tilde x))|^2 d \tilde x\right)^{\frac{1}{2}} \leq \tilde C \Vert v \Vert_{H^1(S_H^m)},
\end{equation} 
%where $\tilde C$ is independent of $m$.

%as follows: let $\gamma_-^m v \in H^{\frac{1}{2}}(\Gamma_m)$ be the trace of $v$ on $\Gamma_m$. Using the right inverse to the Trace operator, we may extend $v$ onto $\mathbb{R}^{n} \backslash \Omega_m$, and then restrict this extension to $S_H \backslash S_H^m$. Moreover there exists a constant $\tilde C$ such that 
%\[
%\Vert v \Vert_{H^1(S_H\backslash S_H^m)}  \leq \tilde C \Vert v \Vert_{H^1(S_H^m)} , \Vert v\Vert_{H^{\frac{1}{2}}(\Gamma)} \leq \tilde C\Vert v\Vert_{H^1(S_H^m)} .
%\]
Now, let $v \in \mathcal{D}({S_H^m})$, fix $\delta >0$ and choose $w_k \in \mathcal{D}({S_H})$ such that 
\newline $\Vert w- w_k\Vert_{H^1(S_H)} < \delta$. Then
\begin{eqnarray}
c_m(w_k,v) & = & \nonumber \int_{S_H^m} \nabla w_k . \nabla \bar v - k^2 w_k \bar v dx
+ \int_{ \Gamma_H}  \bar v T \gamma_-w_k ds - \int_{\Gamma_m} ik \beta^* w_k \bar v ds \\\nonumber
       & = &   c(w_k,v) - \int_{S_H \backslash \overline{S_H^m}} \nabla w_k. \nabla \bar v -k^2w_k \bar v dx + \int_{ \Gamma } ik \beta w_k \bar v ds\\ &-& \int_{ \Gamma_m}  ik \beta^* w_k \bar v ds \\\nonumber
   & = & c(w,v) + c(w_k-w,v) - \int_{S_H \backslash \overline{S_H^m}} \nabla w_k. \nabla \bar v -k^2w_k \bar v dx + \int_{ \Gamma } ik \beta w_k \bar v ds\\& -& \int_{ \Gamma_m}  ik \beta^* w_k \bar v ds\\\nonumber 
   & = &  -\int_{S_H^m} g \bar v dx + c(w_k-w,v) - \int_{S_H \backslash \overline{S_H^m}} \nabla w_k . \nabla \bar v -k^2w_k \bar v +g \bar v dx\\\nonumber  & + & \int_{\Gamma} ik \beta w_k \bar v ds - \int_{ \Gamma_m} ik \beta^* w_k \bar v ds. \\
  \label{error}
\end{eqnarray}
 Now define $\mathcal{H}_m:\mathcal{D}({S_H^m}) \rightarrow \mathbb{C}$ by 
\begin{equation} \label{H_def}
\hspace*{5ex} \mathcal{H}_m(v): =- \int_{S_H \backslash \overline{S_H^m}} \nabla w_k . \nabla \bar{v} -k^2w_k \bar{v} +g \bar{v} dx + \int_{ \Gamma} ik \beta w_k \bar{v} ds - \int_{ \Gamma_m} ik \beta^* w_k \bar{v} ds. 
\end{equation}
To show that $\mathcal{H}_m$ defines a continuous anti-linear functional on $H^1(S_H^m)$, we first of all note that
%Thus we may write 
%\[
%b_m(w_n,v)= b(w_n-w,v)- \int_{S_a^m} g \bar{v} ds  - \mathcal{H}_m(v) \mbox{ for all }  v \in H^1(S_H^m).
%\]
%Since $\Gamma_m \in C^{\infty}( \mathbb{R}^{n-1})$, then by lemma \ref{rellich} and lemma  ref{smoothG} there exist unique $w_1, w_2$ such that 
%\[
%b_m(w_1, v) = - \int_{S_H^m} g \bar{v}, b_m(w_2,v ) = b(w_n-w,v) - \mathcal{H}_m(v)
%\]
%and
%\[
%\Vert w_1 \Vert_{H^1(S_H^m)} \leq E(k) \Vert g \Vert_{L^2(S_H^m)},\mbox{     } \Vert w_2 \Vert_{L^2(S_H^m)} \leq C (\Vert \mathcal{H}_m \Vert_{{H^1(S_H^m)}^*} +const\Vert w_m-w \Vert_{H^1(S_H)})
%\]
%By lemma \ref{smooth_Lip} $E(k)$ and $C$ are independent of $m$. Clearly $w=w_1 + w_2$. So
%\begin{eqnarray}
%\Vert w \Vert_{H^1(S_H^m)} & \leq & E(k) \Vert g \Vert_{L^2(S_H^m)} + C(\Vert \mathcal{H}_m \Vert_{{H^1(S_H^m)}^*} + const\Vert w_n-w\Vert_{H^1(S_H)})\\ & \leq & E(k)\Vert g \Vert_{L^2(S_H)} + C(\Vert \mathcal{H}_m \Vert_{{H^1(S_H^m)}^*} +const\Vert w_n-w\Vert_{H^1(S_H)}).
%\end{eqnarray}
%Let us estimate $\Vert \mathcal{H}_m \Vert_{{H^1(S_H^m)}^*}$: 
%\begin{equation} \label{functional_H}
%|\mathcal{H}_m(v)| \leq \left| \int_{S_a \backslash \overline{S_a^m}} g \bar{v} +\nabla w_n.\nabla \bar{v} - k^2w_n \bar{v} dx \right| + \left| \int_{ \Gamma} ik \beta w_n \bar{v} ds - \int_{\Gamma_m} ik \beta_m w_n \bar{v} ds \right|.
%\end{equation}
%Firstly, 
\begin{eqnarray} \label{first100}
&& \left| \int_{S_H \backslash \overline{S_H^m}}g  \bar v + \nabla w_k.\nabla \bar v -k^2w_k \bar v dx\right| \\\nonumber &\leq&  \left(k^{-1}\Vert g\Vert_{L^2(S_H\backslash \overline{S_H^m})} + \Vert \nabla w_k\Vert_{L^2(S_H \backslash \overline{S_H^m})}\right.\\&&\left. + k\Vert w_k\Vert_{L^2(S_H \backslash \overline{S_H^m})}\right)2\sqrt{(1+4(n-1)L^2)}\Vert v\Vert_{H^1(S_H^m)}\nonumber
\end{eqnarray}
using (\ref{vbd}).

%|^2dx\int_{S_H \backslash S_H^m} |v|^2 dx\right)^{\frac{1}{2}}\\& + & \left(\int_{S_H \backslash S_H^m} |\nabla w_n|^2 dx \int_{S_H \backslash S_H^m} |\nabla v|^2 dx\right)^{\frac{1}{2}}\\& + & \left(\int_{S_H \backslash S_H^m}k^2|w_n|^2 dx\int_{S_H \backslash S_H^m} |v|^2 dx\right)^{\frac{1}{2}}.
%\end{eqnarray}
To estimate the second term on the right hand side of (\ref{H_def}), define
\newline  $h: S_H \setminus S_H^m \to \mathbb{R}$ by $h(\tilde x,x_n) =J_f(\tilde x)= \sqrt{1+ |\nabla f(\tilde x)|^2}$ for all $\tilde x$ at which $f$ is differentiable. 
In addition let $K= \mbox{supp}w_k$, let $\Vert \cdot \Vert$ denote $\Vert \cdot \Vert_{L^{\infty}(\mathbb{R}^n)}$, and let
\[
l(\tilde x)= \left(\int_{\mathbb{R}^{n-1} \cap K}\left|J_{f_{\epsilon_m}}(\tilde x)- J_f(\tilde x)\right|^2\right)^{\frac{1}{2}}.
\]

Then,  
\begin{eqnarray} \nonumber
&& \left| \int_{ \Gamma_m} ik \beta^* w_k \bar{v} ds - \int_{\Gamma} ik \beta w_k \bar{v} ds \right| \\\nonumber & = & \left|\int_{\mathbb{R}^{n-1}}J_{f_{\epsilon_m}}(\tilde x)ik\beta^* w_k\bar v (\tilde x,f_{\epsilon_m}(\tilde x))d\tilde x %\right.\\&&\left.
-
 \int_{\mathbb{R}^{n-1}}J_f(\tilde x)ik\beta w_k\bar v(\tilde x,f(\tilde x))d\tilde x\right|\\\nonumber  & \leq & \left|\int_{\mathbb{R}^{n-1}}J_f(\tilde x)\int_{f(\tilde x)}^{f_{\epsilon_m}(\tilde x)}\frac{\partial}{\partial x_n}(ik\beta^*w_k\bar v)(x)dx_n d\tilde x\right|\\\nonumber & + & \left|\int_{\mathbb{R}^{n-1}\cap K}\left(J_{f_{\epsilon_m}}(\tilde x)- J_f(\tilde x)\right)ik\beta^*w_k\bar v(\tilde x,f_{\epsilon_m}(\tilde x))d \tilde x\right|\\\nonumber 
 & \leq & \sqrt{1+L^2}\left\{\left(\int_{S_H \setminus S_H^m}k^2|\beta^*|^2|w_k|^2dx\right)^{\frac{1}{2}}\left(\int_{S_H \setminus S_H^m}\left|\frac{\partial v}{\partial x_n}\right|^2dx\right)^{\frac{1}{2}}\right.\\\nonumber & + &\left.\left(\int_{S_H \setminus S_H^m}|\beta^*|^2\left|\frac{\partial w_k}{\partial x_n}\right|^2dx\right)^{\frac{1}{2}}\left(\int_{S_H \setminus S_H^m}k^2| v|^2dx\right)^{\frac{1}{2}}\right\} \\\nonumber & + & l(\tilde x) %\left(\int_{\mathbb{R}^{n-1}\cap K}\left|\sqrt{1+|\nabla f_{\epsilon_m}(\tilde x)|^2}- \sqrt{1+|\nabla f(\tilde x)|^2}\right|^2\right)^{\frac{1}{2}}
 \left(\int_{\mathbb{R}^{n-1}}|ik\beta^*w_k\bar v(\tilde x,f_{\epsilon_m}(\tilde x))|^2d\tilde x\right)^{\frac{1}{2}}\\\nonumber & \leq &
 \sqrt{1+L^2}\left\{k\Vert\beta^*\Vert\left(\int_{S_H\setminus S_H^m}|w_k|^2dx\right)^{\frac{1}{2}} \right.\\\nonumber&&\left.+\Vert\beta^*\Vert\left(\int_{S_H\setminus S_H^m}\left|\frac{\partial w_k}{\partial x_n}\right|^2dx\right)^{\frac{1}{2}}\right\}\sqrt{2(1+4(n-1)L^2)} \Vert v \Vert_{H^1(S_H^m)}
\\ & + & l(\tilde x)
%\left(\int_{\mathbb{R}^{n-1} \cap K}\left|\sqrt{1+|\nabla f_{\epsilon_m}(\tilde x)|^2}- \sqrt{1+|\nabla f(\tilde x)|^2}\right|^2\right)^{\frac{1}{2}}
\Vert\beta^*\Vert\Vert w_k\Vert k^{\frac{1}{2}}\sqrt{\sqrt{1+L^2}\left(1+\frac{2}{k \mu}\right)}\Vert v \Vert_{H^1(S_H^m)}, \label{error_est}
\end{eqnarray}
using lemma \ref{traces} and assuming that $S_H^m$ is an $(L,\mu/2,1)$ Lipschitz domain, which it is, provided $\epsilon_m <\mu/2$. 

We may now write (\ref{error}) as
\begin{equation} \label{error2}
c_m(w_k,v)= - \int_{S_H^m} g \bar{v} dx  +c(w_k-w,v) + \mathcal{H}_m(v), \quad  v \in \mathcal{D}({S_H^m}).
\end{equation}
By the density of $\mathcal{D}({S_H^m})$ in $H^1(S_H^m)$, and the continuity of $c_m, c$ and $\mathcal{H}_m$, (\ref{error2}) must hold for all $v \in H^1(S_H^m)$.
Since $\Gamma_m \in C^{\infty}( \mathbb{R}^{n-1})$ and $\beta^* \in C^{\infty}(\Gamma_m)$ then by lemma \ref{apriori_2} there exist unique $w', w''$ such that 
\[
c_m(w', v) = - \int_{S_H^m} g \bar{v}dx, \quad  c_m(w'',v ) =  c(w_k-w,v) + \mathcal{H}_m(v)
\]
and by lemmas \ref{rellich_2} and \ref{apriori_2} 
\begin{eqnarray*}
&&\Vert w' \Vert_{H^1(S_H^m)} \leq k^{-1}E \Vert g \Vert_{L^2(S_H^m)}, \\ &&\Vert w'' \Vert_{L^2(S_H^m)} \leq \sec\Phi(1+2E) \left\{\Vert \mathcal{H}_m \Vert_{{H^1(S_H^m)}^*} + \Vert c\Vert\Vert w_k-w\Vert_{H^1(S_H)}\right\}.
\end{eqnarray*}
By lemma \ref{smooth_Lip} and the fact that $\partial \beta^*/\partial x_n=0$, $E$ and $\sec\Phi$ are independent of $m$. Clearly $w_k=w' + w''$. So
\begin{eqnarray}
\Vert w_k \Vert_{H^1(S_H^m)} & \leq &\nonumber k^{-1}E \Vert g \Vert_{L^2(S_H^m)} \\\nonumber&+& \sec\Phi(1+2E) \left\{\Vert\mathcal{H}_m \Vert_{{H^1(S_H^m)}^*}+\Vert c\Vert\Vert w_k-w\Vert_{H^1(S_H)}\right\}.\\ 
%& \leq & E(k)\Vert g \Vert_{L^2(S_H)} + C(\Vert \mathcal{H}_m \Vert_{{H^1(S_H^m)}^*} +const\Vert w_n-w\Vert_{H^1(S_H)}).
\end{eqnarray}
Now let $m \to \infty$, using (\ref{error_est}) to estimate $\Vert \mathcal{H}_m \Vert_{H^1(S_H^m)^*}$, using Lebesgue's monotone convergence theorem to show that the terms  $\Vert \cdot\Vert_{L^2(S_H\backslash \overline{S_H^m})} \to 0$, and using lemma \ref{smooth_Lip} part v) we see that
\[
\Vert w_k \Vert_{H^1(S_H)} \leq k^{-1}E\Vert g \Vert_2 + \sec\Phi(1+2E)\Vert c\Vert \delta.
\]
Finally arbitrariness of $\delta>0$ gives the result.%\begin{equation}
%\Vert w \Vert_{H^1(S_H)} \leq \sqrt{E(k)}\Vert g\Vert_2.
%\end{equation}
\end{proof}

Combining lemmas \ref{rough}, \ref{special_case_enough_2} and
\ref{ihl_rem_2} with Corollary \ref{cor_infsup_2}, we have the following
result.
\begin{lemma} \label{roughext}
If $\Gamma$ is given by (\ref{Gamma}) with $f$ Lipschitz with Lipschitz constant $L$, and $\beta \in C(\Gamma)$, satisfies the hypotheses of lemma \ref{rough}, then the
variational problem (\ref{var_prob2}) has a unique solution $u\in
H^1(S_H)$ for every ${\cal G}\in H^1(S_H)^*$ and the solution satisfies the
estimate (\ref{aprioriest}).
\end{lemma}

%\begin{lemma}
%Suppose $\Gamma$ is given by (\ref{Gamma}). Given $g \in L^2(S_H)$ and $\beta \in C^{\infty}(\Gamma)$ such that $\Re(\beta)\geq \eta>0$, then (\ref{weak_form}) has a unique solution which satisfies the (\ref{soln_bd}).
%\end{lemma}

We now show that lemmas \ref{rough} and \ref{roughext} hold for more general $\beta \in L^{\infty}(\Gamma)$. 
\begin{lemma} \label{betaext}
Suppose $\Gamma$ is given by (\ref{Gamma}) with $f$ Lipschitz with Lipschitz constant $L$. Let $H\geq f_+ + \mu$, $g \in L^2(S_H),$ and $\beta \in L^{\infty} ( \Gamma)$ be such that $\Re (\beta) \geq \eta>0$, and suppose $w \in H^1(S_H)$ satisfies
\begin{equation} \label{the_above}
b(w, v) = -(g,v), \quad  v \in H^1(S_H).
\end{equation}
Then 
\[
k\Vert w \Vert_{H^1(S_H)} \leq E \Vert g\Vert_2, 
\]
where $E$ is given by (\ref{E}).
\end{lemma}
\begin{proof}
For $\delta>0$ let $\psi_{\delta} \in C_0^{\infty}(\mathbb{R}^n)$ be such that $\psi_{\delta} >0$, $\psi_{\delta}(x)=0$ if $|x|>\delta$, and such that $\int_{\mathbb{R}^n}\psi_{\delta}(x)dx =1$
Then define, $ \beta _{\delta} \in C^{\infty}(\mathbb{R}^{n-1})$ by 
\[
\beta_{\delta}(\tilde x)= \int_{\mathbb{R}^{n-1}} \beta (\tilde x-\tilde y,f(\tilde x-\tilde y)) \psi_{\delta}(\tilde y) d\tilde y,
\]
and then extend $\beta_{\delta}$ to a function $\beta_{\delta} \in C^{\infty}(\mathbb{R}^n)$ via $\beta_{\delta}(\tilde x,x_n) =\beta_{\delta}(\tilde x)$. It follows that $\beta_{\delta} \in C(\Gamma)$ and that $\beta_{\delta}$ is the restriction to $\Gamma$ of a function $\beta_{\delta} \in C^{\infty}(\mathbb{R}^n)$ such that $\partial \beta_{\delta}/\partial x_n =0$. 
Note that, for $\tilde x \in \mathbb{R}^{n-1}$,
\begin{eqnarray*}
\Re(\beta_{\delta}(\tilde x)) &=& \Re\int_{\mathbb{R}^{n-1}}\beta(\tilde x-\tilde y,f(\tilde x-\tilde y))\psi_{\delta}(\tilde x)d\tilde x \\&=& \int_{\mathbb{R}^{n-1}}\psi_{\delta}(\tilde x )\Re\beta(\tilde x-\tilde y,f(\tilde x-\tilde y))d\tilde x \geq \eta,
\end{eqnarray*}
and
\[
|\beta_{\delta}(\tilde x)| \leq \int_{\mathbb{R}^{n-1}}\psi_{\delta}(\tilde x)|\beta(\tilde x-\tilde y,f(\tilde x-\tilde y))|d\tilde x \leq B \Rightarrow \Vert\beta_\delta \Vert_{L^{\infty}(\mathbb{R}^{n-1})} \leq B.
\]
Further, since $\Re(e^{-i(\pi/2 + \Phi)}\beta) \geq 0$, it follows, by arguing as above, that \newline $\Re(e^{-i(\pi/2 + \Phi)}\beta_{\delta})\geq 0.$ This ensures  
%\[
%\inf_{\tilde x \in \mathbb{R}^{n-1}}\arg\beta_{\delta}(\tilde x)= \int_{\mathbb{R}^{n-1}}|\beta (\tilde x-\tilde y)|\arg\beta(\tilde x-\tilde y) \psi_{\delta}(\tilde y)d\tilde y\geq \Phi\int_{\mathbb{R}^{n-1}}|\beta (\tilde x-\tilde y)| \psi_{\delta}(\tilde y)d\tilde y,
%\]
that $\Phi_{{\delta}}:=\min\{0,\inf_{x \in \mathbb{R}^n}\arg\beta_{\delta}\} \geq \Phi$, which in turn means that $\sec\Phi_{{\delta}}\leq\sec\Phi$.

Fix $\epsilon >0$, and choose $w_m \in \mathcal{D}({S_H})$ such that $\Vert w_m - w\Vert_{H^1(S_H)} < \epsilon.$ 
%Fix $\delta>0$, and consider the function 
%$\beta_{\delta}:=\psi_{\delta'}*\beta$, for some $\delta'>0$.
Standard arguments (e.g. \cite{mclean00} Theorem 3.4) show that %if we choose $\delta$ sufficiently small then 
$\beta_{\delta} \to \beta$ in the normed space $L^2(\supp w_m \cap \Gamma)$. Thus if we choose $\delta$ sufficiently small then 
%For $S\in \mathbb{R}$, let $B_S:=\{(\tilde x,f(\tilde x)) : |\tilde x|\leq S\}$. Then for some $S_1 \in mathbb{R}$, $\mbox{supp}w_m \subseteq B_{S_1}$. Fix $\delta > 0$. Since $\beta|B \in L^2(B_{S_1+1})$, there exists $\beta_{\delta} \in C^{\infty}(B_{S_1+1})$ such that 
\begin{equation}
\left(\int_{\Gamma}k|\beta_{\delta}(s) - \beta(s)|^2|w_m(s)|^2 ds\right)^{\frac{1}{2}}< \sqrt{k}\Vert w_m\Vert_{L^{\infty}(\Gamma)}\Vert\beta_{\delta} -\beta \Vert_{L^2(\supp w_m \cap \Gamma)} <\epsilon.
\end{equation}
%because $w_m$ has compact support and $\Gamma$ is the graph of a Lipschitz function.
%Let $w \in H^1(S_H)$ satisfy (\ref{the_above}). Fix $\epsilon >0$, and choose $w_m \in D(\overline{S_H})$ such that $\Vert w_m - w\Vert_{H^1(S_H)} < \epsilon.$ Then
Now
\begin{equation} \label{good}
c(w_m, v)= -(g,v) + c(w_m -w,v), \quad v \in H^1(S_H),
\end{equation}
so that, for $v \in H^1(S_H)$,
\begin{eqnarray*}
&&\int_{S_H} \nabla w_m . \nabla \bar v - k^2w_m \bar v dx +
 \int_{\Gamma_H} \gamma_-\bar v Tw_m ds - \int_{\Gamma} ik \beta_{\delta} w_m \gamma^*\bar v ds \\ &=&  -\int_{S_H}g \bar v dx + c(w_m -w,v) - \int_{\Gamma}ik(\beta_{\delta} - \beta)w_m \gamma^*\bar v ds.
\end{eqnarray*}
Since $\beta_{\delta}$ satisfies the hypotheses of lemma \ref{rough}, then, by lemmas \ref{roughext}, \ref{rough} and \ref{traces}, (cf proof of lemma \ref{rough}) we obtain
\[
\Vert w_m\Vert_{H^1(S_H)} \leq k^{-1}E \Vert g \Vert_2 + \sec\Phi(1+2E)\left[\Vert c\Vert\epsilon +  \epsilon\sqrt{\sqrt{1+L^2}\left(1+\frac{1}{k\mu}\right)}\right], 
\]
and the result follows by arbitrariness of $\epsilon>0$.
%where 
%\begin{equation} \label{E'}
%E'(k):=3k^2\left(3(H-f_-)^2\left(A(k) +\frac{B(k)^2}{4}\right) + \frac{4(H-f_-)^2}{(\eta k)^2}\right) + \left(1+\frac{3\Vert\beta\Vert}{\eta}\right)^2/4,
%\end{equation}
%\begin{equation} \label{A'}
%A'(k):=\frac{2(H-f_-)}{\eta k}\left(2\kappa +1 +4\frac{\kappa}{\eta}(2L^{\prime}3\Vert \beta\Vert^2 +1) + \frac{1}{\eta}( 3\Vert \beta\Vert +\kappa L^{\prime} 3\Vert \beta\Vert^2) \right)
%\end{equation}
%and
%\begin{equation} \label{B'}
%B'(k):=(H-f_-) \left(\sqrt{3}(2\kappa +1) +10+ 4\sqrt{3}\frac{\kappa}{\eta}(2L^{\prime}3\Vert \beta\Vert^2 + 1) +\frac{\sqrt{3}}{\eta}(3\Vert\beta\Vert +\kappa L^{\prime}3\Vert\beta\Vert^2)\right).
%\end{equation}
%By arbitrariness of $\delta$ we get that
%\begin{equation}
%\Vert w_m \Vert_{H^1(S_H)} \leq  \sqrt{E'(k)}\Vert g\Vert_2 + \mbox{const}\epsilon
%\end{equation}
%so that
%\begin{equation}
%\Vert w\Vert_{H^1(S_H)} \leq \sqrt{E'(k)}\Vert g \Vert_2.
%\end{equation}
%Properties of $\beta_{\delta}$
%\begin{equation}
%\Re((\psi_{\epsilon} * \beta)(x))  =  \Re\int_{|y|<\epsilon} \beta(x-y) \psi_{\epsilon}(y) dy\\  =  \int_{|y|<\epsilon}\Re(\beta(x-y))\psi_{\epsilon}(y) dy\\ \geq  \eta. 
%\end{equation}
%Extend $\beta_{\delta}$ to a function on $\Gamma$ by defining it to be 
%\begin{equation}
%\beta_{\delta}(x):= \beta
%\end{equation}
%Important
\end{proof} 

Theorem \ref{mainresultof} now follows by combining lemmas \ref{betaext}, \ref{special_case_enough_2} and
\ref{ihl_rem_2} with Corollary \ref{cor_infsup_2}.

%\bibliographystyle{plain}
%\bibliography{mathbib}

\chapter{The Transmission problem}\label{trans}
\section{Literature review} 

In this chapter we study the transmission problem -- or the problem of scattering by an inhomogeneous layer -- applying once again the methods and results of \cite{chandmonk} to this problem. Thus in terms of style and approach this work follows on from \cite{chandmonk}, \cite{kirschipmp93}, \cite{els02},  \cite{Bonnetmmas94} and \cite{Szembergmmas98} (c.f.\ the literature review of chapter 2).

As outlined in the introduction, given a source $g \in L^2(\mathbb{R}^n)$ that is confined to a strip, the transmission problem will be to find a solution $u$ to the Helmholtz equation
\[
\Delta u +k^2 u =g \mbox{ in } \mathbb{R}^n,
\] 
for $n=2,3$, where the function $k\in L^{\infty}(\mathbb{R}^n)$ varies in a strip containing the source $g$.

We point out that included in our problem set up -- see our exact formulation in the next section -- is the related problem of `scattering by a rough interface'. Here the problem is to study the scattering of electromagnetic or acoustic waves by a rough interface above and below which $k$ is assumed to take different constant values.

We note that in this chapter we will assume that $k$ is a real-valued function. It then follows that,
in the 2D case, we are modelling the scattering of time harmonic electromagnetic waves by an infinite inhomogeneous \emph{dielectric} layer at the interface between semi-infinite homogeneous dielectric half-spaces, with the magnetic permeability a fixed positive constant in the media, in the transverse electric polarization case.

In \cite{RZ_1} Roach and Zhang showed existence and uniqueness of solution to the problem of scattering by a rough interface in the case $n\geq3$, when the interface was supposed to be the graph of a $C^1$ function that became a flat surface at infinity. In \cite{Z_2} (the 2D case) and \cite{Z_3} (the 3D case), Zhang considers the transmission problem that we consider here and obtains existence and uniqueness results but under assumptions on the behaviour of $k$ at infinity. The best results to date are those obtained by Chandler-Wilde and Zhang in \cite{chand_zhang} which show, in the 2D case, existence and uniqueness of solution to this problem for arbitrary $k \in L^{\infty}(\mathbb{R}^2)$ satisfying certain other restrictions without which one can show the problem to be ill-posed. Our results can be seen as an improvement on these in that they hold in both 2 and 3 dimensions and moreover we slightly generalise the assumptions on $k$ made in \cite{chand_zhang} -- see assumptions 4 and 5 below. We should also point out that in establishing an a priori bound on our solution (see lemma \ref{rellich_3}) we borrow some of the techniques used to derive an a priori bound in \cite{chand_zhang}.

We should also mention the papers of Zhang and Roach \cite{RZ_2}; Zhang \cite{Z_1}; and of Anar and Torun \cite{Anar}; all of whom consider the problem of scattering by a rough interface but this time with transmission conditions across the interface requiring that $u$ and its normal derivative jump across said interface.

\section{The Transmission problem and variational formulation}\label{varf_3}
%In this chapter we make a study of the transmission problem, applying once more the methods and techniques of \cite{chandmonk}.

%we shall define some notation related to the transmission problem and then precisely state this problem and an equivalent variational formulation that will be analyzed in
%the next section. 
In contrast to the problems studied in the the other chapters, the transmission problem is a problem posed on the whole of $\mathbb{R}^n$. As such we will not impose any boundary condition but rather, we shall impose two radiation conditions.
As usual, for $x=(x_1,\ldots,x_n)\in\real^n$ ($n=2,3$) let
$\tx=(x_1,\ldots,x_{n-1})$ so that $x=(\tx,x_n)$.  For $H\in \real$,
let $U_H=\left\{x\;:\; x_n>H\right\}$ and
$\GH:=\left\{x\;:\;x_n=H\right\}$. For $a<b$ let $S(a,b)= U_a \setminus \overline{U_b}$.  %Let $D\subset \real^n$ be a
%connected open set such that for some constants $f_-<f_+$ it holds
%that
%\begin{equation}
%U_{f_+}\subset D \subset U_{f_-}.\label{Weq1}
%\end{equation}
%This definition of $D$ (the domain of the electromagnetic field) allows the
%rough surface $\Gamma=\partial D$ to be more general than the graph of
%a function.  
The variational problem will be posed on the strip 
$S:=S(h_-,h_+)$, for some $h_+>h_-$.

Given a source $g\in L^2(\mathbb{R}^n)$ and given $k \in L^{\infty}(\mathbb{R}^n)$, a real valued function, such that for some $h_+>h_-$, the support of $g$ lies in $ S(h_-,h_+)$, and such that $k=k_+>0$ in $\overline{U_{h_+}}$, and such that $k=k_->0$ in $\mathbb{R}^n\setminus U_{h_-}$, the problem we wish
to analyze is to find a function $u$ such that
\begin{eqnarray}
\Delta u+k^2 u =g\mbox{ in } \mathbb{R}^n,\label{helm_3}
\end{eqnarray}
and such that $u$ satisfies the upward and downward propagating radiation conditions ((UPRC) and (DPRC) respectively) above and below the inhomogeneous layer $S$. To state the DPRC precisely, recall from chapter 1, the fundamental solution of the Helmholtz equation $\Phi$, for given wavenumber $k_*>0$:
$$
\Phi(x,y;k_{*})=\left\{\begin{array}{ll}\displaystyle\frac{\ri}{4}H_0^{(1)}(k_*|x-y|),&n=2,\\
\displaystyle\frac{\exp(\ri k_*|x-y|)}{4\pi
|x-y|},&n=3,\end{array}\right. %\label{Weq2}
$$
for $x,y\in\real^n$, $x\neq y$, where $H_0^{(1)}$ is the Hankel
function of the first kind of order zero.  
Then the UPRC -- see (\ref{uprc}) -- states that
\begin{equation}
u(x)=2 \int_{\Gamma_{h_+}}\frac{\pa\Phi(x,y;k_*)}{\pa x_n} u(y)\,
ds(y):=\mathcal{R}_1(x,u|_{\Gamma_{h_+}},k_*),%\label{uprc}       
\quad x\in U_{h_+},
\end{equation}
for all $h_+$ such that the support of $g$ is contained in $\mathbb{R}^n \setminus U_{h_+}$ and such that $k=k_*$ in $\overline{U_{h_+}}$. 
Similarly the DPRC states that 
\begin{equation}
u(x)=-2 \int_{\Gamma_{h_-}}\frac{\pa\Phi(x,y;k_*)}{\pa x_n} u(y)\,
ds(y):=\mathcal{R}_2(x,u|_{\Gamma_{h_-}},k_*),\label{dprc} \quad x\in \mathbb{R}^n\setminus \overline{U_{h_-}},
\end{equation}
for all $h_-$ such that the support of $g$ is contained in $\overline{U_{h_-}}$ and such that $k=k_*$ in $\mathbb{R}^n \setminus U_{h_-}$.

Let us recall also from chapter 1 that if $u|_{\Gamma_{h_+}} \in L^2(\Gamma_{h_+})$ then we may rewrite the UPRC in terms of the Fourier transform of $u|_{\Gamma_{h_+}}$, $\mathcal{F}(u|_{\Gamma_{h_+}})$: we have that
\begin{eqnarray}
\nonumber u(x)&=&\frac{1}{(2\pi)^{(n-1)/2}}\int_{\real^{n-1}}\exp(\ri[(x_n-h_+)\sqrt{k_*^2-\xi^2}+
\tx\cdot\xi])\mathcal{F}(u|_{\Gamma_{h_+}})(\xi)\,d\xi,\;\; %x\in U_{h_+}
\\&:=&\hat{\mathcal{R}}(x,u|_{\Gamma_{h_+}},k_*),\quad  x\in U_{h_+}. \hspace{2ex}\label{uprcstar_trans}
\end{eqnarray}

For convenience let us take the origin in $\mathbb{R}^n$ to be such that $-h_+=h_-$, and for $x=(\tilde x,x_n) \in \mathbb{R}^n$ let $x'=(\tilde x,-x_n)$. Moreover for any function $v:\mathbb{R}^n\backslash \overline{U_{h_-}} \to \mathbb{C}$ define $v':U_{h_+}\to \mathbb{C}$ via $v'(x)=v(x')$. Let us now remark that the DPRC can be expressed, through reflection, in terms of the UPRC.

\begin{remark} \label{up=down}
\[
u(x)= \mathcal{R}_2(x,u|_{\Gamma_{h_-}},k_*) \quad x \in \mathbb{R}^n\backslash \overline{U_{h_-}} 
\]
 if, and only if, 
\[
u'(x)= \mathcal{R}_1(x,u'|_{\Gamma_{h_+}},k_*) \quad  x \in U_{h_+}.
\]
\end{remark}

Thus it follows by remark {\ref{up=down}} that if $u|_{\Gamma_{h_-}} \in L^2(\Gamma_{h_-})$ then $u(x)$ satisfies the DPRC (\ref{dprc}) if, and only if,
$$
u'(x)= \hat{\mathcal{R}}(x,u'|_{\Gamma_{h_+}},k_*), \quad x\in U_{h_+}.
$$  
 
%Equation (\ref{uprcstar}) is a representation for $u$, in the upper
%half-plane $U_{h_+}$, as a superposition of upward propagating
%homogeneous and inhomogeneous plane waves. A requirement that
%(\ref{uprcstar}) holds is commonly used (e.g. \cite{DeSanto02}) as a
%formal radiation condition in the physics and engineering literature
%on rough surface scattering. The meaning of (\ref{uprcstar}) is
%clear when $F_{h_+}\in L^2(\real^{n-1})$ so that $\hat F_{h_+}\in
%L^2(\real^{n-1})$; indeed the integral (\ref{uprcstar}) exists in
%the Lebesgue sense for all $x\in U_{h_+}$. Recently Arens and Hohage
%\cite{are04a} have explained, in the case $n=2$, in what precise
%sense (\ref{uprcstar}) can be understood when $F_{h_+}\in BC(\Gamma_{h_+})$,
%the space of bounded continuous functions on $\Gamma_{h_+}$, so that
%$\hat F_{h_+}$ must be interpreted as a tempered distribution.

%The above discussion motivates the following precise formulation of
%problem (\ref{helm})-(\ref{db1}). 
We now precisely state the transmission problem. Let $H^1(S)$ denote the standard
Sobolev space,
\[
H^1(S):= \{v \in L^2(S) |\nabla v \in L^2(S)\}
\]
%$\|\cdot\|_{H^1(D)}$ defined by $\Vert
%u\Vert_{H^1(D)}=\{\int_{D}(|\nabla u|^2+|u|^2)dx\}^{1/2}$. The main
%function space in which we set our problem will be the Hilbert space
%$V_H$, defined, for $H\ge f_+$, by
%\[
%V_H:=\left\{\phi|_{S_H}\;:\;\phi\in H_0^1(D)\right\},
%\]
on which we will impose a wave number dependent scalar product
$(u,v)_{H^1(S)} := \int_{S} (\nabla u \cdot \overline{\nabla v} +k_+^2
u\bar v)\,dx$ 
and norm, 
 $\Vert u\Vert_{H^1(S)}=\{\int_{S}(|\nabla
u|^2+k_+^2|u|^2)dx\}^{1/2}$.

\paragraph{{\sc The Transmission Problem}} {\em Given $g\in L^2(\mathbb{R}^n)$, and $k \in L^{\infty}(\mathbb{R}^n)$ such that for some $h_+>h_-$, it holds that
the support of $g$ lies in $\overline{U_{h_-}}\setminus U_{h_+}$, and that $k=k_+$, in $\overline{U_{h_+}}$, and $k=k_-$ in $\mathbb{R}^n \setminus U_{h_-}$, for some $k_+, k_->0$,
 find $u:\mathbb{R}^n\to \complex$ such that $u|_{S(a,b)}\in H^1(S(a,b))$ for every $a<h_-$ and $b>h_+$,
\begin{equation}
\Delta u+k^2u=g\quad\mbox{ in } \mathbb{R}^n\label{whelm}
\end{equation}
in a distributional sense, and such that the following radiation conditions hold:
\begin{equation} \label{rad_1}
u(x)= \hat{\mathcal{R}}(x,u|_{\Gamma_{h_+}},k_+), \quad x \in U_{h_+},
\end{equation}
and
\begin{equation}\label{rad_2}
u'(x)= \hat{\mathcal{R}}(x,u'|_{\Gamma_{h_+}},k_-), \quad  x \in U_{h_+}.
\end{equation}}
%by (\ref{uprcstar})for $x \in U_{h_+}$ with $k_*=k_-$ and $F_{h_+}= u'|_{\Gamma_{h_+}}$. 
\begin{remark}
Additional assumptions on $k \in L^{\infty}(\mathbb{R}^n)$ will be made from section 3 onwards in order to establish well-posedness of the boundary value problem. It is well known that the problem is ill-posed for certain functions $k\in L^{\infty}(\mathbb{R}^n)$.
\end{remark}
\begin{remark}
We note that, as one would hope, the solutions of the above
problem do not depend on the choice of $h_-$ and $h_+$. Precisely, if $u$ is a
solution to the above problem for a given pair $h_+, h_-$ for
which $\supp g\subset \overline{S(h_-,h_+)}$ and $k=k_+$ in $\overline{U_{h_+}}$, and $k=k_-$ in $\mathbb{R}^n \setminus {U_{h_-}}$  then $u$ is a solution for
all pairs $h_-,h_+$ with this property. To see that this is true is a
matter of showing that, if (\ref{rad_1}) and (\ref{rad_2}) hold for one pair $h_+,h_-$ such that $\supp g\subset \overline{S(h_-,h_+)}$ and $k=k_+$ in $\overline{U_{h_+}}$, and $k=k_-$ in $\mathbb{R}^n \setminus {U_{h_-}}$ then (\ref{rad_1}) and (\ref{rad_2}) hold for
all pairs $h_+,h_-$ with this property. It was shown in Lemma \ref{lemma3p2} that if (\ref{rad_1}) holds, with $\mathcal{F}(u|_{\Gamma_{h_+}}) \in H^{\frac{1}{2}}(\Gamma_{h_+})$,
for some $h_+$, then it holds for all larger values of $h_+$.
One way to show that (\ref{rad_1}) holds also for every smaller
value of $h_+$, $\tilde h$ say, for which 
$\supp g\subset \mathbb{R}^n \backslash U_{\tilde h}$ and $k=k_+$ in $\overline{U_{\tilde h}}$, is to consider the
function
\begin{eqnarray*}
v(x) & := & u(x) -\\
&&\frac{1}{(2\pi)^{(n-1)/2}}\int_{\real^{n-1}}\exp(\ri[(x_n-\tilde
h)\sqrt{k_+^2-\xi^2}+ \tx\cdot\xi]) \hat{F}_{\tilde h}(\xi)\,d\xi,\;\;
x\in U_{\tilde h},
\end{eqnarray*}
with $F_{\tilde h}:=u|_{\Gamma_{\tilde h}}$, and show that $v$ is
identically zero. To see this we note that, by Lemma
\ref{lemma3p2}, $v$ satisfies the boundary value problem
of chapter \ref{layer} with $D=U_{\tilde h}$ and $g=0$. That $v\equiv 0$ then follows
from Theorem \ref{th_main1}. Similar arguments apply to the other radiation condition (\ref{rad_2}).% of \cite{chandmonk}.

\end{remark}

%As indicated in the above discussion, it is known that the above
%boundary value problem has a solution in the case $n=2$ when
%$\Gamma$ is the graph of a  sufficiently smooth function.  A main
%result of this paper is to prove that the boundary value problem
%is uniquely solvable, both in two and three dimensions, under much
%more general conditions on the boundary $\Gamma$.  Moreover we
%provide explicit estimates of the norm of the solution in the
%strip $S_H$ as a function of the dimensionless wave number
%\begin{equation}
%\kappa=k(H-f_-). \label{kappadef}
%\end{equation}

We now derive a variational formulation of the transmission problem above. We proceed exactly as in chapters \ref{layer} and \ref{imp} except that it's necessary to introduce more complicated notation. Again we will use standard
fractional Sobolev space notation, except that for convenience we adopt wave
number dependent norms, which are both equivalent to the usual norm. Thus, identifying $\Gamma_{h_{\pm}}$ with $\real^{n-1}$,
$H^s(\Gamma_{h_{\pm}})$, for $s\in \real$, denotes the completion of
$C_0^\infty(\Gamma_{h_{\pm}})$ in the norm $\|\cdot\|_{H^s(\Gamma_{h_{\pm}})}$
defined by
$$
 \|\phi\|_{H^s(\Gamma_{h_{\pm}})} = \left(
\int_{\real^{n-1}}(k_{\pm}^2+\xi^2)^s|{\cal F} \phi(\xi)|^2\,
d\xi\right)^{1/2}.
$$
  We recall \cite{adamsSS} that, for all $a>h_+$ and $ b<h_-$, there
exist continuous embeddings (the trace operators)
\begin{eqnarray*}
\gamma_+^{\downarrow}:H^1(U_{h_+}\setminus U_a)\to H^{1/2}(\Gamma_{h_+}), \quad  
\gamma_+^{\uparrow}:H^1(S) \to H^{1/2}(\Gamma_{h_+}), \\ 
\gamma_-^{\downarrow}:H^1(S)\to H^{1/2}(\Gamma_{h_-}),  \quad  
\gamma_-^{\uparrow}:H^1(U_{b}\setminus U_{h_-}) \to H^{1/2}(\Gamma_{h_-}),
\end{eqnarray*}
such that each operator acting on $\phi$ coincides with the
restriction of $\phi$ to $\Gamma_{h_{\pm}}$ when $\phi$ is $C^\infty$. 
We recall also the following fact that, if $u_+\in
H^1(U_{h_+}\setminus U_{a})$, $u_-\in H^1(S)$, and $\gamma_+^{\downarrow}u_+ =
\gamma_+^{\uparrow} u_-$, then $v\in H^1(S(h_-,a))$, where $v(x) := u_+(x)$, $x \in
U_{h_+}\setminus U_a$, $:= u_-(x)$, $x \in S$. Conversely, if $v\in
H^1(S(h_-,a))$ and $u_+:= v|_{U_{h_+}\setminus U_a}$, $u_-:= v|_{S}$, then
$\gamma_+^{\downarrow}u_+=\gamma_+^{\uparrow}u_-$. 

For given wavenumber $k_*>0$ let $T_{k_*}: \mathbb{R}^{n-1} \to\mathbb{C} $ be defined by
% introduce the operator's $T_+,T_-$, which
%will prove to be Dirichlet to Neumann map's on $\Gamma_{h_+}, \Gamma_{h_-}$ respectively (see
%(\ref{eq:DtNdef}) below), defined by
\begin{equation}
T_{k_*}:=\cF^{-1}M_{z(k_*)}\cF, \label{Tdef_3}
\end{equation}
where $M_{z(k_*)}$ is the operation of multiplying by
\[
z(\xi):=\left\{\begin{array}{ll}-\ri\sqrt{k_*^2-\xi^2}&\mbox{if
}|\xi|\le k_*,
\\[3pt]
\sqrt{\xi^2-k_*^2}&\mbox{for }|\xi|>k_*.
\end{array}\right.
\]
Specifically we are concerned with the maps $T_+:\Gamma_{h_+} \to \mathbb{C}$ and $T_-:\Gamma_{h_-} \to \mathbb{C}$ given by $T_+=T_{k_+}$ and $T_-=T_{k_-}$ which
will prove to be Dirichlet to Neumann maps on $\Gamma_{h_+}, \Gamma_{h_-}$ respectively (see
(\ref{eq:DtNdef_trans}) below). Also, lemma \ref{LTB} shows that
$T_{\pm}:H^{1/2}(\Gamma_{h_{\pm}}) \to H^{-1/2}(\Gamma_{h_{\pm}})$ are bounded and that $\Vert T_{\pm}\Vert =1$.

We now restate lemma \ref{lemma3p2}, in the new notation we have introduced.
% which gives properties of $u$ as defined by (\ref{uprcstar_chap3}) in $U_{h_+}$.

\begin{lemma} \label{lemma3p2_trans}
If $u(x)= \hat{\mathcal{R}}(x,u|_{\Gamma_{h_+}},k_*)$ with $u|_{\Gamma_{h_+}}\in H^{1/2}(\Gamma_{h_+})$, then
$u\in H^1(U_{h_+}\setminus U_a)\cap C^2(U_{h_+})$, for every $a>h_+$,
$$
\Delta u + k_*^2 u = 0 \mbox{ in } U_{h_+},
$$
$\gamma_+^{\downarrow}u = u|_{\Gamma_{h_+}}$, if $u|_{\Gamma_{h_+}}\in C_0^\infty(\Gamma_{h_+})$ then
\begin{equation} \label{eq:DtNdef_trans}
T_{k_*}\gamma_+^{\downarrow}u = -\partial u/\partial x_n|_{\Gamma_{h_+}},
\end{equation}
and
\begin{equation} \label{exact_rc_3}
\int_{\Gamma_{h_+}}  \bar vT_{k_*}\gamma_+^{\downarrow}u\,ds + k_*^2\int_{U_{h_+}} u\bar v \,dx
- \int_{U_{h_+}} \nabla u\cdot \nabla \bar v\, dx = 0, \quad v \in
C_0^{\infty}(\mathbb{R}^n).
\end{equation}
Further, the restrictions of $u$ and $\nabla u$ to $\Gamma_a$ are
in $L^2(\Gamma_a)$, for all $a> h_+$, and
\begin{equation} \label{lem32_3}
\int_{\Gamma_a} \left[ \left|\frac{\partial u}{\partial
x_n}\right|^2 - \left|\nabla_{\tilde x} u\right|^2 +
k_*^2|u|^2\right]\,ds \le -2k_*\Im \int_{\Gamma_a} \gamma_+^{\downarrow}\bar u
T_{k_*} \gamma_+^{\downarrow}u\, ds.
\end{equation}
Moreover, for all $a>h_+$,  it holds that for $x$ in $U_a$, $u(x)=\hat{\mathcal{R}}(x,u|_{\Gamma_a},k_*)$ with
$h_+$ replaced by $a$.
\end{lemma}

Now suppose that $u$ satisfies the Transmission problem. Then
$u|_{S(a,b)}\in H^1(S(a,b))$ for every $a<h_-,b>h_+$ and, by definition, since
$\Delta u+k^2u=g$ in a distributional sense,
\begin{equation} \label{he_ds_3}
\int_{\mathbb{R}^n}[g\bar v + \nabla u\cdot\nabla \bar v - k^2 u\bar v] dx = 0,
\quad v\in C_0^\infty(\mathbb{R}^n).
\end{equation}
Applying Lemma \ref{lemma3p2_trans}, and defining $w := u|_{S}$, it
follows that
$$
\int_{S} [g\bar v + \nabla w\cdot \nabla \bar v - k^2 w\bar v]\,
dx + \int_{\Gamma_{h_+}} \bar v T_+\gamma^{\uparrow}_+w \, ds  + \int_{\Gamma_{h_-}}\bar v T_-\gamma^{\downarrow}_-wds= 0, \quad v\in
C_0^\infty(\mathbb{R}^n).
$$
From the denseness of $\{\phi|_{S}:\phi\in C_0^\infty(\mathbb{R}^n)\}$ in
$H^1(S)$ and the continuity of $\gamma^{\downarrow}_-$ and $\gamma^{\uparrow}_+$, it follows that this
equation holds for all $v\in H^1(S)$.

% Equation (\ref{DtN})
%also holds for every $H\ge f_+$ if $u$ is replaced by $u^s_H:=
%u - u^i_H$, where $u_H^i$ denotes the incident field given by (\ref{uidef})
%with $a=H$.

Let $\|\cdot\|_2$ and $(\cdot,\cdot)$ denote the norm and scalar
product on $L^2(S)$, so that $\Vert
v\Vert_2=\sqrt{\int_{S}|v|^2\,dx}$ and
\[
(u,v)=\int_{S}u\overline{v}\,dx,
\]
and define the sesquilinear form $d:H^1(S)\times H^1(S)\to \complex$ by
\begin{equation} \label{sesqui_3}
d(u,v) = (\nabla u,\nabla v) - (k^2u,v) + \int_{\Gamma_{h_+}} \bar v T_+\gamma^{\uparrow}_+w \, ds  + \int_{\Gamma_{h_-}}\bar v T_-\gamma^{\downarrow}_-wds.
\end{equation}
Then we have shown that if $u$ satisfies the boundary value
problem then $w:= u|_{S}$ is a solution of the following
variational problem: find $u\in H^1(S)$ such that
\begin{equation} \label{weak_form_trans}
d(u,v) = -(g,v), \quad v\in H^1(S).
\end{equation}
%where
%$$
%\tilde g(v) := -(g,v) - \int_{\GH}\left(\frac{\pa u_H^i}{\pa
%x_2}+Tu_H^i\right)\bar vds.
%$$

Conversely, suppose that $w$ is a solution to the variational
problem and define $u(x)$ to be $w(x)$ in $S$, to be $\hat{\mathcal{R}}(x,\gamma^{\uparrow}_+ w,k_+)$
%the
%right hand side of (\ref{uprcstar_chap3}), with $F_H:= \gamma^{\uparrow}_+ w$, 
in
$U_{h_+}$ and to be $l(x)$ in $\mathbb{R}^n\backslash U_{h_-}$ where $l'(x)$, in $U_{h_+}$,  is given by $\hat{\mathcal{R}}(x,\gamma^{\downarrow}_- w,k_-).$
%right hand side of (\ref{uprcstar_chap3}), with $F_H:= \gamma^{\downarrow}_- w$. 
Then, by Lemma \ref{lemma3p2_trans}, $u\in H^1(U_{h_+}\setminus U_b)$ and $u\in H^1(U_{a}\setminus U_{h_-})$
for every $b>h_+ $ and $a<h_-$, with $\gamma^{\downarrow}_+u  = \gamma^{\uparrow}_+w$ and $\gamma^{\uparrow}_-u  = \gamma^{\downarrow}_-w$. Thus
$u|_{S(a,b)}\in H^1(S(a,b))$, $b>a$. Further, from (\ref{exact_rc_3}) and
(\ref{weak_form_trans}) it follows that (\ref{he_ds_3}) holds, so that
$\Delta u + k^2u = g$ in $\mathbb{R}^n$ in a distributional sense. Thus $u$
satisfies the transmission problem.

We have thus proved the following theorem.
\begin{theorem} \label{th_equiv_trans} If $u$ is a solution of the transmission problem then
$u|_{S}$ satisfies the variational problem.  Conversely, if $u$
satisfies the variational problem, %$F_H := \gamma_- u$, 
and the
definition of $u$ is extended to $\mathbb{R}^n$ by setting $u(x)$ equal to 
$\hat{\mathcal{R}}(x,\gamma^{\uparrow}_+ u,k_+)$
%the
%right hand side of (\ref{uprcstar_chap3}), with $F_H:= \gamma^{\uparrow}_+ w$, 
for $x$ in
$U_{h_+}$ and to be $l(x)$ in $\mathbb{R}^n\backslash U_{h_-}$ where $l'(x)$, for $x$ in $U_{h_+}$, is given by $\hat{\mathcal{R}}(x,\gamma^{\downarrow}_- u,k_-)$,
%the
%right hand side of (\ref{uprcstar_chap3}) with $F_H := \gamma^{\uparrow}_+ u$  for $x\in U_{h_+}$, and setting $u(x)$ to be equal to $l(x)$ in $\mathbb{R}^n \backslash U_{h_-}$ where $l'(x)$, in $U_{h_+}$,  is the
%right hand side of (\ref{uprcstar_chap3}), with $F_H:= \gamma^{\downarrow}_- u$, 
then
 the extended function satisfies the transmission problem, with $g$ extended by zero from $S$ to $\mathbb{R}^n$ and $k$ extended from $S$ to $\mathbb{R}^n$ by taking the value $k_+$ in ${U_{h_+}}$ and the value $k_-$ in $\mathbb{R}^n\backslash \overline{U_{h_-}}$.
\end{theorem}

We conclude this section by showing that the sesquilinear form $d(.,.)$ is bounded, establishing an explicit value for the bound.
\begin{lemma}\label{WL4_trans}
For all $u,v\in H^1(S)$,
\[
|d(u,v)|\leq \left[%\max\left\{1,
\frac{k^2_{\infty}}{k_+^2}%\right\}
+\left(1+\frac{1}{k_+(h_+-h_-)}\right) +\left(1+\frac{1}{k_-(h_+-h_-)}\right) 
%\max\left\{1,\frac{k^2_+}{k_0^2}\right\}
\right]
\Vert u\Vert_{H^1(S)}\Vert v \Vert_{H^1(S)}
\]
so that the sesquilinear form $d(.,.)$ is bounded.
\end{lemma}

\begin{proof}
From the definition of the sesquilinear form $d(.,.)$, the
Cauchy-Schwarz inequality and the mapping properties of $T_+$ and $T_-$we have
\begin{eqnarray*}
|d(u,v)|&\leq& \|\nabla u\|_2\|\nabla v\|_2+\frac{k^2_{\infty}k_+^2}{k_+^2}\|u\|_2\|v\|_2 +
\|\gamma^{\uparrow}_+u\|_{H^{1/2}(\Gamma_{h_+})}
\|T_+\|\,\|\gamma^{\uparrow}_+v\|_{H^{1/2}(\Gamma_{h_+})} \\&+&\|\gamma^{\downarrow}_-u\|_{H^{1/2}(\Gamma_{h_-})}
\|T_-\|\,\|\gamma^{\downarrow}_-v\|_{H^{1/2}(\Gamma_{h_-})}.
\end{eqnarray*}
To obtain the desired result %the Cauchy-Schwarz inequality and 
%Lemmas \ref{LTB} and
%\ref{WL3} we obtain the desired estimate. ??Applying the Cauchy-Schwarz inequality and Lemmas \ref{LTB} and
we apply lemma \ref{traces} with $\mu = (h_+-h_-)$.% we obtain the desired estimate.
\end{proof}

\section{Analysis of the variational problem}
\label{hik_3}

%The sesquilinear form $b(.,.)$ is not $V_H$-elliptic if $k \in L^{\infty}(S_H)$ is large.  
In this section we shall establish, under assumptions 4 and 5 below, on the function $k \in L^{\infty}(\mathbb{R}^n)$, that the transmission problem and the equivalent
variational problem are uniquely solvable by using the generalized
Lax-Milgram theory of Babu\v{s}ka.

As usual our analysis will
also apply to the following slightly more general problem: given
${\cal G}\in H^1(S)^*$ find $u\in H^1(S)$ such that
\begin{equation} \label{var_prob2_trans}
d(u,v) = {\cal G}(v), \quad v\in H^1(S).
\end{equation}

The assumptions we make are:
\newline \textbf{Assumption 4.} For some $\beta \in [h_-,h_+]$, $k^2$ is monotonic non-increasing on $U_{h_-}\setminus U_{\beta}$ and monotonic non-decreasing on $U_{\beta}\setminus U_{h_+}$.

 We then set 
\begin{eqnarray*} \tilde k(x)& =  & k_+(x) ,\quad x \in U_{\beta}\setminus U_{h_+}\\  
                                & = &  k_-(x) ,\quad  x \in U_{h_-}\setminus U_{\beta}, 
\end{eqnarray*}
so that assumption 4 implies that $\tilde k^2(x) - k^2(x) \geq 0 $ for all $x \in S$.
\newline \textbf{Assumption 5.} For some $\epsilon >0, \lambda_3>0$ it holds that $\tilde k^2(x) - k^2(x) \geq \lambda_3$, for all $x \in \mathcal{C}:= \{(\tilde x,x_n)|x_n \in [f(\tilde x)-\epsilon,f(\tilde x) +\epsilon]\}$, where $f \in L^{\infty}(\mathbb{R}^n)$ is Lipschitz with Lipschitz constant $L$
%tilde x)-f(\tilde y)| \leq L|\tilde x- \tilde y| \quad \tilde x, \tilde y \in \mathbb{R}^2.
%$$
and such that $\mathcal{C} \subseteq S.$
\begin{remark}
If $k^2 \in C^1(\mathbb{R}^n)$ and $k^2$ satisfies assumption 4 then we may write this assumption succinctly as
\[
\frac{\partial k^2}{\partial x_n}(x_n-\beta) \geq 0.
\]
\end{remark}
%\begin{remark} Without loss of generality we may assume that $\mathcal{C} \subseteq S$.
%\end{remark}

In what follows we always assume that there exists $k_0>0$, such that $k(x)\geq k_0$ $x \in \mathbb{R}^n$. We make the abbreviations 
$\kappa_0:= k_0(h_+-h_-)$, $ \kappa_+:= k_+(h_+-h_-)$, $ \kappa_-:= k_-(h_+-h_-)$ and $ \kappa_{\infty}:= k_{\infty}(h_+-h_-)$. Our main result in this section is then the following:

\begin{theorem}\label{th_main1_trans} If  Assumptions 4 and 5 hold then the variational
problem (\ref{var_prob2_trans}) has a unique solution $u\in H^1(S)$ for every
${\cal G}\in {H^1(S)}^*$ and
\begin{equation}
\Vert u\Vert_{H^1(S)}\le [1+ k_{\infty}^{-1}C_1]\left[k_+ +\frac{k_{\infty}^2}{k_+}\right]\Vert {\cal G}\Vert_{{H^1(S)}^*}
\label{aprioriest_3}
\end{equation}
where
\begin{eqnarray}\nonumber
C_1^2& =&  k_{\infty}\sqrt{2\left[ \frac{P^2}{2k_{\infty}^2}[2\kappa_++ 2\kappa_- +1]^2 + P(h_+-h_-)^2\right]} \\\nonumber & + &4k_{\infty}^2\left[ \frac{P^2}{2k_{\infty}^2}[2\kappa_++ 2\kappa_- +1]^2 + P(h_+-h_-)^2\right]\\\label{C_def_chap4}
\end{eqnarray}
and where 
\[
P=  \kappa_{\infty}^2 + 4\kappa_{\infty}k_{\infty}\sqrt{1+L^2}\{\epsilon + \epsilon^{-1}2\lambda_3^{-1}(1+4\kappa_{\infty}^2)\}.
\]

In particular, the transmission problem and the equivalent
variational problem (\ref{weak_form_trans}) have exactly one solution, and
the solution satisfies the bound
\[
k_{\infty}\Vert w\Vert_{H^1(S)}\leq   C_1\Vert g\Vert_{2}.
\]
\end{theorem}
%Moreover, we shall estimate all stability and inf-sup constants
%explicitly, as functions of the dimensionless parameter $\kappa$, so
%that the $k$-dependence of constants will be apparent.
%In order to motivate the assumption that we will impose on $k \in L^{\infty}(D)$, we will consider a simple case when our boundary value problem is ill-posed. Consider the boundary value problem in 2D, with the boundary $\partial D$, being flat, so that $D$ is a half plane.
%Suppose that $\Im(k^2)=0$, that $k^2 \in C^1(D)$ such that $\partial k^2/\partial x_1 =0, \partial k^2/\partial x_2=-\lambda_1$.
%Then looking for solutions of 
 
To apply the generalized Lax-Milgram theorem %(e.g. \cite[Theorem
%2.15]{ihlenburg}) 
we need to show that $d$ is bounded which we have done in lemma \ref{WL4_trans}; to establish the inf-sup condition
that
\begin{equation}
\alpha := \inf_{0\not=u\in H^1(S)}\sup_{0\not=v\in
H^1(S)}\frac{|d(u,v)|}{\Vert u\Vert_{H^1(S)}\Vert v\Vert_{H^1(S)}} >0;
\label{infsup_trans}
\end{equation}
and to establish the transposed inf-sup condition. Noting that from lemma \ref{WL1} 
$d$ satisfies the following symmetry property, that
\[
d(v,u) =d(\bar u,\bar v) \quad u,v \in H^1(S),
\]
it follows
easily %from Corollary \ref{symmetry} 
that the transposed inf-sup
condition follows automatically if (\ref{infsup_trans}) holds.

%\begin{lemma}\label{WEXL6}
%If $v\in V_H$ and $b(v,v)=0$ then
%\begin{equation}
%b(u,v)=\overline{b(v,u)}, \quad u\in V_H.\label{WEXeqc}
%\end{equation}
%\end{lemma}

%\begin{proof} If $b(v,v)=0$ then $\Im b(v,v)=\Im
%\int_{\GH}\gamma_- \bar v\,T\gamma_- v\,ds=0$ and (\ref{WEXeqc})
%follows from Lemma \ref{WL1}.
%\end{proof}

\begin{lemma}\label{WEXL7_trans} If (\ref{infsup_trans}) holds then, for all non-zero $v\in H^1(S)$,
\[
\sup_{0\not=u\in H^1(S)}\frac{|d(u,v)|}{\Vert u\Vert_{H^1(S)}}>0.
\]
\end{lemma}

\begin{proof}
If (\ref{infsup_trans}) holds and $v\in H^1(S)$ is non-zero then
\[
\sup_{0\not=u\in H^1(S)}\frac{|d(u,v)|}{\Vert u\Vert_{H^1(S)}}=
\sup_{0\not=u\in H^1(S)}\frac{|d(\bar v,u)|}{\Vert u\Vert_{H^1(S)}}\geq
\alpha \Vert v\Vert_{H^1(S)}>0.
\]
This proves the lemma.
\end{proof}

The following result follows from \cite[Theorem 2.15]{ihlenburg}
and Lemmas \ref{WL4_trans} and \ref{WEXL7_trans}.

\begin{corollary} \label{cor_infsup_3} If (\ref{infsup_trans}) holds then
the variational problem (\ref{var_prob2_trans}) has exactly one solution
$u\in H^1(S)$ for all ${\cal G}\in {H^1(S)}^*$. Moreover
$$
\|u\|_{H^1(S)} \le \alpha^{-1}\|{\cal G}\|_{{H^1(S)}^*}.
$$
\end{corollary}

%Our approach to establishing (\ref{infsup}) will be to first show
%that (\ref{infsup}) holds in the case that $\Gamma$ is the graph of
%a smooth function, given by (\ref{}) with $f\in BC(\real^{n-1})\cap
%C^\infty(\real^{n-1})$.
To show (\ref{infsup_trans}) we will establish an
a priori bound for solutions of (\ref{var_prob2_trans}), from which the
inf-sup condition will follow by the following easily established
lemma (see \cite[Remark 2.20]{ihlenburg}).

\begin{lemma} \label{ihl_rem_3}
Suppose that there exists $C>0$ such that, for all $u\in H^1(S)$ and
${\cal G}\in {H^1(S)}^*$ satisfying (\ref{var_prob2_trans}) it holds that
\begin{equation} \label{boundA_trans}
\|u\|_{H^1(S)} \le C \|{\cal G}\|_{{H^1(S)}^*}.
\end{equation}
Then the inf-sup condition (\ref{infsup_trans}) holds with $\alpha\ge
C^{-1}$.
\end{lemma}

The following lemma reduces the problem of establishing
(\ref{boundA_trans}) to that of establishing an a priori bound for
solutions of the special case (\ref{weak_form_trans}).
\begin{lemma} \label{special_case_enough_trans}
Suppose there exists $\tilde C>0$ such that, for all $u\in H^1(S)$ and
$g\in L^2(S_H)$ satisfying (\ref{weak_form_trans}) it holds that
\begin{equation} \label{boundB_trans}
\|u\|_{H^1(S)} \le k_{\infty}^{-1}\tilde C\, \|g\|_2.
\end{equation}
Then, for all $u\in H^1(S)$ and ${\cal G}\in {H^1(S)}^*$ satisfying
(\ref{var_prob2_trans}), the bound (\ref{boundA_trans}) holds with
$$
C  \le \left(1 +  k_{\infty}^{-1}\tilde C\left[k_+ + \frac{k_{\infty}^2}{ k_+}\right]\right)
%$\frac{k_+^2}{k_0^2}\left[1 + 2  \,\tilde C\frac{k_{\infty}^2}{k_0k_+}\right].
$$
\end{lemma}
\begin{proof}
 Suppose $u\in H^1(S)$ is a solution of
\begin{equation}
d(u,v)={\cal G}(v),\quad v\in H^1(S), \label{uweak_trans}
\end{equation}
where ${\cal G}\in {H^1(S)}^*$.   Let $d_0:H^1(S)\times H^1(S)\to \C$ be
defined by
\[
d_0(u,v)=(\nabla u,\nabla
v)+k_+^2(u,v)+\int_{\Gamma_{h_+}}\gamma_+^{\uparrow}\overline{v}\,T_+\gamma_+^{\uparrow}u\,ds
+ \int_{\Gamma_{h_-}}\gamma_-^{\downarrow}\overline{v}\,T_-\gamma_-^{\downarrow}u\,ds,
\]
for $u,v\in H^1(S)$.
It follows from Lemma \ref{WL1} that $d_0$ satisfies %is $V_H$-elliptic, in that
\[
\Re \,d_0(v,v)\ge \Vert v\Vert_{H^1(S)}^2, \quad v\in H^1(S).
\]
Thus the problem of finding $u_0\in H^1(S)$ such that
\begin{equation}
d_0(u_0,v)={\cal G}(v),\quad v\in H^1(S), \label{u0weak_trans}
\end{equation}
has a unique solution which satisfies
\begin{equation}
\Vert u_0\Vert_{H^1(S)}\leq \Vert{\cal G}\Vert_{{H^1(S)}^*}. \label{uoap_trans}
\end{equation}
Furthermore, defining $w=u-u_0$ and using (\ref{uweak_trans}) and
(\ref{u0weak_trans}), we see that
\[
d(w,v)=d(u,v)-d(u_0,v)={\cal G}(v)-({\cal
G}(v)-k_+^2(u_0,v)-(k^2u_0,v))=((k_+^2 + k^2)u_0,v),
\]
for all $v\in H^1(S)$. Thus $w$ satisfies (\ref{weak_form_trans}) with
$g=-(k_+^2 +k^2)u_0$. It follows, using (\ref{uoap_trans}) and (\ref{boundB_trans}),
% and Lemma \ref{WL3}DO WE NEED THIS, 
that
\begin{eqnarray} \label{boundC_trans}
\Vert w\Vert_{H^1(S)}\leq k_{\infty}^{-1} \tilde C(k_+^2 +k_{\infty}^2)\Vert u_0\Vert_2\leq k_{\infty}^{-1}\tilde C\left[k_+ + \frac{k_{\infty}^2}{k_+ }\right]
\Vert {\cal G}\Vert_{{H^1(S)}^*}.%= %2\tilde C\frac{k_{\infty}^2k_+}{k_0^3}\Vert {\cal G}\Vert_{V_H^*}.
\end{eqnarray}
The bound (\ref{boundA_trans}), with 
$$ 
C \le \left(1 +  k_{\infty}^{-1}\tilde C\left[k_+ + \frac{k_{\infty}^2}{ k_+}\right]\right),
$$ 
follows
from  (\ref{uoap_trans}) and (\ref{boundC_trans}).
\end{proof}

In lemma \ref{rellich_3} below we will need to make use of the following result, a trace lemma whose proof can be found in the appendix.
\begin{lemma}\label{trace_trans}
Let $f: \mathbb{R}^{n-1} \to \mathbb{R}$ be a bounded Lipschitz function with Lipschitz constant $L$ and let $\mathcal{C}:= \{(\tilde x,x_n)|x_n \in [f(\tilde x)-\epsilon,f(\tilde x) +\epsilon]\}$.
Then for $w \in H^1(\mathcal{C})$ it holds that
\[
\epsilon \int_{\Gamma} |w|^2ds \leq \sqrt{1+L^2}\left\{\epsilon^2\left\Vert \frac{\partial w}{\partial x_n}\right\Vert^2_{L^2(\mathcal{C})} + \Vert w\Vert^2_{L^2(\mathcal{C})}\right\}.
\]
%\int_{\Gamma} |w|^2ds \leq \left\vert \frac{\partial w}{\partial x_n}\right\vert_{L^2(\mathcal{C})} + \left\vert w\right\vert_{L^2(\mathcal{C})}.
%\]
\end{lemma}

Following these preliminary lemmas we turn now to establishing the a
priori bound (\ref{boundB_trans}), at first just for the case when
$k \in C^{\infty}(\mathbb{R}^n)$. 

\begin{lemma}\label{rellich_3}
Let $h_+>h_-$, $g\in L^2(S)$ and
suppose that $k \in C^{\infty}(\mathbb{R}^n)$ (with $k=k_+$ in $U_{h_+}$ and $k=k_-$ in $\mathbb{R}^n \backslash \overline{U_{h_-}})$ satisfies assumptions 4 and 5. Then, if $w\in H^1(S)$ satisfies
\begin{equation}
d(w,\phi)= -(g,\phi),\quad \phi\in H^1(S) %\label{bgprob}
\end{equation}
then
\begin{eqnarray*}
k_{\infty}^2\Vert w\Vert_{H^1(S)}^2\leq  C_1^2 \Vert g\Vert_2^2
\end{eqnarray*}
where $C_1^2$ is given by (\ref{C_def_chap4}).
%\begin{eqnarray*}
%C& =&  k_+\sqrt{2\left[ \frac{2P^2}{k^2}[2(h_+-\beta)k_++ 2(\beta -h_-)k_- +1]^2 + 4P\max\{|h_+-\beta|,|h_--\beta|\}^2\right]} \\ & + &2k_+^22\left[ \frac{2P^2}{k^2}[2(h_+-\beta)k_++ 2(\beta -h_-)k_- +1]^2 + 4P\max\{|h_+-\beta|,|h_--\beta|\}^2\right]
%\end{eqnarray*}
%and where 
%\[
%P=  \kappa_+^2 + 8\kappa_+k_+\sqrt{1+L^2}\{\epsilon + \epsilon^{-1}\lambda_3^{-1}(1+4\kappa_{\infty}^2)\}.
%\]

%If in addition
%$\nu_n\leq -\alpha^2<0$ for some constant $\alpha>0$ and if $f_-$ is
%chosen so that $(f(x_1)-f_-)>(H_{min}-f_-)$ for all $x_1$ then
%\[
%\left\Vert\frac{\partial w}{\partial \nu}\right\Vert_{L^2(\Gamma)}
%\leq \frac{(H-f_-)}{\alpha\sqrt{H_{min}-f_-}}\left(\kappa +\frac{1}{2}+\sqrt{2}\right)
%\Vert g\Vert_2.
%\]
\end{lemma}

\begin{proof}%The proof of this lemma is
%motivated  by \cite{Melenk,cum04}, where a Rellich identity is used
%to prove estimates for solutions of the Helmholtz equation posed on
%bounded domains, by the proofs of the basic inequalities for rough
%surface scattering problems in \cite{chandsjam98,zhangsjam98}, and
%by the estimates derived for the diffraction grating problem in
%\cite{els02}.

 Let $r=|\tx|$. For $A\ge 1$ let $\phi_A\in
C_0^{\infty}(\real)$ be such that $0\leq \phi_A\leq 1$,
$\phi_A(r)=1$ if $r\le A$ and $\phi_A(r)=0$ if $r\ge A+1$ and
finally such that $\Vert\phi_A'\Vert_{\infty}\leq M$ for some fixed
$M$ independent of $A$.

Extending the definition of $w$ to the whole of $\mathbb{R}^n$ by letting \newline $w(x)=\hat{\mathcal{R}}(x,\gamma^{\uparrow}_+ w,k_+)$
%the
%right hand side of (\ref{uprcstar_chap3}), with $F_H:= \gamma^{\uparrow}_+ w$, 
for $x$ in
$U_{h_+}$ and by letting $w(x)=l(x)$ in $\mathbb{R}^n\backslash U_{h_-}$ where $l'(x)$, for $x$ in $U_{h_+}$, is given by $l'(x)=\hat{\mathcal{R}}(x,\gamma^{\downarrow}_- w,k_-)$, 
%$w$ in $U_{h_+}$ by
%(\ref{uprcstar}) with $F_{h_+}:= \gamma_-w$, and $w$ in $\mathbb{R}^n \setminus \overline {U_{h_-}}$ by (ref{  }) with $F_{h_-}:=$ 
it follows from Theorem
\ref{th_equiv_trans} that $w$ satisfies the transmission problem, with
$g$ extended by zero from $S$ to $\mathbb{R}^n$. By standard
interior regularity results (e.g. \cite{mclean00} Theorem 4.16) it holds, since $g\in
L^2(S)$ and $k \in C^{0,1}_{\rm loc}(\mathbb{R}^n)$, that $w\in
H^2_{\rm loc}(\mathbb{R}^n)$. Further, $w\in H^2(U_d\setminus U_c)$ for
$c>d>h_+$ and for $d<c<h_-$. 
% (though $w\in H^2(S_c)$ is not clear without some
%further constraint on the  of $\Gamma$ at infinity).
Moreover, by Lemma \ref{lemma3p2_trans}, $w(x) = \hat{\mathcal{R}}(x,w|_{\Gamma_c},k_+)$ (with $h_+$ replaced by $c$) for $x \in U_c$ for all $c>h_+$. Similarly $w'(x)=\hat{\mathcal{R}}(x,w'|_{\Gamma_{-d}},k_-)$ (with $h_+$ replaced by $-d$) for $x \in U_{-d}$ for all $-d>h_+$.  Thus $w$ satisfies the transmission problem
with $h_+,h_-$ replaced by $c,d$, respectively, for all $c>h_+$ and $d<h_-$ and so, by Theorem
\ref{th_equiv_trans},
\begin{equation} \label{eqstst_3}
\int_{S(d,c)}(\nabla w\cdot\nabla\bar v-k^2 w\bar v)\,dx =
-\int_{\Gamma_c}\gamma^{\uparrow}_+\bar v\, T_+\gamma^{\uparrow}_+ w \,ds -\int_{\Gamma_d}\gamma^{\downarrow}_-\bar v\, T_-\gamma^{\downarrow}_- w \,ds- \int_{S} \bar
v g \,dx,
\end{equation}
for all $c>h_+,d<h_-$.

In view of this regularity and since $w$ satisfies the boundary
value problem, we have, for all $a>h_+,b<h_-$,
\begin{eqnarray*}
\lefteqn{2\Re\int_{S(b,a)}\phi_A(r)(x_n-\beta)g\frac{\partial\bar{w}}{\partial x_n}\,dx}\\
&=& 2\Re\int_{S(b,a)}\phi_A(r)(x_n-\beta)(\Delta w+k^2w)
\frac{\partial\bar{w}}{\partial x_n}\,dx\\
&=&
\int_{S(b,a)}\left\{2\Re\left\{\nabla\cdot\left(\phi_A(r)(x_n-\beta)\frac{\partial\bar{w}}{\partial
x_n} \nabla w\right)\right\}-2\phi_A(r)\left|\frac{\partial
w}{\partial x_n}\right|^2
\right.\\&&\left.-2\Re\left[(x_n-\beta)\phi_A(r)\frac{\partial \nabla
\bar w}{\partial x_n}.\nabla w\right]\right.\\&&\left. -2
\phi_A'(r)(x_n-\beta)\frac{\tx}{|\tx|}\cdot
\Re\left(\nabla_{\tx}w\frac{\partial \bar{w}}{\partial
x_n}\right)\right\}\,dx\\&&+ 2\Re\int_{S(b,a)} k^2(x_n-\beta)\phi_A(r)\frac{\partial \bar w}{\partial
x_n}wdx. %+ %i\Im(k^2)(x_n-\beta)\phi_A(r)\frac{\partial \bar w}{\partial
%x_n}wdx.
\end{eqnarray*}
 Using the divergence theorem and integration by parts
 \begin{eqnarray*}
\lefteqn{2\Re\int_{S(b,a)}\phi_A(r)(x_n-\beta)g\frac{\partial\bar{w}}{\partial x_n}\,dx}\\
&=&(a-\beta)\int_{\Gamma_b}\phi_A(r)\left\{ \left|\frac{\partial
w}{\partial x_n}\right|^2- \left|\nabla_{\tx}w\right|^2+
k_+^2\left|w\right|^2\right\}\,ds\\&&+
(\beta-b)\int_{\Gamma_b}\phi_A(r)\left\{ \left|\frac{\partial
w}{\partial x_n}\right|^2- \left|\nabla_{\tx}w\right|^2+
k_-^2\left|w\right|^2\right\}\,ds\\&&
+\int_{S(b,a)}\left\{\phi_A(r)\left(|\nabla
w|^2-k^2|w|^2-2\left|\frac{\partial w}{\partial
x_n}\right|^2\right)
\right.\\&&\left.-2\phi_A'(r)(x_n-\beta)\Re\left(\frac{\partial\bar{w}}{\partial
x_n} \frac{\partial w}{\partial r}\right)\right\}\,dx.\\&&
-\int_{S(b,a)}\phi_A(r)\frac{\partial k^2}{\partial x_n}(x_n -\beta)|w|^2dx.
% +2\Re\int_{S(b,a)} i\Im (k^2)\phi_A(r)(x_n-\beta)\frac{\partial \bar w}{\partial
%x_n}wdx.
\end{eqnarray*}
%Using the fact that $w=0$ on $\Gamma$, so that $\nabla w=(\partial
%w/\partial \nu)\nu$ and
%\[
%\frac{\partial w}{\partial x_n}=e_n\cdot\nabla w=e_n\cdot
%\nu\frac{\partial w}{\partial \nu}=\nu_n\frac{\partial w}{\partial
%\nu},
%\]
Now, rearranging terms we find that
\begin{eqnarray*}
\lefteqn{2\int_{S(b,a)}\phi_A(r)\left|\frac{\partial w}{\partial x_n}\right|^2\,dx +\int_{S(b,a)}\phi_A(r)\frac{\partial k^2}{\partial x_n}(x_n -\beta)|w|^2dx} \\
&=& (a-\beta)\int_{\Gamma_a}\phi_A(r)\left\{ \left|\frac{\partial
w}{\partial x_n}\right|^2- \left|\nabla_{\tx}w\right|^2+
k_+^2\left|w\right|^2\right\}\,ds\\&&(\beta-b)\int_{\Gamma_b}\phi_A(r)\left\{ \left|\frac{\partial
w}{\partial x_n}\right|^2- \left|\nabla_{\tx}w\right|^2+
k_-^2\left|w\right|^2\right\}\,ds\\&&
+\int_{S(b,a)}\left\{\phi_A(r)\left(|\nabla w|^2-k^2|w|^2\right)
\right.\\&&\left.-2\phi_A'(r)(x_n-\beta)\Re\left(\frac{\partial\bar{w}}{\partial
x_n} \frac{\partial w}{\partial
r}\right)\right\}\,dx\\&&-2\Re\int_{S(b,a)}\phi_A(r)(x_n-\beta)g\frac{\partial\bar{w}}{\partial
x_n}\,dx.% +2\Re\int_{S(b,a)}i\Im (k^2)(x_n-\beta)\frac{\partial \bar w}{\partial
%x_n}wdx.
\end{eqnarray*}
We now wish to let $A\to \infty$.  The only problem is the term
involving $\phi'_A$ which we estimate as follows. Let
$S(b,a)^{t}=\left\{x\in S(b,a)\;:\; |\tx|< t\right\}$ for $t\ge 1$. Then
\begin{eqnarray*}
&&\left|\int_{S(b,a)}\left\{2\phi_A'(r)(x_n-\beta)\Re\left(\frac{\partial\bar{w}}{\partial
x_n} \frac{\partial w}{\partial r}\right)\right\}\,dx\right|\\&&\leq
2M(a-b)\int_{S(b,a)^{A+1}\setminus {S(b,a)}^A}|\nabla
w|^2\,dx\to 0
\end{eqnarray*}
as $A\to\infty$, where the convergence follows from the fact that
$w\in H^1(S(b,a))$.  In addition since $w\in H^2(U_d\setminus U_c)$,
for $c>d>h_+$, and for $d<c<h_-$ $\nabla w|_{\Gamma_a}\in H^{1/2}(\Gamma_a)$ and  $\nabla w|_{\Gamma_b}\in H^{1/2}(\Gamma_b)$ and
so, by the Lebesgue dominated and monotone convergence theorems, (note that $(\partial k^2/\partial x_n)(x_n-\beta) \geq 0$ by assumption 4),
\begin{eqnarray}
\lefteqn{2\int_{S(b,a)}\left|\frac{\partial w}{\partial x_n}\right|^2\,dx  +\int_{S(b,a)}\frac{\partial k^2}{\partial x_n}(x_n -\beta)|w|^2dx}\nonumber\\
&=& (a-\beta)\int_{\Gamma_a}\left\{ \left|\frac{\partial w}{\partial
x_n}\right|^2- \left|\nabla_{\tx}w\right|^2+
k_+^2\left|w\right|^2\right\}\,ds\nonumber\\&&(\beta-b)\int_{\Gamma_b}\left\{ \left|\frac{\partial w}{\partial
x_n}\right|^2- \left|\nabla_{\tx}w\right|^2+
k_-^2\left|w\right|^2\right\}\,ds\nonumber\\&&
+\int_{S(b,a)}\left(|\nabla
w|^2-k^2|w|^2-2\Re\left((x_n-\beta)g\frac{\partial\bar{w}}{\partial
x_n}\right)\right)\,dx.% +2\Re \int_{S(b,a)}i\Im (k^2)(x_n-\beta)\frac{\partial \bar w}{\partial
%x_n}w dx.
\nonumber\\\label{s13_3}
\end{eqnarray}
Now, since $w$ satisfies the boundary value problem, including the
radiation condition's (\ref{rad_1}) and (\ref{rad_2}), applying Lemma \ref{lemma3p2_trans}
it follows that
\begin{eqnarray}
\int_{\Gamma_a}\left\{ \left|\frac{\partial w}{\partial
x_n}\right|^2- \left|\nabla_{\tx}w\right|^2+
k_+^2\left|w\right|^2\right\}\,ds&\leq& 
%2k_+\Im\int_{\Gamma_a}
%\overline{w}\frac{\partial w}{\partial x_n}\,ds\nonumber\\&=&
-2k_+\Im\int_{\Gamma_a}\gamma^{\downarrow}_+\bar{w}T_+\gamma^{\downarrow}_+w\,ds
\label{s14_3}
\end{eqnarray}
and that
\begin{eqnarray}
\int_{\Gamma_b}\left\{ \left|\frac{\partial w}{\partial
x_n}\right|^2- \left|\nabla_{\tx}w\right|^2+
k_-^2\left|w\right|^2\right\}\,ds&\leq& 
%2k_+\Im\int_{\Gamma_a}
%\overline{w}\frac{\partial w}{\partial x_n}\,ds\nonumber\\&=&
-2k_-\Im\int_{\Gamma_b}\gamma^{\uparrow}_-\bar{w}T_-\gamma^{\uparrow}_-w\,ds
\label{s14b_3}
\end{eqnarray}
%on applying the Plancherel identity (\ref{plancherel}), noting
%(\ref{ft_ugama}) and (\ref{ft_unorm}). 
Further, setting $v=w$ in
(\ref{eqstst_3}) we get
\begin{equation}
\int_{S(b,a)}\left(|\nabla
w|^2-k^2|w|^2\right)\,dx=-\int_{\Gamma_a}\gamma^{\uparrow}_+\bar{w}T_+\gamma^{\uparrow}_+w\,ds-\int_{\Gamma_b}\gamma^{\downarrow}_-\bar{w}T_-\gamma^{\downarrow}_-w\,ds-
\int_{S}g\bar{w}\,dx, \label{s15_3}
\end{equation}
for $a>h_+$, $b<h_-$, so that, by Lemma \ref{WL1},
\begin{equation}
\int_{S(b,a)}[|\nabla w|^2-k^2|w|^2]\,dx\leq
-\Re\int_{S}g\bar{w}\,dx \label{s16_3}
\end{equation}
and
\begin{equation}
%-\int_{S(b,a)}\Im (k^2)|w|^2dx + 
\Im\int_{\Gamma_a}\gamma^{\uparrow}_+\bar{w}T_+\gamma_+^{\uparrow}w\,ds +\Im\int_{\Gamma_b}\gamma^{\downarrow}_-\bar{w}T_-\gamma^{\downarrow}_-w\,ds=-\Im\int_{S}g\bar{w}\,dx,
\end{equation}
which means, in view of Lemma \ref{WL1} %and the fact that $\Im (k^2) \geq 0$, 
that 
%\begin{eqnarray}
%\hspace{7ex} \int_{S(b,a)}\Im (k^2)|w|^2dx\leq\Im\int_{S}g\overline{w}\,dx, 
%\label{s17}
%\end{eqnarray}
%and that
\begin{eqnarray}
 -2k_+\Im\int_{\Gamma_a}\gamma_+^{\uparrow}\bar{w}T_+\gamma_+^{\uparrow}w\,ds\leq 2k_+\Im\int_{S}g\bar{w}\,dx,\label{s17a_3} \\-2k_-\Im\int_{\Gamma_b}\gamma^{\downarrow}_-\bar{w}T_-\gamma^{\downarrow}_-w\,ds\leq 2k_-\Im\int_{S}g\bar{w}\,dx. 
\label{s17b_3}
\end{eqnarray}
Using (\ref{s17a_3}) in (\ref{s14_3}) and (\ref{s17b_3}) in (\ref{s14b_3}) then using the resulting
equation and (\ref{s16_3}) in (\ref{s13_3}), we get that
\begin{eqnarray}\label{m1}
\nonumber&&2\int_{S(b,a)}\left|\frac{\partial w}{\partial
x_n}\right|^2\,dx +\int_{S(b,a)}\frac{\partial k^2}{\partial x_n}(x_n -\beta)|w|^2 dx\\\nonumber&\le&  2(a-\beta)k_+\Im\,\int_{S}g\bar w\,dx + 2(\beta -b)k_-\Im\int_{S}g\bar wdx 
  \\\nonumber &-& \Re \int_{S} \left[g\bar w + 2(x_n-\beta)g\frac{\partial \bar
w}{\partial x_n}\right]\,dx \\\nonumber  &\le&  2(a-\beta)k_+\Vert g\Vert_2\Vert w\Vert_2  + 2(\beta -b)k_-\Vert g\Vert_2\Vert w\Vert_2  
 +\Vert g\Vert_2\Vert w\Vert_2  \\&+& 2\left|\int_{S}(x_n-\beta)g\frac{\partial \bar
w}{\partial x_n}\,dx\right|. %+2\left|\int_{S(b,a)}\Im (k^2)(x_n-\beta)\frac{\partial \bar w }{\partial
%x_n}w dx\right| \\ 
\end{eqnarray}
Now, applying the Cauchy-Schwarz inequality to the last term in (\ref{m1}) and then using that $2ab \leq a^2 + b^2$, for $a,b>0$, we get that %Since this equation holds for all $a>H$ and $\nu_n 
\begin{eqnarray}\label{m2}
\nonumber&&\int_{S}\left|\frac{\partial w}{\partial
x_n}\right|^2\,dx +\int_{S(b,a)}\frac{\partial k^2}{\partial x_n}(x_n -\beta)|w|^2 dx\\\nonumber&\le&  [2(a-\beta)k_++ 2(\beta -b)k_- +1]\Vert g\Vert_2\Vert w\Vert_2\\&+& (h_+-h_-)^2\Vert g\Vert_2^2. %+2\left|\int_{S(b,a)}\Im 
\end{eqnarray}
Making use of assumption 4, and since $a>h_+$ and $b<h_-$ were arbitrary we get that 
\begin{eqnarray}\label{m3}
\left\Vert\frac{\partial w}{\partial
x_n}\right\Vert_2^2 &\le&   \mathcal{E} %+2\left|\int_{S(b,a)}\Im 
\end{eqnarray}
and 
\begin{eqnarray}\label{m4}
\int_{S}\frac{\partial k^2}{\partial x_n}(x_n -\beta)|w|^2 dx\leq  \mathcal{E},
\end{eqnarray}
where 
$$\mathcal{E}:=[2\kappa_++ 2\kappa_- +1]\Vert g\Vert_2\Vert w\Vert_2\\+ (h_+-h_-)^2\Vert g\Vert_2^2.
$$
%(k^2)(x_n-\beta)\frac{\partial \bar w }{\partial< 0$, on
%$\Gamma$, 
%\[
%\frac{\partial \Re k^2}{\partial x_n} \geq -\lambda_1 -\lambda_2\Im (k^2,
%\]
Now 
\begin{eqnarray*}
\int_{S}\frac{\partial k^2}{\partial x_n}(x_n -\beta)|w|^2 dx &=& (h_+-\beta)\int_{\Gamma_{h_+}}k_+^2|w|^2ds - (h_--\beta)\int_{\Gamma_{h_-}}k_-^2|w|^2ds
\\&-&\int_{S}k^2\frac{ \partial }{\partial x_n}[(x_n -\beta)|w|^2] dx.
\end{eqnarray*}
Note also that
\[
\int_{S}\tilde{k}^2\frac{ \partial }{\partial x_n}[(x_n -\beta)|w|^2] dx
=(h_+-\beta)\int_{\Gamma_{h_+}}k_+^2|w|^2ds - (h_--\beta)\int_{\Gamma_{h_-}}k_-^2|w|^2ds,
\]
so that
\[
\int_{S}\frac{\partial k^2}{\partial x_n}(x_n -\beta)|w|^2 dx =\int_{S}(\tilde{k}^2 -k^2)\frac{ \partial }{\partial x_n}[(x_n -\beta)|w|^2] dx.
\]
In addition,
\[
\frac{ \partial }{\partial x_n}[(x_n -\beta)|w|^2]= |w|^2 +2(x_n-\beta)\Re\left(\bar{w}\frac{\partial w}{\partial x_n}\right)\geq \frac{|w|^2}{2} - 2(h_+-h_-)^2\left|\frac{\partial w}{\partial x_n}\right|^2.
\]
Thus we get that 
\[
\int_{S}\frac{\partial k^2}{\partial x_n}(x_n -\beta)|w|^2  dx \geq \int_{S}(\tilde{k}^2 -k^2)
\frac{|w|^2}{2}dx - \int_{S}(\tilde{k}^2 -k^2)2(h_+-h_-)^2\left|\frac{\partial w}{\partial x_n}\right|^2dx.
\]
Using this and assumption 5 and using (\ref{m3}) and (\ref{m4}) we arrive at
\begin{eqnarray*}
\frac{\lambda_3}{2}\int_{\mathcal{C}}|w|^2dx \leq \int_{S}(\tilde{k}^2 -k^2)
\frac{|w|^2}{2}dx &\leq & %[2(h_+-\beta)k_++ 2(\beta -h_-)k_- +1]\Vert g\Vert_2\Vert w\Vert_2^2\\&+& \max\{|h_+-\beta|,|h_--\beta|\}^2\Vert g\Vert_2^2+
\mathcal{E}+ \left|\int_{S}(\tilde{k}^2 -k^2)2(h_+-h_-)^2\left|\frac{\partial w}{\partial x_n}\right|^2dx\right|\\&\leq& \mathcal{E}%[2(h_+-\beta)k_++ 2(\beta -h_-)k_- +1]\Vert g\Vert_2\Vert w\Vert_2^2\\&+& \max\{|h_+-\beta|,|h_--\beta|\}^2\Vert g\Vert_2^2 
+ 4\kappa^2_{\infty}\left \Vert \frac{\partial w}{\partial x_n}\right\Vert_2^2\\&\leq& [1+ 4\kappa^2_{\infty}]\mathcal{E}.%[2(h_+-\beta)k_++ 2(\beta -h_-)k_- +1]\Vert g\Vert_2\Vert w\Vert_2^2\\&+& \max\{|h_+-\beta|,|h_--\beta|\}^2\Vert g\Vert_2^2
\end{eqnarray*}
Now using this with the trace inequality, lemma \ref{trace_trans}, we get that 
\[
\int_{\Gamma}|w|^2 ds \leq \sqrt{1+L^2}\left\{\epsilon\mathcal{E} + \epsilon^{-1}2\lambda_3^{-1}(1+4\kappa_{\infty}^2)\mathcal{E}
\right\}.
\]
%Noting that $\mathcal{C_+}$ and $\mathcal{C_-}$ are Lipschitz domains 
We next make use of the Friedrich's inequality, lemma \ref{friedrichs} with $\zeta=1$ to get
\begin{eqnarray*}
k_{\infty}^2\Vert w \Vert_2^2& \leq& \kappa_{\infty}^2\left \Vert \frac{\partial w}{\partial x_n}\right\Vert_2^2 + 4\kappa_{\infty}k_{\infty}\int_{\Gamma}|w|^2ds
\\&\leq& \kappa_{\infty}^2\mathcal{E} + 4\kappa_{\infty}k_{\infty}\sqrt{1+L^2}\left\{\epsilon\mathcal{E} + \epsilon^{-1}2\lambda_3^{-1}(1+4\kappa_{\infty}^2)\mathcal{E}
\right\}\\&=& P \mathcal {E}
\end{eqnarray*}
where the dimensionless parameter $P$ is defined as
\[
P=  \kappa_{\infty}^2 + 4\kappa_{\infty}k_{\infty}\sqrt{1+L^2}\{\epsilon + \epsilon^{-1}2\lambda_3^{-1}(1+4\kappa_{\infty}^2)
\}.  
\]
Thus using $ab \leq a^2/2\eta + b^2\eta/2$ for $a,b>0$, $\eta >0$, we get that
\begin{eqnarray*}
k_{\infty}^2\Vert w \Vert_2^2 & \leq& P\left\{[2\kappa_++ 2\kappa_- +1]\Vert g\Vert_2\Vert w\Vert_2+ (h_+-h_-)^2\Vert g\Vert_2^2\right\} \\&\leq& \frac{P^2}{2k_{\infty}^2}[2\kappa_++ 2\kappa_- +1]^2 \Vert g \Vert_2^2
\\ & +& \frac{k_{\infty}^2}{2}\Vert w\Vert_2^2 + P(h_+-h_-)^2\Vert g\Vert_2^2,
\end{eqnarray*}
so that 
\begin{eqnarray*}
k_{\infty}^2\Vert w \Vert_2^2&\leq&2\left[ \frac{P^2}{2k_{\infty}^2}[2\kappa_++ 2\kappa_- +1]^2 + P(h_+-h_-)^2\right]\Vert g\Vert_2^2.\\%&=& 2\left[ \frac{2P^2}{k^2}[2\kappa_++ 2\kappa_- +1]^2 + 4P(h_+-h_-)^2\right]\Vert g\Vert_2^2
\end{eqnarray*}

Thus using (\ref{s16_3}) we have 
\begin{eqnarray*}
k_{\infty}^2\Vert w\Vert_{H^1(S)}^2& \leq & k_{\infty}^2 \Vert g\Vert_2\Vert w\Vert_2 + 2k_{\infty}^4 \Vert w\Vert_2^2\\&\leq & k_{\infty}\sqrt{2\left[ \frac{P^2}{2k_{\infty}^2}[2\kappa_++ 2\kappa_- +1]^2 + P(h_+-h_-)^2\right]}\Vert g\Vert_2^2 \\ & + &4k_{\infty}^2\left[ \frac{P^2}{2k_{\infty}^2}[2\kappa_++ 2\kappa_- +1]^2 + P(h_+-h_-)^2\right]\Vert g\Vert_2^2
\end{eqnarray*}

\end{proof}
We now proceed to establish that lemma \ref{rellich_3} holds for arbitrary $k \in L^{\infty}(D)$ satisfying assumptions 4 and 5.
\begin{lemma}\label{arb_k_3}
Let $h_+>h_-$, $g\in L^2(S)$ and
suppose that $k \in L^{\infty}(\mathbb{R}^n)$ (with $k=k_+$ in $U_{h_+}$ and $k=k_-$ in $\mathbb{R}^n \backslash \overline{U_{h_-}})$ satisfies assumptions 4 and 5. Then, if $w\in H^1(S)$ satisfies
\begin{equation}
d(w,\phi)= -(g,\phi), \quad \phi\in H^1(S), \label{bgprob_3}
\end{equation}
then
\[
k_{\infty}\Vert w\Vert_{H^1(S)}\leq  C_1\Vert g\Vert_{2}
\]
where $C_1$ is given by (\ref{C_def_chap4}).
\end{lemma}
\begin{proof}
%Let $k \in L^{\infty}(S_H)$ be such that $\frac{\partial \Re k^2}{\partial x_n} \geq -\lambda_1 -\lambda_2\Im k^2$, in a distributional sense, i.e.
%\[
%-\int_{S_H}\Re k^2\frac{\partial \bar \phi}{\partial x_n} dx \geq  -\lambda_1 -\lambda_2\Im k^2.
%\]
Extending the definition of $w$ to the whole of $\mathbb{R}^n$ by letting \newline $w(x)=\hat{\mathcal{R}}(x,\gamma^{\uparrow}_+ w,k_+)$
%the
%right hand side of (\ref{uprcstar_chap3}), with $F_H:= \gamma^{\uparrow}_+ w$, 
for $x$ in
$U_{h_+}$ and by letting $w(x)=l(x)$ in $\mathbb{R}^n\backslash U_{h_-}$ where $l'(x)$, for $x$ in $U_{h_+}$, is given by $l'(x)=\hat{\mathcal{R}}(x,\gamma^{\downarrow}_- w,k_-)$, 
%$w$ in $U_{h_+}$ by
%(\ref{uprcstar}) with $F_{h_+}:= \gamma_-w$, and $w$ in $\mathbb{R}^n \setminus \overline {U_{h_-}}$ by (ref{  }) with $F_{h_-}:=$ 
it follows from Theorem
\ref{th_equiv_trans} that $w$ satisfies the transmission problem, with
$g$ extended by zero from $S$ to $\mathbb{R}^n$. Hence by lemma \ref{lemma3p2_trans} (cf the proof of lemma \ref{rellich_3}) it holds that 
%that $w$ satisfies the boundary value problem, with
%$g$ extended by zero from $S_H$ to $D$ and $k$ extended from $S_H$ to $D$ by taking the value $k_+$ in $U_H$. 
%By standard local
%regularity results (e.g. \cite{mclean00} Theorem 4.18) it holds, since $g\in
%L^2(D)$, $w=0$ on $\Gamma$, $k \in C^{0,1}_{\rm loc}(D)$ and the boundary is smooth, that $w\in
%H^2_{\rm loc}(D)$. Further, $w\in H^2(U_b\setminus U_c)$ for
%$c>b>f_+$.
% (though $w\in H^2(S_c)$ is not clear without some
%further constraint on the  of $\Gamma$ at infinity).
%Moreover, by Lemma \ref{lemma3p2}, $w$ is given by the right hand
%side of (\ref{uprcstar}) in $U_b$ for all $b>H$ if $H$ is replaced
%in (\ref{uprcstar}) by $b$ and $F_b$ denotes the restriction of
%$w$ to $\Gamma_b$. Thus $w$ satisfies the boundary value problem
%with $H$ replaced by $b$, for all $b>H$, and so, by Theorem
%\ref{th_equiv},
%Let $b(w,v)= -(g,v)$.  Extending the definition of $w$ onto $U_H$ using the radiation condition \ref{fuck}, and extending $g$ by zero and $k$ by $k_+$ onto $U_H$, it holds by lemma fuck 
$\forall a>h_+, b<h_-$ ,  
\begin{equation} \label{int_S_b_3}
\int_{S(b,a)}\nabla w.\nabla \bar v - k^2w\bar v dx + \int_{\Gamma_a}\gamma_+^{\uparrow}\bar v T_+\gamma^{\uparrow}_+wds + \int_{\Gamma_b}\gamma^{\downarrow}_-\bar v T_-\gamma^{\downarrow}_-wds  
=-(g,v), 
\end{equation}
for $ v\in H^1(S(b,a))$.
%For $x \in \mathbb{R}^n \backslash D$ let $k(x) = k_0$, so that $k$ is now a function on $\mathbb{R}^n$. Note that assumption 2 now holds for almost all $x \in \mathbb{R}^n$.
%\[
%\Re(k^2(x) = \pi(x) - \lambda_1 x_n - \Im(k^2)x_n,
%\]
%where $\pi: \mathbb{R}^n \to \mathbb{R}$ is a monotonic non-decreasing function.

For $\delta>0$, let $\psi_{\delta} \in C_0^{\infty}(\mathbb{R}^n)$ be such that $\psi_{\delta} >0$, $\psi_{\delta}(x)=0$ if $|x|>\delta$ and such that $\int_{\mathbb{R}^n}\psi_{\delta}(x)dx =1$ %and such that %$\psi_{\delta}(x) =\psi_{\delta}(-x)$ 
for $x \in \mathbb{R}^n$.

Next define $\hat{k}:\mathbb{R}^n \to \mathbb{R}$ by
$\hat{k}(x)=k_0$ for $x \in\mathbb{R}^{n-1}\times [\beta-\delta, \beta +\delta]$ and by $\hat{k}(x)=k(x)$ otherwise.   
Then let $\hat{k}^2_{\delta} \in C^{\infty}(\mathbb{R}^n)$ be given by
\[
\hat k_{\delta}^2 := \hat k^2 \ast \psi_{\delta}.% =\Re(k^2)\ast \psi_{\delta} +i%\Im (k^2_{\delta}):= 
%\Im(k^2)\ast \psi_{\delta}. 
%\quad \Re(k^2_{\delta}):=\pi \ast \psi_{\delta} - \lambda_1x_n - \Im(k^2_{\delta})x_n.
\]
Since $a>h_+$, then for all $x \in \Gamma_a$ there exists $\mu>0$ such that if $z \in B_{\mu}(x)$ then $\hat{k}(z)=k_+$. Thus it follows from the definitions of convolution and $\psi_{\delta}$ that $\hat{k}_{\delta} =k_+$ on $\Gamma_a$ provided we choose $\delta \le \mu$. 
A similar argument shows that $\hat{k}_{\delta} =k_-$ on $\Gamma_b$ provided we choose $\delta \le \mu'$ for some $\mu'>0$. Also $\Vert \hat{k}_{\delta} \Vert_{L^{\infty}(\mathbb{R}^n)} \leq k_{\infty}$ and 
\[
\hat{k}^2_{\delta}= \int_{|y|<\delta}\hat{k}^2(x-y)\psi_{\delta}(y)dy>k_0^2.
\]
Let us show that $\hat{k}_{\delta}$ satisfies assumption 4. Note that $\hat{k}$ is monotonic non-increasing on $\mathbb{R}^n \backslash \overline{U_{\beta + \delta}}$ and monotonic non-decreasing on $ U_{\beta -\delta}$. Now for $(\tilde x,x_n) \in S$  and $h>0$  
\begin{eqnarray}\nonumber
&&\hat{k}^2_{\delta}(\tilde x,x_n + h) - \hat{k}^2_{\delta}(\tilde x,x_n)
\\&&= \int_{|y|<\delta}[\hat{k}^2_{\delta}(\tilde x-\tilde y,x_n -y_n+ h) - \hat{k}^2_{\delta}((\tilde x-\tilde y,x_n-y_n)]\psi_{\delta}(y)dy.\label{k_delta}
\end{eqnarray}
Thus for $|y|<\delta$, if $x_n <\beta$ and $x_n +h<\beta$, then $x_n-y_n <\beta+ \delta$ and $x_n-y_n +h<\beta+\delta$, so that from (\ref{k_delta}) $\hat{k}_{\delta}$ is monotonic non-increasing on $\mathbb{R}^n \backslash U_{\beta }$ because $\hat{k}$ is monotonic non-increasing on $\mathbb{R}^n \backslash U_{\beta + \delta}$. Similarly $\hat{k}_{\delta}$ is monotonic non-decreasing on $U_{\beta}$. Thus $\hat{k}_{\delta}$ satisfies assumption 4.

We next show that $\hat{k}_{\delta}$ satisfies assumption 5. We define, for $\delta < \epsilon$ 
\[
\mathcal{C}_{\epsilon- \delta}:= \{ (\tilde x,f(\tilde x) + t) : \tilde x \in \mathbb{R}^{n-1}, |t|< \epsilon- \delta\}.
\]
Then, for $(\tilde x, x_n) \in  \mathcal{C}_{\epsilon- \delta}$, with $x_n \geq \beta$ and $y \in \mathbb{R}^n$ with $|y|<\delta$,
\begin{equation}\label{ppp}
\tilde{k}^2(x) -  \hat{k}^2(x-y) \geq k_+^2 -k^2(x-y).
\end{equation}
Note that, since $x_n = f(\tilde x) +t $ for some $t$ such that $|t| \leq \epsilon -\delta$, $|x_n-y_n-f(\tilde{x})|< \epsilon$. Thus $x-y \in \mathcal{C}$. Now if $x_n - y_n> \beta $ then from (\ref{ppp}) we have that
\[
\tilde{k}^2(x) -  \hat{k}^2(x-y) \geq \lambda_3.%k_+^2 -k^2(x-y).
\] 
On the other hand if $x_n - y_n< \beta$, then, since $x_n - y_n >\beta-\delta$ it holds that $\hat{k}^2 (x-y) = k_0^2$. Thus    
\[
\tilde{k}^2(x) -  \hat{k}^2(x-y) \geq \lambda_3.%k_+^2 -k^2(x-y).
\] 
Arguing in a similar way, one can show that for $(\tilde x,x_n) \in \mathcal{C}_{\epsilon- \delta}$ with  $x_n <\beta$ and for $y \in \mathbb{R}^n$ with $|y|<\delta$, that also 
\[
\tilde{k}^2(x) -  \hat{k}^2(x-y) \geq \lambda_3.%k_+^2 -k^2(x-y).
\] 
Finally then, for x $\in\mathcal{C}_{\epsilon- \delta}$ 
\[
\tilde{k}^2(x) - \hat{k}^2(x)= \int_{|y|<\delta}(\tilde{k}^2(x) - \hat{k}^2(x-y))\psi_{\delta}(y)dy \geq \lambda_3.
\]
Thus $\hat{k}_{\delta}$ satisfies assumption 5 with $\epsilon$ replaced by  $\epsilon- \delta$.

Now, fix $\chi>0$, and choose $w_n \in \mathcal{D}(S(b,a))$ such that 
$$
\Vert w- w_n\Vert_{H^1(S(b,a))}<\chi.
$$ 
Thus (\ref{int_S_b_3}) can be rewritten as
%\begin{eqnarray*}
\begin{eqnarray*} 
&&\int_{S(b,a)}\nabla w.\nabla \bar v - \hat{k}_{\delta}^2w\bar v dx + \int_{\Gamma_a}\gamma^{\uparrow}_+\bar v T_+\gamma^{\uparrow}_+wds + \int_{\Gamma_b}\gamma^{\downarrow}_-\bar v T_-\gamma^{\downarrow}_-wds  
\\&=&-(g,v) + \int_{U_{\beta - \delta} \backslash U_{\beta + \delta}}(k^2- k_0^2)w \bar vdx + \int_{S(b,a)}(\hat{k}^2 - \hat{k}_{\delta}^2)w\bar{v}dx\\
&=&  -(g,v) + \int_{U_{\beta - \delta} \backslash U_{\beta + \delta}}(k^2- k_0^2)w \bar vdx + \int_{S(b,a)}(\hat{k}^2 - \hat{k}_{\delta}^2)w_n\bar{v}dx \\&+& \int_{S(b,a)}(\hat{k}^2 - \hat{k}_{\delta}^2)(w- w_n)\bar{v}dx \quad v\in H^1(S(b,a)).
\end{eqnarray*}
Thus by lemma \ref{rellich_3} we have that
\begin{eqnarray*}
k_{\infty}\Vert w\Vert_{H^1(S)} &&\leq C_1\left\{ \Vert g\Vert_2 + \Vert (k^2 -k_0^2) w \Vert_{L^2(U_{\beta - \delta} \backslash U_{\beta + \delta})} + \Vert(\hat{k}^2 - \hat{k}^2_{\delta})w_n\Vert_{L^2(S(b,a))}\right.\\&&\left.+ \Vert(\hat{k}^2 - \hat{k}^2_{\delta})(w-w_n)\Vert_{L^2(S(b,a))}\right\}. 
\end{eqnarray*}
Note that (see \cite{mclean00} theorem 3.4) if we choose $\delta$ small enough then we can ensure that 
$$\Vert(\hat{k}^2 - \hat{k}^2_{\delta})w_n\Vert_{L^2(S(b,a))} < \chi,
$$
whilst $\Vert(\hat{k}^2 - \hat{k}^2_{\delta})(w-w_n)\Vert_{L^2(S(b,a))}< 2k_{\infty}^4\chi$. In addition, by using Lebesque's monotone convergence theorem, one sees that we can arrange that 
$$
\Vert (k^2 -k_0^2) w \Vert_{L^2(U_{\beta - \delta} \backslash U_{\beta + \delta})} <\chi,
$$ provided $\delta$ is chosen small enough. The result now follows by arbitrariness of $\chi$.

\end{proof}
Theorem \ref{th_main1_trans} now follows by combining lemmas \ref{arb_k_3}, \ref{special_case_enough_trans} and
\ref{ihl_rem_3} with Corollary \ref{cor_infsup_3}.
\part{Integral Equation Methods}
\chapter{The Dirichlet problem for the Helmholtz equation}\label{BIE}
\section{Introduction and literature review}
This section concerns Boundary integral equation (BIE) techniques for solving an acoustic scattering problem, namely the Dirichlet problem for the Helmholtz equation. An excellent introduction to BIE techniques can be found in Kress's Linear Integral equations \cite{kress89}, chapter 6. We wish to begin by informally looking at how BIE techniques can be applied to resolve the problem of scattering by smooth bounded obstacles. So, let $\Omega \subseteq \mathbb{R}^n$, $n=2,3$ be a smooth, bounded domain (obstacle). The essential problem can be stated as follows: Given $g \in BC(\partial \Omega)$ find $u:\mathbb{R}^n \backslash \Omega \to \mathbb{C}$ such that
\begin{equation}
\Delta u +k^2u =0, \mbox{ in } \mathbb{R}^n\backslash \Omega, \label{Helm_ext}
\end{equation}    
and 
\begin{equation}
u=g \mbox{ on } \partial \Omega. \label{Dirich}
\end{equation}

A solution to this problem can be constructed by making use of the fundamental solutions, ($\Phi(x,y)$ $x,y \in \mathbb{R}^n$) to the Helmholtz equations given in chapter 1. We suppose the solution $u$ to be represented as 
\begin{equation}\label{ansatz}
u(x) = \int_{\partial \Omega} \frac{\partial \Phi(x,y)}{\partial \nu(y)} \phi (y)ds(y) -i\eta\int_{\partial \Omega} \Phi(x,y) \phi (y)ds(y),
\end{equation}
for some, as yet unknown, $\phi \in BC(\partial \Omega)$. Here $\eta>0$ and the normal $\nu(y)$ is directed out of $\Omega$. Making this ansatz (\ref{ansatz}), that $u$ is a ``combined single- and double-layer potential'' leads to a \textit{Brakhage-Werner} type integral equation formulation of the problem. Indeed it was Brakhage and Werner \cite{brakham65} who first suggested making this ansatz, although so too did Leis \cite{Leis} and Panich \cite{panich} independently.   

That $u$, given by (\ref{ansatz}) satisfies (\ref{Helm_ext}) then follows immediately provided one can justify an interchange of order of differentiation and integration, which in fact one can. Thus it remains to show our solution satisfies (\ref{Dirich}). It's at this point that we fix $\phi$. If $u$ is given by (\ref{ansatz}) then it may be continuously extended to the boundary $\partial \Omega$ with a limiting value given by (see for example \cite{Col83a} theorem 2.13)%(note the singularity of the double layer potential)
\begin{equation}    
\lim_{x_n \to x} u(x_n) =  \frac{1}{2}\phi (x) + \int_{\partial \Omega} \frac{\partial \Phi(x,y)}{\partial \nu(y)} \phi (y)ds(y) -i\eta\int_{\partial \Omega} \Phi(x,y) \phi (y)ds(y), \quad x \in \partial \Omega.
\end{equation}
Since we must prescribe that $\lim_{x_n \to x} u(x_n) = g(x)$ for $x \in \partial \Omega$,  it follows that 
\begin{equation}    
 g(x) =  \frac{1}{2}\phi (x) + \int_{\partial \Omega} \frac{\partial \Phi(x,y)}{\partial \nu(y)} \phi (y)ds(y) -i\eta\int_{\partial \Omega} \Phi(x,y) \phi (y)ds(y), \quad x\in\partial\Omega,
\end{equation}
which we may write, in operator notation, as 
\[
2g= (I+K_B-i\eta S_B)\phi
\]
where we define, for $\phi \in BC(\Gamma)$, 
\begin{equation}
(S_B\phi)(x):=2\int_{\partial\Omega}\Phi(x,y)\phi(y)ds(y), \quad x \in \partial\Omega,
\end{equation}
and
\begin{equation}\label{K_B_def}
(K_B\phi)(x):=2\int_{\partial \Omega}\frac{\partial \Phi(x,y)}{\partial \nu(y)}\phi(y)ds(y), \quad x \in \partial\Omega.
\end{equation}
Supposing that $A_B := I+K_B-i\eta S_B$ is an invertible operator on $ BC(\partial \Omega)$, it holds that $\phi$ given by
\[
\phi = (A_B)^{-1}g
\]
will, when substituted into (\ref{ansatz}), yield a solution to (\ref{Helm_ext})-(\ref{Dirich}).
Thus the problem is solved once we have established the invertibility of $A_B$.
Classically due to the compactness and smoothness of the boundary $\partial \Omega$, the operators $K_B$ and $S_B$, defined above, are compact so that $I+K_B-i\eta S_B$ is a Fredholm operator of index zero. Thus, in order to prove it is invertible, it is sufficient to show that it is injective.

In studying scattering by rough surfaces, Chandler-Wilde et al (\cite{Cha99a}, \cite{chandmmas96}, \cite{chandsjam98} wished to apply precisely this same approach to their field. However problems arise in doing this,
the first of these being that, due to the slow decay at
infinity of the standard fundamental solution, $\Phi(x,y)$, of the
Helmholtz equation (like $|x-y|^{-(n-1)/2}$ in $n$ dimensions), the
standard boundary integral operators i.e.\ the analogue of $S_B$ and $K_B$ above, are not bounded on any of the
standard function spaces when the surface is unbounded.  
In order to
get a faster decaying kernel they replaced, in 
\cite{zhangworking}, and again in \cite{chandpottheim}, $\Phi(x,y)$ by an appropriate half-space
Green's function for the Helmholtz equation. Specifically, they worked with the function
\begin{equation}\label{GDelta}
  G(x,y) \; := \; \Phi(x,y) - \Phi(x,y'),
\end{equation}
with $y' = (\tilde y,-y_n)$,  which is the Dirichlet Green's function
for the half space $\{x:x_n>0\}$. They then defined the
\emph{single-layer potential operator} by
\begin{equation}\label{SDelta}
  (S\varphi)(x) \; := \;
  2 \int_{\Gamma} G(x,y) \phi(y) \; ds(y),
  \quad x \in \Gamma,
\end{equation}
and the \emph{double-layer potential operator} by
\begin{equation}\label{KDelta}
  (K\varphi)(x) \; := \;
  2 \int_{\Gamma} \frac{\partial G(x,y)}
  {\partial \nu(y)} \phi(y) \; ds(y),
  \quad x \in \Gamma,
\end{equation}
where the normal $\nu(y)$ is directed out of $D$. 

Nevertheless, the boundedness, of the operators $S$ and $K$ on different function spaces, for example $L^2(\Gamma)$ or $BC(\Gamma)$ is not obvious or maybe even true. Indeed the question is somewhat different in 2 and 3 dimensions, which is one of the reasons why the integral equation approach has been applied, in contrast to the variational approach seen in previous chapters, to the two dimensional case first: \cite{chandmmas96}, \cite{chandsjam98}, \cite{Cha99a}, and \cite{zhangworking} (c.f.\ the literature review at the start of chapter 2); and then more recently \cite{chandpottheim} and \cite{chandpottheim2}, to the three dimensional case. Indeed, in three dimensions, a major difficulty is to establish the boundededness of the operators $S$ and $K$ on $L^2(\Gamma)$. However in \cite{chandpottheim}, by employing Fourier techniques, Chandler-Wilde, Heinemeyer and Potthast were able to establish this, in the case that $\Gamma$ is Lyapunov. 
 
The second major problem concerns the invertibility of the operator $A:=I + K -i\eta S$. Due to the fact that the boundary $\Gamma$, the rough surface of the domain, is infinite, $K$ and $S$ fail to be compact operators so that the classical method of inversion of $A$ can no longer be used. To establish the invertibility in the 2D case
generalisations of part of the Riesz theory of compact operators
have been developed \cite{Ros96,chandjmaa00,chandunfinished} which
require only local compactness rather than compactness and enable
existence of solution in $BC(\Gamma)$ to be deduced from uniqueness
of solution. In fact, injectivity of the second kind BIE in
$BC(\Gamma)$ implies well-posedness in $BC(\Gamma)$ and in the space
$L^p(\Gamma)$, $1\le p\le \infty$ \cite{arenshcwI}. But this theory does
not seem relevant for 3D rough surface scattering problems given
that the corresponding boundary integral operators are not
well-defined as operators on $BC(\Gamma)$. In the absence of these tools existence of solution to the BIE was shown in \cite{chandpottheim} and \cite{chandpottheim2}
%corresponding scattering problem) 
by first proving the invertibility of the operator $A$ in the case when the underlying surface is flat; and then extending this result to the general case by perturbation arguments, with the help of an a priori bound in a manner inspired by the somewhat similar arguments used to prove invertibility for second kind boundary integral equations for potential problems in Lipschitz domains (Verchota \cite{Verchota}; Jerison and Kenig \cite{Jerison}-- see below).

In the case when the underlying surface is also Lipschitz, the case we consider here, yet more problems arise. Indeed, even in the case of potential theoretic problems on a bounded domain, complications arise in the Lipschitz case: the double-layer operator is not even well-defined and must be replaced by it's principal-value generalisation; the question of its boundedness was resolved -- by Coifman, Mcintosh and  Meyer \cite{Meyer} -- only by the use of such deep techniques as the method of rotations; and the invertibilty of the boundary integral operator is problematic, since even though the boundary is compact, the fact that it is Lipschitz means that the double-layer operator fails to be compact. Nevertheless Verchota was able to show that the boundary integral operator still remained invertible \cite{Verchota} by making use of the Jerison and Kenig identities \cite{Jerison}. However it should be pointed out that, even when all of these difficulties had been overcome, the solution constructed via the BIE technique satisfied the Boundary value problem -- the Dirichlet problem for Laplace's equation on a bounded domain -- in  a weaker sense than might have been hoped for. Similarly, our problem here -- a Dirichlet problem for the Helmholtz equation in a perturbed half plane -- must be posed in a slightly weaker sense than the one posed in \cite{chandpottheim} when the underlying surface was Lyapunov, in order to be able to apply the BIE technique. See the next section where we precisely state our boundary value problem. 
%However even so, in the Lipschitz case, one is forced to weaken the problem somewhat, in order to be able to apply the BIE technique (see below). 

Our intention in this chapter is simple: take up the advances made by Chandler-Wilde, Heinemeyer and Potthast in \cite{chandpottheim} and \cite{chandpottheim2}, and try and extend their results to the case when the rough surface $\Gamma$ is Lipschitz; and, in order to do this, make use of the results obtained by Verchota et al as summarized in \cite{Wavelets}.  % used

We should briefly pass some remarks on some other papers in this area:
Recently as part of \cite{chandlangdon}, Chandler-Wilde and Langdon were able to make use of the results of Verchota \cite{Verchota} to apply the Brakhage-Werner approach to the problem of scattering by a bounded Lipschitz obstacle.

 Willers \cite{willers87}, and Kress and Tran \cite{Kre00a}, used BIE methods for three dimensional rough surface scattering but only in the case that the rough surface is flat outside a compact set, allowing the authors to reduce their problem to a boundary integral equation on a finite domain; and Nedelec and Starling \cite{nedstar91} and Dobson and Friedman \cite{dobfried92} also considered the same problem but with assumptions of periodicity on the surface, so that they too obtained an integral equation on a bounded domain.      
 
We once more point out that the results we obtain will not be applicable to the case of plane wave scattering (we assume boundary data in $L^2(\Gamma)$.) For a partial theoretical justification for BIE methods for three dimensional rough surface scattering with plane wave incidence, namely a justification, with some provisos, of Green's representation formula, see DeSanto and Martin \cite{desanto98}.

\section{The rough surface scattering problem}    
%we will study the problem of point source rough surface acoustic scattering, via integral equation techniques. 
We begin this section by recalling the notation that we have used throughout. Note that we are only concerned with a 3D setting in this chapter. Thus for
$x=(x_1,x_2,x_3)\in\real^3$ let
$\tilde x=(x_1,x_2)$ so that $x=(\tilde x,x_3)$.  For $H\in \real$
let $U_H:=\left\{x\;:\; x_3>H\right\}$ and
$\Gamma_H:=\left\{x\;:\;x_3=H\right\}$. 

Throughout this chapter we assume that the rough surface is the graph of a bounded and positive Lipschitz function: Let $f: \mathbb{R}^2 \to \mathbb{R}$ be Lipschitz with Lipschitz constant $L>0$, i.e. 
\[
|f(\tilde x) -f(\tilde y)| \leq L|\tilde x- \tilde y| \quad \tilde x,\tilde y \in \mathbb{R}^2.
\]
We then define 
\begin{equation} \label{Gamma_def_chap5}
\Gamma:=\{(\tilde x,f(\tilde x)):\tilde x \in \mathbb{R}^2\},
\end{equation}
$$
D=:\{(\tilde x,x_3):\tilde x \in \mathbb{R}^2,\mbox{  } x_3 >f(\tilde x)\}.
$$ The assumption that $f$ is bounded and positive means that for some constants $0<f_-<f_+$ it holds
that
\begin{equation}
U_{f_+}\subset D \subset U_{f_-}.\label{Weq1}
\end{equation}
%and let $\Gamma=\partial D$ denote the boundary of $\partial D$ (the rough surface)
Further we set $J_f(\tilde x) = \sqrt{1+|\nabla_{\tilde x}f(\tilde x)|^2}$, $\tilde x \in \mathbb{R}^2$ and we define $L'=\sqrt{1+L^2}$ so that $J_f \leq L'$.

As usual we have
$S_H:=D\setminus\overline{U}_H$, for any $H\ge f_+$, and we denote
the unit outward normal to $D$ by $\nu$.

\bigskip 

In chapter 1 we made mention of the fact that we are interested in the scattering of incident waves from a source of compact support. We wish therefore to develop an analysis
that is applicable whenever the incident wave is due to sources of
the acoustic field located in some compact set $M\subset D$. Since
waves with sources in a bounded set $M \subset D$ can be
represented as superpositions of point sources located in the same
set, we will concentrate on the case when the incident field is that
due to a point source located at some point $z\in D$.

Thus we seek to find $u \in C^2(D)$ satisfying the Helmholtz equation
\[
\Delta u +k^2u= \delta_z, \quad \mbox{ in } D,
\]
satisfying the Dirichlet boundary
condition
\begin{equation}\label{BVP2}
  u(x) \; = \; 0, \quad x \in \Gamma,
\end{equation}
and satisfying an appropriate radiation condition. More precisely, writing the \emph{total field} $u$ as 
\begin{equation}
  u \; := \; u^i + u^s,
\end{equation}
where $u^s$ is the \emph{scattered
field} and $u^i$ the \emph{incident acoustic wave} due to the point source so that $u^i=\Phi(\cdot,z)$, 
we see that we seek to find $u^s \in C^2(D)$ such that
\[
\Delta u^s +k^2 u^s =0,
\]
and such that
%\begin{equation}
%  u \; := \; u^i + u^s,
%\end{equation}
%which is the sum of the incident field and the \emph{scattered
%field} $u^s$, we assume on $\Gamma$ the \emph{Dirichlet} boundary
%condition
\begin{equation}
  u^s(x) \; = -u^i(x), \quad x \in \Gamma.
\end{equation}
%and we seek to find $u$ satisfying the Helmholtz equation
%\[
%\Delta u +k^2u= \delta_z.
%\]
%We require that the scattered field is square-integrable in $S_H$, for all $H>f_+$ i.e.\
%\begin{equation}\label{??}
%  u^s  \in L^2(S_H), \mbox{ for all } H>f_+. 
%\end{equation}
%We wish to develop an analysis
%that is applicable whenever the incident wave is due to sources of
%the acoustic field located in some compact set $M\subset D$. Since
%waves with sources in a bounded set $M \subset \real^3$ can be
%represented as superpositions of point sources located in the same
%set, we will concentrate on the case when the incident field is that
%due to a point source located at some point $z\in D$, i.e.\
%$u^i=\Phi(\cdot,z)$. Thus the following is the specific problem that
%we will consider in this paper:

%Thus the following is the specific problem that
%we would, ideally wish to consider:

%\begin{problem}[\textbf{Point source rough surface scattering problem}]
%  Let $u^i=\Phi(\cdot,z)$ be the incident field due to a point source at $z\in D$.
%   Then we seek a scattered field
%  $u^s \in C^2(D) \cap C(\bar{D})$ such that
%  $u^s$ is a solution to the Helmholtz equation (\ref{BVP1}) in
%  $D$, the total field satisfies the sound-soft boundary condition (\ref{BVP2}),
  %(\ref{BVP3}) holds,  
%  and $u^s$ satisfies an appropriate radiation condition. 
  %In the case $\kappa >0$,
  %we also require that the limiting absorption principle
  %(\ref{BVP4}) hold.
%\end{problem}
%\begin{remark}
% Note that $u|_{\Gamma_H} \in L^2(\Gamma_H)$, so that the radiation condition makes sense.
%\end{remark}
We will convert this scattering problem to a boundary value problem.
To do this we will seek the scattered field as the sum of a mirrored
point-source $\Phi'(\cdot,z):=-\Phi(\cdot,z')$, where $z'$ is the
reflection of $z$ in the flat plane $\Gamma_0$,  plus some unknown
remainder $v$, i.e.\ $u^s = v + \Phi'(\cdot,z)$. Note that
$\Phi'(\cdot,z)$ is a solution to the scattering problem in the
special case that $\Gamma=\Gamma_0$. Using the boundary condition
$u^s + \Phi(\cdot,z) = 0$ on $\Gamma=\partial D$, we obtain the
boundary condition on $v$ that
\begin{equation}\label{eq:gdef}
  v(x) = -\{\Phi(x,z) - \Phi(x,z')\} = -G(x,z) =: g(x), \quad x \in \Gamma.
\end{equation}
Clearly $g\in BC(\Gamma)$; moreover it follows from the bound (\ref{Gbound}) below that we establish on $G(x,y)$ in the next section that
$g\in L^2(\Gamma)$ as well. %so that $g\in L^2(\Gamma)\cap BC(\Gamma)$.
Thus $u^s$ satisfies the above scattering problem if and only if $v$
satisfies the following Dirichlet problem with $g$ given by
(\ref{eq:gdef}):

  Given $g \in L^2(\Gamma)\cap BC(\Gamma)$, find $v \in C^2(D)$ which satisfies the Helmholtz equation
 \[
  \Delta v +k^2v = 0  \mbox{ in } D,
  \]
and the Dirichlet boundary condition $
    v = g$ on $\Gamma$.
   %(\ref{BVP3}) 

We intend to take this problem and state it in a more precise fashion. %Some comments are in order. %Although we sought to find a field $u$ such that $u=0$ on $\Gamma$ we never insisted that $u\in C(\overline{D})$. Althought this was done in \cite{chandpottheim} in the case when $\Gamma$ is Lyapunov, we will not do it here.
As stated earlier we will look for a solution to this boundary value
problem as the \emph{combined single- and double-layer potential}
\begin{equation}\label{udef}
  v(x) \; := \; u_2(x) - i\eta \; u_1(x),
  \quad x \in D,
\end{equation}
with some parameter $\eta \geq 0$, where for a given function
$\phi \in L^2(\Gamma)$ we define the \emph{single-layer potential}
\begin{equation}\label{SLP}
  u_1(x) \; := \; \int_{\Gamma} G(x,y) \phi(y) \; ds(y),
  \quad x \in \real^3,
\end{equation}
and the \emph{double-layer potential}
\begin{equation}\label{DLP}
  u_2(x) \; := \; \int_{\Gamma} \frac{\partial G(x,y)}
  {\partial \nu(y)} \phi(y) \; ds(y),
  \quad x \in \real^3.
\end{equation}
Using the ansatz (\ref{udef}) we will covert the boundary value problem to an equivalent BIE which will involve the operator $S$ as defined by (\ref{SDelta}), and also the principal-value generalisation of the operator $K$ defined by (\ref{KDelta}), which is 
\begin{eqnarray}\label{KDelta_2}
  (K\varphi)(x) \; &:=& \;
2 PV  \int_{\Gamma} \frac{\partial G(x,y)}
  {\partial \nu(y)} \phi(y) \; ds(y) \\& = & 2 \nonumber \lim_{\epsilon \to 0} \int_{\mathbb{R}^2 \backslash B_{\epsilon}(\tilde x)} \frac{\partial G(x,y)}
  {\partial \nu(y)} \phi(y) \; J_f(\tilde y)d\tilde y,
  \quad x \in \Gamma.
\end{eqnarray}
From now on when we refer to the operator $K$ we mean this principal-value version; of course in the case when $\Gamma$ is merely Lyapunov (\ref{KDelta_2}) is the same as (\ref{KDelta}).

Using this approach we will however, in the Lipschitz case, run into problems. The first problem arises because the operator $K$ is not a bounded operator on $BC(\Gamma)$ when the underlying surface is Lipschitz; indeed one can show that even if $\phi \in L^{\infty}(\Gamma)$, $K\phi$ may fail to be so. Thus there is little point in assuming that $g\in BC(\Gamma)$; rather we will only assume that $g\in L^2(\Gamma)$. It then follows that we cannot expect $v$ given by (\ref{udef}) to be continuous up to the boundary. Moreover the limiting values of
$v$ up to the boundary can only be computed in a non-tangential sense: In this chapter we define $\Theta(x)\subset D$ for $x \in\Gamma$ to be the cone of `non-tangential' approach to the point $x=(\tilde x,f(\tilde x))$; precisely, we fix $L^*>L$, and then define
\[
\Theta(x):= \{ y \in D \mbox{ such that } y_n-f(\tilde x) \geq L^*|\tilde y-\tilde x|\}.  
\]
The geometrical significance of these `non-tangential approach cones' is that there exists a constant $\alpha>0$ such that for all $x \in \Gamma$ and for all $y \in \Gamma$ and all $z \in \Theta(x)$ it holds that 
\begin{equation} \label{Geom_sig}
|z-x| \leq \alpha |z-y|.
\end{equation}

Writing $\mbox{n.t.}\lim_{x_n \to x}v(x_n)$ to indicate the limit of $v(x_n)$ as $x_n \in \Theta(x)$ approaches  $x\in \Gamma$, we'll
be able to show that
\begin{equation} \label{weak_boundary_cond}
\mbox{n.t.}\lim_{x_n \to x}v(x_n) = g(x) \mbox{ for almost all } x \in \Gamma.
\end{equation}
Thus in order to apply the BIE technique to the scattering problem in the case that the rough surface is Lipschitz, it's necessary to weaken the problem: we can't expect the solution to be continuous up to the boundary and we will have to impose the boundary condition in the weak sense of (\ref{weak_boundary_cond}). 

Thus we will pose our scattering problem with the boundary condition (\ref{weak_boundary_cond}) and also we will impose a radiation condition on our solution -- our usual radiation condition, see the boundary value problem below. However it will be necessary to impose a further boundedness condition on our solution, without which the problem will have a non-unique solution; we require that it satisfy the following: for $x=(\tilde x,f(\tilde x))\in\Gamma$ and $T\geq f_+$ define 
$$
v_T'(x)= \sup_{T \geq t> f(\tilde x)}|v(\tilde x,t)|.
$$
We'll then impose that $v'_T \in L^2(\Gamma)$ for all $T \geq f_+$. 
Thus the following is the exact boundary value problem we wish to consider: 

\bigskip

{\sc The Bondary Value Problem.} {\em Given $g \in L^2(\Gamma)$, find $v \in C^2(D)$ such that 
\[
\Delta v + k^2 v = 0, \quad  \mbox{ in } D,
\]
$v_T' \in L^2(\Gamma)$, for all $T\geq f_+$, the radiation condition (\ref{uprcstar}) holds with $F_H=v|_{\Gamma_H}$ (with $k_+$ replaced by $k$) for all $H\geq f_+$ and such that
\[
\mbox{n.t.}\lim_{y \to x} v(y) = g(x), \mbox{ for almost all } x \in \Gamma.
\]}  
\newline \begin{remark}
Note that $v|_{\Gamma_H} \in L^2(\Gamma_H)$ for all $H\geq f_+$ by the restriction on $v'_T$ for $T\geq f_+$. This means that the radiation condition (\ref{uprcstar}) makes sense.
\end{remark}
\begin{remark}
In his study of the similar problem, the Potential problem on a bounded Lipschitz domain, in order to obtain a unique solution to his problem, Verchota insisted that his solution $v$ be such that $v^* \in L^2(\Gamma)$, where  
$$v^*(x):= \sup_{y \in \Theta(x)}|v(y)|, \mbox{  for almost all }  x\in\Gamma.
$$ 
In our work we opt to impose the weaker condition on $v'$ -- note $\Vert v_T'\Vert_{L^2(\Gamma)}\leq \Vert v^*\Vert_{L^2(\Gamma)}$ for all $T \geq f_+$ -- as this will be sufficient to prove uniqueness and is easier to show than a condition on $v^*$; indeed it is not clear that $v^* \in L^2(\Gamma)$ in our case.
\end{remark}

%\begin{figure}[ht]
%  \centering
%  \psfrag{Gamma_f}{$\Gamma$}
%  \psfrag{Gamma_0}{$\Gamma_0$}
%  \psfrag{f-}{$f^{-}$}
%  \psfrag{f+}{$f^{+}$}
%  \psfrag{D}{$D$}
%  \psfrag{fixed point source phi(.,z)}{Point source $\Phi(\cdot,z)$}
 % \includegraphics{Bilder/setting.eps}
%  \caption{Geometrical setting of the scattering problem}
%\end{figure}

%%%%%%%%%%%%%%%%%%%%%%%%%%%%%%%%%%%%%%%%%%%%%%%%%%%%%%%%%%%%%%%%%
%%%%%%%%%%%%%%%%%%%%%%%%%%%%%%%%%%%%%%%%%%%%%%%%%%%%%%%%%%%%%%%%%
%
% SECTION: Scattering by rough surfaces in R^3
%
%%%%%%%%%%%%%%%%%%%%%%%%%%%%%%%%%%%%%%%%%%%%%%%%%%%%%%%%%%%%%%%%%
%%%%%%%%%%%%%%%%%%%%%%%%%%%%%%%%%%%%%%%%%%%%%%%%%%%%%%%%%%%%%%%%%

We conclude this section by summarizing the results of Chandler-Wilde, Heinemeyer and Potthast \cite{chandpottheim} and \cite{chandpottheim2}; and also those of Verchota et al \cite{Wavelets}, that we will make use of throughout this chapter.

We start by writing down the boundary value problem of \cite{chandpottheim}:
\newline Given $g \in X:=BC(\Gamma)\cap L^2(\Gamma)$ and $g_{\epsilon} \in X$, for $\epsilon>0$, with $\Vert g_{\epsilon} \to g\Vert_{L^2(\Gamma)} \to 0$ as $\epsilon \to 0$, find $v \in C^2(D) \cap C(\overline{D})$ which satisfies the Helmholtz equation in $D$, the Dirichlet boundary condition $v=g$ on $\Gamma$, the bound
\begin{equation}\label{bound_on_v}
|v(x)| \leq C,  \quad x\in D,
\end{equation}
for $C>0$ and the following limiting absorption principle: that for all sufficiently small $\epsilon>0$, there exists $v_{\epsilon} \in C^2(D)\cap C(\overline{D})$ satisfying $v_{\epsilon} = g_{\epsilon}$ on $\Gamma$, the Helmholtz equation in $D$ with $k$ replaced by $k+i\epsilon$ and the bound (\ref{bound_on_v}), such that for all $x \in D$, $v_{\epsilon}(x) \to v(x)$ as $\epsilon \to 0$.

We then have theorem 2.3 from \cite{chandpottheim}
\begin{theorem}\label{chandpottheim_result1}
If $\Gamma$ is given by (\ref{Gamma_def_chap5}) with $f$ Lyapunov, then the boundary value problem above has at most one solution.
\end{theorem}   

Let $A: L^2(\Gamma)\cap BC(\Gamma) \to L^2(\Gamma)\cap BC(\Gamma)$ be given by $A= I +K -i\eta S$. Note that it is shown in \cite{chandpottheim} that $K$ and $S$ defined by (\ref{KDelta}) and (\ref{SDelta}) are well-defined and bounded operators on $L^2(\Gamma)\cap BC(\Gamma)$ in the case when $\Gamma$ is Lyapunov. We introduce the operator $A'$, the adjoint of $A$, with respect to the bilinear form $(\cdot,\cdot)$ on $L^2(\Gamma)\times L^2(\Gamma)$ defined by
\[
(\phi,\psi) = \int_{\Gamma}\phi(y)\psi(y)ds(y), \quad \phi,\psi \in L^2(\Gamma).
\]
We then have the following theorem concerning the invertibility of $A$ and $A'$: 
 
\begin{theorem}\label{chandpottheim_result2} Suppose $\Gamma$ is given by (\ref{Gamma_def_chap5}) with $f$ Lyapunov. Then $A$ and $A'$ are invertible on $L^2(\Gamma)\cap BC(\Gamma)$ with 
\[
\Vert A^{-1}\Vert_{L^2(\Gamma) \to L^2(\Gamma)}=  \Vert A'^{-1}\Vert_{L^2(\Gamma) \to L^2(\Gamma)}\leq B,
\]
where
\begin{equation} \label{B_label}
B= \frac{1}{2}\left(1+\left(\frac{3k^2 L'}{\eta}[5L' +6L^2] + 6( L' +3L^2)^2\right)^{\frac{1}{2}}\right).
\end{equation}
\end{theorem}
%In order to establish boundedness of the single and double layer operators on $L^2(\Gamma)$ the following result was 
%\theorem}

%\end{theorem}

%In order to state the results of Verchota et al let us introduce the double-layer operator for %the Laplacian, $T$, defined for $\phi \in L^2(\Gamma)$, $x \in \Gamma$ by 
%\begin{eqnarray} \label{Ver_int}
%(T\phi)(x)&=& PV \int_{\Gamma}\frac{(x-y)\cdot\nu(y)}{|x-y|^3}\phi(y)ds(y)\\\nonumber &:= & %\lim_{\epsilon \to 0}\int_{\mathbb{R}^2\backslash B_{\epsilon}(x)}\frac{(x-y)\cdot %\nu(y)}{|x-y|^3}\phi(y)J_f(\tilde y)ds(y).
%\end{eqnarray}

%Then, see \cite{Wavelets} page 262, chapter 15, Theorem 4 and Lemma 2, we have:
%\begin{theorem}\label{Wavelets_them}
%Suppose $\Gamma$ is given by (\ref{Gamma_def_chap5}) with $f$ Lipschitz. For $\phi \in %L^2(\Gamma)$, $T \phi(x)$ given by (\ref{Ver_int}) exists almost everywhere and $T$ is 
%a bounded operator on $L^2(\Gamma)$. Moreover in the case that $f \in C^{\infty}(\mathbb{R}^2)$ %and $\phi \in C^{\infty}_0(\mathbb{R}^2)$ (as a function of $\tilde y$) then at every point $x %\in \Gamma$ we have
%\[
%(T\phi)(x)= -\int_{\mathbb{R}^2}\frac{(\tilde x- \tilde y).\nabla_{\tilde y} (\phi J_f)(\tilde %y)}{|\tilde x-\tilde y|^2}\lambda (\frac{f(\tilde x)-f(\tilde y)}{|\tilde x-\tilde y|})d\tilde y
%\]
%where 
%\begin{equation}\label{lambda_def}
%\lambda(t) := -\int_0^t\frac{1}{(1+s^2)^{\frac{3}{2}}}ds. 
%\end{equation}
%\end{theorem}
In order to state the results of Verchota et al let us introduce the double-layer operator for the Laplacian, $T$, defined for $\phi \in L^2(\mathbb{R}^2)$ and where $x=(\tilde x,f(\tilde x)), y=(\tilde y,f(\tilde y))$ by 
\begin{eqnarray} \label{Ver_int}
(T\phi)(x)&=& %PV \int_{\Gamma}\frac{(x-y)\cdot\nu(y)}{|x-y|^3}\phi(y)ds(y)\\\nonumber &:= & 
\lim_{\epsilon \to 0}\int_{\mathbb{R}^2\backslash B_{\epsilon}(x)}\frac{(x-y)\cdot \nu(y)}{|x-y|^3}\phi(y)d\tilde y.
\end{eqnarray}

Then, see \cite{Wavelets} page 262, chapter 15, Theorem 4 and Lemma 2, we have:
\begin{theorem}\label{Wavelets_them}
Suppose $\Gamma$ is given by (\ref{Gamma_def_chap5}) with $f$ Lipschitz. For $\phi \in L^2(\mathbb{R}^2)$, $T \phi(x)$ given by (\ref{Ver_int}) exists almost everywhere and $T$ is 
a bounded operator on $L^2(\mathbb{R}^2)$. Moreover in the case that $\phi \in C^{\infty}_0(\mathbb{R}^2)$ then at every point $x \in \Gamma$ at which $f$ is differentiable we have
\[
(T\phi)(x)= -\int_{\mathbb{R}^2}\frac{(\tilde x- \tilde y).\nabla_{\tilde y} \phi (\tilde y)}{|\tilde x-\tilde y|^2}\lambda \left(\frac{f(\tilde x)-f(\tilde y)}{|\tilde x-\tilde y|}\right)d\tilde y
\]
where 
\begin{equation}\label{lambda_def}
\lambda(t) := -\int_0^t\frac{1}{(1+s^2)^{\frac{3}{2}}}ds. 
\end{equation}
\end{theorem}
For $\phi \in L^2(\Gamma)$ and $z \in D$ let 
\[
(\mathcal{F}\phi)(z):= \int_{\Gamma}\frac{(z-y)\cdot\nu(y)}{|z-y|^3}\phi(y)ds(y).
\]
Then we will also use the following results (\cite{Wavelets} chapter 15, Theorem 1, pages 259,  264, 265) to establish the jump relations and the $v'_T$ boundedness condition.
\begin{theorem} \label{Wavelets_them_2} Suppose $\Gamma$ is given by (\ref{Gamma_def_chap5}) with $f$ Lipschitz. For each $\phi \in L^2(\Gamma)$ 
\[
\mathrm{n.t.}\lim_{z\to x}(\mathcal{F}\phi)(z)= \left(\frac{4\pi}{2}I + T\right)\phi(x)
\]
for almost all $x \in \Gamma$.
Further $F^*$ is a bounded operator on $L^2(\Gamma)$,
where, for $\phi \in L^2(\Gamma)$, $x \in \Gamma$, 
\[
(F^*\phi)(x) : = \sup_{y\in \Theta(x)}|(\mathcal{F}\phi)(y)|.
\]
\end{theorem}

We now precisely state the main results of this chapter all of which hold of course in the case when $\Gamma$ is given by (\ref{Gamma_def_chap5}) with $f$ Lipschitz:

\begin{theorem}\label{btheo}
The operators $S$ and $K$ are bounded on $L^2(\Gamma)$.
\end{theorem}

\begin{theorem}\label{A_inverse}
The integral operator $A$ is invertible on $L^2(\Gamma)$ with 
\begin{equation}\label{operator_B}
\Vert A^{-1}\Vert_{L^2(\Gamma)\to L^2(\Gamma)} \leq B,
\end{equation}
with $B$ given by (\ref{B_label}). 
\end{theorem}

\begin{theorem}\label{chap5_them}
The Boundary value problem has a unique solution.
\end{theorem}

%%%%%%%%%%%%%%%%%%%%%%%%%%%%%%%%%%%%%%%%%%%%%%%%%%%%%%%%%%%%%%%%%
%%%%%%%%%%%%%%%%%%%%%%%%%%%%%%%%%%%%%%%%%%%%%%%%%%%%%%%%%%%%%%%%%
%
% SECTION Properties of the 3D fundamental solution
%
%%%%%%%%%%%%%%%%%%%%%%%%%%%%%%%%%%%%%%%%%%%%%%%%%%%%%%%%%%%%%%%%%
%%%%%%%%%%%%%%%%%%%%%%%%%%%%%%%%%%%%%%%%%%%%%%%%%%%%%%%%%%%%%%%%%
\section{Properties of the three-dimensional fundamental solution}

We start with an investigation of properties of the fundamental
solution $\Phi(x,y)$ and its derivatives. The key results are the
expansions (\ref{Gdelta exp}) and (\ref{Gdeltanu exp}) needed to
prove mapping properties of the boundary integral operators $S$ and
$K$ in the next section. Note that this section has been copied verbatim from \cite{chandpottheim}.

%================================================================
%================================================================
% Derivatives of \Phi
%================================================================
%================================================================
For the first derivative of $\Phi(x,y)$ with respect to $y_3$ we
calculate
\begin{equation}
  \frac{\partial \Phi(x,y)}{\partial y_3}
  = - \frac{ik}{4\pi} \frac{(x_3-y_3)}{|x-y|^2}
    e^{i\kappa |x-y|}
    \; + \;\frac{1}{4\pi}
    \frac{(x_3-y_3)}{|x-y|^3} e^{i\kappa |x-y|}.
\end{equation}
The second derivative is given by
\begin{eqnarray}
  && \frac{\partial^2 \Phi(x,y)}{\partial y_3^2} \quad
  = \quad \frac{1}{4\pi}
  \bigg\{
    i k \frac{e^{ik |x-y|}}{|x-y|^2}
    \; - \; k^2 \frac{(x_3-y_3)^2}{|x-y|^3}
    e^{i k |x-y|}
    - 2 i k \frac{(x_3-y_3)^2}{|x-y|^4}
    e^{ik |x-y|}  \nonumber \\
    && - \frac{e^{i k |x-y|}}{|x-y|^3}
    \; - \; i k \frac{(x_3-y_3)^2}{|x-y|^4}
    e^{i k |x-y|}
    + 3 \frac{(x_3-y_3)^2}{|x-y|^5} e^{i k |x-y|}
  \bigg\}. \label{secondder}
\end{eqnarray}
For the third derivative with respect to $y_3$ we obtain
\begin{equation} \label{third}
  \frac{\partial^3 \Phi(x,y)}{\partial y_3^{3}} \quad
  = \quad \frac{3k^2}{4\pi}
   \frac{(x_3-y_3)}{|x-y|^3} e^{i k |x-y|}
  \;\; + \;\;
  O \left(
      \frac{1}{|x-y|^4}
    \right).
\end{equation}
This  holds in the sense that, given $c>0$, there exists a constant $C>0$ such that
$$
  \left| \frac{\partial^3 \Phi(x,y)}{\partial y_3^{3}} \quad
  - \quad \frac{3\kappa^2}{4\pi}
   \frac{(x_3-y_3)}{|x-y|^3} e^{i k |x-y|}
  \right|
\le
      \frac{C}{|x-y|^4},
$$
for all $x$, $y\in \real^3$, $x\neq y$, with $x_3$, $y_3\in [0,c]$. The similar equations below, in particular
(\ref{Gdelta exp}) and (\ref{Gdeltanu exp}), are to be understood in
an analogous fashion.
%================================================================
%================================================================
% Expansion of G_{Delta}
%================================================================
%================================================================

We use Taylor's expansion for the fundamental solution $\Phi(x,y)$
with respect to variations of $x_3$ and $y_3$. From Taylor's
theorem, if $g\in C^3[0,\infty)$, then
\begin{equation}\label{taylor}
  g(s) = g(0) \; + \; g'(0)s \; + \;  \frac{1}{2}g^{(2)}(0) s^2
  \; + \frac{1}{3!}\int_0^s (s-t)^2 g^{(3)}(t)\; dt, \quad s > 0.
\end{equation}
Applying (\ref{taylor}) to $g(s) := \Phi(x,\tilde y + s e_3)$, where
$e_3$ is the unit vector in the $x_3$ direction, with $\tilde y  =
(y_1,y_2,0) \in \Gamma_0$ and $s\in [0,c]$ with some constant $c$,
we obtain
\begin{eqnarray}
  \Phi(x,\tilde y  + s e_3)
  & = &
  \frac{1}{4\pi} \frac{e^{i k |x-\tilde y |}}{|x-\tilde y |}
  \; - \;
  \frac{ik}{4\pi} \frac{ x_3 \; e^{ik |x-\tilde y |}}{|x-\tilde y |^2}
  \; s  \label{centraldec1} \\
  && \; + \;
  \frac{ik }{4\pi}\frac{e^{ik|x-\tilde y |}}{|x-\tilde y |^2}
  \; \frac{s^2}{2} \; + \;
  O \left(
      \frac{1}{|x-\tilde y |^3}
    \right). \nonumber
\end{eqnarray}
To estimate the properties of single- and double-layer potentials on
$L^2(\Gamma)$ we need to use Taylor's expansion also with respect to
$x_3$. We treat all the terms of (\ref{centraldec1}) separately and
obtain, after some calculations,
\begin{eqnarray}
  \Phi(\tilde x + h e_3,\tilde y  + s e_3)
  & = &
  \frac{1}{4\pi} \frac{e^{i k |\tilde x-\tilde y |}}{|\tilde x-\tilde y |}
  \label{centraldec3} \\
  && \; + \; \frac{1}{4\pi}
  \frac{ik \; e^{i k|\tilde x-\tilde y |}}{|\tilde x-\tilde y |^2}
  \; \frac{(h-s)^2}{2}
  \; + \;
  O \left(
      \frac{1}{|\tilde x-\tilde y |^3}
    \right). \nonumber
\end{eqnarray}
Altogether we obtain
\begin{equation}\label{Gdelta exp}
  G(\tilde x + h e_3,\tilde y  + s e_3) =
 - \frac{1}{4\pi}
  \frac{i k \; e^{i k |\tilde x-\tilde y |}}{|\tilde x-\tilde y |^2}
  \; 2 h s \; + \;
  O \left(
      \frac{1}{|\tilde x-\tilde y |^3}
    \right),
\end{equation}
in the sense that, given $c>0$,
there exists a constant $C>0$ such that
$$
\left|  G(\tilde x + h e_3,\tilde y  + s e_3) + \frac{2hs}{4\pi}
  \frac{i k \; e^{i k |\tilde x-\tilde y |}}{|\tilde x-\tilde y |^2}\right| \leq
    \frac{C}{|\tilde x-\tilde y |^3},
$$
for all $\tilde x,\tilde y \in\real^2$ with $\tilde x\neq\tilde y$, 
and all $h,s\in[0,c]$. Arguing precisely as in \cite{chandmmas96} in the
case $|x-y|>1$, we can also show the bound that (cf.\
\cite[equations (3.6), (3.8)]{chandmmas96}) there exists a constant $C>0$ such that
\begin{equation} \label{Gbound}
\left|  G(x,y) \right| \leq
    \frac{C(1+x_3)(1+y_3)}{|x-y|^2},
\end{equation}
for all $x,y\in\real^3$ with $x,y\neq 0$ and $x_3,y_3\geq 0$.

%================================================================
%================================================================
% Expansion of nu(y) \cdot \grad_y G_{Delta}
%================================================================
%================================================================
For the normal derivative of $G$, noting that $\partial
\Phi(x,y')/\partial \nu(y)= \partial \Phi(x',y)/\partial \nu(y)$ and
introducing the notation $\bnu(y) := (\nu_1(y),\nu_2(y))$, we derive
\begin{eqnarray}
  4\pi \; \frac{\partial G(x,y)}{\partial \nu(y)}
  & = &
   - i k \; \bnu(y) \cdot (\tilde x-\tilde y )
  \left\{
    \frac{e^{i k |x-y|}}{|x-y|^2}
    \; - \;
    \frac{e^{i k |x-y'|}}{|x-y'|^2}
  \right\} \label{Gdeltapartialnu} \\
  &&  + \quad \bnu(y) \cdot (\tilde x-\tilde y )
  \left\{
    \frac{e^{i k |x-y|}}{|x-y|^3}
    \; - \;
    \frac{e^{i k |x-y'|}}{|x-y'|^3}
  \right\}  \nonumber \\
  && - \quad i k \; \frac{\nu_3(y)(x_3-y_3)}{|x-y|^2} \; e^{i k |x-y|}
  \; + \;
  \frac{\nu_3(y)(x_3-y_3)}{|x-y|^3} \; e^{i k |x-y|} \nonumber \\
  && - \quad i k \; \frac{\nu_3(y)(x_3+y_3)}{|x-y'|^2} \; e^{i k |x-y'|}
  \; + \;
  \frac{\nu_3(y)(x_3+y_3)}{|x-y'|^3} \; e^{i k |x-y'|}.
  \nonumber
\end{eqnarray}
We proceed as in (\ref{centraldec3}) and calculate
\begin{eqnarray}
  \frac{e^{i k |x-y|}}{|x-y|^2}
  & = & \frac{e^{i k |\tilde x-\tilde y |}}{|\tilde x-\tilde y |^2}
   \;+\;
  \frac{i k e^{i k |\tilde x-\tilde y |}}{|\tilde x-\tilde y |^3}
  \frac{(x_3 - y_3)^2}{2}
  \; + \;
  O \left(
      \frac{1}{|\tilde x-\tilde y |^4}
    \right).
  \label{Gdeltanudiff}
\end{eqnarray}
We use this to transform (\ref{Gdeltapartialnu}) into
\begin{eqnarray}
  4\pi \; \frac{\partial G(\tilde x + h e_3,\tilde y  + s e_3)}{\partial \nu(y)}
  & = & - k^2 \bnu(y) \cdot \frac{(\tilde x-\tilde y )}{|\tilde x-\tilde y |}
  \; \frac{ e^{i k |\tilde x-\tilde y |}}{|\tilde x-\tilde y |^2}
  \; 2 h s \label{Gdeltanu exp} \\
  && - \quad i k \nu_3(y)
  \frac{ e^{i k |\tilde x-\tilde y |}}{|\tilde x-\tilde y |^2}
  \; 2 h \; + \;
  O \left(
      \frac{1}{|\tilde x-\tilde y |^3}
    \right), \nonumber
\end{eqnarray}
this equation holding in the same sense as (\ref{Gdelta exp}).

%%%%%%%%%%%%%%%%%%%%%%%%%%%%%%%%%%%%%%%%%%%%%%%%%%%%%%%%%%%%%%%%
% FOURIER Transforms and HANKEL Transform
%%%%%%%%%%%%%%%%%%%%%%%%%%%%%%%%%%%%%%%%%%%%%%%%%%%%%%%%%%%%%%%%
\section{Boundedness of the single- and double-layer potential operators}

In this section we shall establish that $S$ and $K$ are bounded operators on $L^2(\Gamma)$. In order to do this, we split the operators into a local and a
global part, with the help of an appropriate
\emph{cut-off function}.
To this end let $\chi:[0,\infty)\to \real$ be the indicator function 
such that 
\begin{equation}\label{eq:chi}
  \chi(t) :=
    \begin{cases}
      0, & t < 1\\
      1, & t \geq 1.
    \end{cases}
  %\qquad \text{and} \qquad
  %0 \leq \chi(t) \leq 1,\quad \forall t \geq 0.
\end{equation}
Let $A$ with kernel $a$ denote one of
the operators $S$ or $K$, respectively. We define the \emph{global part}
\begin{equation} \label{A1op}
  (A_1\varphi)(x) \; := \; \int_{\Gamma} \chi(|\tilde x-\tilde y|)
  a(x,y) \phi(y) \; ds(y), \quad x \in \Gamma,
\end{equation}
and the \emph{local part}
\begin{equation} \label{A2op}
  (A_2\varphi)(x) \; := \; \int_{\Gamma} \Big(1-\chi(|\tilde x-\tilde y|)\Big)
  a(x,y) \phi(y) \; ds(y), \quad x \in \Gamma.
\end{equation}
This yields the decomposition $A = A_1 + A_2$ and we can study the
mapping properties of $A_1$ and $A_2$ as operators on $L^2(\Gamma)$
separately. We denote by $a_1$ the kernel of $A_1$ and by $a_2$ the kernel of $A_2$. 
Before however looking at the boundedness of $S$ and $K$ let us first make sure that they are well defined operators. 

We first of all remark that the global operators are well-defined for $\phi \in L^2(\Gamma)$ and for all $x \in \Gamma$ as can be seen by applying the Cauchy-Schwarz inequality and using the expansions (\ref{Gdelta exp}) and (\ref{Gdeltanu exp}). For the local part of the operator things are a bit more subtle. For the single layer operator we note that
\[
a_2(x,y) \leq s(\tilde x-\tilde y),
\]
where 
\begin{equation} \label{s_def}
s(\tilde y)= \begin{cases}
      0, & |\tilde y|\geq  1\\
      C/|\tilde y|, & |\tilde y|<1,  
    \end{cases}
\end{equation}
 for some $C>0$. Note that $s \in L^1(\mathbb{R}^2)$. Now for $\phi, \psi \in L^2(\mathbb{R}^2)$ 
\begin{eqnarray*}
&&\left|\int_{\mathbb{R}^2 \times \mathbb{R}^2}a_2(x,y)\phi(\tilde y)\psi(\tilde x)d\tilde xd\tilde y \right|\\& \leq &   \int_{\mathbb{R}^2 \times \mathbb{R}^2}s(\tilde x-\tilde y)|\phi(\tilde y)||\psi(\tilde x)|d\tilde xd\tilde y \\& \leq & \left\{\int_{\mathbb{R}^2 \times \mathbb{R}^2}s(\tilde x-\tilde y)|\phi(\tilde y)|^2d\tilde xd\tilde y \int_{\mathbb{R}^2 \times \mathbb{R}^2}s(\tilde x-\tilde y)|\psi(\tilde x)|^2d\tilde xd\tilde y\right\}^{\frac{1}{2}}\\& = & \Vert s\Vert^{\frac{1}{2}}_{L^1(\mathbb{R}^2)}\Vert\phi\Vert_{L^2(\mathbb{R}^2)} \Vert s\Vert^{\frac{1}{2}}_{L^1(\mathbb{R}^2)}\Vert\psi\Vert_{L^2(\mathbb{R}^2)}. 
\end{eqnarray*}
It follows by Fubini's theorem (see for example \cite{HandS} Theorem 21.13) that as a function of $\tilde y$, $a_2(x,y)\phi(\tilde y)\psi(\tilde x)$ is in $L^1(\mathbb{R}^2)$  for almost all $\tilde x \in \mathbb{R}^2$ so that the local part of the single layer operator is well-defined for almost all $\tilde x \in \mathbb{R}^2$.

We now turn to the local part of the double-layer operator. Making use of (\ref{Gdeltapartialnu}), we see that
\[
4\pi \frac{\partial G(x,y)}{\partial \nu(y)} = \nu(y)\cdot (x-y)\frac{e^{ik|x-y|}}{|x-y|^3} + r(x,y),
\]
where, the local part of $r(x,y)$, $[1- \chi(|\tilde x-\tilde y|)]r(x,y) := r_2(x,y)$ is such that
\begin{equation} \label{a2bound!_prime??}
|r_2(x,y)| \le s(\tilde x-\tilde y ), \quad x,y\in\Gamma, \quad x\neq y
\end{equation}
where $s$ is given by (\ref{s_def}). It follows as above, that the operator with kernel $r_2(x,y)$ is defined for almost all $x \in \Gamma$. This means we may rewrite the operator $K$ as 
\[(K\phi)(x)= (K_1\phi)(x)%\lim_{\epsilon \to 0}\int_{\mathbb{R}^2 \backslash B_{\epsilon}(\tilde x)}\frac{1}{4\pi}\nu(y)\cdot (x-y)\frac{e^{i\kappa|x-y|}}{|x-y|^3}\phi(y)J_f(\tilde y)d\tilde y 
+ \int_{\mathbb{R}^2}r_2(x,y)\phi(\tilde y) J_f(\tilde y)d\tilde y
\]
%We now focus on the operator whose kernal is given by 
%$ 1- \chi(|x-y|) \nu(y)\cdot (x-y)\frac{e^{i\kappa|x-y|}}{|x-y|^3}$. Precisely 
where $K_1$ is defined for $\phi \in L^2(\Gamma)$, $x \in \Gamma$ by  
\begin{equation}\label{K_1}
(K_1\phi)( x) := \lim_{\epsilon \searrow 0} \int_{\mathbb{R}^2 \backslash B_{\epsilon}(\tilde x)}[1-\chi(\tilde x-\tilde y)] \nu(y)\cdot (x-y)\frac{e^{ik|x-y|}}{|x-y|^3}\phi(\tilde y) J_f({\tilde y}) d \tilde y.
\end{equation}
Let us show that $K_1$ is well-defined. Firstly we note that 
\begin{equation} \label{e_exp}
e^{ik|x-y|} = 1 + ik|x-y| + \frac{(ik|x-y|)^2}{2!} + \dots. 
\end{equation} 
We have that
\begin{equation}
 (K_1\phi)(x) = (K_2\phi)(x) +  (K_3\phi)(x), \quad x \in \Gamma,
\end{equation}
where 
\begin{equation}\label{K_2}
(K_2\phi)(x)=\lim_{\epsilon \searrow 0} \int_{\mathbb{R}^2 \backslash B_{\epsilon}(\tilde x)}[1-\chi(\tilde x-\tilde y)] \nu(y)\cdot \frac{(x-y)}{|x-y|^3} \phi(\tilde y) J_f(\tilde y) d \tilde y,
\end{equation}
and 
\begin{equation}\label{K_3}
(K_3\phi)(x) = \lim_{\epsilon \searrow 0} \int_{\mathbb{R}^2 \backslash B_{\epsilon}(\tilde x)}[1-\chi(\tilde x-\tilde y)] \nu(y)\cdot (x-y)\frac{e^{ik|x-y|}-1}{|x-y|^3} \phi(\tilde y) J_f(\tilde y) d \tilde y. 
\end{equation}
From (\ref{e_exp}) it follows that the kernel of $K_3$ is bounded by $s(\tilde x - \tilde y)$, given by (\ref{s_def}), so that $K_3$ is well-defined for almost all $ x \in \Gamma$. 

We rewrite $K_2$ as 
\begin{equation}
(K_2\phi)(x)=(K_4\phi)(x) - (K_5\phi)(x),
\end{equation}
where
\begin{equation}\label{K_4}
(K_4\phi)(x)= \lim_{\epsilon \searrow 0} \int_{\mathbb{R}^{2} \backslash B_{\epsilon}(\tilde x)} \nu(y)\cdot \frac{(x-y)}{|x-y|^3} \phi(\tilde y) J_f(\tilde y) d \tilde y, 
\end{equation}
and where  
\begin{equation} \label{K_5}
(K_5\phi)(x)
= \int_{\mathbb{R}^{2}}\chi(|x-y|) \nu(y)\cdot \frac{(x-y)}{|x-y|^3} \phi(\tilde y) J_f(\tilde y) d \tilde y.
\end{equation}
That $K_4$ is well defined for almost all $x \in \Gamma$ follows from Theorem \ref{Wavelets_them}. Note that $K_5$ is well-defined for all $x \in \Gamma $ by a simple application of the Cauchy-Schwarz inequality.

We now turn to the issue of boundedness of the operators. We shall need in our arguments to make use of Young's inequality.
%consider integral operators
%with kernels of more general type.
 Suppose that
$l:\real^2\times\real^2\to\complex$ is such that $l(\tilde x,\cdot)$ is
measurable for all $\tilde x\in \real^2$, and
%(which is the case if, for example, $l(x,y)$
%is the tensor product, $m(x)n(y)$, of functions $m$ and $n$
%measurable on $\real^2$
let $L$ be the integral operator with kernel $l$, so that for $\psi \in L^2(\mathbb{R}^2)$, 
\begin{equation} \label{genio}
(L\psi)(\tilde x) = \int_{\real^2}l(\tilde x,\tilde y )\psi(\tilde y )\, d\tilde y, \quad
\tilde x\in\real^2.
\end{equation}
When
\begin{equation} \label{lllbound?}
|l(\tilde x,\tilde y )|\le \ell(\tilde x-\tilde y ),
\end{equation}
with $\ell\in L^p(\real^2)$, for some $p\in[1,\infty)$, 
%In this
%case, if $\psi$ is continuous and compactly supported, then
%(\ref{genio}) exists in the Lebesgue sense for all $\bx\in\real^2$,
%$L\psi\in L^s(\real^2)$, for $s\geq 1$, and, then
then from Young's
inequality \cite{Ree75}, it follows that for $s\geq 1$
\begin{equation} \label{youngs}
||L\psi||_{L^s(\real^2)} \le
||\ell||_{L^p(\real^2)}\,||\psi||_{L^r(\real^2)},
\end{equation}
where $ r^{-1} = 1 + s^{-1}-p^{-1}$. 
%Since the set of continuous
%compactly supported functions is dense in $L^r(\real^2)$, we can
%extend the domain of $L$ by density so that $L$ is a bounded
%operator from $L^r(\real^2)$ to $L^s(\real^2)$ with norm $\le
%||\ell||_{L^p(\real^2)}$. Further, if $\ell\in L^1(\real^2)$ and
%$\psi\in L^\infty(\real^2)$, then, trivially, (\ref{genio}) exists
%in the Lebesgue sense for all $\bx\in\real^2$ and (\ref{youngs})
%holds.

We will use the bound (\ref{youngs}) particularly often in the case
$\ell\in L^1(\real^2)$, in which case it implies that
\begin{equation} \label{Lnormbound2}
||L||_{L^2(\real^2)\to L^2(\real^2)} \leq ||\ell||_{L^1(\real^2)}.
\end{equation}
%for $1\leq q\leq \infty$. 

%

% We should point out that our proof below, that the global operator $A_2$ is bounded on $L^2(\Gamma)$, is precisely the same as that in , although there it was done under the assumption that $\Gamma$ is Lyapunov. In any case, for completeness, we include the proof. 
We now show that the local operators are bounded on $L^2(\Gamma)$.
 
% We kick off this section then,
 \begin{lemma} \label{lem:local}
$A_2$ is a bounded operator on $L^2(\Gamma)$. 
%for $1\le q\le \infty$,
%is a bounded operator from $L^\infty(\Gamma)$ to $BC(\Gamma)$, and
%is a bounded operator on $X$. Further, for some $n\in \N$,
%$A_2^n$ is a bounded operator from $L^2(\Gamma)$ to $X$.
\end{lemma}

%%%%%%%%%%%%%%%%%%%%%%%%%%%%%%%%%%%%%%%%%%%%%%%%%%%%%%%%%%%%%%%%
% LOCAL OPERATOR
\begin{proof}
In the single-layer case, the kernel $a_2$ of $A_2$ has compact support and is weakly
singular. Indeed, %Precisely, since (cf. \cite[equation (4.23)]{})
%\begin{equation} \label{lyapunov}
%|\nu(y)\cdot (x-y)| \le |\bx-\by|^{1+\alpha}
%||f||_{BC^{1,\alpha}(\real^2)}, \quad x,y\in\Gamma,
%\end{equation}
%it holds in the double-layer case $A=K$ that, 
for some constant
$C>0$,
\begin{equation} \label{a2bound?}
|a_2(x,y)| \le C\ell(\tilde x-\tilde y ), \quad x,y\in\Gamma, \; x\neq y,
\end{equation}
where
\begin{equation}
\ell(\tilde y ) := \left\{\begin{array}{cc}
               |\tilde y |^{-1}, & |\tilde y |\leq 1, \\
               0, & |\tilde y |>1. \\
             \end{array}\right.
\label{kernel A2}
\end{equation}
%The same bound holds (but is not sharp) in the single-layer case
%$A=S$.
Since $\ell\in L^1(\real^2)$, we see
from (\ref{Lnormbound2}) that $A_2$ is a bounded operator on
$L^2(\Gamma)$ in the single-layer case.

We now move on to the double layer case. Making use of (\ref{Gdeltapartialnu}), we see that
\[
4\pi \frac{\partial G(x,y)}{\partial \nu(y)} = \nu(y)\cdot (x-y)\frac{e^{ik|x-y|}}{|x-y|^3} + r(x,y),
\]
where, the local part of $r(x,y)$, $[1- \chi(|\tilde x-\tilde y|)]r(x,y) := r_2(x,y)$ is such that, for some constant $C>0$
\begin{equation} \label{a2bound!_prime}
|r_2(x,y)| \le C\ell(\tilde x-\tilde y ), \quad x,y\in\Gamma, \; x\neq y,
\end{equation}
where $\ell$ is given by (\ref{kernel A2}). It follows as above, that the operator with kernel $r_2(x,y)$ is bounded on $L^2(\Gamma)$. 

We now focus on the operator %whose kernal is given by 
%$ 1- \chi(|x-y|) \nu(y)\cdot (x-y)\frac{e^{i\kappa|x-y|}}{|x-y|^3}$. Precisely we are 
 $K_1$, defined by (\ref{K_1}). %for $\phi \in L^2(\Gamma)$ by  
%\begin{equation}
%K_1(\phi)(\tilde x) := \lim_{\epsilon \searrow 0} \int_{1> |\tilde x -\tilde y|>\epsilon} \nu(y)\cdot (x-y)\frac{e^{i\kappa|x-y|}}{|x-y|^3}\phi(\tilde y) \sqrt{1+ |\nabla_{\tilde y}f(\tilde y)|^2} d \tilde y.
%\end{equation}
%Firstly,  since $ e^{i\kappa|x-y|} = 1 + i\kappa|x-y| + \frac{(i\kappa|x-y|)^2}{2!} + \dots $, 
Again, we have that
\begin{equation}
 (K_1\phi)(x) = (K_2\phi)(x) +  (K_3\phi)(x),
\end{equation}
where $K_2$ and $K_3$ are given by (\ref{K_2}) and (\ref{K_3}) respectively. 
%\begin{equation}
%K_2(\phi(\tilde x))=\lim_{\epsilon \searrow 0} \int_{1> |\tilde x -\tilde y|>\epsilon} \nu(y)\cdot (x-y){1} \phi(\tilde y) \sqrt{1+ |\nabla_{\tilde y}f(\tilde y)|^2} d \tilde y,
%\end{equation}
%and 
%\begin{equation}
%K_3(\phi(\tilde x)) = \lim_{\epsilon \searrow 0} \int_{1> |\tilde x -\tilde y|>\epsilon} \nu(y)\cdot (x-y)\frac{i\kappa|x-y| + \frac{(i\kappa|x-y|)^2}{2!} + \dots}{|x-y|^3} \phi(\tilde y) \sqrt{1+ |\nabla_{\tilde y}f(\tilde y)|^2} d \tilde y. 
%\end{equation}
Since the kernel of $K_3$ is bounded by $\ell(\tilde x - \tilde y)$, with $\ell$ given by (\ref{kernel A2}), it follows again, that $K_3$ is bounded on $L^2(\Gamma)$. 

We rewrite $K_2$ as 
\begin{equation}
(K_2\phi)(x)=(K_4\phi)(x) - (K_5\phi)(x)
\end{equation}
where $K_4$ and $K_5$ are given by (\ref{K_4}) and (\ref{K_5}) respectively.
%\begin{equation}
%K_4(\phi(\tilde x))= \lim_{\epsilon \searrow 0} \int_{\mathbb{R}^{n-1} \cap |\tilde x -\tilde y|>\epsilon} \nu(y)\cdot (x-y){1} \phi(\tilde y) \sqrt{1+ |\nabla_{\tilde y}f(\tilde y)|^2} d \tilde y, 
%\end{equation}
%and where  
%\begin{equation} 
%K_5(\phi(\tilde x))
%= \int_{\mathbb{R}^{n-1} \cap |\tilde x -\tilde y|>1} \nu(y)\cdot (x-y)1 \phi(\tilde y) \sqrt{1+ |\nabla_{\tilde y}f(\tilde y)|^2} d \tilde y.
%\end{equation}
That $K_4$ is a bounded operator on $L^2(\Gamma)$ follows from Theorem \ref{Wavelets_them}. Thus, to complete the proof we need to show that $K_5$ is also a bounded operator on $L^2(\Gamma)$.

We begin by noting that $K_5$ is bounded as an operator from $L^2(\Gamma)$ into $L^{\infty}(\Gamma)$. Indeed, for all $x \in \Gamma$, a simple application of the Cauchy-Schwarz inequality shows that 
\begin{equation}
|(K_5\phi)(x)| \leq \mathcal{C} \Vert \phi\Vert_{L^2(\Gamma)},
\end{equation}
with $\mathcal{C}$ given by 
\begin{equation}
\mathcal{C}= L'\left\{\int_{G}\frac{1}{|\tilde x -\tilde y|^4}    d \tilde y \right\}^{\frac{1}{2}},
\end{equation}
where 
$$
G = \mathbb{R}^2 \backslash B_{1}(\tilde x),
$$
so that $\mathcal{C}$ is finite and bounded independently of $\tilde x$, as one sees by changing the last integral to polar coordinates and evaluating it.   

Now, for each $n= (n_1,n_2) \in \mathbb{Z}^2$ we let $\Lambda_n$ be the indicator function such that if $\tilde x \in \mathbb{R}^2$ is such that $n_1\leq x_1<n_1 +1$ and such that $ n_2 \leq x_2 <n_2 +1$ then $\Lambda_n(\tilde x)=1 $ and which is $0$ otherwise. Then, letting $\phi_n :=\phi \Lambda_n$ for $\phi \in L^2(\Gamma)$ we have that 
\[
\phi= \sum_{n \in \mathbb{Z}^2} \phi_n. %\begin{equation}
\]%\ell() := \left\{\begin{array}{cc}
%and by the continuity of $K_5$ as an operator from $L^2(\Gamma)$ into $L^{\infty}(\Gamma)$ that
%\[
%(K_5\phi)(x)= \sum_{n \in \mathbb{Z}^2} (K_5\phi_n)(x),
%\]
%for almost all $x \in \Gamma$. 
Now, for $\tilde x \in \mathbb{R}^2$ we let $\mathcal{N}(\tilde x)$ be the set of those $n\in \mathbb{Z}^2$ such that 
$$\mbox{dist} (\tilde x,\supp(\phi_n)) <1.
$$ Note that $ \mathcal{N}(\tilde x)$ contains no more than $9$ elements, and also, that if $x=(\tilde x,f(\tilde x)) \in \Gamma$ is such that  
$$
\mbox{dist} (\tilde x,\supp(\phi_m))>1 
$$ then 
\[
(K_5 \phi_m)(x) = (K_4\phi_m)(x).
\]
Thus, for almost all $x \in \Gamma$, 
\begin{eqnarray*}
(K_5\phi)(x)  %\sum_{n \in\mathbb{Z}^2}(K_5\phi_n)(x)
& = &\sum_{n \in\mathcal{N}(\tilde x)}(K_5\phi_n)(x)+ K_5\left(\sum_{n \not\in\mathcal{N}(\tilde x)}\phi_n\right)(x)\\&=&\sum_{n \in\mathcal{N}(\tilde x)}(K_5\phi_n)(x)+ K_4\left(\sum_{n \not\in\mathcal{N}(\tilde x)}\phi_n\right)(x)\\&=& \sum_{n \in\mathcal{N}(\tilde x)}(K_5\phi_n)(x)+ (K_4\phi)(x)%\sum_{n \in \mathbb{Z}^2}(K_4\phi_n)(x) 
\\&-& \sum_{n \in\mathcal{N}(\tilde x)}(K_4\phi_n)(x).
\end{eqnarray*}
We define, for $m \in\mathbb{Z}^2$,  $T(m):= \{n \in \mathbb{Z}^2: %\mbox{ there exixts } \tilde x \mbox{ such that} m_1<x_1<m_1+1, m_2 <x_2 <m_2+1< \mbox {such that} 
\mbox{dist}(\supp \Lambda_m ,\supp\Lambda_n)< 1\}$.  
In what follows we use that for $a_j\geq 0, j=1,\dots, n$
\[
(a_1 + \dots  + a_n)^2 \leq n(a_1^2 + \dots +a_n^2),
\]
and that $J_f \leq L'$.

So
\begin{eqnarray*}
&&\int_{\mathbb{R}^2}|(K_5\phi)(x)|^2 J_f(\tilde x) d \tilde x \\& \leq & 3\left\{  \int_{\mathbb{R}^2}\left|\sum_{n \in\mathcal{N}(\tilde x)}(K_5\phi_n)(x)\right|^2 J_f(\tilde x) d \tilde x + \int_{\mathbb{R}^2}\left|%\sum_{n \in\mathbb{Z}^2}
(K_4\phi)(x)\right|^2 J_f(\tilde x) d \tilde x\right.\\  &+&\left. \int_{\mathbb{R}^2}\left|\sum_{n \in\mathcal{N}(\tilde x)}(K_4\phi_n)(x)\right|^2 J_f(\tilde x) d \tilde x\right\} \\ & \leq & 3\left\{ \sum_{m \in \mathbb{Z}^2} \int_{\supp(\phi_m)}9\sum_{n \in\mathcal{N}(\tilde x)}|(K_5\phi_n)(x)|^2 J_f(\tilde x) d \tilde x + \Vert K_4\phi\Vert_{L^2(\Gamma)}^2\right.\\&&\left. +\sum_{m \in \mathbb{Z}^2} \int_{\supp(\phi_m)}9\sum_{n \in\mathcal{N}(\tilde x)}|(K_4\phi_n)(x)|^2 J_f(\tilde x) d \tilde x\right\}\\& \leq & 3\left\{ \sum_{m \in \mathbb{Z}^2} \int_{\supp(\phi_m)}9\sum_{n \in T(m)}|(K_5\phi_n)(x)|^2 J_f(\tilde x) d \tilde x + \Vert K_4\phi\Vert_{L^2(\Gamma)}^2\right.\\&&\left. +\sum_{m \in \mathbb{Z}^2} \int_{\supp(\phi_m)}9\sum_{n \in T(m)}|(K_4\phi_n)(x)|^2 J_f(\tilde x) d \tilde x\right\}\\& \leq & 3\left\{ 9\sum_{m \in \mathbb{Z}^2} \mathcal{C}^2\sum_{n \in T(m)}\Vert \phi_n\Vert_{L^2(\Gamma)}^2L'+ \Vert K_4\phi\Vert_{L^2(\Gamma)}^2 +9\sum_{m \in \mathbb{Z}^2} \sum_{n \in T(m)}\Vert K_4 \phi_n \Vert_{L^2(\Gamma)}^2\right\} \\ & \leq & 3\left\{ 9^2 \mathcal{C}^2L'\sum_{m \in \mathbb{Z}^2}\Vert \phi_m\Vert_{L^2(\Gamma)}^2+ \Vert K_4\Vert^2\Vert\phi\Vert_{L^2(\Gamma)}^2 +9^2\sum_{m \in \mathbb{Z}^2}\Vert K_4\Vert^2 \Vert \phi_m \Vert_{L^2(\Gamma)}^2\right\} \\
& \leq & 3[ 81\mathcal{C}^2L' + \Vert K_4 \Vert^2 + 81 \Vert K_4\Vert^2] \Vert \phi\Vert^2_{L^2(\Gamma)}.
\end{eqnarray*}
The proof is complete.
\end{proof}% 

We now turn to the global operators. To show that they are bounded on $L^2(\Gamma)$ we simply  use the proof employed in \cite{chandpottheim} to prove the same result but in the case when $\Gamma$ was Lyapunov. The necessary changes are trivial, but for completeness we'll go over the argument again here. Looking back to (\ref{genio}), one case of relevance to the argument is that in which
\begin{equation} \label{prodconvmult}
l(\tilde x,\tilde y )=m_1(\tilde x)\ell(\tilde x-\tilde y )m_2(\tilde y ),
\end{equation}
 with $m_1,m_2\in
BC(\real^2)$, $\ell\in L^2(\real^2)$, $\cF\ell\in
L^\infty(\real^2)$. In this case, if $\psi\in L^2(\real^2)$, $L$ is a bounded operator on $L^2(\real^2)$
with norm
\begin{equation} \label{Lnormbound}
||L||_{L^2(\real^2)\to L^2(\real^2)} \le 2\pi
||m_1||_{BC(\real^2)}\, ||\cF \ell||_{L^\infty(\real^2)}
\,||m_2||_{BC(\real^2)},
\end{equation}
see for example \cite{Ree75}. Examining (\ref{Gdelta exp}) and (\ref{Gdeltanu exp}) we see that
large parts of the kernels of the operators $S$ and $K$ have the
form (\ref{prodconvmult}), where moreover $\ell$ has  certain
symmetries that simplify the calculation of its Fourier transform.
For $\tilde y \in\real^2$ let
$r:=|\tilde y |$ and $\hat y:=\tilde y /|\tilde y |$. The specific symmetries that arise
are those where $\ell$ has the form
\begin{equation}\label{eq:radial2}
  \ell(\tilde y ) = F(r)Y^j_n(\hat y),
\end{equation}
where
\begin{equation} \label{Fdef}
  F(r) \; := \; \frac{e^{ikr}}{\beta + r^2},
  \quad r \geq 0,
\end{equation}
for some $\beta>0$ and with $n=0$ or $1$, and $j=0,...,n$, where the functions $Y^j_n$ are
spherical harmonics of order $n$ defined on the unit circle
$\Omega\subset\real^2$ by
\begin{equation} \label{Ydef?}
Y^0_0(\hat y) := 1, \quad Y^0_1(\hat y) := \cos \theta, \quad Y^1_1(\hat y)
:= \sin \theta,\quad \hat y=(\cos\theta,\sin\theta)\in\Omega.
\end{equation}

%\begin{equation} \label{Fdef?}
%  F(r) \; := \; \frac{e^{i\kappa r}}{\beta + r^2},
%  \quad r \geq 0,
%\end{equation}
%for some $\beta>0$. 
%The relevance of this example to the operators
%$S$ and $K$ is that the explicitly written terms on the right hand
%side of (\ref{Gdelta exp}) and (\ref{Gdeltapartialnu}) all take the
%form (\ref{prodconvmult}) if $\ell$ is given by (\ref{eq:radial2})
%and (\ref{Fdef}) with $\beta=0$, and $\tilde x + h e_3$ and $\tilde y  + s e_3$
%lie on $\Gamma$.

We now recall a result from \cite{chandpottheim} needed for the proof.

 \begin{lemma} \label{lem:bound}
 If $\ell$ is given by
 (\ref{eq:radial2}) and (\ref{Fdef}) with
$\beta>0$ and $n=j=0$ or $n=1$ and $j=0$ or $1$, then $\cF \ell\in
L^\infty(\real^2)$. Further in the case that we replace $k$ by $k+i\epsilon$ in the definition of $\ell$ to get a new function $\ell_{\epsilon}$ say, then the result still holds and moreover $\Vert \cF \ell -\cF \ell_{\epsilon}\Vert_{L^\infty(\mathbb{R}^2)} \to 0$ as $\epsilon \to 0$.  
\end{lemma}

With these preliminaries in place,             % 0, & |\by|>1. \\
we now prove a lemma on the mapping
properties of $A_1$. Together, lemmas \ref{lem:local} and
\ref{lem:global} provide a proof of Theorem \ref{btheo}.
\begin{lemma} \label{lem:global}
$A_1$ is a bounded operator on $L^2(\Gamma)$.
% and is a bounded
%operator from $L^2(\Gamma)$ to $X$.
\end{lemma}%%%%%%%%%%%%%%%%%%%%%%%%%%%%%%%%%%%%%%%%%%%%%%%%%%%%%%%%%%%%%%%%
% GLOBAL OPERATOR
\begin{proof}
 From the decompositions
(\ref{Gdelta exp}) and (\ref{Gdeltanu exp}) it follows that the
kernel $a_1$ of $A_1$ can be written, in both the cases $A=S$ and
$A=K$, in the form
\begin{equation} \label{a1split}
a_1(x,y) = l^*(\tilde x,\tilde y ) + l(\tilde x,\tilde y ),
\end{equation}
where $l^*$ is a sum of terms each of the form (\ref{prodconvmult}),
with $m_1,m_2\in BC(\real^2)$ and $\ell$ given by
 (\ref{eq:radial2}) and (\ref{Fdef}) with
$\beta=1$, and with $n=0$ or $1$. Further, $l^*$ can be chosen so
that $l$ satisfies the bound, for some constant $C>0$,
\begin{equation} \label{lbbound}
|l(\tilde x,\tilde y )| \le C \tilde\ell(\tilde x-\tilde y ), \quad \tilde x,\tilde y \in\real^2,
\end{equation}
where $\tilde\ell(\tilde y ):= (1+|\tilde y |)^{-3}$, so that $\tilde\ell\in
L^1(\real^2)$. In detail, in the case $A=S$ we see from (\ref{Gdelta
exp}) that an appropriate choice is to take
\begin{equation} \label{la}
l^*(\tilde x,\tilde y ) = -\frac{ik f(\tilde x)f(\tilde y )}{2\pi}
\frac{e^{ik|\tilde x-\tilde y |}}{1+|\tilde x-\tilde y |^2},
\end{equation}
while, in the case $A=K$ we see from (\ref{Gdeltanu exp}) that we
can take
\begin{equation} \label{lab}
l^*(\tilde x,\tilde y ) = -\frac{k^2 f(\tilde x)f(\tilde y )}{2\pi}\bnu(y)\cdot
\frac{\tilde x-\tilde y }{|\tilde x-\tilde y |} \frac{e^{ik|\tilde x-\tilde y |}}{1+|\tilde x-\tilde y |^2}
-\frac{ik f(\tilde x)\nu_3(y)}{2\pi}
\frac{e^{ik|\tilde x-\tilde y |}}{1+|\tilde x-\tilde y |^2}.
\end{equation}
 It follows from (\ref{Lnormbound}) and Lemma
\ref{lem:bound} applied to the integral operator with kernel $l^*$,
and (\ref{Lnormbound2}) applied to the integral operator with kernel
$l$, that $A_1$ is a bounded operator on $L^2(\Gamma)$.
\end{proof}
          %  \end{array}\right.
%
At this point we introduce some notation that will help us to emphasize the underlying surface on which the operators $S$ and $K$ given by (\ref{SDelta}) and (\ref{KDelta_2}) respectively are defined: We will write $S$ and $K$ as $S_f$ and $K_f$ respectively if the integrals in (\ref{SDelta}) and (\ref{KDelta_2}) are defined over the surface $\Gamma$ given by (\ref{Gamma_def_chap5}).

\begin{remark}\label{uniform_bound}
Examining the proofs of lemmas \ref{lem:global} and \ref{lem:local} and also the proof of theorem \ref{Wavelets_them} we see that the norms of the operators $S$ and $K$ depend, in terms of the underlying surface $\Gamma$, only on the constants $L$ and $L'$ and also on the maximum height of the Lipschitz function. This means that, given constants $0<C_1<C_2$, the operators $S_h$ and $K_h$ are uniformly bounded on $L^2(\Gamma)$ for all $h \in  B$ where 
\begin{eqnarray*}
B(C_1,C_2):=\{ f:\mathbb{R}^2 \to \mathbb{R}   : C_1\leq f \leq C_2 \mbox{ and } f \mbox{ is Lipschitz with constant L }\}.
\end{eqnarray*} 
We will make use of this fact in lemma \ref{contthe} below. 
\end{remark}

%(which is the case if, for example, $l(x,y)$
%is the tensor product, $m(x)n(y)$, of functions $m$ and $n$
%measurable on $\real^2$

%

%

% Further, denoting $\ell$ and $L$ by
%$\ell_{\kappa_1}$ and $L_{\kappa_1}$ to indicate their dependence on
%$\kappa_1$, we have that $||\cF
%\ell_{\kappa_1}-\cF\ell_0||_{L^\infty(\real^2)}\to 0$ as
%$\kappa_1\to 0$, so that $L_{\kappa_1}$ tends to $L_0$ in norm as
%$\kappa_1\to 0$.
% \end{lemma}

%%%%%%%%%%%%%%%%%%%%%%%%%%%%%%%%%%%%%%%%%%%%%%%%%%%%%%%%%%%%%%%%%
%%%%%%%%%%%%%%%%%%%%%%%%%%%%%%%%%%%%%%%%%%%%%%%%%%%%%%%%%%%%%%%%%
%
% Properties of single- and double-layer potentials
%
%%%%%%%%%%%%%%%%%%%%%%%%%%%%%%%%%%%%%%%%%%%%%%%%%%%%%%%%%%%%%%%%%
%%%%%%%%%%%%%%%%%%%%%%%%%%%%%%%%%%%%%%%%%%%%%%%%%%%%%%%%%%%%%%%%%

%%%%%%%%%%%%%%%%%%%%%%%%%%%%%%%%%%%%%%%%%%%%%%%%%%%%%%%%%%%%%%%%%

\section{Properties of the layer-potentials}

As part of the proof of Theorem \ref{chap5_them} we need to show that our
modified single- and double-layer potentials $u_1$ and $u_2$, over
the unbounded surface $\Gamma$, satisfy certain properties that we wish our solution to have.
%behave in a similar way to the
%corresponding standard layer potentials supported on a smooth
%bounded surface. This is done in the following theorem in which $M
%:= \{x:0<x_3<f()\}$ denotes the region between $\Gamma$ and
%$\Gamma_0$.
We begin with the following lemma.
%%%%%%%%%%%%%%%%%%%%%%%%%%%%%%%%%%%%%%%%%%%%%%%%%%%%%%%%%%%%%%%%
% THEOREM    JUMP RELATIONS
%%%%%%%%%%%%%%%%%%%%%%%%%%%%%%%%%%%%%%%%%%%%%%%%%%%%%%%%%%%%%%%%
\begin{theorem}
\label{jr thm} Let $u_1$ and $u_2$ denote the single- and
double-layer potentials with density
  $\phi\in L^2(\Gamma)$, defined by (\ref{SLP}) and (\ref{DLP}), respectively.
  It holds that:

  (i) For $n=1,2$, $u_n\in C^2(D)$ and $\Delta u_n + k^2 u_n = 0$ in $D$;
  
  (ii) Given constants $C_2>C_1>0$ and $\epsilon>0$, %and a compact subset $S$ of $H^+$, 
  there exists a constant $C_{\epsilon}>0$ such that
    \begin{equation} \label{layerbound2}
    |u_n(x)| \leq C_{\epsilon}\norm{\phi}_{L^2(\Gamma)}  , \quad n=1,2,
  \end{equation}
  for all $x\in D$ with $|x_3-f(x_1,x_2)| >\epsilon$, all $\phi\in L^2(\Gamma)$, and all $f \in B(C_1,C_2)$.%all $\in S$,
  %where $C>0$ is independent of $\epsilon$.
 % \[
 % C=.  
 % \]%and all $f\in
  %B=B(C_1,C_2)$.
 % (iv) $u_1$ and $u_2$ satisfy the radiation condition (\ref{ }) 
 %  (v) $u_1',  u_2' \in L^2(\Gamma)$. 

  (iii) For $u_1$ and $u_2$ we have the non-tangential jump relations:%can be continuously extended from $D$ to $\bar D$ and
  %from $M$ to $\bar M$, with limiting values
  \begin{equation}\label{jump rel 1}
    \mathrm{n.t.}\lim_{y \to x}{u_{1}(y)} \; = \;  \int_{\Gamma} G(x,y) \phi(y) \; ds(y),
    \quad x \in \Gamma,
  \end{equation}
  and
  \begin{equation}\label{jump rel 2}
   \mathrm{n.t.}\lim_{y\to x} u_{2}(y) \; = PV \int_{\Gamma} \frac{\partial G(x,y)}
    {\partial \nu(y)} \phi(y) \; ds(y) \; + \;
    \frac{1}{2} \phi(x), \quad x \in \Gamma,
  \end{equation}

  %u_{n,\pm}(x) := \lim_{\epsilon\to 0+} u_n(x\pm\epsilon \nu(x))$,
  %for  $n=1,2$ and $x\in\Gamma$,
  %and $\nu(x)$ denotes the unit normal at $x\in\Gamma$ directed into
  %$D$.

  %(iii) Given constants $C_2>C_1>0$ and a compact subset $S$ of $H^+$, there exists a constant
  %$C>0$ such that
  %\begin{equation} \label{layerbound}
  %  |u_n(x)| \leq C  ||\varphi||_X, \quad x\in D\cup M,\; n=1,2,
  %\end{equation}
  %for all $\varphi\in X$, $\in S$, and $f\in B=B(C_1,C_2)$;

\end{theorem}

\begin{proof}

To prove (ii), we recall that $G(x,y)$ satisfies the bound (\ref{Gbound}) and point out
that, by interior elliptic regularity estimates for solutions of the
Helmholtz equation (e.g.\ \cite[Lemma 2.7]{chandprs98}), it follows
that $\nabla_yG(x,y)$ satisfies the same bound %for$\nabla_yG(x,y)$ satisfy the
with a different constant $C$. Precisely, if $ a(x,y)$ denotes the kernel of $u_1$ or $u_2$, then for every $\epsilon>0$ there exists $C_{\epsilon}>0$  such that  
\begin{equation} \label{bound_on_a2}
|a(x,y)| \leq C_{\epsilon} \frac{(1+x_3)(1+y_3)}{1+|x-y|^2},
\end{equation}
for all $x,y\in\real^3$ with $x_3,y_3\geq 0$ and $|x-y|\geq
\epsilon$. %Under the assumption that $|x_3-f(\tilde x)|>\epsilon$, and because $\Gamma$ is the graph of a Lipschitz function, with Lipschitz constant $L$, then for any $y \in \Gamma$ it holds that $|x-y| \geq \epsilon\sin(\pi/2 -\tan^{-1}L)$,
%so that for any $y \in \Gamma$
%\[
%|x-y|^2 \geq \frac{1}{2} \epsilon^2\sin^2(\pi/2 -\tan^{-1}L) + \frac{1}{2}|x-y|^2.
%\]
%This leads to the bound, that for $y \in \Gamma$
%\[
%|a(x,y)| \leq 2C \frac{(1+x_3)(1+y_3)}{\min[\epsilon^2\sin^2(\pi/2 -\tan^{-1}(L)),1][1+|x-y|^2]}.
%\]
%Applying Cauchy-Schwarz we have that $|u(x)|$ is bounded, for
%$|x_3-f(\tilde x)|>\epsilon$, by 
%$$
% \frac{2C (1+f_+)}{\min[\epsilon^2\sin^2(\pi/2 -\tan^{-1}(L)),1]}I(x)
%\norm{\phi}_{L^2(\Gamma)},
%$$  where
%\begin{eqnarray*}
%[I(x)]^2 &=& (1+x_3)^2\int_{\Gamma}\frac{ds(y)}{(1+|x-y|^2)^2}\\
%&\leq &(1+x_3)^2(1+||\nabla f||_{L^{\infty}(\Gamma)})^{1/2}
%\int_{\real^2}\frac{d\by}{(1+|\tilde x-\by|^2+ (x_3-f(\by))^2)^2}.
%\end{eqnarray*}
%Thus, for some constant $c>0$ it holds, for all $x\in\{y:y_3>0\}$
%and all $f\in B$, that $ [I(x)]^2 \leq cF(x_3) $ where
%$$
%F(x_3) := (1+x_3)^2\int_0^\infty \frac{r\;dr}{(1 + x_3^2+r^2)^2} =
%\frac{(1+x_3)^2}{x_3^2}\int_0^\infty \frac{s\;ds}{(x_3^{-2} +
%1+s^2)^2} \,.
%$$
%Clearly $F$ is bounded on $[0,\infty)$. Thus the first term
%satisfies the bound.
Applying Cauchy-Schwarz, 
%proof of (\ref{layerbound}), 
we see that it holds that
$$
|u_n(x)| \leq C_\epsilon (1+C_2)I(x)
\norm{\varphi}_{L^2(\Gamma)},\quad n=1,2,
$$
for all $x\in D$ with $|x_3-f(x_1,x_2)|\geq \epsilon$ and all
$f\in B(C_1,C_2)$, where
\begin{eqnarray*}
[I(x)]^2 &=& (1+x_3)^2\int_{\Gamma}\frac{ds(y)}{(1+|x-y|^2)^2}\\
&\leq &(1+x_3)^2L'
\int_{\real^2}\frac{d\tilde y}{(1+|\tilde x-\tilde y|^2+ (x_3-f(\tilde y))^2)^2}.
\end{eqnarray*}
Thus, for some constant $c>0$ it holds, for all $x\in\{y:y_3>0\}$
and all $f\in B(C_1,C_2)$, that $ [I(x)]^2 \leq cL'F(x_3) $ where
$$
F(x_3) := (1+x_3)^2\int_0^\infty \frac{r\;dr}{(1 + x_3^2+r^2)^2} =
\frac{(1+x_3)^2}{x_3^2}\int_0^\infty \frac{s\;ds}{(x_3^{-2} +
1+s^2)^2} \,.
$$
Clearly $F$ is bounded on $[0,\infty)$. Thus the first term
satisfies the bound (\ref{layerbound2}). 
%In view of the bound on $I(x)$ already
%shown above, we see that we have established (\ref{layerbound2}).

We now establish (i). This is clear when
$\phi$ is compactly supported. The general case follows from the
density in $L^2(\Gamma)$ of the set of those elements that
are compactly supported, from the bound (\ref{layerbound2}), and
from the fact that limits of uniformly convergent sequences of
solutions of the Helmholtz equation satisfy the Helmholtz equation
(e.g. \cite[Remark 2.8]{chandprs98}).

%(iii). 
To prove (iii), we make use of the following continuous \emph{cut-off function}. Let $\chi_c:[0,\infty) \to \mathbb{R}$ be a continuous function with 
\begin{equation}\label{eq:chi_c}
  \chi_c(t) :=
    \begin{cases}
      0, & t < 1/2\\
      1, & t \geq 1.
    \end{cases}
 \qquad \text{and} \qquad
  0 \leq \chi_c(t) \leq 1,\quad \forall t \geq 0.
\end{equation}
 %use the cut-off function $\chi$ given by (\ref{eq:chi}).
Let $u$ denote one of $u_1$ and $u_2$, and let $a$ denote the kernel
of $u$ so that $a(x,y):=G(x,y)$ and $a(x,y):=\partial
G(x,y)/\partial \nu(y)$ in the respective cases. We have, for $x\in
D$, that
\begin{eqnarray}
\hspace*{-2ex} u(x)
  & = & \int_{\Gamma}
  \chi_c(|x-y|)a(x,y)\phi(y)\;ds(y)  + \int_\Gamma\big[1-\chi_c(|x-y|)\big]
    a(x,y) \phi(y) \; ds(y). \nonumber
\end{eqnarray}
The first term has a continuous kernel that is bounded at infinity
by the estimate (\ref{Gdelta exp}) or (\ref{Gdeltanu exp}), and,
since $\phi\in L^2(\Gamma)$, is continuous in $\{x:x_3>0\}$. Thus there is no problem in computing its value on $\Gamma$.
To deal with the second term suppose initially we are in the single-layer case. The kernel splits into two parts; since the term 
$$
[1-\chi_c(x-y)]\frac{e^{ik|x-y'|}}{|x-y'|}
$$
is continuous and has compact support it follows that
\[
\int_{\Gamma}[1-\chi_c(x-y)]\frac{e^{ik|x-y'|}}{|x-y'|}\phi(y)ds(y)
\]
is continuous in $\{x:x_3>0\}$.  We are thus left to deal with 
\[
 \mathrm{n.t.}\lim_{x^n\to x}\int_{\Gamma}[1-\chi_c(x^n-y)]\frac{e^{ik|x^n-y|}}{|x^n-y|}\phi(y)ds(y).
\]
In the case that $\phi$ is a smooth function with compact support, so that it belongs to $L^{\infty}(\Gamma)$, the dominated convergence theorem implies that
\begin{equation}\label{ntjr_1}
\mathrm{n.t.}\lim_{x^n\to x}\int_{\Gamma}[1-\chi_c(x^n-y)]\frac{e^{ik|x^n-y|}}{|x^n-y|}\phi(y)ds(y)= \int_{\Gamma}[1-\chi_c(x-y)]\frac{e^{ik|x-y|}}{|x-y|}\phi(y)ds(y).   % The
\end{equation}
The dominated convergence theorem is applicable due to the bound, that follows from (\ref{Geom_sig}) and the triangle inequality, that for $x \in \Gamma$ and for all $x^n \in \Theta(x)$,
and all $y \in \Gamma$ that 
\begin{equation}\label{Geom_sig_+1}
|x-y|\leq (\alpha+1)|x^n-y|. 
\end{equation}
To prove (\ref{ntjr_1}) for general $\phi \in L^2(\Gamma)$ it is sufficient, that the maximal operator $H^*(x):= \sup_{y \in \Theta(x)}|H(y)|$ is bounded on $L^2(\Gamma)$, where  $H(z)$, $z \in \mathbb{R}^3$ is given by 
\[
H(z) = \int_{\Gamma} [1-\chi_c(z-y)]\frac{e^{ik|z-y|}}{|z-y|}\phi(y)ds(y).
\]
But this is straightforward; as is seen by using again the bound (\ref{Geom_sig_+1}) so that the kernel of $H$, $h(z,y)$ say, satisfies that
\begin{equation}
|h(z,y)| \le [\alpha+1]\ell(\tilde x-\tilde y), \quad x,y\in\Gamma, \; x\neq y,
\end{equation}
where
\begin{equation}\label{well}
\ell(\tilde y) := \left\{\begin{array}{cc}
               |\tilde y|^{-1}, & |\tilde y|\leq 1, \\
               0, & |\tilde y|>1. \\
             \end{array}\right.
\end{equation}
%The same bound holds (but is not sharp) in the single-layer case
%$A=S$.
Since $\ell\in L^1(\real^2)$, we see
from (\ref{Lnormbound2}) that $H^*$ is a bounded operator on
$L^2(\Gamma)$.

For the double-layer case we see from (\ref{Gdeltapartialnu}) that the kernel is composed of parts that are continuous with compact support and other singular parts. We need only examine the singular parts. In the first place, for the layer potential with kernel 
\[
-[1-\chi_c(x-y)]ik\nu(y)\cdot(x-y)\frac{e^{ik|x-y|}}{|x-y|}
\]
a similar argument as above shows that 
\begin{eqnarray}\label{ntjr_2}
&&\mathrm{n.t.}\lim_{x^n\to x}\int_{\Gamma}[1-\chi_c(x^n-y)]ik\nu(y)\cdot(x^n-y)\frac{e^{ik|x^n-y|}}{|x^n-y|}\phi(y)ds(y)\\&&= \int_{\Gamma}[1-\chi_c(x-y)]ik\nu(y)\cdot(x-y)\frac{e^{ik|x-y|}}{|x-y|}\phi(y)ds(y).   % The
\end{eqnarray}
We now examine $$
\mathrm{n.t.}\lim_{x^n \to x} (\mathcal{K}_1\phi)(x^n)
$$ where 
\[
(\mathcal{K}_1\phi)(x^n)=\int_{\Gamma}[1-\chi_c(|x^n-y|)]\nu(y)\cdot(x^n-y)\frac{e^{ik|x^n-y|}}{|x^n-y|^3}ds(y).
\]
Firstly,  since 
$$
 e^{ik|x^n-y|} = 1 + ik|x^n-y| + \frac{(ik|x^n-y|)^2}{2!} + \dots, 
 $$ 
we have that
\begin{equation}
 (\mathcal{K}_1\phi)(x^n) = (\mathcal{K}_2\phi)(x^n) +  (\mathcal{K}_3\phi)(x^n),
\end{equation}
where 
\begin{equation}
(\mathcal{K}_2\phi)(x^n)= \int_{\mathbb{R}^2}[1-\chi_c(|x^n-y|)] \nu(y)\cdot \frac{(x^n-y)}{|x^n-y|^3}\phi(\tilde y) J_f(\tilde y) d \tilde y,
\end{equation}
and 
\begin{equation}
(\mathcal{K}_3\phi)(x^n) = \int_{\mathbb{R}^2}[1-\chi_c(|x^n-y|)] \nu(y)\cdot (x^n-y)\frac{%ik|x^n-y| + \frac{(ik|x^n-y|)^2}{2!} + \dots 
e^{ik|x^n-y|} -1}{|x^n-y|^3} \phi(\tilde y) J_f(\tilde y)d \tilde y. 
\end{equation}
Since the kernel of $\mathcal{K}_3$ is bounded by $C \ell(\tilde x - \tilde y)$, given by (\ref{well}), for some $C>0$, it follows again, that $\mathrm{n.t.}\lim_{x^n \to x}(\mathcal{K}_3\phi)(x^n) = (\mathcal{K}_3\phi)(x)$. 

We rewrite $\mathcal{K}_2$ as 
\begin{equation}
(\mathcal{K}_2\phi)(x^n)=(\mathcal{K}_4\phi)(x^n) - (\mathcal{K}_5\phi)(x^n)
\end{equation}
where
\begin{equation}
(\mathcal{K}_4\phi)(x^n)=  \int_{\mathbb{R}^{2} } \nu(y)\cdot \frac{(x^n-y)}{|x^n-y|^3} \phi(\tilde y) J_f(\tilde y) d \tilde y, 
\end{equation}
and where  
\begin{equation} 
(\mathcal{K}_5\phi)(x^n)
= \int_{\mathbb{R}^{2} } \chi_c(|x^n-y|)\nu(y)\cdot \frac{(x^n-y)}{|x^n-y|^3} \phi(\tilde y) J_f(\tilde y) d \tilde y.
\end{equation}
That 
\[
\lim_{x^n \to x}\mathcal{K}_4(x^n)= \frac{4\pi}{2}\phi(x) + PV\int_{\Gamma}\frac{(x-y)\cdot\nu(y)}{|x-y|^3}\phi(y)ds(y)
\]
is just the statement of theorem \ref{Wavelets_them_2}. Whereas $(\mathcal{K}_5\phi)$ is continuous on $\{x:x_3>0\}$. The proof is complete.

\end{proof}

%-----------------------------------------------------------------
%
%-----------------------------------------------------------------
%We finish this section by proving that the single- and double-layer
%potential operators depend continuously on variations in the
%boundary $\Gamma$.  In the statement of the following theorem,
%$B=B(C_1,C_2)$ is the set defined in Remark \ref{remark1}, for some
%constants $C_2>C_1>0$. We use the notation $A_f$ for either $S$ or
%$K$ defined on a surface $\Gamma_f$ given by some $f\in B$. With the
%help of the isomorphism
%\begin{equation} \label{iso}
%  I_f: L^2(\Gamma_f) \to L^2(\real^2), \quad
%  (I_f\varphi)(\by) = \varphi( (\by,f(\by))),
%  \quad \by \in \real^2,
%\end{equation}
%we associate  $A_f$ with  the element $\tilde A_f= I_f A_f I_f^{-1}$
%of the set of bounded linear operators on $L^2(\real^2)$ for each $f
%\in B$. Denoting the kernel of $\tilde A_f$ by $a_f$, we see that,
%where $x=(\bx,f(\bx))$, $y=(\by,f(\by))$, and $a(x,y):=G(x,y)$ or
%$a(x,y):=\partial G(x,y)/\partial \nu(y)$, in the respective cases
%$A_f=S$ and $A_f=K$, it holds that
%$$
%a_f(\bx,\by) = a(x,y)J_f(\by), \quad J_f(\by) := \sqrt{1+|\nabla
%f(\by)|^2}.
%$$
We next establish that the solution we construct via the combined layer potential satisfies the `$v_T'$' boundedness condition. 
\begin{lemma}\label{u_prime_lemma}
For $x \in D$, let $v(x)= u_2(x) - i\eta u_1(x)$ where $u_1$ and $u_2$ denote the single- and
double-layer potentials with density
  $\phi\in L^2(\Gamma)$ defined by (\ref{SLP}) and (\ref{DLP}) respectively.
  %Let $u(x)$ be given by   for $x\in D$. T
  Then for all $T\geq f_+$ the function $v_T' \in L^2(\Gamma).$
\end{lemma}
\begin{proof}
We first of all consider the global part of the layer potentials (c.f.\ the proof of lemma \ref{lem:global}). Fix $T\geq f_+$. For the single-layer potential, %using the expansion (\ref{Gdelta exp}), 
we define, for $\phi \in L^2(\Gamma)$, $x =(\tilde x,f(\tilde x))\in \Gamma$ and $T\geq t> f(\tilde x)$, \begin{eqnarray*}
p(\tilde x, t):& = & \int_{\Gamma} \chi(|(\tilde x- \tilde y|)G((\tilde x, t),y)\phi(y)ds(y). \end{eqnarray*}
%\\& = & t\int_{\Gamma} -\chi(|(\tilde x,t)-y|)\frac{1}{2\pi}\frac{i\kappa e^{i\kappa|\bx-\by|}}{|\bx-\by|^2}f(\tilde y)\phi(y)ds(y)
%+ \int_{\Gamma} \chi(|(\tilde x,t)-y|)l(\bx-\by)\phi(y)ds(y),
%\end{eqnarray*}
Using the expansion (\ref{Gdelta exp}), we may write $p(\tilde x,t)$ as 
\begin{eqnarray*}
p(\tilde x, t)=t\int_{\Gamma} -\frac{1}{2\pi}\frac{ik e^{ik|\tilde x-\tilde y|}}{1 + |\tilde x-\tilde y|^2}f(\tilde y)\phi(y)ds(y)
+ \int_{\Gamma} l(\tilde x-\tilde y)\phi(y)ds(y),
\end{eqnarray*}
where $l$ satisfies the bound, for some constant $C>0$ independent of $t$,
\begin{equation} \label{lbbound_prime}
|l(\tilde x,\tilde y)| \le C \tilde\ell(\tilde x-\tilde y), \quad \tilde x,\tilde y\in\real^2,
\end{equation}
where $\tilde\ell(\tilde y):= (1+|\tilde y|)^{-3}$, so that $\tilde\ell\in
L^1(\real^2)$. It now follows as in the proof of lemma \ref{lem:global} that $p'_T \in L^2(\Gamma)$: (\ref{Lnormbound}) and Lemma
\ref{lem:bound} applied to the integral operator
\[
t\int_{\Gamma} -\frac{1}{2\pi}\frac{ik e^{ik|\tilde x-\tilde y|}}{1+|\tilde x-\tilde y|^2}f(\tilde y)\phi(y)ds(y),
\]
and (\ref{Lnormbound2}) applied to the integral operator with kernel
$l$, show this to be true.%at $u'_{S,G}$ is a bounded operator on $L^2(\Gamma)$. 

We argue in a similar vain for the global part of the double-layer potential. %Using the expansion (\ref{Gdeltanu exp}), 
We define, for $\phi \in L^2(\Gamma)$, $x=(\tilde x,f(\tilde x)) \in \Gamma$ and $T\geq t> f(\tilde x),$
\begin{eqnarray*}
q(\tilde x, t):& = &\int_{\Gamma} \chi(|\tilde x-\tilde y|)\frac{\partial G((\tilde x,t),y)}{\partial \nu(y)}ds(y),
\end{eqnarray*}
and then again, using the expansion (\ref{Gdeltanu exp}) we rewrite this as 
\begin{eqnarray*}
q(\tilde x, t)
&& =  t\int_{\Gamma} \left[-\frac{1}{2\pi}k^2 \bnu(y)\cdot\frac{(\tilde x-\tilde y)}{|\tilde x-\tilde y|}\frac{e^{ik|\tilde x-\tilde y|}}{1+|\tilde x-\tilde y|^2}f(\tilde y) \right.\\&&\left.- \frac{ik\nu_3(y)}{2\pi}\frac{e^{ik|\tilde x-\tilde y|}}{1+ |\tilde x-\tilde y|^2}\right]\phi(y)ds(y)\\
&&+ \int_{\Gamma}l(\tilde x-\tilde y)ds(y),
\end{eqnarray*}
where $l$ satisfies (\ref{lbbound_prime}) for some constant $C>0$ independent of $ t  $. 
Again, it now follows as in the proof of lemma \ref{lem:global} that $q'_T \in L^2(\Gamma)$.
%: (\ref{Lnormbound}) and Lemma
%\ref{lem:bound} applied to the integral operator
%\[
%t\int_{\Gamma} -\frac{1}{2\pi}\frac{i\kappa e^{i\kappa|\bx-\by|}}{1+|\bx-\by|^2}f(\tilde y)\phi(y)ds(y),
%\]
%and (\ref{Lnormbound2}) applied to the integral operator with kernel
%$l$, show that $u'_{S,G}$ is a bounded operator on $L^2(\Gamma)$. 

%\begin{equation} \label{lbbound??}
%|l(\bx,\by)| \le C \tilde\ell(\bx-\by), \quad \bx,\by\in\real^2,
%\end{equation}
%where $\tilde\ell(\by):= (1+|\by|)^{-3}$, so that $\tilde\ell\in
%L^1(\real^2)$.
 % \frac{i \kappa \; e^{i \kappa |\bx-\by|}}{|\bx-\by|^2}
 % \; 2 h s \; + \;
 % O \left(
 %     \frac{1}{|\bx-\by|^3}
We now turn our attention to the local part of the layer potentials (c.f.\ the proof of lemma \ref{lem:local}). For the single layer potential we wish to consider, for $\phi \in L^2(\Gamma)$, $x=(\tilde x,f(\tilde x)) \in \Gamma$ and $T\geq t> f(\tilde x)$,
\begin{eqnarray*}
s(\tilde x, t):= \int_{\Gamma}[1- \chi(|\tilde x-\tilde y|)]G((\tilde x, t), y)ds(y).
\end{eqnarray*}
For $x \in \Gamma$, making use of the inequality (\ref{Geom_sig}), a simple application of the triangle inequality shows that for any $z \in \Theta(x)$ and for all $y \in \Gamma$ 
\begin{equation} \label{Cone_Geom}
|x-y|\leq [\alpha +1] |z-y|.
\end{equation} 
It follows,
exactly as in the proof of lemma \ref{lem:local}, that we have that for some constant
$C>0$,
\begin{equation} \label{a2bound_prime}
|[1-\chi(|\tilde x-\tilde y|)]G((\tilde x, t),y)| \le C\ell(\tilde x-\tilde y), \quad x,y\in\Gamma, \; x\neq y,
\end{equation}
where
\begin{equation}
\ell(\tilde y) := \left\{\begin{array}{cc}
               |\tilde y|^{-1}, & |\tilde y|\leq 1, \\
               0, & |\tilde y|>1. \\
             \end{array}\right.
\label{kernel A2_prime}
\end{equation}
%Now, recalling that for any $t \in \Theta(x)$ there exists a constant $\alpha>0$ such that 
%$|x-y|\leq \alpha |t-y|$ for all $y \in \Gamma$, we now conclude that 
%\begin{equation} %\label{a2bound_prime}
%|\chi(|(\tilde x,t)-y|)G((\tilde, t),y)| \le C\ell(\bt-\by), \quad x,y\in\Gamma, \; x\neq y,
%\end{equation}
 %The same bound holds (but is not sharp) in the single-layer case
%$A=S$.
Since $\ell\in L^1(\real^2)$, we see
from (\ref{Lnormbound2}) that $s'_T \in L^2(\Gamma)$.

For the local part of the double layer operator, we define for $\phi \in L^2(\Gamma)$, $x=(\tilde x,f(\tilde x)) \in \Gamma$ and $T\geq t>f(\tilde x)$,
\[
w(\tilde x,t):= \int_{\Gamma} [1-\chi(|\tilde x-\tilde y|)]\frac{\partial G((\tilde x,t),y)}{\partial \nu(y)} \phi(y) ds(y). 
\]
Making use of (\ref{Gdeltapartialnu}), we see that
\[
4\pi \frac{\partial G((\tilde x,t),y)}{\partial \nu(y)} = \nu(y)\cdot ((\tilde x,t)-y)\frac{e^{ik|(\tilde x,t)-y|}}{|(\tilde x,t)-y|^3} + r((\tilde x,t),y),
\]
where, the local part of $r((\tilde x,t),y)$, $[1- \chi(|\tilde x-\tilde y|)]r((\tilde x,t),y) := r_2((\tilde x,t),y)$ is such that, for some constant $C>0$
\begin{equation} \label{a2bound!}
|r_2((\tilde x,t),y)| \le C\ell(\tilde x-\tilde y), \quad x,y\in\Gamma, \; x\neq y,
\end{equation}
where $\ell$ is given by (\ref{kernel A2_prime}). Here again we have made use of the inequality (\ref{Cone_Geom}). It follows as above, that if $R_2(\tilde x,t)$ denotes the layer potential with kernel $r_2((\tilde x,t),y)$ then ${R_2}'_T$ is bounded on $L^2(\Gamma)$. 

We now focus on 
%\begin{equation}
%u_{} 1- \chi(|(\tilde x,t)-y|) \nu(y)\cdot ((\tilde x,t)-y)\frac{e^{i\kappa|(\tilde x,t)-y|}}{|\tilde x-y|^3}.
%\end{equation}
the quantity $m(\tilde x,t)$, defined for $\phi \in L^2(\Gamma)$, $x=(\tilde x,f(\tilde x)) \in \Gamma$ and $T\geq t> f(\tilde x)$ by  
\begin{equation}
m(\tilde x,t) :=  \int_{\Gamma}[1- \chi(|\tilde x-\tilde y|)] \nu(y)\cdot ((\tilde x,t)-y)\frac{e^{ik|(\tilde x,t)-y|}}{|(\tilde x,t)-y|^3} \phi(y)ds(y).% \sqrt{1+ |\nabla_{\tilde y}f(\tilde y)|^2} d \tilde y.
\end{equation}
We proceed as in the proof of lemma \ref{lem:local}.
Firstly recall that 
$$ e^{ik|(\tilde x,t)-y|} = 1 + ik|(\tilde x,t)-y| + \frac{(ik|(\tilde x,t)-y|)^2}{2!} + \dots. 
$$ 
Then if
\begin{equation}
 m(\tilde x,t)= m_2(\tilde x,t) +  m_3(\tilde x,t),
\end{equation}
where 
\begin{equation}
m_2(\tilde x,t)= \int_{\Gamma}[1- \chi(|\tilde x-\tilde y|)] \frac{\nu(y)\cdot ((\tilde x,t)-y)}{|(\tilde x,t)-y|^3} \phi( y) ds(y),%\sqrt{1+ |\nabla_{\tilde y}f(\tilde y)|^2} d \tilde y,
\end{equation}
and where 
\begin{equation}
m_3(\tilde x,t) =  \int_{\Gamma}[1- \chi(|\tilde x-\tilde y|)] \nu(y)\cdot ((\tilde x,t)-y) \frac{e^{ik|x-y|}-1}{|(\tilde x,t)-y|^3} \phi( y) ds(y),%\sqrt{1+ |\nabla_{\tilde y}f(\tilde y)|^2} d \tilde y. 
\end{equation}
then since the kernel of $m_3$ is bounded by $C\ell(\tilde x - \tilde y)$, for some $C>0$ and with $\ell$ given by (\ref{kernel A2_prime}), it follows again, that ${m_3}_T'$ is bounded on $L^2(\Gamma)$. 

We rewrite $m_2$ as 
\begin{equation}
m_2(\tilde x,t)=m_4(\tilde x,t) - m_5(\tilde x,t)
\end{equation}
where
\begin{equation}
m_4(\tilde x,t)=  \int_{\Gamma} \frac{\nu(y)\cdot ((\tilde x,t)-y)}{|(\tilde x,t)-y|^3} \phi(y)ds(y),% \sqrt{1+ |\nabla_{\tilde y}f(\tilde y)|^2} d \tilde y, 
\end{equation}
and where  
\begin{equation} 
m_5(\tilde x,t)
= \int_{\Gamma}\chi(|\tilde x-\tilde y|) \frac{\nu(y)\cdot ((\tilde x,t)-y)}{|(\tilde x,t)-y|^3} \phi(y)ds(y).% \sqrt{1+ |\nabla_{\tilde y}f(\tilde y)|^2} d \tilde y.
\end{equation}
That ${m_4}'_T$ is in $L^2(\Gamma)$ follows from theorem \ref{Wavelets_them_2} which states that in fact $m_4^*$ is in $L^2(\Gamma)$. Thus, to complete the proof we need to show that ${m_5}_T'$ is also in $L^2(\Gamma)$.

To emphasize its dependence on $\phi$ we write $(m_5\phi)(\tilde x,t)$ instead of just $m_5(\tilde x,t)$. We begin by noting that ${m_5}'_T$, viewed as an operator acting on $\phi$, is bounded from $L^2(\Gamma)$ into $L^{\infty}(\Gamma)$. We use the bound (\ref{Cone_Geom}), and then, just as in the proof of lemma \ref{lem:local} we see that for all $x \in \Gamma$, a simple application of the Cauchy-Schwarz inequality yields that 
\begin{equation}
|{m_5}'_T(x)| \leq [\alpha +1]^{2}\mathcal{C} \Vert \phi\Vert_{L^2(\Gamma)},
\end{equation}
with $\mathcal{C}$ given by 
\begin{equation}
\mathcal{C}= L'\left\{\int_{G}\frac{1}{|\tilde x -\tilde y|^4}  d \tilde y \right\}^{\frac{1}{2}},
\end{equation}
where 
\[
G=\mathbb{R}^2 \backslash B_{1}(\tilde x)
\]
so that $C$ is finite and bounded independently of $\tilde x$, as one sees by changing the last integral to polar coordinates and evaluating it.   

Now, for each $n= (n_1,n_2) \in \mathbb{Z}^2$ we let $\Lambda_n$ be the indicator function such that if $\tilde x \in \mathbb{R}^2$ is such that $n_1\leq x_1<n_1 +1$ and such that $ n_2 \leq x_2 <n_2 +1$ then $\Lambda_n(\tilde x)=1 $ and which is $0$ otherwise. Then, letting $\phi_n :=\phi \Lambda_n$ for $\phi \in L^2(\Gamma)$ we have that 
\[
\phi= \sum_{n \in \mathbb{Z}^2} \phi_n. %\begin{equation}
\]%\ell(\by) := \left\{\begin{array}{cc}
%and by the continuity of $u_5$ as an operator from $L^2(\Gamma)$ into $L^{\infty}(\Gamma)$ that
%\[
%(u_5\phi)(\tilde x,f(\tilde x))= \sum_{n \in \mathbb{Z}^2} (u_5\phi_n)(\tilde x,f(\tilde x)).
%\]
Now, for $\tilde x \in \mathbb{R}^2$ we let $\mathcal{N}(\tilde x)$ be the set of those $n\in \mathbb{Z}^2$ such that 
$$
\mathrm{dist} (\tilde x,\supp(\phi_n)) <1.
$$
Note that $ \mathcal{N}(\tilde x)$ contains no more than $9$ elements, and also, that %if 
%$$\mathrm{dist} (\tilde x,\supp(\phi_m))\geq1 
%$$
\[
\left(m_5 \sum_{m \notin \mathcal{N}(\tilde x)}\phi_m\right)(\tilde x,t) = \left(m_4\sum_{m \notin \mathcal{N}(\tilde x)}\phi_m\right)(\tilde x,t).
\]
Thus, for $\tilde x \in \mathbb{R}^2$, $T\geq t > f(\tilde x)$, 
\begin{eqnarray*}
(m_5\phi)(\tilde x,t) %= \left(m_5\sum_{n \in\mathbb{Z}^2}\phi_n\right)(\tilde x,t)
& = &\sum_{n \in\mathcal{N}(\tilde x)}(m_5\phi_n)(\tilde x,t)+ \left(m_5\sum_{n \not\in\mathcal{N}(\tilde x)}\phi_n\right)(\tilde x,t)\\&=&\sum_{n \in\mathcal{N}(\tilde x)}(m_5\phi_n)(\tilde x,t)+ \left(m_4\sum_{n \not\in\mathcal{N}(\tilde x)}\phi_n\right)(\tilde x,t)\\&=& \sum_{n \in\mathcal{N}(\tilde x)}(m_5\phi_n)(\tilde x,t)+ (m_4\phi)(\tilde x,t)\\& -& \sum_{n \in\mathcal{N}(\tilde x)}(m_4\phi_n)(\tilde x,t).
\end{eqnarray*}
In what follows we use that for $a_j\geq 0, j=1,\dots, n$
\[
(a_1 + \dots  + a_n)^2 \leq n(a_1^2 + \dots +a_n^2),
\]
and that $J_f \leq L'$,
and we define for $m \in \mathbb{Z}^2$, 
$$
T(m):=\{n \in \mathbb{Z}^2:\dist(\supp\Lambda_m,\supp\Lambda_n)<1\}.
$$
So
\begin{eqnarray*}
&&\int_{\mathbb{R}^2}|({m_5}'_T\phi)(x)|^2 J_f(\tilde x) d \tilde x \\&&\leq 3\left\{  \int_{\mathbb{R}^2}\sup_{T\geq t>f(\tilde x)}\left|\sum_{n \in\mathcal{N}(\tilde x)}(m_5\phi_n)(\tilde x,t)\right|^2 J_f(\tilde x) d \tilde x + \Vert {m_4}'_T\phi\Vert^2_{L^2(\Gamma)}%\int_{\mathbb{R}^2}\sup_{T\geq t\geq f(\tilde x)}\left |\sum_{n \in\mathbb{Z}^2}(u_4\phi_n)(\tilde x,t)\right|^2 J_f(\tilde x) d \tilde x 
\right.\\&& 
+ \left. \int_{\mathbb{R}^2}\sup_{T\geq t> f(\tilde x)}\left |\sum_{n \in\mathcal{N}(\tilde x)}(m_4\phi_n)(\tilde x,t)\right|^2 J_f(\tilde x) d \tilde x\right\} \\&&= 3\left\{ \sum_{m \in \mathbb{Z}^2} \int_{\supp(\phi_m)}\sup_{T\geq t> f(\tilde x)}\left|\sum_{n \in\mathcal{N}(\tilde x)}(m_5\phi_n)(\tilde x,t)\right|^2 J_f(\tilde x) d \tilde x + \Vert {m_4}'_T\phi\Vert_{L^2(\Gamma)}^2 \right.\\&&\left.+\sum_{m \in \mathbb{Z}^2} \int_{\supp(\phi_m)}\sup_{T\geq t> f(\tilde x)}\left|\sum_{n \in\mathcal{N}(\tilde x)}(m_4\phi_n)(\tilde x,t)\right|^2 J_f(\tilde x) d \tilde x\right\}\\&& \leq 3\left\{ 9\sum_{m \in \mathbb{Z}^2} \int_{\supp(\phi_m)}\sum_{n \in T(m)}\sup_{T\geq t\geq f(\tilde x)}|(m_5\phi_n)(\tilde x,t)|^2 J_f(\tilde x) d \tilde x + \Vert {m_4}'_T\phi\Vert_{L^2(\Gamma)}^2 \right.\\&&\left.+9\sum_{m \in \mathbb{Z}^2} \int_{\supp(\phi_m)}\sum_{n \in T(m)}\sup_{T\geq t> f(\tilde x)}|(m_4\phi_n)(\tilde x,t)|^2 J_f(\tilde x) d \tilde x\right\}\\ &&\leq 3\left\{ 9\sum_{m \in \mathbb{Z}^2} [\alpha +1]^{4}\mathcal{C}^2L'\sum_{n \in T(m)}\Vert \phi_n\Vert_{L^2(\Gamma)}^2+ \Vert {m_4}'_T\phi\Vert_{L^2(\Gamma)}^2\right.\\ &&+ \left.9\sum_{m \in \mathbb{Z}^2}\sum_{n \in T(m)} \Vert{m_4}'_T\phi_n\Vert^2_{L^2(\Gamma)}\right\} \\&& \leq 3\left\{ 81[\alpha+1]^{4}\mathcal{C}^2L'\sum_{m \in \mathbb{Z}^2}\Vert \phi_m\Vert_{L^2(\Gamma)}^2+ \Vert {m_4}'_T\Vert^2\Vert\phi\Vert_{L^2(\Gamma)}^2 \right.\\&+& \left.9\sum_{m \in \mathbb{Z}^2}\sum_{n \in T(m)}\Vert {m_4}'_T\Vert^2 \Vert\phi_n\Vert_{L^2(\Gamma)}^2\right\} \\
&&\leq3[ 81[\alpha +1]^{4}\mathcal{C}^2L' + \Vert {m_4'}_T \Vert^2 + 81 \Vert {m_4'}_T\Vert^2] \Vert \phi\Vert^2_{L^2(\Gamma)}.
\end{eqnarray*}
The proof is complete.
\end{proof}
\begin{remark}
In the above proof we showed that for all $T\geq f_+$, there exists $C>0$ such that 
\begin{equation}\label{u'_bd}
\Vert v_T' \Vert_{L^2(\Gamma)} \leq C \Vert \phi\Vert_{L^2(\Gamma)},
\end{equation}
whenever $v(x)$ is given by 
\begin{equation} \label{the_ansatz}
v(x)= u_2(x) - i\eta u_1(x)
\end{equation}
 where $u_1$ and $u_2$ denote the single- and
double-layer potentials with density
  $\phi\in L^2(\Gamma)$, defined by (\ref{SLP}) and (\ref{DLP}), respectively . In fact if $v$ is given by (\ref{the_ansatz}) but with $k$ replaced by $k+ i\epsilon$ for $\epsilon \in [0,1]$ in (\ref{SLP}) and (\ref{DLP}), then, examining the proof of lemma \ref{u_prime_lemma} we see that the same bound holds with the same constant $C>0$. We'll make use of this fact in the next lemma. 
\end{remark}

We next show, by mimicking part of the proof of lemma 3.3 of \cite{chandpottheim2}, that the solution we construct via the combined layer-potential satisfies the radiation condition.
\begin{lemma}\label{rad_lemma}
Let $v(x)$ be given by (\ref{the_ansatz}) with $u_1(x)$ defined by (\ref{SLP}) and $u_2(x)$ defined by (\ref{DLP}). Then for all $H>f_+$, $v$ satisfies the radiation condition (\ref{uprcstar}) with $F_H=v|_{\Gamma_H}$ (and with $k_+$ replaced by $k$).
\end{lemma}
\begin{proof}
Fix $H > f_+$. Let $v(x)$ be given by (\ref{the_ansatz}) and define, for $x \in {U_H}$  
\begin{equation}\label{u_rad}
u(x):= \frac{1}{2\pi}\int_{\mathbb{R}^2}\exp(i[(x_3 - H) \sqrt{k^2 -\xi^2} + \tilde x\cdot \xi])\hat{\psi_H}(\xi)d\xi,
\end{equation}
where $\psi_H:=v|_{\Gamma_H}$, $\hat{\psi_H}$ denotes the Fourier transform of $\psi_H$ and where $\sqrt{k^2-\xi^2}=i\sqrt{\xi^2-k^2}$ for $|\xi|>k$. Note that, since $v_T'\in L^2(\Gamma)$ for any $T\geq f_+$ by lemma \ref{u_prime_lemma}, it follows that $\psi_H \in L^2(\Gamma_H)$ so that its Fourier transform is well-defined. %Note further that $\psi_H \in C^2(\Gamma_H)$ and that, by interior elliptic regularity estimates it follows that the second order partial derivatives of $u$ decay at least as rapidly as $|x|^-2$ on $\Gamma_H$. Thus $\psi_H \in H^2(\Gamma_H)$, so that by Cauchy- we see that $\hat{\psi_H}\in L^1(\mathbb{R}^2)$ with    
We'll show that $u=v$ in $U_H$.

 We note first that $v$ restricted to $\overline{U_H}$ satisfies the boundary value problem of \cite{chandpottheim} -- that we wrote down toward the end of section 2 of this chapter -- in the case that we set $\Gamma= \Gamma_H$, $g = \psi_H$, and define $g_{\epsilon}$ to be the restriction of $v$ to $\Gamma_H$ when $v$ is defined by (\ref{the_ansatz}) but with $k$ replaced by $k+i\epsilon$ in the definition of the Dirichlet Green's function $G$. (Note that $g_{\epsilon} \to g$ in $L^2(\Gamma_H)$ as $\epsilon \to 0$ because it is easy to show that $g_{\epsilon} \to g$ pointwise and then we can apply the dominated convergence theorem, both $g$ and $g_{\epsilon}$ being dominated by $v_T'$ for any $T\geq H$, and with $v_T'$ being bounded by (\ref{u'_bd}).) Indeed, that $g \in L^2(\Gamma_H)$ follows from lemma \ref{u_prime_lemma}. Also $g$ is continuous. In addition using the bounds 
\[
|G(x,y)| ,|\nabla_y G(x,y)| \leq C \frac{(1+x_3)(1+y_3)}{|x-y|^2}
\]
that follow from (\ref{Gbound}) and interior elliptic regularity estimates (c.f.\ the proof of lemma \ref{jr thm} (ii)), and noting that since $H-f_+:= \delta > 0$, $|x-y| >\delta$, for $x \in \overline{U_H}$ and $y \in \Gamma$, a simple application of the Cauchy-Schwarz inequality shows that $g$ is bounded, and so belongs to the space $X: =L^2(\Gamma)\cap BC(\Gamma)$ of \cite{chandpottheim}.  Moreover $v \in C^2(U_H) \cap C(\overline{U_H})$, satisfies the Helmholtz equation in $U_H$ and the Dirichlet boundary condition. That $v|_{\overline{U_H}}$ satisfies the bound (\ref{bound_on_v}) follows from lemma \ref{jr thm}(ii). Finally one can show that the limiting absorption principle holds; by applying the dominated convergence theorem for example, with $v_{\epsilon}$ being given by (\ref{the_ansatz}) but with $k$ replaced by $k+i\epsilon$ in the definition of the Dirichlet Green's function $G$. 

We next show that $u$ also satisfies this same boundary value problem. That $u \in C^2(U_H)\cap C(\overline{U_H})$ satisfies the Helmholtz equation and the Dirichlet boundary condition can be seen by harking back to lemma \ref{lemma3p2}. To show that $u$ satisfies the bound (\ref{bound_on_v}) it is sufficient to show that $\hat{\psi_H} \in L^1(\mathbb{R}^2)$. For this we note that $\psi_H \in C^2(\Gamma_H)$ and that by interior elliptic regularity estimates for solutions of the Helmholtz equation (e.g.\ \cite[Lemma 2.7]{chandprs98}) (c.f.\ the proof of lemma \ref{jr thm}(ii)) the second order partial derivatives of $v$ decay at least as rapidly as $|x|^{-2}$. Thus $\psi_H \in H^2(\Gamma_H)$, so that, by the Cauchy-Schwarz inequality
\begin{eqnarray*}
\left\{\int_{\mathbb{R}^2}|\hat{\psi_H}(\xi)|d\xi\right\}^2 &&\leq \int_{\mathbb{R}^2}(1+\xi^2)^{-2}d\xi\int_{\mathbb{R}^2}|\hat{\psi_H}(\xi )|^2(1+\xi^2)^2d\xi\\&& = \left[\int_{\mathbb{R}^2}(1+\xi^2)^{-2}d\xi\right]\Vert\psi_H\Vert_{H^2(\Gamma_H)}^2.
\end{eqnarray*}    
Finally it's easy to show that $u$ satisfies the limiting absorption principle with $u_{\epsilon}$ being defined by (\ref{u_rad}) but with $\psi_{H}= v_{\epsilon}|_{\Gamma_H}.$ 

It follows now by theorem \ref{chandpottheim_result1} that $u=v$ in $U_H$.
\end{proof}

\section{Uniqueness and existence results} \label{sec:exist}

In this section we prove uniqueness and existence for our integral equation
formulation and for the boundary value problem. %As the first step in this

We begin by proving uniqueness of solution for the boundary value problem. 
%Due to \cite{Cha95a},

\begin{lemma}\label{unique}
The Boundary Value problem has at most one solution.
\end{lemma}
\begin{proof}
Let $v$ satisfy the boundary value problem and let $f:\mathbb{R}^2 \to \mathbb{R}$ denote the Lipschitz function with Lipschitz constant $L$ that defines $\Gamma$ (see (\ref{Gamma_def_chap5})). By lemma \ref{smooth_Lip} there exists a sequence of functions, $f_n:\mathbb{R}^2 \to \mathbb{R}$, such that each $f_n$ is Lyapunov, such that each $f_n$ is Lipschitz with Lipschitz constant $L$, such that the $f_n$ converge to $f$ in the norm $\Vert \cdot \Vert_{L^{\infty}(\mathbb{R}^2)}$, such that $f_n\geq f + \epsilon_n$ for some $\epsilon_n >0$, and such that the $f_n$ are decreasing. %In addition we may assume that $f_n <f_+$ for all $n \in \mathbb{N}$.
Let $\Gamma_n:=\{(\tilde x,f_n(\tilde x)):\tilde x \in \mathbb{R}^2\}$ and let $D_n:=\{(\tilde x,x_3):\tilde x \in \mathbb{R}^2, x_3>f_n(\tilde x)\}\subset D$. Let $v_n:= v|_{\Gamma_n}$, so that $v_n \in L^2(\Gamma_n)$ because
\[
\int_{\mathbb{R}^2}|v_n(\tilde x,f_n(\tilde x))|^2 \sqrt{1+|\nabla f_n(\tilde x)|^2}d\tilde x \leq  L' \int_{\mathbb{R}^2}|v_T'(\tilde x,f(\tilde x))|^2 \sqrt{1+|\nabla f(\tilde x)|^2}d\tilde x,
\]
for any $T>f_n$. Note that $v_n$ is also continuous and (c.f. the proof of lemma \ref{rad_lemma}) is bounded. Since $\Gamma_n$ is a Lyapunov curve, it follows by theorem \ref{chandpottheim_result2}, that there exists $\phi_n \in L^2(\Gamma_n)\cap BC(\Gamma_n)$ such that
\[
\phi_n = (I +K_{f_n} - i\eta S_{f_n})^{-1}v_n.
\]
Now, for $x \in D_n$ let 
\[
u(x)= \int_{\Gamma_n}\frac{\partial G(x,y)}{\partial \nu(y)}\phi_n(y)ds(y) -i\eta\int_{\Gamma_n} G(x,y)\phi_n(y)ds(y).
\]
Note that by theorem 5.5 of \cite{chandpottheim} $u \in C^2(D_n)\cap C(\overline{D_n})$, 
\[
\Delta u +k^2u= 0, \quad \mbox{ in } D_n
\]
and that $u|_{\Gamma_n}= v_n$. Further, by lemma \ref{rad_lemma} $u$ satisfies the radiation condition (\ref{uprcstar}) for $H>\Vert f_n\Vert_{L^{\infty}(\mathbb{R}^2)}$ and by lemma \ref{u_prime_lemma} $u_T' \in L^2(\Gamma_n)$ for all $T\geq \Vert f_n\Vert_{L^{\infty}(\mathbb{R}^2)}$. We now show that $u=v$ in $D_n$. To do this we'll show that in $D_n$, $w:=u-v$ satisfies the boundary value problem of chapter 2 in the case that $g=0$ and $k$ is constant; theorem \ref{th_main1} will then imply that $w=0$ in $D_n$. 

It's clear that $w \in C^2(D_n)\cap C(\overline{D_n})$ satisfies the homogeneous Helmholtz equation in $D_n$, that $w=0$ on $\Gamma_n$ and also that $w$ satisfies the radiation condition (\ref{uprcstar}) for $H\geq \Vert f_n\Vert_{L^{\infty}(\mathbb{R}^2)}$. Moreover for $T>H>\Vert f_n\Vert_{L^{\infty}(\mathbb{R}^2)}$ and where $S_{H_n}= D_n\backslash \overline{U_H}$ we have that
\begin{eqnarray*}
\int_{S_{H_n}}|v(x)|^2dx &=& \int_{\mathbb{R}^2}\int_{f_n(\tilde x)}^H|v(\tilde x,x_3)|^2dx_3d\tilde x
\\&\leq&\int_{\mathbb{R}^2}(H-f_n(\tilde x))|v_T'(\tilde x)|^2d\tilde x\leq (H-f_-)\Vert v_T'\Vert^2_{L^2(\Gamma)}.
\end{eqnarray*}
Thus $v \in L^2(S_{H_n})$ for all $H> \Vert f_n \Vert_{L^{\infty}(\mathbb{R}^2)}$, and a similar calculation shows that the same is true of $u$; thus it's also true of $w.$ We now show that since $w=0$ on $\Gamma_n$ it follows that $w$ is in the space $V_{H_n}$ as defined in chapter 2. 

We first of all observe that $w \in C^1(\overline{D_n})$ by adapting the proof of theorem 3.27.\ of \cite{Col83a}. We fix an arbitrary point $x$ of the boundary $\Gamma_n$, and we let $\Omega \subset D_n$ be a small neighbourhood of the boundary about $x$ such that it's boundary, $\partial \Omega$, is smooth and coincides with $\Gamma_n$ in such a way that $x$ is an interior point of $\partial \Omega$. 

We now solve the problem find $w^*:\Omega \to \mathbb{C}$ such that 
$w^* \in C^2(\Omega)$ 
\[
\Delta w^*+k^2w^*=0, \mbox{ in } \Omega,
\]
and such that $w^*=w$ on $\partial \Omega$
by letting $w^*$ take the form 
\[
w^*(x) = \int_{\partial \Omega}\frac{\partial \Phi(x,y)}{\partial \nu(y)}\phi(y), \quad x \in \Omega 
\]
where $\phi = (I+K_B)^{-1}w$, where $K_B$ is given by (\ref{K_B_def}) and where $I+K_B$ is invertible as a map on $C(\partial \Omega)$ because $\partial \Omega$ is smooth. Provided we choose $\Omega$ small enough then it follows by the uniqueness lemma, lemma 3.26.\ of \cite{Col83a}, for solutions of the Helmholtz equation in domains with small diameter that $w=w^*$ in $\Omega$. Further we note that the boundary data $w$ of the above problem can be decomposed as the sum $w=w_1 + w_2$ where $w_1$ is smooth, where $w_1 =w$ in a compact subset of $\Gamma_n \cap \partial \Omega$ that contains $x$ and where $w_1$ vanishes outside a yet larger compact subset of $\Gamma_n \cap \partial \Omega$ containing $x$. If we then set
\[
\phi_1 =(I+K_B)^{-1}w_1
\]
and 
\[
\phi_2 =(I+K_B)^{-1}w_2
\]
then it follows by theorem 2.30 of \cite{Col83a} that
$\phi_1 \in C^{1,\alpha}(\partial \Omega)$ for $\alpha \in (0,1)$ and that
$\phi_2$ is zero in a compact neighbourhood of $x.$ 
These facts are then enough to conclude, with the help of theorem 2.23.\ of \cite{Col83a}, that  $w \in C^1(\overline{\Omega})$. 

We now show that $w \in H^1(S_{H_n})$ for $H>\Vert f_n\Vert_{L^{\infty}(\mathbb{R}^2)}$.
For $R>0$ we let $\theta_R \in C^{\infty}(\mathbb{R}^2)$ be such that:
$\theta_R(\tilde x) = 1$ if $\tilde x \in B_R(0)$, $\theta_R(\tilde x)=0 $ if $\tilde x \notin B_{R+1}(0)$ and such that $0\leq \theta_R\leq 1$. We let $\alpha \in C^{\infty}(\mathbb{R})$ be such that for $x_3 \in \mathbb{R}$, $\alpha(x_3)=1$, if $x_3<H$ such that $\alpha(x_3)= 0$ if $x_3>H+1$ and such that $0\leq \alpha \leq 1$. We also choose $\theta_R$ and $\alpha$ so that $|\nabla\theta_R \alpha|\leq C$, for some $C>0$ independent of $R$. We then apply Green's theorem to the functions $w, \theta_R\alpha w \in C^1(\overline{D_n})$ to get that 
\begin{eqnarray*}
\int_{D_n\backslash\overline{U_{H+1}}}\theta_R\alpha \bar w\Delta w+ \nabla({\theta_R\alpha \bar w})\cdot\nabla w \;dx =0, 
\end{eqnarray*}
so that 
\begin{eqnarray*}
\int_{D_n\backslash\overline{U_{H+1}}}-\theta_R\alpha k^2|w|^2+ \theta_R\alpha|\nabla w|^2 + \bar{w}\nabla({\theta_R\alpha })\cdot\nabla w \;dx =0. 
\end{eqnarray*}
Hence 
\begin{eqnarray*}
\int_{D_n\backslash\overline{U_{H+1}}}\theta_R\alpha|\nabla w|^2 \;dx\leq \int_{D_n\backslash\overline{U_{H+1}}}\theta_R\alpha k^2|w|^2 \;dx + \int_{D_n\backslash\overline{U_{H+1}}}|w||\nabla w|Cdx,
\end{eqnarray*}
so that
\begin{eqnarray*}
\int_{D_n\backslash\overline{U_{H+1}}}\theta_R\alpha|\nabla w|^2 \;dx &-&\int_{D_n\backslash\overline{U_{H+1}}}\frac{|\nabla w|^2}{2} \;dx\\&\leq& \int_{D_n\backslash\overline{U_{H+1}}}\left[\theta_R\alpha k^2+\frac{C^2}{2}\right]|w|^2 \;dx.
% + \int_{D_n\backslash\overline{U_{H+1}}}|w||\nabla w|C\;dx,
\end{eqnarray*}
Now let $R\to \infty$. Using the monotone convergence theorem we get that 
\begin{eqnarray*} 
\frac{1}{2}\int_{D_n \backslash \overline{U_H}}|\nabla w|^2 dx\leq \int_{D_n\backslash\overline{U_{H+1}}}\left[k^2+\frac{C^2}{2}\right]|w|^2 dx.
\end{eqnarray*}
Thus $w \in H^1(S_{H_n})$ and since $w=0$ on $\Gamma_n$, we see that $w \in V_{H_n}$. 
 
Having shown that $w$ is a solution of the boundary value problem of chapter 2 in the case that $g=0$ and $k=0$ we conclude that $w=0$, i.e.\ $u=v$ in $D_n$. Thus for $x \in D_n$ we have obtained the representation 
\begin{eqnarray*}     
v(x)= \int_{\Gamma_n}\frac{\partial G(x,y)}{\partial \nu(y)}\phi_n(y)ds(y) -i\eta\int_{\Gamma_n} G(x,y)\phi_n(y)ds(y).
\end{eqnarray*} 
Now fix $x \in D$. There exists $N \in \mathbb{N}$ such that for all $n\geq N$ the above representation is valid and also $|x_3 -f_n(\tilde x)| >\epsilon$, for some $\epsilon >0$. Thus applying theorem \ref{jr thm}(ii) and then theorem \ref{chandpottheim_result2} 
we have that
\[
|v(x)| \leq C_{\epsilon}\Vert \phi_n\Vert_{L^2(\Gamma_n)} \leq C_{\epsilon}\Vert (I + K_{f_n} -i\eta S_{f_n})^{-1}\Vert\Vert v_n\Vert_{L^2(\Gamma_n)}\leq C_{\epsilon}B\Vert v_n\Vert_{L^2(\Gamma_n)}.
\]
Since $v_n \to 0$ pointwise for almost all $\tilde x \in \mathbb{R}^2$ as $n \to \infty$ and $|v_n(\tilde x,f_n(\tilde x))|\leq |v_T'(\tilde x,f(\tilde x))|$ for $\tilde x \in \mathbb{R}^2$ and $T \geq\Vert f_N\Vert_{L^{\infty}(\mathbb{R}^2)}$ with $v_T' \in L^2(\Gamma)$, we use the dominated convergence theorem to deduce that $\Vert v_n \Vert_{L^2(\Gamma_n)} \to 0$ as $n \to\infty$. Thus $v(x)=0$ for all $x \in D$.
\end{proof}
 
 %Theorem 1, (see also \cite{Ros96}, Theorem 3.1), a solution $u \in
%C^2(G) \cap C(\overline{G})$ to the Helmholtz equation (\ref{BVP1})
%with $\Im(\kappa) > 0$ on an open set $G \subset \real^{n}$ which
%satisfies the growth condition $|u(x)| \leq C e^{\theta |x|}$, with
%some constant $\theta < \Im(\kappa)$, and the boundary condition
%$u(x)=0$ for $x \in
%\partial G$ will vanish identically on $G$. This result directly implies
%uniqueness for the scattering problem and the boundary value problem
%for $\kappa_1>0$. For $\kappa_1=0$ uniqueness is a consequence of
%the limiting absorption principle we require, i.e. of the
%convergence (\ref{BVP4}).

Next we turn to establishing
existence of solution. We will need the following lemma whose proof is fairly routine but lengthy. To facilitate the proof we introduce, for a given bounded Lipschitz function $f$, the isomorphism
\[
I_f:L^2(\Gamma) \to L^2(\mathbb{R}^2), \quad (I_f \phi)(\tilde y) = \phi((\tilde y,f(\tilde y))), \quad \tilde y \in \mathbb{R}^2. 
\]
We then associate $S_f$ with the element $\tilde {S_f}=I_fS_fI_f^{-1}$ of the set of bounded linear operators on $L^2(\mathbb{R}^2)$. Denoting the kernel of $\tilde {S_f}$ by $s_f$ we see that, where $x=(\tilde x,f(\tilde x))$ and $y=(\tilde y,f(\tilde y))$, it holds that 
\[
s_f(\tilde x, \tilde y)= G(x,y)J_f(\tilde y).
\]
An analogous statement is true of $\tilde {K_f}=I_f K_f I_{f}^{-1}$.
\begin{lemma}\label{contthe}
Let $f:\mathbb{R}^2 \to \mathbb{R}$ be a bounded Lipschitz function with Lipschitz constant $L$ and for $n \in \mathbb{N}$ let $f_n:\mathbb{R}^2 \to\mathbb{R}$ be a sequence of smooth Lipschitz functions also with Lipschitz constant $L$, such that $f_n \to f$ in $L^{\infty}(\mathbb{R}^2)$ and such that $\nabla_{\tilde x}f_n \to  \nabla_{\tilde x}f$ in $L^p(K)$ for compact $K \subseteq \mathbb{R}^2$, with $1<p<\infty$. 
%and such that $\nabla_{\tilde x}f_n(\tilde x) \to  \nabla_{\tilde x}f(\tilde x)$ for almost all $\tilde x \in \mathbb{R}^2$.
Let $\Gamma_n:=\{(\tilde x,f_n(\tilde x)): \tilde x \in \mathbb{R}^2\}$. Then for all $\phi \in L^2(\mathbb{R}^2)$%  The single- and double-layer potential operators depend continuously o

i)\begin{equation} \label{strong_conv}
\lim_{n \to \infty}\Vert \tilde A_{f_n}\phi - \tilde A_f\phi\Vert_{L^2(\mathbb{R}^2)}= 0,
\end{equation}
and
ii)\[
\Vert A_{f}I_{f}^{-1}\phi\Vert_{L^2(\Gamma)}= \lim_{n \to \infty}\Vert A_{f_{n}}I_{f_n}^{-1}\phi\Vert_{L^2(\Gamma_n)}.
\]%  boundary $\Gamma_f$ of the unbounded domain $D_f$ in the sense that
%  \begin{equation}\label{normest}
%    \sup_{\substack{f,g \in B \\ \norm{f-g}_{BC^{1,\alpha}(\real^2)} \leq \epsilon}}
%    \norm{\tilde A_f - \tilde A_g}_{L^2(\real^2)\to L^2(\real^2)} \to 0, \quad \epsilon \to 0.
%  \end{equation}
\end{lemma}

\begin{proof}
%We will first prove that 
%\begin{equation} \label{strong_conv}
%\lim_{n \to \infty}\Vert \tilde A_{f_n}\phi - \tilde A_f\phi\Vert_{L^2(\mathbb{R}^2)}= 0. 
%\end{equation}
To prove i) we first of all remark that, since the operators $\tilde A_{f_n}$ and $\tilde A_f$ are uniformly bounded by a constant $C$ say, see remark \ref{uniform_bound}, it holds for $\phi_k \in C^{\infty}_0(\mathbb{R}^2)$ that
\begin{eqnarray*}
\Vert \tilde A_{f_n}\phi - \tilde A_f\phi\Vert_{L^2(\mathbb{R}^2)}& \leq& \Vert (\tilde A_{f_n} - \tilde A_f)\phi_k\Vert_{L^2(\mathbb{R}^2)}+ \Vert (\tilde A_{f_n} - \tilde A_f)(\phi -\phi_k)\Vert_{L^2(\mathbb{R}^2)}\\&& \Vert (\tilde A_{f_n} - \tilde A_f)\phi_k\Vert_{L^2(\mathbb{R}^2)}+ 2C\Vert\phi -\phi_k\Vert_{L^2(\mathbb{R}^2)}.
\end{eqnarray*}
This calculation shows that we need only establish (\ref{strong_conv}) in the case that $\phi$ is a smooth function with compact support.

Similarly to how we proceeded when proving Theorem \ref{btheo}, we
decompose the operator $\tilde A_{f}-\tilde A_{f_n}$ into a global and a
local part, i.e.\ $\tilde A_f- \tilde A_{f_n} = A_1 + A_2$ with
$A_1,A_2$ defined similarly to (\ref{A1op}) and (\ref{A2op}) except that here we will employ a smooth cut of function, $\chi_c:[0,\infty) \to \mathbb{R}$ defined by 
\begin{equation}\label{eq:chi_c_2}
  \chi_c(t) :=
    \begin{cases}
      0, & t < 1/2\\
      1, & t \geq 1.
    \end{cases}
 \qquad \text{and} \qquad
  0 \leq \chi_c(t) \leq 1,\quad \forall t \geq 0.
\end{equation}

%As before, we denote the kernels of $A_1$ and $A_2$ by $a_1$ and
%$a_2$, so that $a=a_1+a_2$. Correspondingly we split $\tilde A_f$ as
%$\tilde A_f = \tilde A_{f,1}+\tilde A_{f,2}$ and $\tilde a_f$ as
%$\tilde a_f = \tilde a_{f,1}+\tilde a_{f,2}$.
% We
%now carry out the proof for the case of the single-layer operator.
%The necessary changes for the double-layer operator are
%straightforward.

% THE GLOBAL PART %%%%%%%%%%%%%%%%%%%%%%%%%%%%%%%%%%%%%%%%%%%%%%
\textbf{The global operator}. The kernel of the global operator $A_1$
is given by
\begin{equation}\label{eq:dependence boundary}
  a_1(\tilde x,\tilde y):=\chi_c(|\tilde x-\tilde y|)[a_f(\tilde x,\tilde y)-a_{f_n}(\tilde x,\tilde y)].
\end{equation}
We look at the double-layer case only, the single-layer case is simpler. We let $x_f= (\tilde x, f(\tilde x))$, $x_{f_n}=(\tilde x,f_n(\tilde x))$ and write $\nu$ as $\nu_{f}$ or $\nu_{f_n}$ to indicate its dependence on $f$ or $f_n$ respectively.
We need to examine the integral operator with kernel
\begin{eqnarray}\nonumber
&&\chi_c(|\tilde x-\tilde y|)\left\{\nu_{f}(\tilde y)\cdot \nabla_{y}G(x_f,y_f)J_f(\tilde y) - 
\nu_{f_n}(\tilde y)\cdot \nabla_{y}G(x_{f_n},y_{f_n})J_{f_n}(\tilde y)\right\}\\\nonumber&&=
\chi_c(|\tilde x-\tilde y|)(\nu_{f}(\tilde y)-\nu_{f_n}(\tilde y))\cdot \nabla_{y}G(x_f,y_f)J_f(\tilde y)\\ \nonumber&&+
\chi_c(|\tilde x-\tilde y|)\nu_{f_n}(\tilde y)\cdot [\nabla_yG(x_f,y_f)-\nabla_{y}G(x_{f_n},y_{f_n})]J_{f}(\tilde y) \\&&+\chi_c(|\tilde x-\tilde y|)\nu_{f_n}(\tilde y)\cdot \nabla_{y}G(x_{f_n},y_{f_n})[J_{f}(\tilde y)-J_{f_n}(\tilde y)].\label{double_global}
\end{eqnarray}
To deal with the first term of (\ref{double_global}) we note that from (\ref{Gdeltanu exp}) there exists $C>0$ such that 
\[
\chi_c(|\tilde x -\tilde y|)|\nabla_yG(x_f,y_f)| \leq C|\tilde x-\tilde y|^{-2}
\]
for $\tilde x,\tilde y \in \mathbb{R}^2$ and then we use (\ref{youngs}) with $s=2$, $p=2$ and $r=1$, and then finally use that 
\[
\left\Vert \left(\frac{\partial f}{\partial y_i}- \frac{\partial f_n}{\partial y_i}\right)J_f\phi\right\Vert_{L^1(\mathbb{R}^2)}\leq \left\Vert \frac{\partial f}{\partial y_i}- \frac{\partial f_n}{\partial y_i}\right\Vert_{L^2(\supp \phi)}\Vert J_f\phi\Vert_{L^2(\mathbb{R}^2)},
\]
for $i=1,2$.
The third term of (\ref{double_global}) is dealt with in a similar manner.

%For the second term of (\rf{double_global}) 

%In the single-layer case we use the expansion (\ref{a1split}) and equation (\ref{la}),
%denoting $l$ by $l_f$ or $l_{f_n}$ to indicate its dependence on $f$ or $f_n$ respectively. We obtain
%\begin{eqnarray} \label{expand}
%  a_1(\tilde x,\tilde y) & = &
%  -\frac{ik}{2\pi} \frac{e^{i k |\tilde x-\tilde y|}}{1+|\tilde x-\tilde y|^2}
%  \;  \Big\{ f(\tilde x)[f(\tilde y)-{f_n}(\tilde y)] + [f(\tilde x)-{f_n}(\tilde %x)]{f_n}(\tilde y) \Big\} J_f(\tilde y)\nonumber \\
% && \nonumber \\
%  && + \Big(l_f(\tilde x,\tilde y)-l_{f_n}(\tilde x,\tilde y)\Big) J_f(\tilde y) \nonumber \\
%  &&+ \Big(-\frac{ik {f_n}(\tilde x){f_n}(\tilde y)}{2\pi}
%\frac{e^{ik|\tilde x-\tilde y|}}{1+|\tilde x-\tilde y|^2}
%+l_{f_n}(\tilde x,\tilde y)\Big)(J_f(\tilde y)-J_{f_n}(\tilde y)),
%\end{eqnarray}
%for $\tilde x,\tilde y\in\real^2$, $\tilde x\neq \tilde y$.

%The integral operator whose kernel is the first term of
%(\ref{expand}) can be bounded using Lemma \ref{lem:bound} and
%(\ref{Lnormbound}). %Similarly, the integral operator whose kernel is
%the last term of (\ref{expand}) can be bounded using Lemma
%\ref{lem:bound}, (\ref{Lnormbound}), (\ref{lbbound}) and
%(\ref{Lnormbound2}), noting that Remark \ref{remark1} guarantees the
%uniformity of (\ref{lbbound}) for $f\in B$. 
To bound the integral
operator whose kernel is the second term of (\ref{double_global}), we
construct, for every $\eta\in (0,1)$, a function $\ell_\eta\in
L^2(\real^2)$ such that
\begin{equation} \label{bbb}
\left|\chi_c(\tilde x-\tilde y)[\nabla_yG(x_f,y_f)-\nabla_{y}G(x_{f_n},y_{f_n})]\right|\leq
\ell_\eta(\tilde x-\tilde y), \quad \tilde x,\tilde y\in\real^2,
\end{equation}
whenever $\norm{f-f_n}_{L^{\infty}(\real^2)}$ is
sufficiently small, and such that $||\ell_\eta||_{L^2(\real^2)}\to
0$ as $\eta\to 0$, and then we use the estimate (\ref{youngs}) with $s=2$, $p=2$ and $r=1$.
%Together, the bounds on the three parts of $A_1$ show
%(\ref{normest}) for the global part of the operator.

The construction of $\ell_\eta$ is as follows: %We choose a constant $C>0$ so that (\ref{lbbound})
%holds for all $f_n$ and for $f$ (this is possible -- see the expansions (\ref{Gdelta exp}) and (\ref{Gdeltanu exp})). We note that $||J_{f}||_{L^\infty(\real^2)}\leq L'$. Then,
%where $\tilde \ell\in L^1(\real^2)$ is defined as in Lemma
%\ref{lem:global}, 
we set
$$ \ell_\eta(\tilde y) := \left\{\begin{array}{cc}
                    \eta & 1/2 < |\tilde y| < \eta^{-1}, \\
                    0     &      |\tilde y| < 1/2,           \\
                                        2C|\tilde y|^{-2} & \mbox{otherwise.} \\
                  \end{array}\right.
$$
Clearly this satisfies that $||\ell_\eta||_{L^2(\real^2)}\to 0$ as
$\eta\to 0$. Since, for every $\eta\in (0,1)$, $
\left|\nabla_yG(x_f,y_f)-\nabla_{y}G(x_{f_n},y_{f_n})\right|\to 0 $ as
$\norm{f-{f_n}}_{L^{\infty}(\real^2)}\to 0$, and uniformly so in $\tilde x$ and $\tilde y$ for $ 1/2 \leq
|\tilde x-\tilde y| \leq \eta^{-1}$, the bound (\ref{bbb}) holds.

\textbf{The local operator.}
The kernel of the local operator $A_2$
is given by
\begin{equation}\label{eq:dependence boundary?}
  a_2(\tilde x,\tilde y):=(1-\chi_c(|\tilde x-\tilde y|))[a_f(\tilde x,\tilde y)-a_{f_n}(\tilde x,\tilde y)].
\end{equation}
%In the single-layer case we make the abreviations $x_f:=(\tilde x, f(\tilde x))$, $y_f:=(\tilde y, f(\tilde y))$ etcetera and then 
In the single-layer case we see that
\begin{eqnarray}\label{SL_dep_exp}
\nonumber a_2(\tilde x,\tilde y)&=& [1-\chi_c(|\tilde x-\tilde y|)]\frac{1}{4\pi}\left\{\frac{e^{ik|x_f-y_f|}}{|x_f- y_f|}- \frac{e^{ik|x_{f_n}-y_{f_n}|}}{|x_{f_n}- y_{f_n})}\right\}J_f(\tilde y)\\\nonumber & -&   [1-\chi_c(|\tilde x-\tilde y|)]\frac{1}{4\pi}\left\{\frac{e^{ik|x_{f_n}-y_{f_n}|}}{|x_{f_n}- y_{f_n}|}\right\}[J_{f_n}(\tilde y) - J_f(\tilde y)]\\\nonumber &-& [1-\chi_c(|\tilde x-\tilde y|)]\frac{1}{4\pi}\left\{\frac{e^{ik|x_f-y'_f|}}{|x_f- y'_f|}- \frac{e^{ik|x_{f_n}-y'_{f_n}|}}{|x_{f_n}- y'_{f_n})}\right\}J_f(\tilde y)\\ & +&   [1-\chi_c(|\tilde x-\tilde y|)]\frac{1}{4\pi}\left\{\frac{e^{ik|x_{f_n}-y'_{f_n}|}}{|x_{f_n}- y'_{f_n}|}\right\}[J_{f_n}(\tilde y) - J_f(\tilde y)].
\end{eqnarray}
For the integral operator whose kernel is given by the first (and similarly third) term of (\ref{SL_dep_exp}) we construct for every $\eta \in (0,1)$ a function $l_{\eta}\in L^1(\mathbb{R}^2)$ such that 
\begin{eqnarray}\nonumber\label{bbb_2}
&&\left|[1-\chi_c(|\tilde x-\tilde y|)]\frac{1}{4\pi}\left\{\frac{e^{ik|x_f-y_f|}}{|x_f- y_f|}- \frac{e^{ik|x_{f_n}-y_{f_n}|}}{|x_{f_n}- y_{f_n})}\right\}J_f(\tilde y)\right|\leq l_\eta(\tilde x- \tilde y),\\
\end{eqnarray}
whenever $\Vert f-f_n\Vert_{L^{\infty}(\mathbb{R}^2)}$ is sufficiently small and such that $\Vert l_{\eta}\Vert_{L^1(\mathbb{R}^2)}\to 0$ as $\eta \to 0$, and then we use the estimate (\ref{Lnormbound2}).% (provided $\Vert f-f_n\Vert_{L^{\infty}}(\mathbb{R}^2)$ is sufficiently small, and then use .

We define $l_{\eta}$ by 
$$ \ell_\eta(\tilde y) := \left\{\begin{array}{cc}
                    0     & 1     < |\tilde y|,   \\
                    \eta & \eta < |\tilde y| < 1, \\
                                        {2L'}/{|\tilde y|} & \mbox{otherwise.} \\
                  \end{array}\right.
$$
Clearly this satisfies that $||\ell_\eta||_{L^1(\real^2)}\to 0$ as
$\eta\to 0$. Since, for every $\eta\in (0,1)$, 
$$
\left| \frac{e^{ik|x_f-y_f|}}{|x_f- y_f|}- \frac{e^{ik|x_{f_n}-y_{f_n}|}}{|x_{f_n}- y_{f_n}|}\right|\to 0 
$$ 
as
$\norm{f-{f_n}}_{L^{\infty}(\real^2)}\to 0$, and uniformly so in $\tilde x$ and $\tilde y$ for $ \eta\leq
|\tilde x-\tilde y| \leq 1$, the bound (\ref{bbb_2}) holds.

For the integral operators whose kernels are the second and fourth terms of (\ref{SL_dep_exp}) we again review them as integral operators with kernels 
\begin{eqnarray}\nonumber 
&&- [1-\chi_c(|\tilde x-\tilde y|)]\frac{1}{4\pi}\left\{\frac{e^{ik|x_{f_n}-y_{f_n}|}}{|x_{f_n}- y_{f_n}|}\right\}%[J_{f_n}(\by) - J_f(\by)]
\\& &+   [1-\chi_c(|\tilde x-\tilde y|)]\frac{1}{4\pi}\left\{\frac{e^{ik|x_{f_n}-y'_{f_n}|}}{|x_{f_n}- y'_{f_n}|}\right\}%[J_{f_n}(\by) - J_f(\by)].
\end{eqnarray}
acting on the function $ [J_{f_n}(\tilde y) - J_f(\tilde y)]\phi(y).$ We then make use of (\ref{Lnormbound2}) and the inequality
\[
\Vert [J_{f_n} - J_f]\phi \Vert_{L^2(\mathbb{R}^2)}\leq\Vert  J_{f_n} - J_f\Vert_{L^4(\supp \phi)} \Vert \phi \Vert_{L^4(\mathbb{R}^2)}.
\]

We finally examine the local part of the double-layer operator. 
For $x, y \in \mathbb{R}^3$, $ x\neq y$ we define 
\[
T(x,y) = \nabla_yG(x,y) - \frac{(x-y)e^{ik|x-y|}}{|x-y|^3},
\]
and note that $T(x,y)$ is uniformly continuous in $x$ and $y$ provided $|x-y|>\epsilon$, for some $\epsilon >0$, and also that $|T(x_f,y_f)| \leq C|\tilde x- \tilde y|^{-1}$, for some $C>0$ and $\tilde x, \tilde y \in \mathbb{R}^2, \tilde x\neq \tilde y$, (this can be seen from (\ref{Gdeltapartialnu}) for example). We need to examine the integral operator with kernel   
\begin{eqnarray}\nonumber
&&[1-\chi_c(|\tilde x-\tilde y|)]\left\{\nu_{f}(\tilde y)\cdot \nabla_{y}G(x_f,y_f)J_f(\tilde y) - 
\nu_{f_n}(\tilde y)\cdot \nabla_{y}G(x_{f_n},y_{f_n})J_{f_n}(\tilde y)\right\}\\\nonumber&&=
[1-\chi_c(|\tilde x-\tilde y|)]\left\{\nu_{f}(\tilde y)\cdot T(x_f,y_f)J_f(\tilde y) - 
\nu_{f_n}(\tilde y)\cdot T(x_{f_n},y_{f_n})J_{f_n}(\tilde y)\right\}\\\nonumber&&+
[1-\chi_c(|\tilde x-\tilde y|)]\left\{\nu_{f}(\tilde y)\cdot \frac{(x_f-y_f)e^{ik|x_f-y_f|}}{|x_f-y_f|^3}J_f(\tilde y)  
\right.\\&&\left. - \nu_{f_n}(\tilde y)\cdot\frac{(x_{f_n}-y_{f_n})e^{ik|x_{f_n}-y_{f_n}|}}{|x_{f_n}-y_{f_n}|^3} J_{f_n}(\tilde y)\right\}.\label{everyone}
\end{eqnarray}
We rewrite the first term on the right hand side of (\ref{everyone}) as 
\begin{eqnarray}\nonumber
&&[1-\chi_c(|\tilde x-\tilde y|)]\left\{(\nu_{f}(\tilde y)-\nu_{f_n}(\tilde y))\cdot T(x_f,y_f)J_f(\tilde y)\right.\\\nonumber &&\left. 
+\nu_{f_n}(\tilde y)\cdot [T(x_f,y_f)-T(x_{f_n},y_{f_n})]J_f(\tilde y) \right.\\&&\left.+  
\nu_{f_n}(\tilde y)\cdot T(x_{f_n},y_{f_n})(J_f(\tilde y) - J_{f_n}(\tilde y))\right\}. \label{Texp}
\end{eqnarray}

To handle the terms in (\ref{Texp}) we argue similarly to how we did in the single-layer case above, noting that we may once again construct an analogous function $\ell_\eta \in L^1(\mathbb{R}^2)$ for $\eta \in (0,1)$ by exploiting the properties of $T$ described above.
%Looking at the expansion (\ref{Gdeltanu exp}) we see that we only need concern ourselves with two terms. Firstly  we write
%\begin{eqnarray} \label{first_term}
%\nonumber &&-ik\nu_{f}(y)\cdot(x_f- y_f)\frac{e^{ik|x_f-y_f|}}{|x_f-y_f|^2} -  ik\nu_{f_n}(y)\cdot(x_{f_n}- y_{f_n})\frac{e^{ik|x_{f_n}-y_{f_n}|}}{|x_{f_n}-y_{f_n}|^2}\\\nonumber &=& 
%-ik\nu_{f}(y)\cdot\left[(x_f- y_f\frac{e^{ik|x_f-y_f|}}{|x_f-y_f|^2} -  (x_{f_n}- y_{f_n})\frac{e^{ik|x_{f_n}-y_{f_n}|}}{|x_{f_n}-y_{f_n}|^2}\right]\\ && +ik[\nu_{f}(\by)-\nu_{f_n}(\by)]\cdot(x_{f_n}- y_{f_n})\frac{e^{ik|x_{f_n}-y_{f_n}|}}{|x_{f_n}-y_{f_n}|^2}.
%\end{eqnarray}%\textbf{The local operator}. For the local operator we argue in a
%For the first term on the right hand side of (\ref{first_term}) we apply a similar argument to the one we used in the proof on the local part of the single layer operator. For the second term we see that 
%\[
%\nu_{f}(\by) -\nu_{f_n(\by)} = \left(\frac{\partial f}{\partial y_1}- \frac{\partial f_n}{\partial y_1}, \frac{\partial f}{\partial y_2}- \frac{\partial f_n}{\partial y_2}\right)
%\]
%so that for $i=1,2$
%\[
%\left| ik\left [\frac{\partial (f-f_n)}{\partial y_i}(y) \right ](\hat x_i- \hat y_i)\frac{e^{ik|x_{f_n}-y_{f_n}|}}{|x_{f_n}-y_{f_n}|}\right|
%\leq k\left|\left [\frac{\partial (f-f_n)}{\partial y_i}(y) \right ]\right|\left|{|\tilde x-\tilde y|}\right|
%\]
%and we can argue as before.

Finally, from (\ref{everyone}), we need to look at the integral operator 
\begin{eqnarray*}
&&PV_{\epsilon \to 0}\int_{\mathbb{R}^2\backslash B_{\epsilon}(\tilde x)}(1-\chi_c(|\tilde x-\tilde y|))\frac{\nu_{f}(\tilde y) \cdot (x_f-y_f)}{|x_f-y_f|^3}e^{ik|x_f-y_f|}\phi(y)J_f(\tilde y) d\tilde y \\&-& \int_{\mathbb{R}^2}(1-\chi_c(|\tilde x-\tilde y|))\frac{\nu_{f_n}(\tilde y) \cdot (x_{f_n}-y_{f_n})}{|x_{f_n}-y_{f_n}|^3}e^{ik|x_{f_n}-y_{f_n}|}\phi(y)J_{f_n}(\tilde y) d\tilde y \\& =& \int_{\mathbb{R}^2}(1-\chi_c(|\tilde x-\tilde y|))\frac{\nu_{f}(\tilde y) \cdot (x_f-y_f)}{|x_f-y_f|^3}(e^{ik|x_f-y_f|}-1)\phi(y)J_f(\tilde y) d\tilde y \\&+&PV_{\epsilon \to 0}\int_{\mathbb{R}^2\backslash B_{\epsilon}(\tilde x)}\frac{\nu_{f}(\tilde y) \cdot (x_f-y_f)}{|x_f-y_f|^3}\phi(y)J_f(\tilde y) d\tilde y \\ &-& \int_{\mathbb{R}^2}\chi_c(|\tilde x-\tilde y|)\frac{\nu_{f}(\tilde y) \cdot (x_{f}-y_{f})}{|x_f-y_f|^3}\phi(y)J_f(\tilde y) d\tilde y\\
& -& \int_{\mathbb{R}^2}(1-\chi_c(|\tilde x-\tilde y|))\frac{\nu_{f_n}(\tilde y) \cdot (x_{f_n}-y_{f_n})}{|x_{f_n}-y_{f_n}|^3}(e^{ik|x_{f_n}-y_{f_n}|}-1)\phi(y)J_{f_n}(\tilde y) d\tilde y \\&-&\int_{\mathbb{R}^2}\frac{\nu_{f_n}(\tilde y) \cdot (x_{f_n}-y_{f_n})}{|x_{f_n}-y_{f_n}|^3}\phi(y)J_{f_n}(\tilde y) d\tilde y \\ &+& \int_{\mathbb{R}^2}\chi_c(|\tilde x-\tilde y|)\frac{\nu_{f_n}(\tilde y) \cdot (x_{f_n}-y_{f_n})}{|x_{f_n}-y_{f_n}|^3}\phi(y)J_{f_n}(\tilde y) d\tilde y.
\end{eqnarray*} 
For the integral operator
\begin{eqnarray*}
 &&\int_{\mathbb{R}^2}(1-\chi_c(|\tilde x-\tilde y|))\frac{\nu_{f}(\tilde y) \cdot ((x_f-y_f)}{|x_f-y_f|^3}(e^{ik|x_f-y_f|}-1)\phi(y)J_f(\tilde y) d\tilde y
\\ & - &\int_{\mathbb{R}^2}(1-\chi_c(|\tilde x-\tilde y|))\frac{\nu_{f_n}(\tilde y) \cdot (x_{f_n}-y_{f_n})}{|x_{f_n}-y_{f_n}|^3}(e^{ik|x_{f_n}-y_{f_n}|}-1)\phi(y)J_{f_n}(\tilde y) d\tilde y
\end{eqnarray*}
we use the expansion 
$$ 
e^{ik|x-y|} = 1 + ik|x-y| + \frac{(ik|x-y|)^2}{2!} + \dots, 
$$ 
and then deal with this integral operator similarly to how we dealt with the single-layer local operator. 

For the integral operator given by  
\begin{eqnarray*}
&&\int_{\mathbb{R}^2}\chi_c(|\tilde x-\tilde y|)\frac{\nu_{f}(\tilde y) \cdot (x_{f}-y_{f})}{|x_f-y_f|^3}\phi(y)J_f(\tilde y) d\tilde y
\\&&-\int_{\mathbb{R}^2}\chi_c(|\tilde x-\tilde y|)\frac{\nu_{f_n}(\tilde y) \cdot (x_{f_n}-y_{f_n})}{|x_{f_n}-y_{f_n}|^3}\phi(y)J_{f_n}(\tilde y) d\tilde y,
\end{eqnarray*}
we notice that both of the kernels are bounded by 
\[
t(\tilde x,\tilde y)= \frac{\chi_c(|\tilde x -\tilde y|)}{|\tilde x - \tilde y|^2}
\]
and that since $t(\tilde y) \in L^2(\mathbb{R}^2)$, the integral operator with kernel $t(\tilde x,\tilde y)$ is bounded, by Young's inequality (\ref{youngs}), from $L^2(\mathbb{R}^2)$ to $L^1(\mathbb{R}^2).$ We then make similar arguments to those we made in the double-layer global case.
%\[
%\Vert [J_{f_n} - J_f]\phi \Vert_{L^1(\mathbb{R}^2)}\leq\Vert  J_{f_n} - J_f\Vert_{L^2(\supp \phi)} \Vert \phi \Vert_{L^2(\mathbb{R}^2)},
%\] 
%we can arrive at the desired result.

We turn to looking at 
\begin{eqnarray*}
&&PV_{\epsilon \to 0}\int_{\mathbb{R}^2\backslash B_{\epsilon}(\tilde x)}\frac{\nu_{f}(\tilde y) \cdot (x_f-y_f)}{|x_f-y_f|^3}\phi(y)J_f(\tilde y) d\tilde y \\&- &\int_{\mathbb{R}^2}\frac{\nu_{f_n}(\tilde y) \cdot (x_{f_n}-y_{f_n})}{|x_{f_n}-y_{f_n}|^3}\phi(y)J_{f_n}(\tilde y) d\tilde y.
\end{eqnarray*}%similar way as for the global operator, in particular in a similar
But, as usual we need only consider
\begin{eqnarray*}
&&PV_{\epsilon \to 0}\int_{\mathbb{R}^2\backslash B_{\epsilon}(\tilde x)}\frac{\nu_{f}(\tilde y) \cdot (x_f-y_f)}{|x_f-y_f|^3}\phi(y)J_{f_n}(\tilde y) d\tilde y \\&- &\int_{\mathbb{R}^2}\frac{\nu_{f_n}(\tilde y) \cdot (x_{f_n}-y_{f_n})}{|x_{f_n}-y_{f_n}|^3}\phi(y)J_{f_n}(\tilde y) d\tilde y.
\end{eqnarray*}
Now since $\phi(y)J_{f_n}(\tilde y) \in C^{\infty}_0(\mathbb{R}^2)$, we apply theorem \ref{Wavelets_them} to get that
\begin{eqnarray*}
&&PV_{\epsilon \to 0}\int_{\mathbb{R}^2\backslash B_{\epsilon}(\tilde x)}\frac{\nu_{f}(\tilde y) \cdot (x_f-y_f)}{|x_f-y_f|^3}\phi(y)J_{f_n}(\tilde y) d\tilde y \\& = &-\int_{\mathbb{R}^2}\frac{(\tilde x-\tilde y).\nabla_{\tilde y}(\phi J_{f_n})(\tilde y)}{|\tilde x- \tilde y|^2}\lambda \left(\frac{f(\tilde x)-f(\tilde y)}{|\tilde x-\tilde y|}\right)d\tilde y
\end{eqnarray*}
where $\lambda$ is given by (\ref{lambda_def}). 

Thus we need to examine, for $i=1,2$, and where $\hat x_i$ denotes the unit vector in the $i$th direction, 
\begin{eqnarray*}
 &&\int_{\mathbb{R}^2}\frac{(\hat x_i-\hat y_i)}{|\tilde x- \tilde y|}\frac{\partial}{\partial y_i} (\phi J_{f_n})(\tilde y)\left[\lambda \left(\frac{f(\tilde x)-f(\tilde y)}{|\tilde x-\tilde y|}\right) -\lambda \left(\frac{f_n(\tilde x)-f_n(\tilde y)}{|\tilde x-\tilde y|}\right)\right]d \tilde y\\&&:=\int_{\mathbb{R}^2}v(\tilde x,\tilde y,f,f_n,\phi J_{f_n})d \tilde y.
\end{eqnarray*}

We split this up into a local and global part and first of all examine
\begin{eqnarray*}
V(\tilde x):= \int_{\mathbb{R}^2}(1-\chi_c(|\tilde x-\tilde y|))v(\tilde x,\tilde y,f,f_n, \phi J_{f_n})%(\frac{(\hat x_i-\hat y_i)}{|\tilde x-\tilde y|}\frac{\partial}{\partial y_i}(\phi J_{f_n})(\tilde y)\left[\lambda \left(\frac{f(\tilde x)-f(\tilde y)}{|\tilde x-\tilde y|}\right) -\lambda \left(\frac{f_n(\tilde x)-f_n(\tilde y)}{|\tilde x-\tilde y|}\right)\right]
d \tilde y.
\end{eqnarray*}

Noting that $\lambda$ is a uniformly continuous function we can construct for every $\eta \in (0,1)$ a function $l_\eta \in L^1(\mathbb{R}^2)$ such that
\begin{eqnarray*} 
\left|(1-\chi_c(|\tilde x-\tilde y|)\frac{(\hat x_i-\hat y_i)}{|\tilde x-\tilde y|}\left[\lambda \left(\frac{f(\tilde x)-f(\tilde y)}{|\tilde x-\tilde y|}\right) -\lambda \left(\frac{f_n(\tilde x)-f_n(\tilde y)}{|\tilde x-\tilde y|}\right)\right]\right| \leq l_n(\tilde x-\tilde y)
\end{eqnarray*}

 % we know that %way as for the integral operator corresponding to the second 
whenever $\Vert f-f_n\Vert_{L^2(\mathbb{R}^2)}$ is sufficiently small and such that $\Vert l_\eta \Vert_{L^1(\mathbb{R}^2)} \to 0$ as $\eta \to 0$.
The construction of $l_\eta$ is as follows
$$\ell_\eta(\tilde y) := \left\{\begin{array}{cc}
                    \eta & \eta < |\tilde y| < 1, \\
                                        2{\Vert \lambda \Vert_{L^{\infty}(\mathbb{R}^2)}}/{|\tilde y|} & |\tilde y | <\eta, \\
                    0                   & \mbox{otherwise.} \\
                  \end{array}\right.
$$
Thus for the integral operator $V$ we have from (\ref{youngs}) that 
\[
\Vert V\Vert_{L^2(\mathbb{R}^2)} \leq \Vert l_\eta\Vert_{L^1(\mathbb{R}^2)}\left \Vert\frac{\partial}{\partial y_i}(\phi J_{f_{n}})\right\Vert_{L^2(\supp \phi)}. 
\]

We next look at
\begin{eqnarray*}
W(\tilde x):= \int_{\mathbb{R}^2}\chi_c(|\tilde x-\tilde y|)v(\tilde x,\tilde y,f,f_n,\phi J_{f_n})%\frac{(\hat x_i-\hat y_i)}{|\tilde x-\tilde y|}\frac{\partial}{\partial y_i}(\phi J_{f_n})(\tilde y)\left[\lambda \left(\frac{f(\tilde x)-f(\tilde y)}{|\tilde x-\tilde y|}\right) -\lambda \left(\frac{f_n(\tilde x)-f_n(\tilde y)}{|\tilde x-\tilde y|}\right)\right]
d \tilde y.
\end{eqnarray*}

Once again, we construct for every $\eta \in (0,1)$ a function $l_\eta \in L^3(\mathbb{R}^2)$ such that
\begin{eqnarray*} 
\left|\chi_c(|\tilde x-\tilde y|)\frac{(\hat x_i-\hat y_i)}{|\tilde x-\tilde y|}\left[\lambda \left(\frac{f(\tilde x)-f(\tilde y)}{|\tilde x-\tilde y|}\right) -\lambda \left(\frac{f_n(\tilde x)-f_n(\tilde y)}{|\tilde x-\tilde y|}\right)\right]\right| \leq l_n(\tilde x-\tilde y)
\end{eqnarray*}

 % we know that %way as for the integral operator corresponding to the second 
whenever $\Vert f-f_n\Vert_{L^2(\mathbb{R}^2)}$ is sufficiently small and such that $\Vert l_\eta \Vert_{L^3(\mathbb{R}^2)} \to 0$ as $\eta \to 0$.
The construction of $l_\eta$ is as follows:
$$\ell_\eta(\tilde y) := \left\{\begin{array}{cc}
                    \eta & 1/2 < |\tilde y| < \eta^{-1}, \\
                                        2{\Vert \lambda \Vert_{L^{\infty}(\mathbb{R}^2)}}/{|\tilde y|} & \eta^{-1}<|\tilde y |, \\
                    0                   & \mbox{otherwise.} \\
                  \end{array}\right.
$$
Thus for the integral operator $W$ we have from (\ref{youngs}) that 
\[
\Vert W\Vert_{L^2(\mathbb{R}^2)} \leq \Vert l_\eta\Vert_{L^3(\mathbb{R}^2)}\left \Vert\frac{\partial}{\partial y_i}(\phi J_{f_{n}})\right \Vert_{L^{6/7}(\supp \phi)}. 
\]
ii)We fix $\epsilon>0$ and then fix $\psi \in C^{\infty}_0(\mathbb{R}^2)$ such that 
\[
\Vert \tilde A_f \phi - \psi\Vert_{L^2(\mathbb{R}^2)}<\epsilon.
\] 
Then for $n \in \mathbb{N}$ we have that 
\begin{eqnarray*}
\Vert A_fI_f^{-1}\phi\Vert_{L^2(\Gamma)}&&= \Vert \tilde A_f \phi \sqrt{J_f}\Vert_{L^2(\mathbb{R}^2)}\\&&\leq \Vert (\tilde A_f \phi -\tilde A_{f_n} \phi )\sqrt{J_f}\Vert_{L^2(\mathbb{R}^2)}+\Vert \tilde A_{f_n} \phi \sqrt{J_{f_n}}\Vert_{L^2(\mathbb{R}^2)}\\&&+\Vert \tilde A_{f_n} \phi [\sqrt{J_{f_n}}-\sqrt{J_f}]\Vert_{L^2(\mathbb{R}^2)}\\&&\leq
\sqrt{L'} \Vert (\tilde A_f \phi -\tilde A_{f_n} \phi )\Vert_{L^2(\mathbb{R}^2)}+\Vert  A_{f_n}I_{f_n}^{-1} \phi \Vert_{L^2(\Gamma_n)}\\&&+\Vert (\tilde A_{f_n} \phi -\psi) [\sqrt{J_{f_n}}-\sqrt{J_f}]\Vert_{L^2(\mathbb{R}^2)}+\Vert \psi [\sqrt{J_{f_n}}-\sqrt{J_f}]\Vert_{L^2(\mathbb{R}^2)}.
\end{eqnarray*}
Thus provided $n$ is chosen large enough we use part i) to get that
\begin{eqnarray*}
\Vert A_fI_f^{-1}\phi\Vert_{L^2(\Gamma)}&&\leq \sqrt{L'}\epsilon +\Vert A_{f_n}I_{f_n}^{-1}\phi\Vert_{L^2(\Gamma_n)}\\&&+ 2\sqrt{L'}\Vert\tilde A_{f_n}\phi -\psi\Vert_{L^2(\mathbb{R}^2)} + \Vert\psi\Vert_{L^4(\mathbb{R}^2)}\Vert \sqrt{J_{f_n}}-\sqrt{J_f}\Vert_{L^4(\supp\psi)}\\&&\leq
\sqrt{L'}\epsilon +\Vert A_{f_n}I_{f_n}^{-1}\phi\Vert_{L^2(\Gamma_n)}\\&&+2\sqrt{L'}\Vert\tilde A_{f}\phi -\psi\Vert_{L^2(\mathbb{R}^2)}+2\sqrt{L'}\Vert\tilde A_{f_n}\phi -\tilde A_f \phi\Vert_{L^2(\mathbb{R}^2)}\\&&+\epsilon\Vert\psi\Vert_{L^4(\mathbb{R}^2)}\\&&\leq
\sqrt{L'}\epsilon +\Vert A_{f_n}I_{f_n}^{-1}\phi\Vert_{L^2(\Gamma_n)}+2\sqrt{L'}\epsilon+2\sqrt{L'}\epsilon+\epsilon\Vert\psi\Vert_{L^4(\mathbb{R}^2)}.
\end{eqnarray*}
From the arbitrariness of $\epsilon>0$ we now conclude that
\[
\Vert A_fI_f^{-1}\phi\Vert_{L^2(\Gamma)}\leq \lim_{n \to \infty}\Vert A_{f_n}I_{f_n}^{-1}\phi\Vert_{L^2(\Gamma_n)}.
\]
The reverse inequality is proved analogously. The proof is complete.%(\ref{expand}). In particular, where $a_2$ is the kernel of the
\end{proof}
%in the mildly rough case. But first we prove a
%preliminary lemma which shows that to establish unique solvability
%of the integral equation in the space $X$ it is enough to study
%solvability in $L^2(\Gamma)$.

We are now in a position to prove theorem \ref{A_inverse} on the invertibility of $A$.

%\begin{lemma}
%The integral operator $A$ is invertible on $L^2(\Gamma)$ with 
%\[
%\Vert A^{-1}\Vert_{L^2(\Gamma)} \leq B,
%\]
%with $B$ given by 
%\end{lemma}
\begin{proof}
We choose by lemma \ref{smooth_Lip} a sequence of Lipschitz functions $f_n \in C^{\infty}(\mathbb{R}^2)$, $n \in \mathbb{N}$, such that each $f_n$ has Lipschitz constant $L$, such that each $f_n$ is Lyapunov, such that $\Vert f_n - f\Vert_{L^{\infty}(\mathbb{R}^2)} \to 0$ and such that $\nabla_{\tilde x}f_n \to  \nabla_{\tilde x}f$ in $L^p(K)$ for compact $K \subseteq \mathbb{R}^2$, with $1<p<\infty$. For brevity we denote by $A_n$ the integral operator $A_{f_n}$ and by $A$ the integral operator $A_f$.
% when the underlying surface is $f_n$. 

By lemma \ref{contthe} we know that for $\phi \in L^2(\mathbb{R}^2)$
\[
\Vert AI_f^{-1}\phi \Vert_{L^2(\Gamma)} = \lim_{n\to \infty}\Vert A_n I_{f_n}^{-1}\phi \Vert_{L^2(\Gamma_n)}.
\]
Since by theorem \ref{chandpottheim_result2} we have that for all $n \in \mathbb{N}$
\[
\Vert A_n I_{f_n}^{-1}\phi \Vert_{L^2(\Gamma_n)} \geq B^{-1} \Vert I_{f_n}^{-1}\phi\Vert_{L^2(\Gamma_n)},     
\]
it follows that for all $\phi \in L^2(\mathbb{R}^2)$,
\begin{equation} \label{bounded_below}
\Vert AI_{f}^{-1}\phi \Vert_{L^2(\Gamma)}\geq B^{-1} \Vert I_{f}^{-1}\phi\Vert_{L^2(\Gamma)}.     
\end{equation}
This shows that $A$ is bounded below. We now establish that the adjoint of $A$, $A'$ is also bounded below. Together, the two bounds along with theorem \ref{btheo} ensure the invertibility of $A$, whilst the bound (\ref{operator_B}) follows from (\ref{bounded_below}).

Fix $\phi \in L^2(\mathbb{R}^2)$. We know that by theorem  \ref{chandpottheim_result2} each of the $A_n$ is invertible on $L^2(\Gamma_n)$ and moreover that  
\[
\Vert \tilde A_{n}^{-1}\phi\Vert_{L^2(\mathbb{R}^2)}\leq \Vert A_n^{-1} I_{f_n}^{-1}\phi \Vert_{L^2(\Gamma_n)} \leq B \Vert I_{f_n}^{-1}\phi \Vert_{L^2(\Gamma_n)} \leq B\sqrt{L'}\Vert \phi\Vert_{L^2(\mathbb{R}^2)}.
\]
This shows that that sequence $\tilde A_{n}^{-1} \phi$ is bounded in $L^2(\mathbb{R}^2)$. Identifying $L^2(\mathbb{R}^2)$ with it's dual, which we may do by the Riesz representation theorem, we see that by the Banach-Alaoglu theorem (\cite{Pederson} theorem 2.52), we may extract a $w^*$-limit $t \in L^2(\mathbb{R}^2)$ from this sequence. Thus we have that 
\[
\lim_{n \to \infty} (\tilde A_n^{-1}\phi, \psi )_2 = (t, \psi)_2, \quad \psi \in L^2(\mathbb{R}^2),
\]
where 
\[
(u,v)_2= \int_{\mathbb{R}^2}u \bar v \;dx,
\] 
and where $\Vert t \Vert_{L^2(\mathbb{R}^2)} \leq  C\Vert \phi \Vert_{L^2(\mathbb{R}^2)}$
where $C=B \sqrt{L'}$.
Now,
\begin{eqnarray*}
\Vert A'I_f^{-1}\phi \Vert_{L^2(\Gamma)}\geq \Vert \tilde A' \phi\Vert_{L^2(\mathbb{R}^2)}& = & \sup_{\{v : \Vert v \Vert \leq 1\}}|(\tilde A'\phi,v)_2|\\
& \geq & |(\tilde A'\phi,t{C^{-1}}\Vert \phi \Vert^{-1}_{L^2(\mathbb{R}^2)})_2|\\
& = & {C^{-1}}\Vert \phi \Vert^{-1}_{L^2(\mathbb{R}^2)}|(\phi,\tilde At)_2|\\ 
& = & {C^{-1}}\Vert \phi \Vert^{-1}_{L^2(\mathbb{R}^2)}\lim_{n \to \infty}|(\phi, \tilde A_nt)_2|\\ 
& = & {C^{-1}}\Vert \phi \Vert^{-1}_{L^2(\mathbb{R}^2)}\lim_{n \to \infty}|(\tilde A_n'\phi, t)_2|\\ 
& = & {C^{-1}}\Vert \phi \Vert^{-1}_{L^2(\mathbb{R}^2)}\lim_{n \to \infty}\lim_{m \to \infty}|(\tilde A_n'\phi, \tilde A_m^{-1}\phi)_2|\\ 
& = & {C^{-1}}\Vert \phi \Vert^{-1}_{L^2(\mathbb{R}^2)}\lim_{n \to \infty}\lim_{m \to \infty}|(\tilde A_m'^{-1}\tilde A_n'\phi, \phi)_2|\\ &=& C^{-1}\Vert \phi \Vert_{L^2(\mathbb{R}^2)}\\&\geq& C^{-1}\sqrt{L'}^{-1}\Vert I_f^{-1}\phi \Vert_{L^2(\Gamma)}.
\end{eqnarray*}
The proof is complete. 
\end{proof}

   %\begin{lemma} \label{L2X}
%Suppose that the integral equation (\ref{central ieq}) has exactly
%one solution $\varphi\in L^2(\Gamma)$ for every $g\in L^2(\Gamma)$.
% Then also
%(\ref{central ieq}) has exactly one solution $\varphi\in X$ for
%every $g\in X$, so that $(I+K-i\eta S)^{-1}$ exists and is bounded
%as an operator on $X$.
%\end{lemma}

%{\em Proof}. If the assumptions of the lemma hold then equation
%(\ref{central ieq}) has exactly one solution $\varphi\in
%L^2(\Gamma)$ for every $g\in X\subset L^2(\Gamma)$.  Further,
%defining $A=K-i\eta S$, it holds that $ \varphi = A\varphi + 2g $
%and, by induction, that, for every $n\in \mathbb{N}$,
%$$
%\varphi = A^n \varphi + 2(A^{n-1}+...+A^0)g.
%$$
%Now, by Theorem \ref{btheo}, $A$ is a bounded operator on $X$ and,
%by Corollary \ref{cor:map}, $A^n$ is a bounded operator from
%$L^2(\Gamma)$ to $X$, for some $n\in\mathbb{N}$. Thus $\varphi\in
%X$. We have shown that (\ref{central ieq}) has exactly one solution
%$\varphi\in X$ for every $g\in X$, so that $(I+K-i\eta S)^{-1}$
%exists as an operator on $X$. Since $X$ is a Banach space it follows
%as a standard corollary of the open mapping theorem that $(I+K-i\eta
%S)^{-1}$ is bounded.
%\endproof

We conclude this chapter by proving our main result, theorem \ref{chap5_them} concerning existence and uniqueness of solution to the boundary value problem.

{\em proof of theorem \ref{chap5_them}.}
By lemma \ref{unique} we know that the boundary value problem has at most one solution. We construct a solution $v$ to the boundary value problem by supposing that for $x \in D$, $v(x)$ is given by (\ref{the_ansatz}) with $u_1(x)$ defined by (\ref{SLP}) and $u_2(x)$ defined by (\ref{DLP}) and with the density $\phi$ such that
\[
\phi = A^{-1}g
\]
possible by theorem \ref{A_inverse}. We then see that by theorem \ref{jr thm} $v \in C^2(D)$, satisfies the Helmholtz equation in $D$, and also the non-tangential boundary condition on $\Gamma$. Further by lemmas \ref{rad_lemma} and \ref{u_prime_lemma} $v$ satisfies the radiation condition (\ref{uprcstar}) for all $H>f_+$ and satisfies that $v_T' \in L^2(\Gamma)$ for all $T\geq f_+$.

\begin{appendix}
\chapter{Trace results}
\begin{lemma} 
%For $u,v \in W_H$
%\[
%b(u,v) = (\nabla u,\nabla v) -k^2(u,v) + \int_{\Gamma_H\backslash \Upsilon}\gamma_- \bar vT \gamma_- u ds - \int_{\Gamma \backslash \Upsilon} ik\beta \gamma^* w\gamma^* \bar v ds.
%\]

Let $D$ be an $(L,\mu,N)$ Lipschitz domain, and let $S_H= D \backslash \overline{U_H}$ for $H \geq f_+ + \mu$.
For $u \in H^1(S_H)$,
\begin{eqnarray*}
\Vert \gamma_- u\Vert_{H^{\frac{1}{2}}(\Gamma_H )} \leq \sqrt{\left(1+ \frac{1}{k\mu}\right)}\Vert u \Vert_{H^1(S_H)},
\end{eqnarray*} 
and, the map $\gamma^*: \mathcal{D}({S_H}) \to L^2(\Gamma)$ such that $\gamma^*u$ is $u$ restricted to $\Gamma$, for $u \in \mathcal{D}({S_H})$, extends to a bounded linear operator $\gamma^*: H^1(S_H) \to L^2(\Gamma)$ with  
\begin{eqnarray*}
k\Vert \gamma^* u \Vert^2_{L^2(\Gamma)} \leq N\sqrt{1+L^2}\left(1+\frac{1}{ k\mu}\right)\Vert u \Vert^2_{H^1(S_H)}.
\end{eqnarray*}
\end{lemma}
\begin{proof}
For $u \in \mathcal{D}({S_H})$, define, for $x_n \in [H-\mu,H]$, $\hat u (\xi,x_n) = (\mathcal{F} u(\cdot ,x_n))(\xi)$. Let $S= \mathbb{R}^{n-1} \times [H-\mu,H]$, and let 
$\phi :  S \to \mathbb{R}$, be such that $\phi(\xi,x_n) = [{(x_n -H) + \mu}]/{\mu}$.
We have
\begin{eqnarray*}
\vert \hat{u}(\xi,H)\vert^2=\int_{H-\mu}^H\frac{\partial}{\partial
x_n}\phi| \hat{u}(\xi,x_n)|^2\,dx_n & = &
2\Re\int_{H-\mu}^H\phi\overline{\hat{u}(\xi,x_n)}\frac{\partial}{\partial
x_n} (\hat{u}(\xi,x_n))\,dx_n.\\
& + &\int_{H-\mu}^H\frac{\partial \phi}{\partial x_n}|\hat{u}(\xi,x_n)|^2
 dx_n.\end{eqnarray*}
Thus,
\begin{eqnarray*}
\|u\|_{H^{1/2}(\Gamma_H)}^2&=&\int_{\real^{n-1}}|\sqrt{\xi^2+k^2}|\,|\hat{u}(\xi,H)|^2\,d\xi\\
&\leq&2\int_{S}|\sqrt{\xi^2+k^2}|\,|\hat{u}(\xi,x_n)|\,
\left\vert\frac{\partial}{\partial x_n}\hat{u}(\xi,x_n)
\right\vert\,d\xi\,dx_n\\ & + & \int_{S}|\sqrt{\xi^2+k^2}|\,|\hat{u}(\xi,x_n)|^2\,
\frac{1}{\mu}d\xi dx_n\\
&\leq&2\left\{\int_{S}|\xi^2+k^2|\,|\hat{u}(\xi,x_n)|^2\,
d\xi\,dx_n\right\}^{1/2}\left\{\int_S
\left\vert\frac{\partial}{\partial x_n}\hat{u}(\xi,x_n)
\right\vert^2\,d\xi\,dx_n\right\}^{1/2}  \\
& + &  \left\{\frac{1}{k\mu} \int_{S}|\xi^2 + k^2||\hat{u}(\xi,x_n)|^2 d \xi dx_n\right\}.%^{\frac{1}{2}}.
\end{eqnarray*}
Now, by Parseval's theorem,
\begin{eqnarray*}
\int_{S}\xi^2\,|\hat{u}(\xi,x_n)|^2\,d\xi\,dx_n&=&
\int_{S}\left\vert\cF(\nabla_{\tilde{x}}(u)(\cdot,x_n))(\xi)\right\vert^2
\,d\xi\,dx_n\\
&=&\int_{S}\left\vert\nabla_{\tilde{x}}u(x)\right\vert^2\,dx.
\end{eqnarray*}
Applying Parseval's theorem again, and using $2ab\leq a^2 + b^2$, for $a,b\geq 0,$  
\begin{eqnarray*}
\|u\|_{H^{1/2}(\Gamma_H)}^2 & \leq & 2\left\{\int_S\left\{k^2
|u(x)|^2+|\nabla_{\tilde{x}}u(x)|^2\right\}\,dx\right\}^{\frac{1}{2}}\left\{\int_S\left\vert
\frac{\partial}{\partial x_n}
u(x)\right\vert^2\,dx\right\}^{1/2} \\
& + &  \left\{\frac{1}{k\mu}\int_S k^2|u(x)|^2 + |\nabla_{\tilde x}u(x)|^2 dx\right\}%^{\frac{1}{2}}
\\ & \le & \left(1+\frac{1}{k\mu}\right)\|u\|_{H^1(S_H)}^2.
\end{eqnarray*}
%Now using (\ref{u_bds_w}), (\ref{conv1}), (\ref{conv2}) and (\ref{conv3}), with the fact that $w=u$ everywhere on $\Gamma_H$, 
The first result now follows because of the density of $\mathcal{D}({S_H})$ in $H^1(S_H).$
 
%For the second part let $\{\phi_j\}_{j \in J}$ be a $C^{\infty}$ partition of unity for $\Gamma$ subordinate to the $\{O_j\}_{j \in J}$.%  fix $j \in J$, and consider 
%\[
%Y_j:=\{y \in \Gamma : \exists \mbox{ an open ball of radius  } \epsilon \mbox{ and centre }y \mbox{ contained in } O_j.\} 
%\]
For the second part, note that $S_H$ is an $(L,\mu,N+1)$ Lipschitz domain; let ${S_H}_j$ be the $\Omega_j$ of definition \ref{LipDef}. Define $U_i:=\{y \in \Gamma: B_{\mu}(y) \subseteq O_i, \mbox{ and } B_{\mu}(y) \not\subseteq O_j \mbox{ if } j<i\}\subseteq O_i.$ Note that $\Gamma$ is the disjoint union of the $\{U_i\}_{i \in J}$, and that, by definition,
\[
\int_{\Gamma}|u(s)|^2 ds= \sum_{i=1}^{\infty} \int_{U_i}|u(s)|^2ds.
\]
Fix $j \in J$. Rotate ${S_H}_j$ into the epigraph, of a Lipschitz function $f_j$, and let $e_n$ denote the vertical unit vector, after this rotation. %In view of definition \ref{LipDef},
For $y \in U_j$, $y + te_n \in O_j \cap  S_H= O_j \cap  {S_H}_j$, provided $0<t<\mu$. Let $S:= \{(\tilde y, y_n+te_n): y \in U_j,  0\leq t \leq\mu\}$. Denote by $\phi :S \rightarrow \mathbb{R}$, the function such that $\phi(\tilde y,y_n+te_n):=1- t /\mu$, for $(\tilde y,y_n+te_n) \in S$. Note that, after a suitable change of coordinates, and where $K=\{\tilde x \in \mbox{supp}f_j|(\tilde x,f_j(\tilde x)) \in U_j \}\subseteq \mathbb{R}^{n-1}$,
\[
\int_{ U_j} |u(s)|^2 ds = \int_{K }|u(\tilde x,f_j(\tilde x)|^2\sqrt{1+|\nabla f_j(\tilde x)|^2} d\tilde x.
\]
Then 
\begin{eqnarray*}
&&\int_{K}k|u(\tilde x,f_j(\tilde x)|^2\sqrt{1+|\nabla f_j(\tilde x)|^2} d\tilde x\\& = & \int_{K}\sqrt{1+|\nabla f_j(\tilde x)|^2}\int_{t=\mu}^{t=0}k\frac{\partial}{\partial x_n }(\phi|u(\tilde x,f_j(\tilde x)+te_n)|^2) dx_n d\tilde x  \\ %& = & L'\int_{\mathbb{R}^{n-1} \cap O_j}\int^{x_n =f_j}_{x_n =f_j + \epsilon /2} k\overline{\phi u(\tilde x, x_n)}\frac{\partial \phi u(\tilde x,x_n)}{\partial x_n} dx_n d\tilde x
& \leq & \sqrt{1+L^2}\left\{\int_{S} 2k|u(x)|\left|\frac{\partial u(x)}{\partial x_n}\right|dx + \int_{S} \frac{k}{\mu}|u(x)|^2 dx\right\}.
\end{eqnarray*}
Use of the Cauchy-Schwarz inequality and $2ab\leq a^2 +b^2$, $a,b>0$, gives 
\begin{eqnarray}\nonumber
\int_{ U_j}k|u(s)|^2 ds & \leq &\sqrt{1+L^2}\left\{ \int_{S}k^2|u(x)|^2 dx +\int_{S}\left|\frac{\partial u(x)}{\partial x_n}\right|^2dx \right.\\ &+&\left. \frac{1}{\mu k}\int_{S}k^2|u(x)|^2 dx\right\} \\\label{inequality}  %\left k^2\Vert u\Vert^2_{L^2(S)} + \frac{1}{2}\Vert \nabla u \Vert^2_{L^2(S)} \\ 
& = & \sqrt{1+L^2}\left( 1+ \frac{1}{\mu k}\right) \Vert u\Vert^2_{H^1(O_j \cap S_H)},
\end{eqnarray} 
since $S \subseteq O_j \cap S_H$. 
%Now fix $\mathcal{F}_n$, one of the $N$ subfamilies of definition \ref{LipDef}. Since the $O_j$ in this family are mutually disjoint, 
Repeat this argument for all $j \in J$. %Note that each point on $\Gamma$ belongs to at least one $Y_j$, and that 
Note that property (iii) of definition \ref{LipDef} implies that %(prove by induction on $N\geq 1$)
\[
\sum_{j \in J}\Vert u\Vert^2_{H^1(O_j\cap S_H)} \leq N \Vert u \Vert^2_{H^1(S_H)}.
\]
Then summing inequality (\ref{inequality}) over finite $j$ and letting $j \to \infty$, implies 
\begin{eqnarray*}
\int_{\Gamma} k|u(s)|^2ds \leq N\sqrt{1+L^2}\left(1+ \frac{1}{\mu k}\right) \Vert u\Vert^2_{H^1(S_H)}.
\end{eqnarray*}
%Summing the same inequlity for each subfamily $\mathcal{F}_n$ gives the result for smooth $u$, since each $y \in \Gamma$ belongs to at least one $O_j$. 
The density of $\mathcal{D}({S_H})$ in $H^1(S_H)$ gives the bound in the general case.
\end{proof}

\begin{lemma}
Let $f: \mathbb{R}^{n-1} \to \mathbb{R}$ be a bounded Lipschitz function with Lipschitz constant $L$ and let $\mathcal{C}:= \{(\tilde x,x_n)|x_n \in [f(\tilde x)-\epsilon,f(\tilde x) +\epsilon]\}$.
Then for $w \in H^1(\mathcal{C})$ it holds that
\[
\epsilon \int_{\Gamma} |w|^2ds \leq \sqrt{1+L^2}\left\{\epsilon^2\left\Vert \frac{\partial w}{\partial x_n}\right\Vert^2_{L^2(\mathcal{C})} + \Vert w\Vert^2_{L^2(\mathcal{C})}\right\}.
\]
%\int_{\Gamma} |w|^2ds \leq \left\vert \frac{\partial w}{\partial x_n}\right\vert_{L^2(\mathcal{C})} + \left\vert w\right\vert_{L^2(\mathcal{C})}.
%\]
\end{lemma}
\begin{proof}
For $w \in \mathcal{D}(\mathcal{C})=\{v|_{\mathcal{C}}: v \in C^{\infty}_0(\mathbb{R}^n)\}$ and $x_n >f(\tilde x)$
\[
w(\tilde x, f(\tilde x)) = -\int_{f(\tilde x)}^{x_n} \frac{\partial w(\tilde x, y_n)}{\partial y_n}dy_n + w(\tilde x,x_n).
\]
Thus
\begin{eqnarray*}
|w(\tilde x, f(\tilde x))|^2 &&\leq 2\left\{\left|\int_{f(\tilde x)}^{x_n}\frac{\partial w(\tilde x, y_n)}{\partial y_n}dy_n\right|^2 + |w(\tilde x,x_n)|^2\right\}\\&& 
\leq  2\left\{[x_n - f(\tilde x)]\int_{f(\tilde x)}^{x_n}\left|\frac{\partial w(\tilde x, y_n)}{\partial y_n}\right|^2dy_n + |w(\tilde x,x_n)|^2\right\},
\end{eqnarray*}
so that
\begin{eqnarray*}
&&\int_{\mathbb{R}^{n-1}}\int_{f(\tilde x)}^{f(\tilde x) + \epsilon} |w(\tilde x, f(\tilde x))|^2 \sqrt{ 1+ | \nabla_{\tilde x} f(\tilde x)}|^2 dx_n d\tilde x \\&&\leq 2\sqrt{1+L^2}\left\{ \epsilon\int_{\mathbb{R}^{n-1}}\int_{f(\tilde x)}^{f(\tilde x) + \epsilon}\int_{f(\tilde x)}^{x_n}\left|\frac{\partial w(\tilde x, y_n)}{\partial y_n}\right|^2dy_n dx_nd\tilde x\right.\\&&\left.+\int_{\mathbb{R}^{n-1}}\int_{f(\tilde x)}^{f(\tilde x) + \epsilon} |w(\tilde x,x_n)|^2dx_nd\tilde x\right\},
\end{eqnarray*}
and finally so that 
\[
\epsilon \int_{\Gamma} |w|^2ds \leq 2\sqrt{1+L^2}\left\{\epsilon^2\left\Vert \frac{\partial w}{\partial x_n}\right\Vert^2_{L^2(\mathcal{C}_+)} + \Vert w\Vert^2_{L^2(\mathcal{C}_+)}\right\},
\]
where $\mathcal{C}_+:= \{(\tilde x,x_n)|x_n \in [f(\tilde x),f(\tilde x) +\epsilon]\}$.

By arguing identically in the region below $\Gamma$, 
$\mathcal{C}_-: =\{(\tilde x,x_n)|x_n \in [f(\tilde x)-\epsilon,f(\tilde x) ]\}$, one obtains the necessary bound, for all $w \in \mathcal{D}(\mathcal{C})$: Since this space is dense in $H^1(\mathcal{C})$ the result holds for all $w$ in this space.
\end{proof}

\end{appendix}

%\bibliography{mathbib}

\begin{thebibliography}{10}

\bibitem{Abr65}
{\sc M.~Abramowitz and I.A. Stegun}, {\em Handbook of Mathematical Functions},
  Dover, New York, 1965.

\bibitem{adamsSS}
{\sc R.~A. Adams}, {\em Sobolev Spaces}, Academic Press, New York, 1975.

\bibitem{Anar}
{\sc I.~E. Anar and G.~Torun}, {\em The electromagnetic scattering problem: the case of TE polarized waves}, Appl. Math. Comp., 169, (2005), 339-354.

\bibitem{arenshcwI}
{\sc T.~ Arens, S.~N. Chandler-Wilde and K.~O. Haseloh}, {\em Solvability and Spectral Properties of Integral Equations on the Real Line: {I}. {W}eighted Spaces of continuous functions}, J. Math. Anal. Appl., 2002, 272, 276-302.

\bibitem{arenshcwII}
{\sc T.~ Arens, S.~N. Chandler-Wilde, K.~O. Haseloh}, {\em Solvability and Spectral Properties of Integral Equations on the Real Line: {II}. {$L^p$} Spaces and Applications}, J. Int. Equ. Appl.,
2003, 15, 1-35.

\bibitem{arensexist}
{\sc T.~Arens}, {\em Existence of solution in elastic wave scattering by
  unbounded rough surfaces}, Math. Meth. Appl. Sci., 25 (2002), pp.~507--528.

%\bibitem{Are03}
%{\sc T.~ Arens, S~N. Chandler-Wilde, and K.~O. Haseloh}, {\em Solvability
%  and spectral properties of integral equations on the real line. {II}. {$L\sp
%  p$}-spaces and applications}, J. Integral Equations Appl., 15 (2003),
%  pp.~1--35.

\bibitem{are04a}
{\sc T.~Arens and T.~Hohage}, {\em On radiation conditions for rough surface scattering problems}, IMA J. Appl. Math., 70:839-847, 2005. 

\bibitem{Bonnetmmas94}
{\sc A.~S. Bonnet-Bendhia and P.~ Starling}, {\em Guided Waves by Electromagnetic Gratings and Non-Uniqueness Examples for The Diffraction Problem}, Math. Methods in Appl. Sciences, 1994, 17,
305--338.

\bibitem{brakham65}
{\sc H.~Brakhage and P.~Werner}, {\em {\"U}ber das {D}irichletsche
  {A}u{\ss}enraumproblem f{\"u}r die {H}elmholtzsche {S}chwingungsgleichung},
  Arch. Math., 16 (1965), pp.~325--329.

\bibitem{chandima93}
{\sc S.~N. Chandler-Wilde}, {\em Some uniform stability and convergence results for integral
equations on the real line and projection methods for their solution}, IMA J. Num. Anal., 1993,
13, 509--535.

\bibitem{chandip95}
{\sc S.~ N. Chandler-Wilde and C.~R. Ross}, {\em Uniqueness results for direct and inverse scattering by infinite surfaces in a lossy medium}, Inverse Problems, 1995, 10, 1063-1067.

\bibitem{chandmmas96}
{\sc S.~N. Chandler-Wilde and C.~R. Ross}, {\em Scattering by rough surfaces: the {Dirichlet} problem for the {Helmholtz} equation in a non-locally perturbed half-plane}, Math. Meth. Appl. Sci., 1996, 19, 959--976.

\bibitem{chandmmas97}
{\sc S.~N. Chandler-Wilde}, {\em The Impedance Boundary Value Problem for the Helmholtz Equation in a Half-Plane} Math. Methods Appl. Sci., Vol. 20, 813-840 (1997).

\bibitem{chandprs99}
{\sc S.~N. Chandler-Wilde, C.~R. Ross and B.~Zhang}, {\em Scattering by infinite one dimensional rough surfaces}, Proc. R. Soc. Lond, Ser. A, 455 (1999) pp.~3767--3787.  

\bibitem{Cha02b}
{\sc S.~N. Chandler-Wilde and A.~T. Peplow}, {\em A boundary integral equation
  formulation for the {H}elmholtz equation in a locally perturbed half-plane},
  Z. Ang. Math. Mech., 85 (2005), pp.~79--88.


\bibitem{chandsjam98}
{\sc S.~N. Chandler-Wilde and B.~ Zhang}, {\em A uniqueness result for scattering by infinite rough surfaces}, SIAM J. Appl. Math., 1998, 58, 1774-1790.


%\bibitem{Cha95a}
%{\sc S.~N. Chandler-Wilde and C.~R. Ross}, {\em Uniqueness results for direct
%  and inverse scattering by infinite surfaces in a lossy medium}, Inverse
%  Problems, 11 (1995), pp.~1063--1067.

%\bibitem{}
%\leavevmode\vrule height 2pt depth -1.6pt width 23pt, {\em Scattering by rough
%  surfaces: the {D}irichlet problem for the {H}elmholtz equation in a
%  non-locally perturbed half-plane}, Math. Methods Appl. Sci., 19 (1996),
%  pp.~959--976. 

\bibitem{chandrs04}
{\sc S.~N. Chandler-Wilde, S.~ Langdon, L.~ Ritter}, {\em A high-wavenumber boundary-element method for an acoustic scattering problem}, Phil. Trans. R. Soc., Series A, 362, 647--671, 2004.

\bibitem{chandlangdon}
{\sc S.~N. Chandler-Wilde, S.~Langdon}, {\em A Galerkin Boundary Element Method For High Frequency Scattering by Convex Polygons}, submitted for publication

\bibitem{Cha99a}
{\sc S.~N. Chandler-Wilde, C.~R. Ross, and B.~Zhang}, {\em Scattering by
  infinite one-dimensional rough surfaces}, Proc. R. Soc. Lond. A Math., 455 (1999), pp.~3767--3787.

\bibitem{chandprs98}
{\sc S.~N. Chandler-Wilde and B.~Zhang}, {\em Electromagnetic scattering by an
  inhomogeneous conducting or dielectric layer on a perfectly conducting plate},
  Proc. R. Soc. Lon. A, 454 (1998), pp.~519--542.

%\bibitem{Cha98a}
%\leavevmode\vrule height 2pt depth -1.6pt width 23pt, {\em A uniqueness result
%  for scattering by infinite rough surfaces}, SIAM J. Appl. Math., 58 (1998),
%  pp.~1774--1790 (electronic).

%\bibitem{chandsjma99}
%\leavevmode\vrule height 2pt depth -1.6pt width 23pt, {\em Scattering of
%  electromagnetic waves by rough interfaces and inhomogeneous layers}, SIAM J.
%  Math. Anal., 30 (1999), pp.~559--583.

\bibitem{chandunfinished}
%\leavevmode\vrule height 2pt depth -1.6pt width 23pt, 
{\sc S.~N. Chandler-Wilde and B.~Zhang}, {\em A generalised
  collectively compact operator theory with an application to second kind
  integral equations on unbounded domains}, J. Integral Equat. Appl., 14
  (2002), pp.~11--52.


\bibitem{chandpottheim}
{\sc S.~N. Chandler-Wilde, E.~Heinemeyer, R.~Potthast}, {\em Acoustic Scattering by mildly rough unbounded surfaces in three dimensions}, SIAM Journal on Applied Mathematics 66 (2006), 1002-1026.  


\bibitem{chandpottheim2}
{\sc S.~N. Chandler-Wilde, E.~Heinemeyer, R.~Potthast}, {\em A well-posed integral equation formulation for 3D rough surface scattering}, 2006 Published online in Proceedings of the Royal Society of London, Series A DOI:10.1098/rspa.2006.1752.  

\bibitem{chandjmaa00}
{\sc S.~N. Chandler-Wilde, B.~Zhang, and C.~R. Ross}, {\em On the solvability
  of second kind integral equations on the real line}, J. Math. Anal. Appl.,
  245 (2000), pp.~28--51.

\bibitem{chandmonk}
{\sc S.~N. Chandler-Wilde and P.~Monk},
{\em Existence, Uniqueness and Variational Methods for scattering by unbounded rough surfaces}, SIAM J. Math. Anal. 37, 598-618, 2005.

%\bibitem{chandzamm05}
%{\sc S.~N. Chandler-Wilde A.~T. Peplow}, {\em A boundary integral equation formulation for the Helmholtz equation in a locally perturbed half-plane}, Zeitschrift f\"ur Angewandte Mathematik und Mechanik, 2005, 2, 79--88.
 
\bibitem{chand_zhang}
{\sc S.~N. Chandler-Wilde and B.~Zhang}, {\em Scattering of electromagnetic waves by rough interfaces and inhomogeneous layers}, Siam. J. Math. Anal. Vol. 30, No. 3, pp. 559-583. 
  
\bibitem{chandmonk2}
{\sc S.~N. Chandler-Wilde and P.~Monk}, {\em Wave-number explicit bounds in time-harmonic scattering}, Submitted for publication 2006.

\bibitem{Simon's_thesis}
{\sc S.~N. Chandler-Wilde}, {\em Ground effects in environmental sound propagation}, Phd thesis, University of Bradford, U.K. 1988.

\bibitem{SPM}
{\sc S.~N. Chandler-Wilde, P.~Monk}, {\em Hadamard's Criterion for a Boundary Value Problem in Layer Scattering}, J. Comp. Appl. Math., published online 2006. 

%\bibitem{Chew04}
%{\sc W.~C. Chew, J.~M. Song, T.~J. Cui, S.~Velarnparambil, M.~L. Hastriter, and
%  B.~Hu}, {\em Review of large scale computing in electromagnetics with fast
%  integral equation solvers}, CMES-Comp. Model. Eng., 5 (2004), pp.~361--372.

\bibitem{HH}
{\sc X.~Claeys and H.~Haddar}, {\em Scattering from infinite rough tubular surfaces}, submitted to J. Math. Meth. Appl. Sci. 

\bibitem{Col98b}
{\sc D.~Colton and R.~Kress}, {\em Inverse acoustic and electromagnetic
  scattering theory}, Springer-Verlag, Berlin, 2nd~ed., 1998.

\bibitem{Col83a}
{\sc D.~L. Colton and R.~Kress}, {\em Integral equation methods in
  scattering theory}, John Wiley,
  New York, 1983.

\bibitem{cumfeng}
{\sc P.~ Cummings and X.~ Feng}, {\em Sharp Regularity Coefficient Estimates for Complex-Valued Acoustic and Elastic Helmholtz Equations}, Mathematical Models and Methods in Applied Sciences, 16, (2006), no.1, 139-160. 

\bibitem{DeSanto02}
{\sc J.~A. DeSanto}, {\em Scattering by rough surfaces}, in Scattering:
  Scattering and Inverse Scattering in Pure and Applied Science,
  R.~Pike,
  P.~Sabatier, eds., Academic Press, 2002, pp.~15--36.

\bibitem{desanto98}
{\sc J.~A. DeSanto and P.~A. Martin}, {\em On the derivation of boundary
  integral equations for scattering by an infinite two-dimensional rough
  surface}, J. Math. Phys., 39 (1998), pp.~894--912.

\bibitem{dobfried92}
{\sc D.~Dobson and A.~Friedman}, {\em The time harmonic {M}axwell equations in
  a doubly-periodic structure}, J. Math. Anal. Appl., 166 (1992), pp.~507--528.

\bibitem{Ned_1}
{\sc M.~Dur\'an, I.~Muga and J.-C.~N\'ed\'elec}, {\em The Helmholtz equation with impedance in a half plane}, C.R. Acad. Sci. Paris, Ser. I340 (2005), 483-488.

\bibitem{Ned_2}
{\sc M.~Dur\'an, I.~Muga and J.-C.~N\'ed\'elec}, {\em The Helmholtz equation with impedance in a half space}, C.R. Acad. Sci. Paris, Ser. I341 (2005), 561-566.

\bibitem{eidvin}
{\sc D.~M. Eidus and A.~A. Vinnik}, {\em On radiation conditions for domains with infinite boundaries}, Soviet Math. Dokl., 1974, 15, 12--15.

\bibitem{elfguerin}
{\sc T.~M. Elfouhaily and C.~Gu\'erin}, {\em A Critical survey of approximate scattering wave 
theories from random rough surfaces}, Waves Random Media 14 (2004) R1-R40.

\bibitem{els02}
{\sc J.~ Elschner and M.~ Yamamoto}, {\em An inverse problem in periodic diffractive optics: reconstruction of {L}ipschitz grating profiles}, Appl. Analysis, 2002, 81, 1307--28.

\bibitem{Erd54}
{\sc A.~Erd\'elyi et~al.}, {\em Tables of Integral Transforms (Part II)},
  McGraw--Hill, London, 1954.

\bibitem{Evans}
{\sc L.~C. Evans}, {\em Partial Differential Equations}, Graduate studies in mathematics, Vol. 19, Amer. Math. Soc., Providence Rhode Island, 1998.

\bibitem{Federer}
{\sc H.~Federer}, {\em Geometric Measure Theory},
Springer-Verlag, Berlin, 1969.

\bibitem{fengsheen}
{\sc X.~Feng and D.~Sheen}, {\em An Elliptic Regularity Coefficient Estimate for a problem arising from a frequency domain treatment of waves}, Trans. Amer. Math. Soc., volume 346, number 2, 1994.

\bibitem{HandS}
{\sc E.~Hewitt and K.~Stromberg}, {\em Real and Abstract Analysis}, Springer-Verlag New York, 1965.

\bibitem{ihlenburg}
{\sc F.~Ihlenburg}, {\em Finite Element Analysis of Acoustic Scattering}, volume 132 of
 {\em Applied Mathematical Sciences}, Springer, Berlin, 1998.

\bibitem{Jerison}
{\sc D.~Jerison and C.~Kenig}, {\em An identity with applications to harmonic measures}, Bull. Amer. Math. Soc., 2, 1980, 447-451.

\bibitem{kirschipmp93}
{\sc A.~ Kirsch}, {\em Diffraction by periodic structures}, Inverse Problems in Mathematical Physics, edited by P{\"a}v{\"a}rinta, L.  and Somersalo, E., Springer, Berlin, 87--102, 1993.

\bibitem{kress89}
{\sc R.~Kress}, {\em Linear Integral Equations},
  Springer, Berlin, 1989.

\bibitem{Kre00a}
{\sc R.~Kress and T.~Tran}, {\em Inverse scattering for a locally perturbed
  half-plane}, Inverse Problems, 16 (2000), pp.~1541--1559.

\bibitem{Leis}
{\sc R.~Leis}, {\em Zur Dirichtletschen Randwertaufgabe des Aussenraums der Schwingungsgleichung}, 1965, Math. Z. 90, 205-211. (doi:10.1007/BF01119203).



\bibitem{mclean00}
{\sc W.~McLean}, {\em Strongly Elliptic Systems and Boundary Integral
  Equations}, CUP, 2000.

\bibitem{Melenk}
{\sc J.~M. Melenk}, {\em On Generalized Finite Element Methods}, Phd thesis, University of Maryland, College Park, MD, 1995.

\bibitem{Wavelets}
{\sc Y.~Meyer and R.~Coifman}, {\em Wavelets: Calder\'on-Zygmund and Multilinear Operators}, Cambridge Studies in Advanced Mathematics 48, 1997


\bibitem{Meyer}
{\sc R.~R. Coifmann, A.~Mcintosh and Y.~Meyer}, {\em L'int\'egrale de Cauchy d\'efinit
 un op\'erateur born\'e sur les courbes Lipschitziennes}, Ann. of Math., 116, 1982, 361-387.

\bibitem{minskii}
{\sc V.~ Minskii}, {\em {S}ommerfeld radiation conditions for second-order differential operator in a domain with infinite border}, J. Differ. Equat., 1983, 48, 157--176.

\bibitem{Morse}
{\sc P.~M. Morse and U.~Ingard}, {\em Theoretical Acoustics}, New York: McGraw-Hill, 1968.

\bibitem{nedstar91}
{\sc J.-C. N\'ed\'elec and F.~Starling}, {\em Integral equation
methods in a
  quasi-periodic diffraction problem for the time-harmonic {M}axwell's
  equations}, SIAM J. Math. Anal., 22 (1991), pp.~1679--1701.

\bibitem{odeh}
{\sc F.~M. Odeh}, {\em Uniqueness Theorems for the Helmholtz Equation in Domains with Infinite Boundaries}, J. Math. and Mech. vol 12, No. 6, 1963.

\bibitem{ogil91}
{\sc J.~A. Ogilvy}, {\em Theory of Wave Scattering from Random Rough Surfaces},
  Adam Hilger, Bristol, 1991.

\bibitem{panich}
{\sc O.~I. Panich}, {\em On the question of the solvability of the exterior boundary-value problems for the wave equation and Maxwell's equations}, 1965, Usp.Mat. Nauk20A, 221-226. 

\bibitem{petit80}
{\sc R.~Petit}, {\em Electromagnetic Theory of Gratings}, Springer, Berlin, 1980.

\bibitem{Pederson}
{\sc G.~K. Pederson}, {\em Analysis Now}, Springer-Verlag, New York Inc., 1988.

\bibitem{Monk}
{\sc P.~Monk}, {\em Finite Element Methods for Maxwell's Equations}, Oxford Science Publications, Clarendon Press, 2003. 

\bibitem{Ree75}
{\sc M.~Reed and B.~Simon}, {\em Methods of modern mathematical physics: Part
  II Fourier Analysis, Self-Adjointness}, Academic Press, New York, 1975.

%\bibitem{Rel_1}
%{\sc F.~ Rellich}, {\em Darstellung der Eigenwerte von $\Delta u +\lambda u=0$ durch ein Randintegral}, Math. Z., Vol 46, 1940.

\bibitem{Rel_2}
{\sc F.~ Rellich}, {\em \"Uber das asymptotische Verhalten der L\"osungen von $\Delta u + \lambda u=0$ in unendlichen Gebieten}, Jber. Deutsch. Math. Verein 53, 1943, 57-65.

\bibitem{RZ_1}
{\sc G.~F. Roach and B.~Zhang}, {\em On Sommerfeld radiation conditions for the diffraction problem with two unbounded media}, Proc. R. Soc. Edin., 121A, 149-161, 1992. 

\bibitem{RZ_2}
{\sc G.~F. Roach and B.~Zhang}, {\em A transmission problem for the reduced wave equation in inhomogeneous media with an infinite interface}, Proc. R. Soc. Lond., A (1992) 436, 121-140. 

\bibitem{Ros96}
{\sc C.~Ross}, {\em Direct and Inverse Scattering by Rough Surfaces}, PhD
  thesis, Brunel Uni., 1996.

\bibitem{saillardsent01}
{\sc M.~Saillard and A.~Sentenac}, {\em Rigorous solutions for electromagnetic
  scattering from rough surfaces}, Waves Random Media, 11 (2001),
  pp.~R103--R137.

\bibitem{Szembergmmas98}
{\sc B.~ Strycharz-Szemberg}, {\em A Direct and Inverse Scattering Transmission Problem for   Periodic Inhomogeneous Media}, Math. Methods in Appl. Sciences, 1998, 21, 969-983.

\bibitem{torres93}
{\sc R.~H. Torres and G.~V. Welland}, {\em The {H}elmholtz-equation and
  transmission problems with {L}ipschitz interfaces}, Indiana U. Math. J., 42
  (1993), pp.~1457--1485.

%\bibitem{tsangitap95}
%{\sc L.~Tsang, C.~H. Chan, K.~Pak, and H.~Sangani}, {\em {Monte-Carlo}
%  simulations of large-scale problems of random rough surface scattering and
%  applications to grazing incidence with the BMIA/canonical grid method}, IEEE
%  Tran. Ant. Prop., 43 (1995), pp.~851--859.

\bibitem{Verchota}
{\sc G.~Verchota}, {\em Layer Potentials and regularity for the Dirchlet problem for Laplace's equation in Lipschitz domains}, J. Funct. Anal., 1984, 59, 572-611. (doi:10.101/0022-1236(84)90066-1).

\bibitem{vogel}
{\sc V.~Vogelsang}, {\em Das {A}usstrahlungsproblem fur elliptische {D}ifferentialgleichungen in {G}ebieten mit Unbeschranktem {R}and}, Math. Z., 1975, 144, 101--124.

\bibitem{voron94}
{\sc A.~G. Voronovich}, {\em Wave Scattering from Rough Surfaces},
2nd ed., Springer,
  Berlin, 1998.

\bibitem{warnick01}
{\sc K.~F. Warnick and W.~C. Chew}, {\em Numerical simulation methods for rough
  surface scattering}, Waves Random Media, 11 (2001), pp.~R1--R30.

\bibitem{Wilcox84}
{\sc C.~H. Wilcox}, {Scattering Theory for Diffraction Gratings}, Applied Mathematical Sciences
46, Springer-Verlag, 1984. 

\bibitem{willers87}
{\sc A.~Willers}, {\em The {H}elmholtz equation in disturbed half-spaces},
  Math. Meth. Appl. Sci., 9 (1987), pp.~312--323.

%\bibitem{xia03}
%{\sc M.~Xia, C.~H. Chan, S.~Li, B.~Zhang, and L.~Tsang}, {\em An efficient
%  algorithm for electromagnetic scattering from rough surfaces using a single
%  integral equation and multilevel sparse-matrix canonical-grid method.},
%  {IEEE} Trans. Antennas Prop., AP-51 (2003), pp.~1142--1149.

\bibitem{Z_1}
{\sc B.~Zhang}, {\em On transmission problems for wave propagation in two locally perturbed half-spaces}, Math. Proc. Camb. Phil. Soc. (1994), 115, 545. 

\bibitem{Z_2}
{\sc B.~Zhang}, {\em Commutator Estimates, Besov spaces and scattering problems for the acoustic wave propagation in perturbed stratified fluids}, Proc. R. Soc. Edin., 121A, 149-161, 1992. 

\bibitem{Z_3}
{\sc B.~Zhang}, {\em On radiation conditions for acoustic propagators in perturbed stratified fluids}.

\bibitem{zhangworking}
{\sc B.~Zhang and S.~N. Chandler-Wilde}, {\em Integral equation methods for
  scattering by infinite rough surfaces}, Math. Methods Appl. Sci., 26 (2003),
  pp.~463--488.

\end{thebibliography}
\end{document}